\documentclass[reqno,12pt,letterpaper]{amsart}
\usepackage[proof]{zlmacros}

\title[Internal waves in aquariums with characteristic corners]{Internal waves in aquariums\\ with characteristic corners}
\author{Zhenhao Li}
\address{Department of Mathematics, Massachusetts Institute of Technology, Cambridge, MA 02139}

\begin{document}

\begin{abstract}
    We give a precise microlocal description of the singular profile that forms in the long-time propagation of internal waves in an effectively two-dimensional aquarium. We allow domains with corners, such as polygons appearing in the experimental set ups of Maas et al.~\cite{maas_observation_97}. This extends the previous work of Dyatlov--Wang--Zworski~\cite{DWZ}, which considered domains with smooth boundary. We show that in addition to singularities that correspond to attractors in the underlying classical dynamics, milder singularities propagate out of the corners as well. 
\end{abstract}

\maketitle

\section{Introduction}
Let $\Omega \subset \R^2 = \{x = (x_1, x_2) \mid x_j \in \R\}$ be an open, bounded, and simply-connected domain with Lipschitz and piecewise-smooth boundary. We consider a model of internal waves in the domain $\Omega$ that describes the evolution of the stream function $u$ of a stable-stratified fluid under time-periodic forcing with profile $f \in \CIc(\Omega;\mathbb R)$. It is given by the following Poincar\'e problem:
\begin{equation}\label{eq:internal_waves}
    (\partial_t^2 \Delta + \partial_{x_2}^2) u = f(x) \cos \lambda t, \quad u|_{t = 0} = \partial_t u|_{t = 0} = 0, \quad u|_{\partial \Omega } = 0,
\end{equation}
where $\lambda \in (0, 1)$ and $\Delta := \partial_{x_1}^2 + \partial_{x_2}^2$. The equation models small perturbations of a stable-stratified, invicid, hydrostatic, and nonrotating two-dimensional Boussinesq fluid. Because the fluid is nonrotating, there exists a stream function $u$ which is related to fluid velocity by $\mathbf v = (\partial_{x_2} u, -\partial_{x_1} u)$. See for instance Dauxois et al. \cite{Dauxois_18} for a recent review of the equation in the physics literature.

In this paper, we study the limiting profile of the solution $u$ to the internal waves equation~\eqref{eq:internal_waves}. This expands upon the work of Dyatlov--Wang--Zworski in \cite{DWZ} by allowing corners in the domain $\Omega$. Mathematically, the corners introduce new singularities and change the microlocal structure of the singularities in the limiting profile. Moreover, the physical experiments conducted with fluids in a tank are always done in domains with corners. The results of this paper explain how the mathematical model exhibits strong singularities along the periodic orbit and mild singularities that stem from the corners as observed in the experiments. See Figure~\ref{fig:experiment}.
\begin{figure}
    \centering
    \includegraphics[scale = 4]{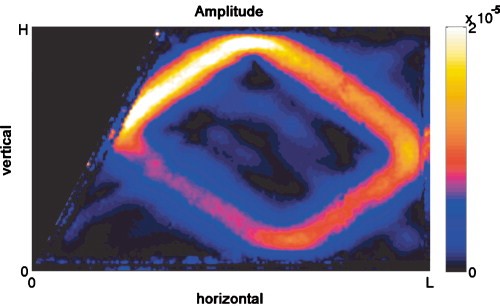}
    \caption{The experiment of Hazewinkel et al. \cite{Hazewinkel_2010} in which the long-time behavior of a stable-stratified fluid in an oscillating trapezoidal tank is observed. We see the strongest singularities along a parallelogram (which is the attractor of the dynamics)
    but also milder singularities stemming from the corners, which could be explained by the results of this paper.}
    \label{fig:experiment}
\end{figure}

To understand the evolution of the solution to ~\eqref{eq:internal_waves}, it is helpful to rearrange the fourth-order equation to the form of an evolution problem. If $\Omega$ has Lipschitz boundary, then the Dirichlet Laplacian has an inverse $\Delta_\Omega^{-1} : H^{-1}(\Omega) \to H^1_0(\Omega)$. Therefore, we can rewrite~\eqref{eq:internal_waves} as
\begin{equation}\label{eq:evolution_problem}
    (\partial_t^2 + P) w = f \cos \lambda t, \quad w|_{t = 0} = \partial_t w|_{t = 0} = 0, \quad f \in \CIc(\Omega; \R), \quad u = \Delta_\Omega^{-1} w,
\end{equation}
where the operator $P$ is given by 
\begin{equation}\label{eq:P_def}
    P := \partial_{x_2}^2 \Delta_\Omega^{-1}: H^{-1}(\Omega) \to H^{-1} (\Omega).
\end{equation}
Since $\Omega$ has Lipschitz and piecewise-smooth boundary, $P$ is non-negative, bounded, self-adjoint, and $\mathrm{Spec}(P) = [0, 1]$ on the Hilbert space $H^{-1}(\Omega)$ with the inner product 
\[\langle u, w \rangle_{H^{-1}(\Omega)} := \langle \nabla \Delta_\Omega^{-1} u, \nabla \Delta_\Omega^{-1} w \rangle_{L^2(\Omega)}.\]
See Ralston \cite[\S3]{Ralston_73} for details. Therefore, the evolution problem~\eqref{eq:evolution_problem} can be solved using the functional calculus of $P$ by 
\begin{equation}\label{eq:functional_solution}
    \begin{gathered}
        w(t) = \Re(e^{i \lambda t} \mathbf W_{t, \lambda}(P) f) \quad \text{where} \\
        \mathbf W_{t, \lambda}(z) = \int_0^t \frac{\sin(s \sqrt{z})}{\sqrt{z}} e^{-i \lambda s} \, ds = \sum_{\pm} \frac{1 - e^{-i t}(\lambda \pm \sqrt{z})}{2 \sqrt{z}(\sqrt z \pm \lambda)}.
    \end{gathered}
\end{equation}
In fact, it is easy to see that we have the distributional limit
\[\mathbf W_{t, \lambda} (z) \to (z - \lambda^2 + i0)^{-1} \quad \text{as} \quad t \to \infty \quad \in \quad \mathcal D'_z((0, \infty)).\]
Therefore, the focus of this paper will be study the spectrum of $P$ near $\lambda^2$. In particular, we will establish a limiting absorption principle for $P$ near $\lambda^2$ in Theorem~\ref{thm:spectral}, and use this to deduce the following main result using Stone's formula. Under reasonable geometric and dynamical assumptions, we can understand the long-time behavior of the solution $u$ to the internal waves equation~\eqref{eq:internal_waves} via the following decomposition.
\begin{theorem}\label{thm:evolution}
    Assume that $\Omega \subset \R^2$ has straight characteristic corners with respect to $\lambda \in (0, 1)$ and satisfies the Morse--Smale condition (see Definitions~\ref{def:MS}) . Then the solution to the internal waves equation~\eqref{eq:internal_waves} can be decomposed as 
    \begin{equation}
        u(t) = \Re\big(e^{i \lambda t} u_+ \big) + r(t) + e(t)
    \end{equation}
    where
    \begin{equation}
        u_+ := \Delta_\Omega^{-1} (P - \lambda^2 + i0)^{-1} f,
    \end{equation}
    with $(P-\lambda^2+i0)^{-1}$ defined in Theorem~\ref{thm:spectral} below, and 
    \[\|r(t)\|_{H^1_0(\Omega)} \le C \quad \text{for all} \quad t \ge 0, \qquad e(t) \to 0 \quad \text{in} \quad \bar H^{\ha - \beta}(\Omega) \]
    for any $\beta > 0$. 
\end{theorem}
While this decomposition looks similar to \cite[Theorem 1]{DWZ}, the structure of $u_+$ is very different due to the corners. The precise microlocal description of $u_+$ will be given in Theorem~\ref{thm:spectral}. Physically, Theorem~\ref{thm:evolution} means that up to a small error and a term uniformly bounded in energy space $H^1_0$ for all times, the limiting solution to the internal waves equation is given by $\Re(e^{i \lambda t} u_+)$. 

If the characteristic ratio defined in Definition \ref{def:corner_type} satisfies $|\log \alpha(\lambda, \kappa)| \le \sqrt{3} \pi$ for every corner $\kappa$, then $u_+$ is bounded in energy space away from the attractor of the underlying classical dynamics. Thus the singularities caused by the corners are much milder than the singularities caused by the underlying dynamics. In all experiments, the characteristic ratio condition is satisfied, which explains why we primarily see singularities forming along the attractor. If the characteristic ratio condition is not satisfied, we still have a characterization of the regularity of $u_+$, but it no longer lies in energy space away from the attractors (see Theorem~\ref{thm:spectral}). For a heuristic justification of why the corner leads to additional milder singularities, see the remark after Theorem~\ref{thm:spectral}. 

\subsection{Assumptions on the domain}
Throughout the paper, $\Omega \subset \R^2_{x_1, x_2}$ will always be an open, bounded, and simply-connected domain with Lipschitz and piecewise-smooth boundary. The rest of the assumptions are formulated in terms of the relevant underlying classical dynamics of the system. In view of~\eqref{eq:evolution_problem}, we wish to study operators of the form $P - \lambda^2$, which is the same as studying the second order differential operator
\begin{equation}\label{eq:P(lambda)def}
    P(\lambda):= (P - \lambda^2) \Delta_\Omega = (1 - \lambda^2) \partial_{x_2}^2 - \lambda^2 \partial_{x_1}^2, \quad \lambda \in (0, 1).
\end{equation}
For $\lambda \in (0, 1)$, observe that $P(\lambda)$ is the $1+1$ wave operator with the ``speed of light'' given by $c = \lambda/\sqrt{1 - \lambda^2}$. Furthermore, the light rays associated with $P(\lambda)$ travel along the level sets of the functions 
\begin{equation}\label{eq:dual_factor}
    \ell^\pm(x, \lambda) := \pm \frac{x_1}{\lambda} + \frac{x_2}{\sqrt{1 - \lambda^2}}.
\end{equation}
\begin{figure}
    \centering
    \includegraphics{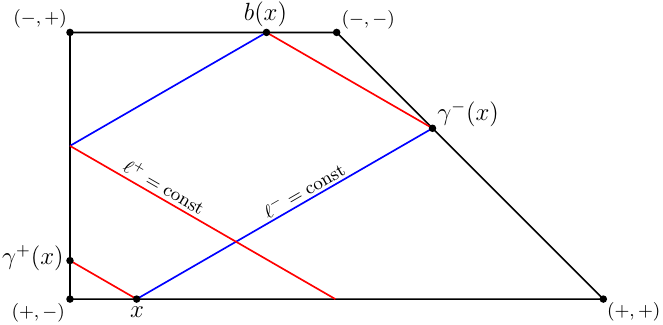}
    \caption{An example depicting a $\lambda$-simple domain with straight characteristic corners for some $\lambda \in (0, 1)$ corresponding to the level sets $\ell^\pm$. The corner types defined in Definition~\ref{def:corner_type} are labeled, as well as the involutions $\gamma^\pm$ and the chess billiard map $b$.}
    \label{fig:defs}
\end{figure}
\begin{definition}\label{def:lambda_simple}
    Let $\lambda \in (0, 1)$. We say that $\Omega$ is \textup{$\lambda$-simple with straight characteristic corners} (or just \textup{$\lambda$-simple} for short in the context of this paper) if:
    \begin{enumerate}
        \item each of the functions $\ell^\pm(\bullet, \lambda): \partial \Omega \to \R$ has a unique maximum and minimum, which are achieved at
    \begin{equation}\label{eq:characteristic_set}
        \begin{aligned}
            x_{\max}^\pm &:= \mathrm{arg\, max}_{x \in \partial \Omega} \, \ell^\pm(x, \lambda), \\
            x_{\min}^\pm &:= \mathrm{arg\, min}_{x \in \partial \Omega} \, \ell^\pm(x, \lambda).
        \end{aligned}
    \end{equation}
    The four points must be distinct and $\ell^\pm(x^\mp_{\max}) \neq \ell^\pm(x^\mp_{\min})$. 
    \item $\partial \Omega$ is smooth except possibly at $x^\pm_{\max}$ and $x^\pm_{\min}$, and $\ell^\pm(x, \lambda): \partial \Omega \to \R$ has no critical points except possibly at $x^\pm_{\max}$ and $x^\pm_{\min}$.
    \item In the case that $\partial \Omega$ is smooth near $x^\pm_{\max}$ or $x^\pm_{\min}$, they must be nondegenerate critical points of $\ell^\pm(\bullet, \lambda): \partial \Omega \to \R$.

    \item In the case that $\partial \Omega$ is not smooth near $\kappa \in \{x^\pm_{\max}, x^\pm_{\min}\}$, $\partial \Omega$ admits a local parameterization in the form
    \[\theta \mapsto \begin{cases} \kappa + \theta \mathbf v_1 & \theta > 0, \\ \kappa - \theta \mathbf v_2 & \theta < 0\end{cases}\]
    for sufficiently small $|\theta|$ and linearly independent $\mathbf v_1, \mathbf v_2 \in \R^2$. Such points are called \textup{corners} and will be denoted by the set $\corner$. 
    \end{enumerate} 
\end{definition}
We believe that the results of this paper should still hold without (4). That is, the boundary need not be ``straight'' near the corners. However, assuming straight corners significantly simplifies the computations of the Schwartz kernel of the restricted boundary reduced operator in \S\ref{sec:kernel_computation}. With curved boundaries near the corner, these computations become bulkier and we need to consider lower order terms in the b-calculus, which the machinery in this paper can handle, but does not add mathematical insight. Experiments have also been conducted in domains with straight corners. 

The behavior of the internal waves equation near corners is very different from the behavior near other extrema of $\ell^\pm$ on $\partial \Omega$. It will be useful to further categorize the corners as follows. 
\begin{definition}\label{def:corner_type}
    A characteristic corner $\kappa \subset \partial \Omega$ is a \textup{type-$(\mu, \nu)$ corner} for $\mu, \nu \in \{+, -\}$ if and only if
    \[\mu\cdot \ell^{-\nu}(x - \kappa, \lambda) \ge 0\]
    for all $x \in \Omega$. 

    In particular, near a type-$(\mu, \nu)$ corner $\kappa$, there exists a neighborhood of $U$ of $\kappa$ and $\alpha_\pm(\lambda, \kappa) > 0$ such that
    \begin{multline}
        U \cap \partial \Omega = \{x \in \R^2: \ell^{\nu}(x, \lambda) = - \alpha_-(\lambda, \kappa) \ell^{-\nu}(x, \lambda),\ \mu \cdot \ell^{\nu}(x, \lambda) \in (-\delta, 0] \} \\
        \cup \{x \in \R^2: \ell^{\nu} = \alpha_+(\lambda, \kappa)\ell^{-\nu}(x, \lambda), \ \mu \cdot \ell^{\nu} (x, \lambda) \in [0, \delta) \}
    \end{multline}
    for some small $\delta > 0$. Furthermore, we define
    \[\alpha(\lambda, \kappa) = \alpha_\kappa(\lambda) := \alpha_+(\lambda, \kappa)/\alpha_-(\lambda, \kappa)\]
    as the \textup{characteristic ratio} of the corner $\kappa$. 
\end{definition}
See Figure~\ref{fig:defs} for a clear diagram of the corner types, since the signs in Definition~\ref{def:corner_type} above can be confusing. We will often omit the dependence on $\lambda$ and/or $\kappa$ in the notation for $\alpha(\lambda, \kappa)$ and $\alpha_\pm(\lambda, \kappa)$ when the context is clear.

\begin{figure}
    \centering
    \includegraphics[scale = 0.75]{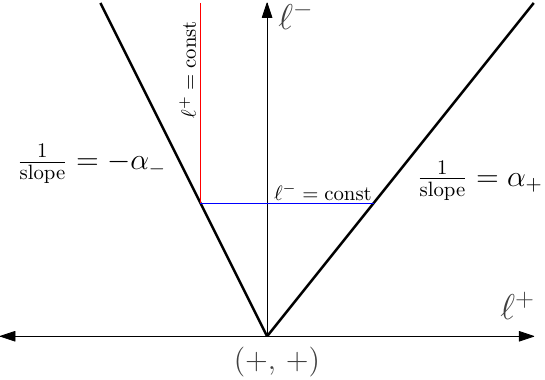}
    \caption{A diagram of a neighborhood of a type $(+, +)$ corner in $(\ell^+, \ell^-)$ coordinates. In this picture, the level sets of $\ell^+$ are vertical and the level sets of $\ell^-$ are horizontal. The corner is fixed by $\gamma^-$.}
    \label{fig:alpha_def}
\end{figure}

\begin{Remarks}
    1. The upshot of this definition is so that we can take advantage of some reflection symmetries. Let $\mathrm{Ref}_1(x_1, x_2) = (-x_1, x_2)$ and $\mathrm{Ref}_2(x_1, x_2) = (x_1, -x_2)$. Then
    \begin{itemize}
        \item if $\kappa$ is a type-$(+, -)$ corner of $\Omega$, then $\mathrm{Ref}_1(x)$ is a type-$(+, +)$ corner of $\mathrm{Ref}_1(\Omega)$.
        \item if $\kappa$ is a type-$(-, -)$ corner of $\Omega$, then $\mathrm{Ref}_2(x)$ is a type-$(+, +)$ corner of $\mathrm{Ref}_2(\Omega)$.
        \item if $\kappa$ is a type-$(-, +)$ corner of $\Omega$, then $\mathrm{Ref}_1 \circ \mathrm{Ref}_2(x)$ is a type-$(+, +)$ corner of $\mathrm{Ref}_1 \circ \mathrm{Ref}_2(\Omega)$.
    \end{itemize}
    Furthermore, the characteristic ratio $\alpha(\lambda)$ as well as the $\alpha_\pm(\lambda)$ parameters remain invariant under $\mathrm{Ref}_j$. Therefore, it generally suffices to study the local behavior near a type-$(+, +)$ corner. 

    \noindent
    2. In a neighborhood of a type-$(\mu, \nu)$ corner $\kappa$, we also have the convenient choice of ``blown-up'' coordinates given by 
    \begin{equation}\label{eq:good_corner_coords}
        r = \mu \cdot \ell^\nu(\bullet - \kappa, \lambda), \quad \tau = \frac{\ell^{-\nu}(\bullet - \kappa, \lambda)}{\ell^\nu(\bullet - \kappa, \lambda)}, \quad (r, \tau) \in [0, \delta)_r \times [-\alpha_-, \alpha_+]_\tau.
    \end{equation}
    These are basically polar coordinates centered at a corner that resolves the singularity at the corner in some sense. 
\end{Remarks}

Assume that $\Omega$ is $\lambda$-simple. Then there exist unique, continuous, and orientation-reversing involutions $\gamma^\pm(\bullet, \lambda) = \gamma^\pm_\lambda = \gamma^\pm: \partial \Omega \to \partial \Omega$ that satisfy
\begin{equation}\label{eq:gamma_def}
    \ell^\pm(x) = \ell^\pm(\gamma^\pm(x)).
\end{equation}
Observe that $\gamma^\pm$ fixes $x^\pm_{\mathrm{min}}$ and $x^\pm_{\mathrm{max}}$ and exchanges the two points at which level curves of $\ell^\pm(\bullet, \lambda)$ intersects the boundary. Now we can define the \textit{chess billiard map} $b(\bullet, \lambda) = b_\lambda = b: \partial \Omega \to \partial \Omega$ by the composition
\begin{equation}\label{eq:b_def}
    b:= \gamma^+ \circ \gamma^-.
\end{equation}
$b$ is a continuous orientation preserving homeomorphism. See Figure~\ref{fig:defs} for an example of $\gamma^\pm$ and $b$. 

Denote the set of periodic points of $b$ by
\begin{equation}
    \Sigma_\lambda := \{x \in \partial \Omega: b^n(x) = x\ \text{for some $n \ge 1$}\}.
\end{equation}
\begin{definition}\label{def:MS}
    Let $\lambda \in (0, 1)$. Then $\Omega$ is \textup{Morse--Smale with straight characteristic corners with respect to $\lambda$} (or just Morse--Smale in the context of this paper) if
    \begin{enumerate}
        \item $\Omega$ is $\lambda$-simple (see Definition~\ref{def:lambda_simple}),
        \item $\Sigma_\lambda \neq \emptyset$ and $\Sigma_\lambda \cap \corner = \varnothing$, which implies that $b$ is smooth in a neighborhood of $\Sigma_\lambda$,
        \item $\Sigma_\lambda$ consists of hyperbolic periodic points, that is $ \partial_x b^{n}(x, \lambda) \neq 1$ for all $x \in \Sigma_\lambda$ where $n$ is the minimal period. 
    \end{enumerate}
    Furthermore, we denote the attracting and repulsive periodic points by 
    \begin{equation}\label{eq:periodic_points}
        \Sigma_\lambda^\pm = \{x \in \Sigma_\lambda: \pm \log (\partial_x b^n(x, \lambda)) < 0\}
    \end{equation}
    respectively. 
\end{definition}
We note that the above definition is a slight abuse of language, because what we really mean that the associated chess-billiard map $b_\lambda$ is Morse--Smale.

\subsection{Spectral result}
Now we describe $u_+$ from Theorem~\ref{thm:evolution}. To describe the singularities, we need to keep track of the trajectories coming out of the corners. Define
\begin{equation}
    \begin{aligned}
        \mathscr O_\lambda^+ :=& \{x \in \partial \Omega: x = b_\lambda^k(\gamma^+_\lambda(\kappa)), \ k \in \N_0, \ \text{$\kappa$ is type-$(\pm, +)$}\} \\
        &\quad \cup \{x \in \partial \Omega: x = b_\lambda^{k}(\kappa), \ k \in \N, \ \text{$\kappa$ is type-$(\pm, -)$}\}, \\
        \mathscr O_\lambda^- :=& \{x \in \partial \Omega: x = b_\lambda^{-k}(\kappa), \ k \in \N, \ \text{$\kappa$ is type-$(\pm, +)$}\} \\
        &\quad \cup \{x \in \partial \Omega: x = b_\lambda^{-k}(\gamma^-_\lambda\kappa), \ k \in \N_0, \ \text{$\kappa$ is type-$(\pm, -)$}\}.
    \end{aligned}
\end{equation}
Note that the limit points of $\mathscr O^\pm_\lambda$ are precisely $\Sigma^\pm_\lambda$, and we emphasize that this means that for all sufficiently small $\delta > 0$,
\begin{equation}
    \mathscr O^\pm_\lambda \cap \Sigma^\mp_\lambda(\delta) = \emptyset.
\end{equation}
Furthermore, $\mathscr O^\pm_\lambda \setminus \Sigma^\pm_\lambda(\delta)$ is a finite set for all $\delta > 0$. 
\begin{figure}
    \centering
    \includegraphics{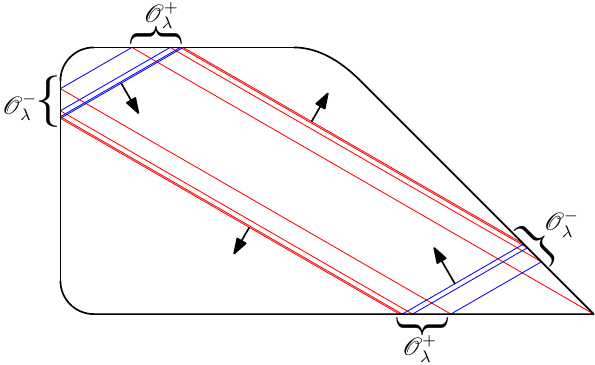}
    \caption{Diagram of the trajectories coming out of the corner $\mathscr O^\pm_\lambda$. The arrows indicate the conormal directions contained in $\Lambda^-$.}
    \label{fig:pm_orbits}
\end{figure}

We will describe the singularities of the solution to the internal waves equation~\eqref{eq:internal_waves} via their wavefront sets (see \eqref{eq:wavefront_def}), which is a subset of the cotangent bundle $T^* \Omega$. Roughly speaking, the singularities will lie on the conormal bundle of the trajectories coming out of the corners and on the conormal bundle of the periodic cycle.

For $x \in \partial \Omega$, define
\begin{equation}
    \Gamma^\pm_\lambda(x) := \{u \in \Omega: \ell^\pm(y, \lambda) = \ell^\pm(x, \lambda)\}. 
\end{equation}
The periodic trajectory is given by 
\begin{equation}
    \Gamma_\lambda := \bigcup_{x \in \Sigma^+_\lambda} (\Gamma^+_\lambda(x) \cup \Gamma^-_\lambda(x))
\end{equation}
The conormal bundle of $\Gamma^\pm_\lambda(y)$ can be split into positive and negative frequencies using the orientation on $\partial \Omega$. In particular, fix a positively-oriented b-parameterization $\mathbf x: \mathbb S^1 \to \partial \Omega$. We define for $\mu \in \{+, -\}$ the positive and negative conormal bundles
\begin{equation}\label{eq:N^*_pm}
    N^*_\mu \Gamma^\pm_{\lambda}(x) := \{(y, \mu \tau \partial_\theta \ell^\pm(\mathbf x(\theta))|_{\theta = \mathbf x^{-1}(x)} d\ell^\pm(y, \lambda)) : y \in \Gamma^\pm_\lambda(x),\ \tau > 0\}. 
\end{equation}
Note that 
\[N^*_\pm \Gamma^\pm_\lambda(y) = N_\mp^* \Gamma^\pm_\lambda( \gamma_{\lambda}^\pm(y)).\]
If $\mathcal K \neq \emptyset$, define
\begin{equation}\label{eq:lagrangians}
    \Lambda^\pm(\lambda) := \overline{\bigcup_{x \in \mathscr O^+_\lambda} N^*_\pm \Gamma^-_\lambda(x)  \cup \bigcup_{x \in \mathscr O^-_\lambda} N^*_\mp \Gamma^+_\lambda(x)}.
\end{equation}
On the other hand, if $\mathcal K = \emptyset$, put 
\[\Lambda^\pm(\lambda) = \bigcup_{x \in \Sigma^+_\lambda} N^*_\pm \Gamma^+_\lambda(x)  \cup \bigcup_{x \in \Sigma^-_\lambda} N^*_\mp \Gamma^-_\lambda(x),\]
in which case the notation is the same as in \cite{DWZ}. Observe that $\Lambda^\pm(\lambda)$ does not intersect the conormal bundle of $\Gamma^\pm_\lambda(\kappa)$ for any $\kappa \in \mathcal K$. We separately define these special rays and denote them by
\begin{equation}\label{eq:sr_def}
    \sr_\lambda = \sr := \bigcup_{\substack{\kappa \in \mathcal K, \\ \nu \in \{+, -\}}} \Gamma^\nu_\lambda(\kappa).
\end{equation}
The corresponding conormal bundle is then denoted by $N^*(\sr)$. Microlocally, the singularity of the limiting singular profile in the limiting absorption principle lives on the conormal bundle of $\sr$ as well as $\Lambda^\pm$ depending on if the spectrum is approached from the upper or lower half complex plane. 

\begin{theorem}\label{thm:spectral}
    Let $\mathcal J \subset (0, 1)$ be such that $\Omega$ is Morse--Smale with respect to every $\lambda \in \mathcal J$. Then the spectrum of $P$ is absolutely continuous in $\mathcal J^2 = \{\lambda^2: \lambda \in \mathcal J\}$. The limits
    \begin{equation}
        (P - \lambda^2 \pm i0)^{-1} f  = \lim_{\epsilon \to 0+} (P - (\lambda \mp i\epsilon)^2) f  \quad \text{in} \quad \mathcal D'(\Omega)
    \end{equation}
    exist and 
    \begin{equation}\label{eq:lim_regularity}
        (P - \lambda^2 \pm i0)^{-1} f \in \bar H^{-\frac{3}{2} - }(\Omega).
    \end{equation}
    Furthermore, 
    \begin{equation}\label{eq:wavefront}
        \WF((P - \lambda^2 \pm i0)^{-1} f) \subset \Lambda^\pm(\lambda) \cup N^*(\sr).
    \end{equation}
    For every $\chi \in C^\infty(\overline \Omega)$ such that $\supp \chi \cap \Gamma_\lambda = \emptyset$, 
    \begin{equation}\label{eq:limit_conorm_reg}
        \chi \cdot (P - \lambda^2 \pm i0)^{-1} f \in \mathcal A^{[s]}(\widetilde \Omega; \Lambda^\pm(\lambda), \sr, \ff)
    \end{equation}
    where $s < 0$ and 
    \begin{equation}
        s < \min_{\kappa \in \mathcal K} \Re\Big(\mathfrak l_{\lambda, \kappa} - \frac{3}{2}\Big), \qquad \mathfrak l_{\lambda, \kappa} := \frac{2 \pi i}{i \pi - \log \alpha(\lambda, \kappa)}.
    \end{equation}
    Here, $\widetilde \Omega$ is the blown-up domain defined in~\eqref{eq:Omega_tilde_set}, and $\mathcal A^{[s]}$ are conormal spaces defined in~\eqref{eq:conormal_def}.
\end{theorem}

\begin{Remarks}
1. To see why additional singularities are expected to appear as a result of the corners, consider the equation 
\[P(\lambda) u = f, \qquad f \in \CIc(\Omega)\]
where $P(\lambda)$ is defined in~\eqref{eq:P(lambda)def}. The relevant solutions are the limit of $H^1_0(\Omega)$ solutions to $P(\omega)u = f$ as $\omega \to \lambda$ from the upper half plane, and these must take the form 
\begin{equation}\label{eq:solution_near_corner}
    u(r, \tau) \sim \sum_{k = 1}^\infty c_k r^{\mathfrak l_{\lambda, \kappa} k} \Big[(\tau + i0)^{\mathfrak l_{\lambda, \kappa} k} - \alpha_+^{\mathfrak l_{\lambda, \kappa} k} \Big]
\end{equation}
where $(r, \tau)$ are the blown-up coordinates defined in~\eqref{eq:good_corner_coords}, and the exponentials are defined using the branch of log on $\C \setminus i(-\infty, 0)$. We will provide rigorous justification for this expansion in Proposition~\ref{prop:polyhom_EU}, but it is easy to see that these solutions are essentially found by separation of variables near the corner, and one can directly check that $w$ satisfy the boundary condition $w(r, - \alpha_-) = w(r, \alpha_+) = 0$. Furthermore, $u$ has a singularity on the special ray $sr = \{\tau = 0\}$. Heuristically, we see that $u_+$ from Theorem~\ref{thm:evolution} should look like~\eqref{eq:solution_near_corner} near the corner, and the limits $(P - \lambda^2 - i0)^{-1}$ should look like $\Delta u_+$ near the corner. This provides some justification for the numerology appearing in Theorem~\ref{thm:spectral}.

\noindent
2. We see that if the characteristic ratio satisfies $(\log \alpha_\kappa)^2 < 3 \pi^2$ for every $\kappa \in \mathcal K$, then
\[\chi \cdot (P - \lambda^2 \pm i0)^{-1} f \in H^{-1}(\Omega)\]
for every $\chi \in C^\infty(\overline \Omega)$ such that $\supp \chi \cap \Gamma_\lambda = \emptyset$. Numerically, this is roughly $\frac{1}{230.75} < \alpha_\kappa(\lambda) < 230.76$. In experimental and numerical setups, the characteristic ratio is well within that range, so away from the periodic trajectories, $u_+ = \Delta_\Omega^{-1}(P - \lambda^2 \pm i0)^{-1} f$ lies in the energy space $H^1(\Omega)$. Therefore, it is difficult to visually observe these singularities compared to the much stronger singularities near the attractor.
\end{Remarks}

\subsection{Related works}
The formation of singularities along the periodic trajectory of the chess-billiard map was first predicted in the physics literature by Maas--Lam \cite{maas_lam_95}. This has since been experimentally observed by by Maas et al. \cite{maas_observation_97}, Hazewinkel et al. \cite{Hazewinkel_2010}, Brouzet \cite{brouzet_16}, and many more. The same equation~\eqref{eq:internal_waves} also has other physical applications. In particular, it also describes inertial waves in a rotating fluid stratified in angular momentum, see Maas \cite{Maas2001}.  
Predating the experimental physics work, the  Poincar\'e problem and has been studied in the math literature by John \cite{John_41}, Aleksandryan \cite{Aleksandryan_60}, and Ralston \cite{Ralston_73}. More recently, Colin~de~Verdi\`ere--Saint-Raymond \cite{Verdiere_Raymond_20} and Dyatlov--Zworski \cite{Dyatlov_Zworski_19} considered a model of the internal waves equation on surfaces without boundary by considering certain 0-th order pseudodifferential operators on surfaces. They proved, using different methods, that singularities form along certain dynamical attractors associated with the system. The viscosity limits of these operators on surfaces were then studied by Galkowski--Zworski \cite{Galkowski_Zworksi_22} and Wang \cite{Wang_22}, and spectral properties of 0-th order pseudodifferential operators were studied by Zhongkao Tao \cite{Tao_19}. The result for planar domains with smooth boundary was later proved in \cite{DWZ}. Furthermore, in the absence of hyperbolic attractors in the underlying dynamics, Colin~de~Verdi\`ere--Li proved in \cite{ZLC} that the solutions remain bounded in energy space if the underlying dynamics is strongly ergodic.

\subsection{Organization of the paper}
We summarize the necessary properties of Morse--Smale dynamics in \S\ref{sec:dynamical_setup} and give examples of domains that satisfy the hypothesis of our paper. In \S\ref{sec:microlocal}, we outline the necessary microlocal tools from the classical calculus and the b-calculus. The main result of \S\ref{sec:microlocal} is Proposition \ref{prop:full_b-est}, which is a b-elliptic estimate for pseudodifferential operators in the full calculus. This will allow us to propagate singularities through the corner. 

In \S\ref{sec:reduction_to_boundary}, we reduce the spectral problem of Theorem~\ref{thm:spectral} to a problem on the boundary. This is given by the restricted boundary reduced operator $d\mathcal C_\omega$, and we explicitly compute the Schwartz kernel of this operator. In \S\ref{sec:propagation_estimates}, we describe how singularities are propagated by $d\mathcal C_\omega$. In particular, this describes how the corner creates singularities and how they propagate. Ultimately, the propagation estimates gives us a global semi-Fredholm estimate for $d \mathcal C_\omega$, which is needed in establishing the limiting absorption principle in \S\ref{sec:LAP}. The other ingredient needed for the limiting absorption principle is uniqueness of a limiting problem, which is given in the beginning of \S\ref{sec:LAP}. Finally, in \S\ref{sec:evolution}, we deduce the main evolution result of Theorem~\ref{thm:evolution} using Stone's formula.

\section{Dynamical setup}\label{sec:dynamical_setup}
Since $\Omega$ is a domain with corners, $\partial \Omega$ is only piecewise-smooth and Lipschitz. There two useful perspectives that we will take when considering the boundary. First, $\partial \Omega$ is homeomorphic to $\mathbb S^1$, and smoothly so away from the corners. Second, we can disconnect $\partial \Omega$ at the corners and view $\partial \Omega$ as the union of embedded (or immersed if there is only one corner) one-dimensional manifolds with boundary, the smooth structure of which are induced from the embedding $\partial \Omega \to \R^2$. It is then sensible to use parameterizations of $\partial \Omega$ that preserves both perspectives.
\begin{definition}
    A \textup{b-parameterization} of $\partial \Omega$ is a bi-Lipschitz positively-oriented homeomorphism
    \[\theta \in \mathbb S^1 = \R/\Z \quad \mapsto \quad \mathbf x(\theta) \in \partial \Omega \subset \R^2\]
    such that $\mathbf x$ is smooth on $\mathbb S^1 \setminus \mathbf x^{-1}(\mathcal K)$. 
\end{definition}
Note that such parameterizations preserve the smooth structure near the endpoints of the one-dimensional manifolds. For example, the arclength parameterization of $\partial \Omega$ is a b-parameterization with the appropriate normalization. 

\subsection{Morse--Smale dynamics}
In the following lemma, we show that under the Morse--Smale condition, we can always find a b-parameterization so that the chess billiard map near periodic points is locally affine linear. This is possible since the periodic points are hyperbolic. 
\begin{lemma}\label{lem:linear_coords}
    Let $\Omega$ be Morse--Smale with respect to $\lambda \in (0, 1)$ and let $\Sigma^\pm_\lambda$ be the attracting/repulsive periodic points. There exists a b-parameterization $\mathbf x$ such that $\partial_\theta \gamma^\pm(\mathbf x(\theta))$ is locally constant in a small neighborhood of $\mathbf x^{-1}(\Sigma_{\lambda})$, and $\pm \log( \partial_\theta b(\mathbf x(\theta), \lambda)) < 0$ in a small neighborhood of $\mathbf x^{-1}(\Sigma_{\lambda}^\pm)$ respectively. 
\end{lemma}
\begin{proof}
    We drop $\lambda$ from the notation. Let $x_0 \in \Sigma^+_\lambda$ with period $n$. It suffices to construct $\mathbf x$ near the orbit of $x_0$ and $\gamma^-(x_0)$. It follows from the local linearization lemma in \cite[Lemma 6.2]{DWZ} that there exists a b-parameterization $\tilde{\mathbf x}$ such that that $\tilde{\mathbf x}(0) = x_0$ and $b^n(\tilde{\mathbf x}(\theta)) =: c\theta$ for some $0 < c < 1$, and $\theta \in (-\delta, \delta)$, $\delta > 0$ sufficiently small.  

    Let $\theta_k^+ = \tilde{\mathbf x}^{-1}(b^k(x_0))$, $k = 0, \dots, n - 1$ denote the $k$-th point in the orbit of $0$ lifted to $\mathbb S^1$. Let $\theta_k^- = \gamma^-(\theta)$
    \begin{align*}
        &\mathbf x(\theta) := b^k(\widetilde{\mathbf x}(c^{-\frac{1}{n}}(\theta - \theta_k^+))) \quad \text{for} \quad \theta \in (\theta_k^+ - \delta, \theta_k^+ + \delta), \\
        &\mathbf x(\theta) := \gamma^{-}(\mathbf x(\theta_k^- - \theta + \theta_k^+)) \quad \text{for} \quad \theta \in (\theta_k^- - \delta, \theta_k^- + \delta)
    \end{align*}
    Upon possibly $\delta > 0$, this is well defined in $\delta$-neighborhoods of $\theta_k^\pm$ and extends to all of $\mathbb S^1$ to a b-parameterization. If there are multiple periodic orbits in $\Sigma^+$, this construction simply be repeated since it is completely local. Such $\mathbf x$ fits the description of the lemma.  
\end{proof}

The most important property of the Morse--Smale condition is that it is an open condition in $\lambda$. This is one of the key components to showing that the spectral measure of $P$ is absolutely continuous (in fact H\"older continuous) near $\lambda^2$. Since the periodic points depend smoothly on $\lambda$ and are away from the corners, the same result from \cite{DWZ} still holds with a similar proof. 
\begin{lemma}
    The set of $\lambda \in (0, 1)$ for which $\Omega$ is Morse--Smale with respect to $\lambda$ is open, and the periodic set $\Sigma_\lambda$ depends smoothly on $\lambda$ as long as $\Omega$ is Morse--Smale with respect to $\lambda$. 
\end{lemma}
\begin{proof}
    Fix $\lambda_0 \in (0, 1)$ such that $\Omega$ is Morse--Smale with respect to $\lambda_0$. We can identify $\partial \Omega$ with $\mathbb S^1$ via a b-parameterization. Since $\ell^\pm(x, \lambda)$ depends smoothly on $\lambda$, we see that $b^n(\theta, \lambda): \mathbb S^1_\theta \to \mathbb S^1$ depends smoothly on $\lambda$ near $\lambda_0$, and is piece-wise smooth and Lipschitz in $\theta$. Furthermore, $b^n(\theta, \lambda)$ is smooth in both $\lambda$ near $\lambda_0$ and $\theta$ near the periodic points $\Sigma_\lambda$. The lemma then follows by the implicit function theorem since $\partial_\theta b^n(\theta, \lambda_0) \neq 1$ for $\theta \in \Sigma_\lambda$. 
\end{proof}

We refer the interested reader to de Melo--van Strien \cite[Chapter 1]{Melo_Strien} for a more general discussion of one parameter families of circle homeomorphisms.

\subsection{Examples and nonexamples}
We give two examples of domains that fits into the scope of this paper, as well as an interesting and physically relevant corner case that does not fit into the context of this paper. 
\begin{figure}
    \centering
    \includegraphics{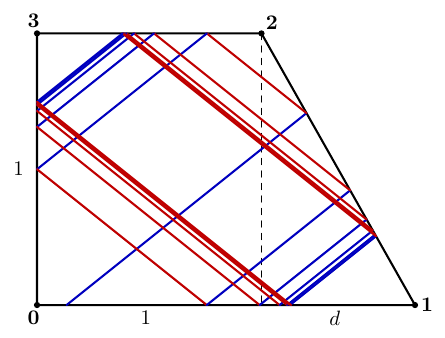}
    \hspace{5mm}
    \includegraphics{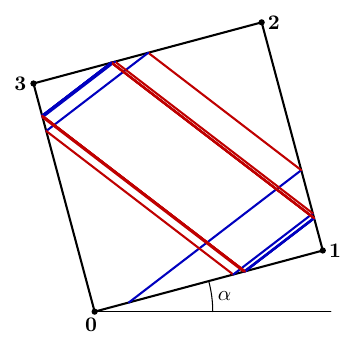}
    \caption{Examples of domains with Morse--Smale dynamics. In the left, we have the trapezoid $\mathcal T_{1, d}$ and on the right we have the tilted square $Q_\alpha$. Graphics courtesy of \cite{DWZ}.}
    \label{fig:domains}
\end{figure}

\subsubsection{Trapezoids}
The most important example that is covered by this paper is the trapazoidal domain. These are the domains often found in physics experiments such as \cite{maas_observation_97}, \cite{Hazewinkel_2010}, and \cite{brouzet_16}. Let $\mathcal T_{a, b} \subset \R^2$ be the open trapezoid with vertices $(0, 0)$, $(a + b, 0)$, $(a, 1)$, $(0, 1)$, $a, b > 0$. See Figure~\ref{fig:domains}. It is easy to see that $\mathcal T_d$ is $\lambda$-simple when 
\begin{equation}
    \lambda \in I_b := \left(\frac{b}{\sqrt{ 1 + b^2}}, 1 \right)
\end{equation}
It was shown by Lenci--Bonanno--Cristador \cite{Lenci_22} that $\Omega$ is Morse--Smale with respect to a full measure set of $\lambda \in I_b$. In other words, trapezoidal domains satisfy the hypothesis of Theorem~\ref{thm:evolution} and \ref{thm:spectral} for a generic choice of $\lambda$.

\subsubsection{Tilted squares}
For $\alpha \in (0, \pi/2)$, define $Q_\alpha \subset \mathcal \R$ as the square with vertices $(0, 0)$, $(\cos \alpha, \sin \alpha)$, $\sqrt{2}( \cos(\alpha + \frac{\pi}{4}), \sin(\alpha + \frac{\pi}{4}))$, $\cos(\alpha + \frac{\pi}{2}, \sin (\alpha + \frac{\pi}{2}))$. See Figure~\ref{fig:domains}. It is shown in \cite[\S 2.5]{DWZ} by direct computation that if 
\[\frac{\sqrt{1 - \lambda^2}}{\lambda} = \tan \beta, \quad \text{for} \quad 0 < \alpha < \pi/8, \quad \frac{\pi}{4} - \alpha < \beta < \frac{\pi}{4} + \alpha,\]
then $Q_{\alpha}$ is Morse--Smale with respect to $\lambda$. 

We mention that when $\alpha = 0$, that is for the untilted square, $\Omega$ is \textit{not} Morse--Smale for any $\lambda \in (0, 1)$. In this case, the chess-billiard map is generically ergodic, and it can be shown via Fourier series that the solution to~\eqref{eq:internal_waves} remains bounded in energy space $H^1$ for all times for a generic choice of $\lambda$. See \cite[\S 4.1]{ZLC} for details.

\subsubsection{Exotic corners}
There are several corner cases that do not fit into the context of this paper but are still interesting. We mention here a case that significantly affects the dynamics. In a domain with corner, it is possible for a corner to be an extremum of both $\ell^+_\lambda$ and $\ell^-_\lambda$ at the same time. See Figure~\ref{fig:exotic_corner}. Note that this is clearly not possible if the boundary is smooth. In this case, the chess-billiard map $b$ is still well-defined, but has a corner as a fixed point. 
\begin{figure}
    \centering
    \includegraphics[scale = 0.8]{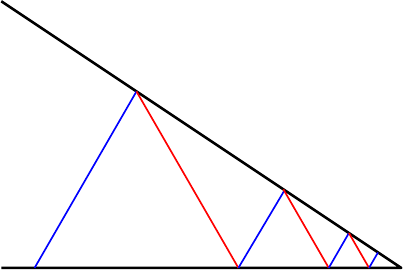}
    \caption{A domain with corner where the corner is a fixed point of the chess billiard map. Such corners do not fit into the scope of this paper.}
    \label{fig:exotic_corner}
\end{figure}
In particular, all triangular domains have a vertex that is a fixed point of the chess-billiard map. In all trapezoidal domains, such corners also appear for all $\lambda \in (0, 1)$ sufficiently small. These types of corners are referred to as \textit{subcritical corners} in the physics literature. See Wunsch \cite{wunsch}.

\section{Microlocal and b-calculus preliminaries}\label{sec:microlocal}
In this section, we give an overview of some standard tools from microlocal analysis needed in the proof. The basics of the Kohn--Nirenberg calculus on $\mathbb S^1$ are established in \S\ref{sec:S1microlocal}, which will be needed to analyze the boundary reduced equation away from corners. We will need the b-calculus for analysis near the corners, and the basics are estabished in \S\ref{sec:small_b_calc}-\ref{sec:boundary_terms}. In particular, we prove a version of the b-elliptic estimates for operators in the full calculus (see \S\ref{sec:boundary_terms}). Generally in the literature, elliptic estimates are proved for b-differential operators, which belong to the small calculus since their Schwartz kernels are supported on the diagonal and vanish near the boundary. However, the operators that appear in this paper have Schwartz kernels that do not vanish near the boundaries, so we will need to consider boundary terms more carefully. 

\subsection{Microlocal analysis on \texorpdfstring{$\mathbb S^1$}{S1}}\label{sec:S1microlocal}
Via a b-parameterization, $\partial \Omega$ can be identified with $\mathbb S^1 = \R/\Z$. We emphasize here that $\mathbb S^1$ and $\partial \Omega$ clearly do not have the same smooth structure. Nevertheless, it is very helpful to consider pseudodifferential operators on $\mathbb S^1$ since the operators we will encounter are often of the form of a pseudodifferential operator composed with a Heaviside cutoff. For a thorough introduction to the Kohn--Nirenberg calculus of pseudodifferential operators, we refer the reader to \cite[Chapter 18.1]{H3} for details. We give the exposition needed for this paper below. Define set of 1-periodic Kohn--Nirenberg symbols $S^m(T^* \mathbb S^1)$ as the set of $a \in C^\infty(\R_x \times \R_\xi)$ that satisfies
\begin{equation}\label{eq:KN_seminorm}
    a(x + 1, \xi) = a(x, \xi), \qquad |\partial_x^\alpha \partial_\xi^\beta a(x, \xi)| \le C_{\alpha, \beta} \langle \xi \rangle^{m - \beta},
\end{equation}
where $\langle \xi \rangle := \sqrt{1 + |\xi|^2}$. Since we only consider symbols on $T^* \mathbb S^1$, there is no ambiguity when we denote $S^m = S^m(T^* \mathbb S^1)$. Define the residual class of symbols
\begin{equation*}
    S^{-\infty} = \bigcap_{m \in \R} S^m
\end{equation*}
Pseudodifferential operators are obtained by quantizing symbols. We use the standard quantization procedure and define for $a \in S^m$ the operator $\Op(a): C^\infty(\mathbb S^1) \to C^\infty(\mathbb S^1)$ by 
\begin{equation}\label{eq:std_quant}
    \Op(a) u(x) = \frac{1}{2 \pi} \int_{\R^2} e^{i(x - y) \xi} a(x, \xi) u(y)\, dy d \xi.
\end{equation}
In the integral, $u \in C^\infty (\mathbb S^1)$ is understood as a 1-periodic function, and the integral is understood in the sense of oscillatory integrals (see \cite[\S1]{GS_94}). $\Op(a)$ also extends to an operator from $\mathcal D'(\mathbb S^1) \to \mathcal D'(\mathbb S^1)$. One can verify that this quantization can be equivalently given in Fourier series by
\begin{equation}\label{eq:op_fourier_series}
    \begin{gathered}
        \Op(a) u(x) = \sum_{k, n \in \Z} e^{2 \pi i n x} a_{n - k} (k) \hat u(k), \\
        a_\ell(k) := \int_0^1 a(x, 2 \pi k) e^{-2 \pi i \ell x}\, dx, \quad \hat u(k) := \int_0^1 u(x) e^{2 \pi i kx} \, dx.
    \end{gathered}
\end{equation}
From this representation, it is clear that the symbol of an operator is not unique, and that when $a \in S^m$ function of $\xi$ only, $\Op(a)$ is a Fourier multiplier:
\[\Op(a)u(x) = \sum_{k \in \Z} e^{2 \pi i nx} a(2 \pi k) \hat u(k).\]

The pseudodifferential operators of order $m$ on $\mathbb S^1$ are then defined as 
\begin{equation}\label{eq:KN_class}
    \Psi^m(\mathbb S^1) := \{\Op(a): a \in S^m\}
\end{equation}
for $m \in \R$ and $m = \infty$. Note that $\Psi^{-\infty}(\mathbb S^1) = \bigcap_{m \in \R} \Psi^s(\mathbb S^1)$, and it follows from~\eqref{eq:op_fourier_series} that
\begin{equation}\label{eq:smoothing}
    \Psi^{-\infty}(\mathbb S^1) = \{R: C^\infty(\mathbb S^1) \to \mathcal D'(\mathbb S^1) :\text{the Schwartz kernel lies in $C^\infty(\mathbb S^1 \times \mathbb S^1)$}\}.
\end{equation}
Therefore, $R \in \Psi^{-\infty}(\mathbb S^1)$ extends to an operator from $\mathcal D'(\mathbb S^1) \to C^\infty( \mathbb S^1)$.

For a (possibly $\omega$-dependent) operator $A = \Op(a) \in \Psi^m(\mathbb S^1)$ for some $a \in S^m$, we define the principal symbol as
\begin{equation}
    \sigma_m(A) = [a] \in S^m/S^{m - 1}.
\end{equation}
It is clear from \eqref{eq:op_fourier_series} that the full symbol $a$ is not unique. On the other hand, the principal symbol is canonically defined. In fact, we have a short exact sequence
\[0 \rightarrow \Psi^{m - 1} \rightarrow \Psi^{m} \xrightarrow{\sigma_m} S^m/S^{m - 1} \to 0. \]
In our notation, we will often drop the brackets around the principal symbol, and it will be implicitely understood as an element of the quotient class. In particular, we write $\sigma_m(A) = b$ for any $b$ that satisfies $a - b \in S^m - 1$ uniformly in $\omega$. 

The operators we encounter will almost always be $\omega$-dependent families of operators, and we write $A = A_\omega \in \Psi^m(\mathbb S^1)$ when the symbols of these operators are given by $a = a_\omega \in S^m(T^* \mathbb S^1)$ and satisfy \ref{eq:KN_seminorm} uniformly in the parameter $\omega$. Then $\omega$ dependence will often be suppressed in the notation. 

Next we characterize the composition of pseudodifferential operators: let $a \in S^m$ and $b \in S^\ell$ then
\begin{equation}\label{eq:classical_comp}
    \begin{gathered}
        \Op(a) \Op(b) \in \Psi^{m + \ell}(T^* \mathbb S^1), \quad \Op(a) \Op(b) = \Op(a \# b), \\
        \text{where} \quad a \# b = e^{i D_\xi D_y} a(x, \xi) b(y, \eta) \big|_{\substack{y = x \\ \eta = \xi}} \in S^{m + \ell}.
    \end{gathered}
\end{equation}
Here, we use the H\"ormander convention of $D_x := -i\partial_x$. In other words, we have the asymptotic summation formula
\begin{equation}\label{eq:comp_expansion}
    a \# b \sim \sum_{k = 0}^\infty \frac{i^k}{k!} D_\xi^k a(x, \xi) D_x^k b(x, \xi),
\end{equation}
which is understood as
\[a\# b - \sum_{k = 0}^{N - 1} \frac{i^k}{k!} D_\xi^k a(x, \xi) D_x^k b(x, \xi) \in S^{m + \ell - N}.\]

Next, the definition of pseudodifferential operators do not depend on the choice of coordinates on $\mathbb S^1$. We have the following change of variables property. 
\begin{lemma}\label{lem:CoV}
    Let $a \in S^m$ and let $\psi: \mathbb S^1 \to \mathbb S^1$ be a diffeomorphism. Then there exists $\tilde a \in S^m$ such that 
    \begin{equation}
        \psi^* \Op(a) (\psi^{-1})^* = \Op (\tilde a).
    \end{equation}
    Furthermore, $\tilde a$ has asymptotic expansion 
    \begin{equation}\label{eq:change_expansion}
        \tilde a \sim \sum_{k = 0}^\infty L_k(a \circ \tilde \psi),
    \end{equation}
    where $L_k$ are differential operators of order $2k$ on $T^* \mathbb S^1$ and map $S^m \to S^{m - k}$, and $\widetilde \psi$ is the lifted symplectomorphism of $\psi$, i.e. 
    \[\widetilde \psi: T^* \mathbb S^1 \to T^* \mathbb S^1, \quad \widetilde \psi(x, \xi) := (\psi(x), (d \psi(x))^{-1} \xi).\]
    Futhermore, $L_0 = 1$, and $L_k = L'_k D_\xi$ for some differential operator $L'_k$ for every $k \ge 1$. 
\end{lemma}

We will need to characterize how projections onto positive (or negative) frequencies behave under commutation and change of variable. Define the Fourier projections
\begin{equation}\label{eq:pm_projections}
    \Pi^\pm = \Op(\chi^\pm)
\end{equation}
where $\chi_\pm \in C^\infty(\T^* \mathbb S^1)$ are such that 
\begin{gather*}
    \supp \chi^+ \subset \{\xi > 0\}, \quad \chi^+ = 1 \quad \text{on} \quad \{\xi \ge 1\} \\
    \supp \chi^- \subset \{\xi < 1\}, \quad \chi^+ = 1 \quad \text{on} \quad \{\xi \le 0\}.
\end{gather*}
These are chosen so that $\Pi^+ + \Pi^- = \id$. Then the following lemma is an immediate consequence of~\eqref{eq:comp_expansion} and~\eqref{eq:change_expansion}
\begin{lemma}\label{lem:Heaviside_identities}
    For any $\omega$-dependent families of $A \in \Psi^{m}$ and orientation-reversing diffeomorphism $\psi:\mathbb S^1 \to \mathbb S^1$, we have
    \[[\Pi^\pm, A] \in \Psi^{-\infty} \quad \text{and} \quad \psi^* \Pi^\pm (\psi^{-1})^* = \Pi^\mp + \Psi^{-\infty}.\] 
\end{lemma}

Despite the nonuniqueness of the full symbol, it still determines a notion of essential support for pseudodifferential operators. Define the wavefront set by
\begin{equation}\label{eq:wavefront_def}
    \WF(A) := \{(x, \xi) : \xi \neq 0,\ \exists \rho > 0: a(y, \eta) = \mathcal O( \langle \eta \rangle^{-\infty}) \ \text{when} \ |x - y| < \rho, \tfrac{\eta}{\xi} > 0\}^C.
\end{equation}
The principal symbol determines where an operator is elliptic. Define the elliptic set by 
\begin{equation*}
    \Ell(A) := \{(x, \xi) : \xi \neq 0, \exists c, \rho > 0: \sigma_m(A)(y, \eta) \ge c \langle \eta \rangle^m \ \text{when}\ |x - y| < \rho, \tfrac{\eta}{\xi} > 0\}.
\end{equation*}
We give a version of elliptic regularity for pseudodifferential operators here. 
\begin{lemma}\label{lem:elliptic}
    Let $P \in \Psi^m(\mathbb S^1)$ and $A, B \in \Psi^0(\mathbb S^1)$ be such that 
    \[\WF(A) \subset \Ell(P) \cap \WF(B).\] 
    Then there exists $Q  \in \Psi^{-m}(\mathbb S^1)$ and $R \in \Psi^{-\infty}(\mathbb S^1)$ such that 
    \begin{equation}
        A = QBP + R, \qquad \sigma_{-m}(Q) = \frac{\sigma_0(A)}{\sigma_0(B) \sigma_m(P)} 
    \end{equation}
    In particular, we have the estimate
    \[\|Au\|_{H^s} \le C_{s, N}\big(\|BPu\|_{H^{s - m}} + \|u\|_{H^{-N}} \big).\]
    for all $s, N \in \R$. 
\end{lemma}
Finally, we will need the fact that the norm bound on pseudodifferential operators is determined by the sup-norm of the principal symbol up to low frequency effects. 
\begin{lemma}\label{lem:symbol_bound}
    Let $A \in \Psi^0(\mathbb S^1)$ and let 
    \[M := \sup_{x \in \mathbb S^1} \limsup_{\xi \to \infty} \sigma_0(A)(x, \xi).\]
    Then for every $N, \delta > 0$, there exists $C > 0$ such that
    \[\|A u\|_{L^2} \le (M + \delta) \|u\|_{L^2} + C\|u\|_{H^{-N}}.\]
\end{lemma}
Observe that this lemma also gives us explicit control over the constant in the elliptic estimates in Lemma~\ref{lem:elliptic}. We refer the reader to \cite[Chapter 4]{GS_94} for proofs of Lemmas~\ref{lem:elliptic} and~\ref{lem:symbol_bound}. We also remark that the conclusions of both lemmas hold uniformly in $\omega$ for $\omega$-dependent operators that satisfy the hypothesis of the lemmas uniformly.

\subsection{Small b-calculus}\label{sec:small_b_calc}
We develop the small b-calculus on the positive real line $\R_+ := (0, \infty)$. First, we define two important spaces of distributions on $\R_+$. Let $F \subseteq \mathcal D'(\R)$ be a subspace of distributions on the real line. The \textit{supported distributions} is the subspace space $\dot F(\R_+) \subseteq F$ consisting of distributions supported on $\overline \R_+$. The \textit{extendible distributions} is defined to be $\overline F(\R_+) = F/\dot F(\R_-)$. For instance, $\overline C_{\mathrm c}^\infty(\R_+)$ consists of compactly supported functions smooth up to $0$ and indeed equal to the space $C_\mathrm{c}^\infty(\overline \R_+)$. On the other hand, $\dot C_\mathrm{c}^\infty(\R_+)$ consists of smooth and compactly supported functions that vanishes to infinite order at $0$. Furthermore, one sees that $\dot{\mathcal D}'(\R_+)$ is dual to $\overline C_{\mathrm c}^\infty(\R_+)$, and $\overline{\mathcal D}'(\R_+)$ is dual to $\dot C_\mathrm{c}^\infty(\R_+)$. See \cite[B.2]{H3} for a thorough treatment of these spaces.

We wish to define b-operators in such a way that extends differential operators on $\R_+$ of the form 
\[P = \sum_{k = 0}^m c_k(x) (x D_x)^k, \qquad c_k \in \overline C^\infty(\R_+).\]
In particular, we are interested in the behavior of such operators near $0$, so we might as well assume that the coefficients $c_k$ vanishes to infinite order at infinity. Now for their pseudodifferential counterpart, define $S_\bo^m = S_\bo^m(\R_+)$ as the set of $p \in C^\infty(\overline \R_+ \times \R)$ such that for all $\alpha, \beta, N \in \N$, 
\begin{gather}
    |\partial_\xi^\alpha \partial_x^\beta p(x, \xi)| \le C_{\alpha, \beta, N} \langle \xi \rangle^{m - \alpha} \langle x \rangle^{-N} \label{eq:b-symbol_est}\\
    \text{and} \quad \int_\R e^{i t \xi} p(x, \xi) \, d \xi = 0 \quad \text{if $t \ge 1$ and $x \ge 0$} \label{eq:lacunary}
\end{gather}
where $\langle \xi \rangle = \sqrt{1 - |\xi|^2}$. Also put $S^{-\infty}_\bo := \bigcap_{m \in \R} S^{-\infty}_\bo$. For $p \in S^m_\bo$, define the operator $\Op_\bo(p): \dot{\mathscr S}^\infty(\R_+) \to \dot{\mathscr S}^\infty(\R_+)$ by 
\begin{equation}\label{eq:b-quantization}
    \Op_\bo(p) u(x) = \frac{1}{2 \pi} \int_\R \int_{\R_+} e^{i(x - y) \xi}p(x, x\xi) u(y)\, dy d \xi.
\end{equation}
The small b-pseudodifferential operators of order $m$ are denoted by 
\begin{equation}\label{eq:b-calc}
    \Psi^m_\bo(\R_+) := \{\Op_\bo(p): p \in S^m_+\}, \qquad m \in \R \quad \text{or} \quad -\infty.
\end{equation}
The Schwartz kernel $\Op_\bo(p)$ are extendible distributions on the quarter space $\R_+^2$ and is given by 
\begin{equation}\label{eq:b-kernel}
    K(x, y) = \frac{1}{2 \pi} \int_\R e^{i(x - y) \xi} p(x, x \xi)\, d \xi = \frac{1}{2 \pi} x^{-1} \int_{\R} e^{i(1 - \frac{y}{x}) \xi} p(x, \xi)\, d \xi.
\end{equation}
By~\eqref{eq:b-symbol_est}, we see that $K$ vanishes to infinite order as $\frac{y}{x} \to \infty$. Due to~\eqref{eq:lacunary}, $K$ also vanishes to infinite order as $\frac{y}{x} \to 0$. In order to obtain a better description of the Schwartz kernel, it makes sense to blow up the corner of the quarter space $\overline \R_+^2$. We blowup this quarter space at the origin to the space
\begin{equation}\label{eq:R2_blowup}
    \widetilde \R_+^2 := [0, \infty)_r \times [-1, 1]_\tau
\end{equation}
together with a blow-down map $\beta:\widetilde \R_+^2 \to \overline \R_+^2$ by 
\begin{equation}\label{eq:R2_blowdown}
    \beta(r, \tau) = \left( \frac{r(1 + \tau)}{2}, \frac{r(1 - \tau)}{2} \right). 
\end{equation}
See Figure~\ref{fig:quarter_blowup}. In other words, we are essentially introducing ``polar'' coordinates on $\overline \R_+^2$ by 
\begin{equation}\label{eq:quarter_bcoords}
    \tau = \frac{x - y}{x + y}, \quad r = x + y.
\end{equation}
This blow up construction will also serve as the local model for blowing up planar domains with corners later. For a more canonical and general construction of blow ups of double spaces, see \cite[\S4.2]{Melrose93} for details, but the coordinate description suffices for now. In our coordinates, the boundary hypersurface $[-1, 1] \times \{0\} \subset \widetilde \R_+^2$ is called the \textit{front face} of the blow up, or $\ff$ for short. $\{-1\} \times [0, \infty)$ is called the left boundary (lb) and the $\{1\} \times [0, \infty)$ is called the right boundary (rb). See Figure~\ref{fig:quarter_blowup}.

\begin{figure}
    \centering
    \includegraphics[scale = 0.7]{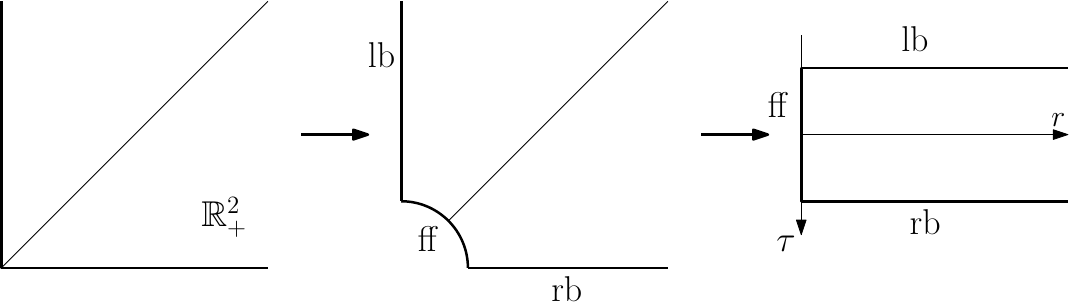}
    \caption{The second and third pictures are diagrams of $\widetilde \R^2_+$ with lb, rb, and ff labeled. The $(r, \tau)$ coordinates are labeled in the third. The diagonal is depicted in each picture.}
    \label{fig:quarter_blowup}
\end{figure}

From~\eqref{eq:b-kernel}, we immediately see that if $K$ is the Schwartz kernel of an element of $\Psi^m_\bo(\R_+)$, then $r(1 - \tau)K(r, \tau)$ is smooth in a neighborhood of lb and rb, and vanishes to infinite order on lb and rb. This is the reason such operators are called ``small'' b-pseudodifferential operators. On the symbol side, vanishing on lb is given by the symbol estimates~\eqref{eq:b-symbol_est} and vanishing on rb is given by~\eqref{eq:lacunary}, which is sometimes called the Lacunary condition (see \cite[\S 18.3]{H3}). As such, we have the following characterization of $\Psi^{-\infty}_\bo(\R_+)$. 

\begin{proposition}
    $K \in L^1_{\mathrm{loc}}(\R^2)$ with $\supp K \subset \overline \R_+^2$ is the Schwartz kernel of $\Op_\bo(p)$ for some $p \in S^{-\infty}_\bo$ if and only if 
    \[F(r, \tau) = \frac{1}{2} r(1 - \tau) K(r, \tau)\]
    is in $C^\infty([\overline \R_+^2; \{0\}])$, vanishes to infinite order on $\lb$ and $\rb$, and for all $\alpha, \beta, N > 0$, we have 
    \begin{equation}
        |\partial_\tau^\alpha \partial_r^\beta F(r, \tau)| \le C_{\alpha, \beta, N} \langle r \rangle^{-N}. 
    \end{equation}
\end{proposition}
See \cite[Theorem 18.3.6]{H3} for a detailed proof. The same product formula from the standard microlocal calculus still holds as well.
\begin{lemma}\label{lem:b-composition}
    If $p_j \in S^{m_j}_\bo$, $j = 1, 2$, then $\Op_\bo(p_1) \Op_\bo(p_2) = \Op(q)$ where $q \in S^{m_1 + m_2}_\bo$ is given by 
    \begin{equation}
        q(x, \xi) = e^{i D_y D_\eta} p_1(x, \eta) p_2(xy, \xi y) \big|_{y = 1, \eta = \xi}.
    \end{equation}
    In particular, $\Op_\bo(p_1 p_2) - \Op_\bo(p_1) \Op_\bo(p_2) \in \Psi^{m_1 + m_2 - 1}_\bo(\R_+)$. 
\end{lemma}
See \cite[Theorem 18.3.11]{H3} for the proof. The upshot of this is that in the small b-calculus, the symbolic elliptic parametrix construction still works. We define uniform ellipticity in the same way. $P = \Op_\bo(p) \in \Psi^{m}_\bo(\R_+)$ is uniformly elliptic in a relatively open set $U \subset \overline \R_+$ if there exists $c> 0$ such that 
\[|a(x, \xi)| \ge c |\xi|^m \quad \text{for} \quad x \in U, \ |\xi| > c^{-1}.\]
\begin{proposition}\label{prop:small_parametrix}
    Let $P \in \Psi^{m}_\bo(\R_+)$ be uniformly elliptic in a relatively open set $U \subset \overline \R_+$, and let $\chi \in \CIc(\overline \R_+)$ be such that $\supp \chi \subset U$. Then there exists $Q \in \Psi^{-m}_\bo(\R_+)$ such that 
    \[P \circ Q - \chi \in \Psi^{-\infty}_\bo(\R_+) \quad \text{and} \quad Q \circ P - \chi \in \Psi^{-\infty}_\bo(\R_+).\]
\end{proposition}

Again, because the small b-calculus is symbolic, the proof of $L^2$ boundedness of 0th-order operators due to H\"ormander still holds. In particular, if $P \in \Psi^0_\bo(\R_+)$, one can use the symbolic calculus to construct an approximate square root to $M - P^*P$ for a sufficiently large $M > 0$, from which one deduces $L^2$-boundedness. See \cite[Theorem 18.3.12]{H3} for details, and we state it below for convenience. 
\begin{lemma}\label{lem:b-L2}
    If $P \in \Psi^0_\bo(\R_+)$, then $P: L^2(\R_+) \to L^2(\R_+)$ is bounded. 
\end{lemma}
We emphasize here that $L^2(\R_+) = L^2(\R_+; dx)$ where $dx$ is the Lebesgue measure on $\R_+$. Also note that there's no need to specify the extendible and supported $L^2$ spaces since they are identical. 

\subsection{Mellin transform and b-Sobolev spaces}
The standard $L^2$ based Sobolev spaces can be defined using the Fourier transform, and thus behave well with respect the translation-invariant differential operators $D_x^k$. However, b-operators on $\R_+$ are generated by differential operators of the form $(xD_x)^k$, which are dilation-invariant. We will define b-Sobolev spaces via the Mellin transform, which is a dilation invariant version of the Fourier transform. For $u \in \dot C_{\mathrm{c}}^\infty (\R_+)$ the Mellin transform is given by 
\begin{equation}\label{eq:mellin}
    \mathcal M u(\sigma) := \int_0^\infty x^{-i \sigma} u(x) \frac{dx}{x}.
\end{equation}
Observe that the Mellin transform reduces to the Fourier transform in logarithmic coordinates $t = \log x$ by
\begin{equation}
    \mathcal M u(\sigma) = \int_{-\infty}^\infty e^{-it\sigma} u(e^t) \, dt.
\end{equation}
It is easy to see from the Paley--Weiner theorem that $\mathcal M u(\sigma)$ is entire for $u \in \dot C^\infty_{\mathrm{c}}(\R_+)$. Furthermore, if $u \in \overline{\mathcal E}'(\R_+)$, then $\mathcal M u(\sigma)$ is holomorphic in $\{\Im \sigma > M\}$ for some sufficiently large $M$. Some basic properties follow immediately:
\begin{gather}
    \mathcal M (xD_x u)(\sigma) = \sigma \mathcal M u(\sigma), \label{eq:Mprop1}\\
    \mathcal M(x^a u)(\sigma) = \mathcal M u(\sigma + ia), \label{eq:Mprop2}\\
    \mathcal M(\Lambda^*_ru)(\sigma) = r^{i\sigma} \mathcal M u(\sigma) \label{eq:Mprop3}
\end{gather}
where $\Lambda_r(x) = rx$ for some $r > 0$. 

Now we define b-Sobolev spaces. Fix a cutoff function
\begin{equation}\label{eq:zero_cutoff}
    \chi_0 \in \CIc(\overline \R_+; [0, 1]), \qquad \text{$\chi_0 \equiv 1$ near $0$}.
\end{equation}
Then
\begin{equation}\label{eq:b-sob}
    H^s_\bo(\R_+) := \{u \in \overline{\mathcal D}'(\R_+): \langle \sigma \rangle^s \mathcal M(\chi_0 u) \in L^2(\R), \ (1 - \chi_0)u \in H^s(\R)\}.
\end{equation}
Clearly this $H^s_\bo(\R_+)$ is independent of the choice of cutoff $\chi$. We will also need the weighted spaces
\begin{equation}\label{eq:weighted_b}
    H^{s, a}_\bo(\R_+) := \rho^a H^s(\R_+), \qquad \rho:= x \chi_0 + (1 - \chi_0)
\end{equation}
It is easy to see that
\begin{equation}
    \chi_0 u \in H^{s, a}_\bo(\R_+) \iff \int_{\{\Im\sigma = -a\}} \langle \sigma \rangle^{2s} |\mathcal M (\chi_0 u)(\sigma)|^2\, d\sigma < \infty
\end{equation}

In particular, these spaces are the same as standard Sobolev spaces away from $0$, and near $0$, observe that for $k \in \N$, 
\begin{equation}
    \chi_0 u \in H^k_\bo(\R_+) \iff \sum_{j = 0}^k \int |(xD_x)^j (\chi_0 u)(x)|^2\, \frac{dx}{x} < \infty,
\end{equation}
The b-Sobolev spaces are the spaces on which b-operators naturally act, since they measure regularity with respect to $x D_x$ near $0$. To address the weights in the b-Sobolev spaces, we need the following lemma. 
\begin{lemma}\label{lem:b-conj}
    Let $\rho$ be as in~\eqref{eq:weighted_b} and $z \in \C$. Then 
    \begin{equation}
        \Psi^m_\bo(\R_+) \ni P \mapsto \rho^{-z} P \rho^z \in \Psi^m_\bo(\R_+)
    \end{equation}
    is an isomorphism. 
\end{lemma}
\begin{proof}
    It suffices to show one direction since the isomorphism follows by replacing $z$ by $-z$. Let $\chi_0$ be as in~\eqref{eq:zero_cutoff} and fix a cutoff subordinate to $\chi_0$:
    \[\widetilde \chi_0 \in \CIc(\overline \R_+; [0, 1]), \quad \widetilde \chi_0 \equiv 1 \ \text{near $0$}, \quad \chi_0 \equiv 1 \ \text{on $\supp \widetilde \chi_0$}.\]
    Observe that
    \[\rho^{-z} P \rho^z - \widetilde \chi_0 \rho^{-z} P \rho^z \widetilde \chi_0 \in \Psi^{m}_\bo(\R_+).\]
    We now show that $\widetilde \chi_0 \rho^{-z} P \rho^z \widetilde \chi_0 \in \Psi^m_\bo(\R_+)$ by constructing the symbol. By Lemma~\eqref{lem:b-composition}, $\widetilde \chi_0 P \widetilde \chi_0 \in \Psi^m_\bo(\R_+)$, and let $a(x, \xi)$ be its symbol. Let 
    \begin{equation}
        b(x, \xi) = \frac{1}{2 \pi} \int_\R e^{i \xi t}(t + 1)^{-z} \hat a(x, t)\, dt, \quad \text{where}\quad \hat a(x, t) = \int_\R e^{-i \xi t} a(x, \xi)\, d \xi
    \end{equation}
    By~\eqref{eq:b-symbol_est} and~\eqref{eq:lacunary}, $\hat a(x, t)$ vanishes to infinite order as $t \to \infty$ and $t \to {-1}$ uniformly and smoothly in $x$ up to $x = 0$, so indeed $b$ is well-defined and $b \in S^m_+(\R_+)$. Note that $\widetilde \chi_0(x) \rho^z(x) = \widetilde \chi_0(x) x^z$, so it follows that
    \begin{align*}
        \widetilde \chi_0 \rho^{-z} P \rho^z \widetilde \chi_0 u(x) &= \frac{1}{2 \pi} \int_\R \int_{\R_+} (x/y)^{-z} e^{i(x - y) \xi} a(x, x \xi) u(y)\, dy d \xi\\
        &=\frac{1}{2 \pi} \int_\R \int_{\R_+} e^{i(x - y)\xi} b(x, x\xi) u(y)\, d\xi dy
    \end{align*}
    as desired.
\end{proof}
It is immediately clear from~\eqref{eq:weighted_b} that $L^2(\R_+) = H^{0, -\ha}_\bo(\R_+)$, so the mapping properties of b-operators on b-Sobolev spaces follow from the $L^2$ bound in Lemma~\ref{lem:b-L2} and the invariance under conjugation in Lemma~\ref{lem:b-conj}. 
\begin{lemma}\label{lem:b-map_prop}
    Let $P \in \Psi^m_\bo(\R_+)$. Then $P$ extends to a continuous operator $P: H^{\ell, a}_\bo(\R_+) \to H^{\ell - m, a}_\bo(\R_+)$ for all $\ell, a \in \R$. 
\end{lemma}

Next we characterize the relationship between b-Sobolev spaces and standard Sobolev spaces. In particular, $H^s(\R_+)$ for $s \in (-\ha, \ha)$ has b-Sobolev characterizations. First, we note that for such $s$, the supported and extendible Sobolev spaces are the same, so there is no ambiguity when we write $H^s(\R_+)$ for $s \in (-\ha, \ha)$. 
\begin{lemma}\label{lem:cutting}
    Let $H$ denote the Heaviside function. Then the multiplication map $u(x) \mapsto H(x) u(x)$ is continuous on $H^s(\R)$ for $s \in (-\ha, \ha)$
\end{lemma}
\begin{proof}
    The lemma is trivial for $s = 0$, and by duality, it suffices to establish the lemma for $s \in (0, \ha)$, which is given in \cite[Chapter 4, Lemma 5.4]{Taylor}
\end{proof}
Note that there are no distributions in $H^s(\R)$, $s \in (-\ha, \ha)$, that is supported on $\{0\}$. Therefore, it follows from Lemma~\ref{eq:cutting} that there is no reason to distinguish between $\dot H^s(\R_+)$ and $\bar H^s(\R_+)$ for $s \in (-\ha, \ha)$, and we will sometimes denote such spaces by just $H^s(\R_+)$. The following lemma relating b-Sobolev spaces and Sobolev spaces is well-known in the community. We give a self-contained proof here for convenience.
\begin{lemma}\label{lem:b_sob_to_sob}
    For $s \in (-\ha, \ha)$, we have an isomorphism $H^s(\R_+) \simeq H_\bo^{s,s - \ha}(\R_+)$ via the identity map.
\end{lemma}
\begin{proof}
    It suffices to prove the lemma for $0 < s < \ha$ since the case $s = 0$ is clear and the range $-\ha < s< 0$ follows by duality. 

    \noindent
    1. First consider $u \in H^{s, s - \ha}_\bo (\R_+)$ for $s \in (0, \ha)$, and we may further assume that $u$ is supported near $0$. We wish to show that the $H^s(\R_+)$ norm is bounded. For $s \in (0, \ha)$, 
    \begin{equation}\label{eq:physical_sobolev}
        \|u\|_{H^s} \simeq \int |u(x)|^2 \, dx + \int_{\R^2} \frac{|u(x) - u(y)|^2}{|x - y|^{2s + 1}}\, dx dy.
    \end{equation}
    where $\simeq$ denotes equivalence in norms. See \cite[\S 7.9]{H1} for a proof of~\eqref{eq:physical_sobolev}. Since $u(x) = 0$ for $x < 0$, the second integral in~\eqref{eq:physical_sobolev} can be separated into
    \begin{equation}
    \begin{gathered}
        \int_{\R^2} \frac{|u(x) - u(y)|^2}{|x - y|^{2s + 1}}\, dx dy = I(u) + 2\int_{\R} |x^{-s + \ha} u(x)|^2\, \frac{dx}{x} \\
        \text{where} \quad I(u) := \int_0^\infty \int_0^\infty \frac{|u(x) - u(y)|^2}{|x - y|^{2s + 1}}\, dx dy
    \end{gathered}
    \end{equation}
    Therefore, it suffices to show that $I(u) < \infty$. 

    \noindent
    2. Substituting $y = rx$, we find 
    \begin{align}
        I(u) &= \int_0^\infty \int_0^\infty \frac{|u(x) - u(rx)|^2}{|x - xr|^{2s + 1}}x \, dx dr \nonumber \\
        &= \int_0^\infty \frac{1}{|1 - r|^{2s + 1}} \int_0^\infty \left|\frac{u(x) - u(rx)}{x^{s -\ha}} \right|^2 \, \frac{dx}{x} dr \nonumber \\
        &= \frac{1}{2 \pi} \int_0^\infty\int_{-\infty}^\infty \frac{|1 - r^{i\sigma + (s -\ha)}|^2}{|1 - r|^{2s + 1}} |\mathcal Mu(\sigma - i(s - \tfrac{1}{2}))|^2\, d\sigma dr, \label{eq:I(u)}
    \end{align}
    where the last equality follows from Parceval's identity together with~\eqref{eq:Mprop2} and~\eqref{eq:Mprop3}. We first integrate~\eqref{eq:I(u)} in $r$. Making the substitution $r = e^q$, we find that for $\sigma > 1$, 
    \begin{align}
        \int_0^\infty \frac{|1 - r^{i\sigma + s -\ha}|^2}{|1 - r|^{2s + 1}}\, dr &\le M + \int_{e^{-1}}^{e} \frac{|1 - r^{i\sigma + s -\ha}|^2}{|1 - r|^{2s + 1}}\, dr \nonumber \\
        &\le M + C\int_{-1}^1 \frac{|1 - \exp(q(i\sigma + s - \ha))|^2}{|q|^{2s + 1}} \, dq \nonumber \\
        &\le M + C\int_{\{|q| < (2\sigma)^{-1}\}} |q|^{-2s + 1} \sigma^{2}\, dq \nonumber \\
        &\qquad \qquad+ C\int_{\{(2\sigma)^{-1} \le |q| \le 1\}} |q|^{-2s - 1}\, dq \nonumber \\
        &\le M + C \sigma^{2s}, \label{eq:big_s}
    \end{align}
    where $M > 0$ is some constant independent of $\sigma$ and $C$ may change from line to line, but remains independent of $\sigma$. For $0 \le \sigma \le 1$, 
    \begin{equation}
        \int_0^\infty \frac{|1 - r^{i\sigma + s -\ha}|^2}{|1 - r|^{2s + 1}}\, dr \lesssim \int_0^{\ha} r^{2s - 1} \, dr  + \int_{\ha}^2 |1 - r|^{-2s + 1} \, dr + \int_2^\infty |r|^{-2} < \infty, \label{eq:small_s}
    \end{equation}
    where the hidden constant is independent of $0 \le \sigma \le 1$. Substituting the bounds~\eqref{eq:big_s} and~\eqref{eq:small_s} into~\eqref{eq:I(u)}, we find that 
    \[I(u) \lesssim \int_{-\infty}^\infty \langle \sigma \rangle^{2s} |\mathcal Mu(\sigma - i(s - \tfrac{1}{2}))|^2\, d\sigma = \|u\|_{H^{s, s - \ha}_\bo}\]
    as desired. 

    \noindent
    3. Now assume that $u \in H^s(\R_+)$. Then $I(u) < \infty$. Observe that for $\sigma > 1$ and again making the substitution $r = e^q$, we have
    \begin{align}
        \int_0^\infty \frac{|1 - r^{i\sigma + s -\ha}|^2}{|1 - r|^{2s + 1}}\, dr &\ge c \int_{-1}^1 \frac{|1 - \exp(q(i\sigma + s - \ha))|^2}{|q|^{2s + 1}} \, dq \\
        &\ge c \int_{\{|q| < (2\sigma)^{-1}\}} |q|^{-2s + 1} \sigma^{2}\, dq \\
        &\ge c \sigma^{2s}
    \end{align}
    for some small $c > 0$. For $0 \le \sigma \le 1$, 
    \[\int_0^\infty \frac{|1 - r^{i\sigma + s -\ha}|^2}{|1 - r|^{2s + 1}}\, dr \ge c \int_0^\ha r^{2s - 1}\, dr \ge c > 0\]
    for some sufficiently small $c > 0$. Again using~\eqref{eq:I(u)}, we find that
    \[\|u\|_{H^{s, s - \ha}_\bo} \lesssim I(u) \le \|u\|_{H^s},\]
    so indeed the two spaces are isomorphic. 
\end{proof}
The above lemma allow us to pass between b-Sobolev spaces and standard Sobolev spaces. In our specific setting, we may consider $\partial \Omega$ as a manifold with boundary at the corners where the smooth structure near the corner is determined by some b-parameterization. In particular, we have the usual supported and extendible Sobolev spaces denoted $\dot H^s(\partial \Omega)$ and $\bar H^s(\partial \Omega)$ respectively. Similarly, b-Sobolev spaces on $\partial \Omega$ as a manifold with boundary will be denoted $H^{s, a}_\bo(\partial \Omega)$. It follows from Lemma~\ref{lem:b_sob_to_sob} that
\[\dot H^s(\partial \Omega) = \bar H^s(\partial \Omega) = H^{s, s - \ha}_\bo (\partial \Omega) \quad \text{for} \quad s \in (-\ha, \ha).\] 
Often times, for distributions supported away the corners, we will just write $H^s(\partial \Omega)$ instead of $\dot H^s(\partial \Omega)$ when there is no ambiguity. 


A slightly subtle point is that the chess billiard map takes corners to non-corners, but continuous functions on $\partial \Omega$ will still get pulled-back to continuous functions by the chess billiard map. When interpreting $\partial \Omega$ as the disjoint union of manifolds with boundary, the fact that the continuity of a function is preserved by the pullback is unclear because we are ignoring the $C^0$ structure at the corners of the boundary. Therefore, it is also useful to interpret $\partial \Omega$ as $\mathbb S^1$ and consider the chess billiard map as a bi-Lipschitz piecewise-smooth homeomorphism on $\mathbb S^1$. We have the following lemma to address these issues. 
\begin{lemma}\label{lem:pullback_continuity}
    Let $f :\mathbb S^1 \to \mathbb S^1$ be a bi-Lipschitz piecewise-smooth homeomorphism. Then the pullback $f^*: C^\infty(\mathbb S^1) \to \mathcal D'(\mathbb S^1)$ extends to an isomorphism 
    \[f^* : H^s(\mathbb S^1) \to H^s(\mathbb S^1)\]
    for $s \in (\ha, \frac{3}{2})$.
\end{lemma}
\begin{proof}
    We may assume that $f(0) = 0$. For $u \in C^\infty (\mathbb S^1)$, 
    \begin{equation}
        f^* u(\theta) = u(0) + \int_0^\theta f'(t) (f^* u')(t)\, dt
    \end{equation}
    This extends to $u \in H^s(\mathbb S^1)$ for $s \in (\ha, \frac{3}{2})$. Indeed, for such $s$, $H^s \hookrightarrow C^0(\mathbb S^1)$. Furthermore, it follows from Lemma~\ref{eq:cutting} that $f' \cdot (f^* u') \in H^{s - 1}(\mathbb S^1)$. 
\end{proof}

Another consequence of Lemma~\ref{lem:b_sob_to_sob} is that multiplication by the Heaviside function has the following mapping property. 
\begin{lemma}\label{lem:b-cutting}
    The map $u(x) \mapsto H(x) u(x)$ is continuous from $H^{s}(\R) \to H^{s, a}_\bo(\R_+)$ for $a \in (-1, 0)$ and $s \ge a + \ha$. 
\end{lemma}
\begin{proof}
    It follows from Lemma~\ref{lem:b_sob_to_sob} that $u(x) \mapsto H(x) u(x)$ is continuous from
    \[H^{a + \ha + k}(\R) \to H^{a + \ha + k, a + k}_\bo(\R_+)\]
    for $k \in \N_0$. The lemma then follows by interpolation. 
\end{proof}

Finally, we introduce polyhomogeneous conormal distributions. 

\begin{definition}
    An \textup{index family} is a set $E \subset \C \times \N_0$ such that 
    \begin{enumerate}
        \item $(z, k) \in E \implies (z + 1, k) \in E$ and $(z, k - 1) \in E$ if $k \ge 1$,

        \item $\{(z, k) \in E: \Re z < C\}$ is a finite set for any $C \in \R$. 
    \end{enumerate}
    Due to condition (2), we can define
    \[\inf E:= \inf\{z: (z, k) \in E \ \text{for some $k \in \N_0$}\}.\]
\end{definition}
Then we define the space of polyhomogeneous conormal distribution with respect to the index family $E$ by
\begin{equation}
    u \in \mathcal A^E_\phg(\R_+) \iff u \in \overline{\mathcal E}'(\R_+) \quad \text{and} \quad u \sim \sum_{(z, k) \in E} a_{k, z} x^z (\log x)^k
\end{equation}
for some coefficients in $a_{k, z} \in \C$. The `$\sim$' in this context means that 
\[\chi_0 \Big[u - \sum_{\substack{(z, k) \in E \\ \Re z \le N}} a_{z, k} x^z (\log x)^k \Big] \in \dot C^N_{\mathrm{c}}(\R_+) \qquad \text{for all $N \in \N$}. \]

Polyhomogeneous distributions can be characterized by their Mellin transforms. 

\begin{lemma}\label{lem:mellin_polyhom}
    Assume that $u \in \overline{\mathcal E}'(\R_+)$. Then $u \in \mathcal A^E_{\phg}(\R_+)$ if and only if 
    \begin{enumerate}
        \item $\mathcal M u(\sigma)$ extends from $\Im \sigma \gg 1$ to a meromorphic function over $\C$,

        \item if $\sigma = -iz$ is a pole of order $k$, then $(z, k) \in E$, 

        \item for all $C > 0$, there exists a $C'$ such that for $|\Im \sigma| < C$,
        \[|\mathcal M u(\sigma)| \le C_N |\Re \sigma|^{-N} \quad \text{if} \quad |\Re \sigma| > C'.\]
    \end{enumerate}
\end{lemma}
See \cite[Lemma 2.2]{Hintz_notes} for a detailed proof of this Lemma.

\subsection{Boundary terms}\label{sec:boundary_terms}
The small b-calculus is especially convenient since it is symbolic. However, the residual class $\Psi^{-\infty}_\bo$ of the small b-calculus is too big. In particular, the mapping property in Lemma~\ref{lem:b-map_prop} does not improve decay near $0$, and indeed, one can check that the embedding $H^{s, a}_\bo \hookrightarrow H^{s - N, a}_\bo$ is not compact for any $N > 0$. This means that with the small parametrix from Proposition~\ref{prop:small_parametrix} is not enough to give Fredholm properties for elliptic b-operators. We are thus led to consider operators with Schwartz kernels that do not vanish to infinite order at lb and rb. Moreover, the restricted boundary reduced operators do not lie in the small b-calculus in the first place. 

Let $\mathcal E = (E_\lb, E_\rb)$ denote a pair of index sets $E_\lb$ and $E_{\rb}$. We define the space $\widetilde \Psi_\bo^{-\infty, \mathcal E}$ by the set of operators with kernel $K$ that satisfies $r(1 - \tau) K(r, \tau)$ is polyhomogeneous conormal to lb and rb with index set $E_\lb$ and $E_\rb$ respectively, smoothly up to $\tau = 0$ and decaying rapidly as $\tau \to \infty$. Here $K(r, \tau)$ is the Schwartz kernel lifted to the blown-up quarter space $\widetilde{\R}^2_+$. More precisely, 
\begin{multline}\label{eq:full_smoothing}
    r(1 - \tau) K(r, \tau) - \sum_{\substack{(z, k) \in E_\lb \\ \Re z \le N}} a_{z, k, \lb}(r, \tau) (1 + \tau)^z (\log (1 + \tau))^k \\
    - \sum_{\substack{(z, k) \in E_\rb \\ \Re z \le N}} a_{z, k, \rb}(r, \tau) (1 - \tau)^z (\log (1 - \tau))^k \in \dot {\mathscr S}((\R_+)_r; \dot C^N_{\mathrm{c}}([-1, 1])_\tau),
\end{multline}
where $a_{z, k, \bullet} \in \dot{\mathscr S}(\R_+; C^\infty([-1, 1]))$. Implicit here is that $r(1 - \tau) K(r, \tau)$ is polyhomogenous conormal with respect to ff with the index set $\N_0$, that is, $r(1 - \tau)K(r, \tau)$ is smooth up to ff. 

While it may seem that looking at the index family of $r(1 - \tau) K(r, \tau)$ rather than of $K(r, \tau)$ is unnatural, we are merely compensating for the fact that the Schwartz kernel of an operator acting on functions is not canonically a function. We are fixing an underlying choice of smooth measure on $\R^+$ to integrate against so that we can write the Schwartz kernel as a function, but it does not transform like one under pullback by the blow-down map $\beta$. See \cite[\S5]{Melrose93} for details, where the b-calculus is developed acting on half-densities to avoid some technical issues, but we merely adjust the weights manually here to avoid these discussions. As a result, we will see that the definition of $\widetilde \Psi^{-\infty, \mathcal E}_\bo(\R_+)$ gives mapping properties with uncomplicated numerology in Lemma~\ref{lem:full_b-map_prop}. We first consider the composition with elements of the small b-calculus. 
\begin{lemma}\label{lem:full_composition}
    Let $\mathcal E = (E_\lb, E_\rb)$ be a pair of indicial families and let $P \in \Psi^m_\bo(\R_+)$. Then 
    \begin{equation}
        \widetilde \Psi_\bo^{-\infty, \mathcal E}(\R_+) \ni R \mapsto P \circ R \in \widetilde \Psi_\bo^{-\infty, \mathcal E}(\R_+)
    \end{equation}
\end{lemma}
See \cite[Proposition 5.46]{Melrose93} for the proof.

Next, we consider the mapping properties of operators in $\widetilde \Psi_\bo^{-\infty, \mathcal E}$. For the sake of completeness, we adapt the proof from \cite{Melrose93} below for operators acting on functions rather than half densities. The proof also applies to a slightly broader class of Schwartz kernels, which we will need later. 
\begin{lemma}\label{lem:full_b-map_prop}
    Let $R \in \widetilde \Psi_\bo^{-\infty, \mathcal E}$. Assume that $-\inf E_\rb < \inf E_\lb$. For any $N >0$ and $a \in (-\inf E_\rb, \inf E_\lb)$, $R$ extends to an operator 
    \[R : H^{s, a}_\bo(\R_+) \to H^{s + N, a}_\bo(\R_+).\] 
\end{lemma}
\begin{proof}
    1. We first make a few simplifications. Let $K$ denote the Schwartz kernel of $R$. The lemma is clear if $K$ is supported away from $0$, so it suffices to assume that $K$ lifted to the blown-up space is supported in a neighborhood of $\ff \subset \widetilde \R_+^2$. Observe that for $z \in \C$, the Schwartz kernel of $x^{-z} R x^z$ is given by $(x/y)^z K(x, y)$. Therefore $x^{-z}R x^z \in \widetilde \Psi^{-\infty, \mathcal E'}_\bo$ where $\mathcal E' = (E'_\lb, E_\rb')$ are index sets satisfying 
    \[\inf E'_\lb = \inf E_\lb + z, \quad \inf E'_\rb = \inf E_\rb -z.\]
    Therefore, we may assume that $\inf E_\lb, \inf E_\rb > 0$ and take $a = 0$. Finally, by Lemma~\ref{lem:full_composition}, it suffices to take $s = 0$. 

    \noindent
    2. First consider the case that $K$ is supported in a small neighborhood of $\lb \cap \ff$. Let $\langle \bullet, \bullet \rangle_\bo$ denote the inner product on $L^2(\R_+; \frac{dx}{x}) = H^0_\bo(\R_+)$. It suffices to show that 
    \[|\langle Ru, u \rangle_\bo| \le C\|u\|_{H^0_\bo}^2.\]
    We compute
    \begin{align*}
        |\langle Ru, u \rangle_\bo| &= \int_{\R_+^2} y K(x, y) u(y) \bar u(x)\, \frac{dy}{y} \frac{dx}{x}\\
        &= \int_0^\delta \int_0^\delta y K(ty, y) u(y) \bar u(ty)\, \frac{dy}{y} \frac{dt}{t}
    \end{align*}
    where we used the assumption that $K$ is supported in a small neighborhood of $\lb \cap \ff$. Since $R \in \widetilde \Psi^{-\infty, \mathcal E}_\bo$, it follows that,
    \begin{equation}\label{eq:smooth-to-ff}
        yK(ty, y) \in C^\infty([0, \delta)_y; \mathcal A_\phg^{E_\lb}(\R_+)),
    \end{equation}
    that is, $yK(ty, y)$ is a family of conormal distributions in $t$ that depends smoothly on $y$. Fix $0 < \epsilon < \inf E_\lb$. Then by the Cauchy--Schwartz inequality, we have 
    \begin{align*}
        |\langle Ru, u \rangle_\bo|^2 &\le \left(\int_0^\delta \int_0^\delta |t^{-\epsilon} y K(ty, y) u(y)|^2\, \frac{dy}{y} \frac{dt}{t} \right) \left( \int_0^\delta \int_0^\delta |t^\epsilon \bar u(ty)|^2 \, \frac{dy}{y}\frac{dt}{t}\right) \\
        &\lesssim \|u\|_{H^0_\bo} \int_0^\delta t^{2\epsilon} \int_0^{\delta t} |u(y)|^2 \,\frac{dy}{y} \frac{dt}{t} \\
        &\lesssim \|u\|_{H^0_\bo}^2.
    \end{align*}
    as desired. It remains to consider if $K$ is supported in a neighborhood of $\rb \cap \ff$. In this case, we simply note that the Schwartz kernel of the adjoint with respect to $\langle \bullet, \bullet \rangle_\bo$ is supported in a neighborhood of $\lb \cap \ff$, and the same bound follows by duality. 
\end{proof}

Define the space
\begin{equation}
    \widetilde \Psi_\bo^{m, \mathcal E}(\R_+) := \Psi_\bo^m(\R_+) + \widetilde \Psi_\bo^{-\infty, \mathcal E}(\R_+) 
\end{equation}

It will be useful to compare b-operators against dilation invariant operators on $\R_+$. We first collect some basic properties about dilation-invariant operators. 

We denote the full space of order $m$ dilation invariant b-operators by $\widetilde \Psi^{m, \mathcal E}_{\mathrm{dil}}(\R_+) \subset \mathcal L(\dot C_{\mathrm c}^\infty(\R_+), \overline{\mathcal D}'(\R_+))$, and
\begin{equation}
    P \in \widetilde \Psi^{m, \mathcal E}_{\mathrm{dil}}(\R_+) \iff \chi_0 P \in \Psi^{m, \mathcal E}_\bo(\R_+) \quad \text{and} \quad M_\lambda^*(Pu) = P(M_\lambda^* u) \ \forall \lambda > 0,
\end{equation}
where again $M_\lambda(x) = \lambda x$. By dilation invariance, the Schwartz kernel $K$ of $P \in \widetilde \Psi^{m, \mathcal E}_{\mathrm{dil}}(\R_+)$ satisfies $K(x, y) = \lambda K(\lambda x, \lambda y)$ for all $\lambda > 0$. In particular, when lifted to the blown-up space $\widetilde \R_+^2$ using coordinates~\eqref{eq:quarter_bcoords}, $r K(r, \tau)$ is independent of $r$, and we have 
\begin{equation}\label{eq:Kff}
    \frac{1}{2} r(1 - \tau)K(r, \tau) = \frac{1}{2}(r(1 - \tau)K(r, \tau))|_{\ff} =: K_\ff(\tau).
\end{equation}
Here, we inserted an extra factor of $\frac{1}{2} r (1 - \tau)$, again to adjust for the fact that we implicitly fixed an underlying choice of density on $\R_+$ to define the Schwartz kernel.

Tracing back the definitions of $\widetilde \Psi^{\infty, \mathcal E}_\bo(\R_+)$ in~\eqref{eq:full_smoothing} and $\Psi^m_\bo(\R_+)$ in~\eqref{eq:b-calc}, we immediately have the following. 
\begin{lemma}\label{lem:Kff}
    Let $\mathcal E = (E_{\lb}, E_\rb)$ be a pair of indicial families and let $K$ be the Schwartz kernel of an operator in $\widetilde \Psi^{m, \mathcal E}_{\mathrm{dil}}(\R_+)$ with $K_\ff$ as in~\eqref{eq:Kff}. Then
    \begin{equation}\label{eq:Kff_asymptotics}
    \begin{gathered}
        \chi_0(\tau) K_\ff (\tau - 1) \in \mathcal A^{E_\lb}_\phg(\R_+), \\
        \chi_0(\tau) K_\ff (1 - \tau) \in \mathcal A^{E_\rb}_\phg(\R_+),
    \end{gathered}
    \end{equation}
    where $\chi_0$ is as defined in~\eqref{eq:zero_cutoff} with $\supp \chi_0 \subset [0, 1)$. Let $\chi \in \CIc((-1, 1); [0, 1])$ such that $\chi \equiv 1$ near $0$. Then 
    \begin{equation}\label{eq:Kff_0}
        \chi(\tau) K_\ff(\tau) = \frac{1}{2 \pi} \int e^{i\tau \xi} p(\xi) \, d \xi
    \end{equation}
    for some $p(\xi)$ such that $\chi_0(x) p(\xi) \in S^m_+$. 
\end{lemma}
Following the lemma, we give the natural notion of ellipticity for dilation invariant operators. 
\begin{definition}
    We say that $P \in \widetilde \Psi^{m, \mathcal E}_{\mathrm{dil}}(\R_+)$ is elliptic if $p(\xi)$ as in~\eqref{eq:Kff_0} satisfies 
\begin{equation}
    |p(\xi)| \ge c|\xi|^m \quad \text{for all} \quad |\xi| \ge c^{-1}.
\end{equation}
\end{definition}
For dilation-invariant operators, it is natural to measure regularity with respect to $xD_x$ on all of $\R_+$ rather than just near $0$. We define the dilation-invariant b-Sobolev spaces on $\R_+$ by 
\begin{equation}
    \begin{gathered}
        H^s_{\mathrm{dil}}(\R_+) := \{u \in \overline{\mathcal D}'(\R_+) : \langle \sigma \rangle \mathcal Mu \in L^2(\R)\}, \\
        H^{s, a}_{\mathrm{dil}}(\R_+) := x^a H^s(\R_+). 
    \end{gathered}
\end{equation}
Comparing with~\eqref{eq:b-sob} and~\eqref{eq:weighted_b}, we see that the only difference is how regularity at infinity is handled. In $H^{s, a}_\bo(\R_+)$ we treat $0$ as the boundary and characterize regularity with vector fields tangent to $0$ and uniformly bounded up to $\infty$. On the other hand, in $H^{s, a}_{\mathrm{dil}}(\R_+)$, we are essentially taking the projective compactification of $\R_+$ and considering $\{0, \infty\}$ as the boundary, then measuring regularity with respect to vector fields tangent to both $0$ and $\infty$.

The advantage of dilation-invariant operators is that their behavior under Mellin transform is very simple. The following lemma is an extension of~\eqref{eq:Mprop1} to all dilation-invariant operators. Define the change of coordinates $\pi: \overline \R_+ \to [-1, 1]$ by
\begin{equation}\label{eq:projective_coords}
    \pi(t) = \frac{t - 1}{t + 1}.
\end{equation}
\begin{lemma}\label{lem:normal_fam}
    Let $\mathcal E = (E_{\lb}, E_\rb)$ be a pair of index families. Let $P \in \Psi^{m, \mathcal E}_{\mathrm{dil}}(\R_+)$ with Schwartz kernel $K$, and let $K_\ff$ be as in~\eqref{eq:Kff}. Then the normal family of $P$ by
    \[N(P, \sigma) := \mathcal M (\pi^* K_\ff)(\sigma)\]
    is holomorphic for $-\Im \sigma \in (-\inf E_\rb, E_\lb)$, and for every $U \Subset (-\inf E_\rb, E_\lb)$, there exists a constant $C_U$ such that 
    \begin{equation}\label{eq:Mk_est}
        N(P, \sigma) \le C_U \langle \sigma \rangle^m \quad \text{for $-\Im \sigma \in U$}
    \end{equation}
    If $P$ is elliptic, we also have a lower bound
    \begin{equation}\label{eq:Mk_est_elliptic}
        N(P, \sigma) \ge C_U^{-1} \langle \sigma \rangle^m \quad \text{for $-\Im \sigma \in U$ and $\Re \sigma > C_U$}.
    \end{equation}
    Finally, for $u \in \overline{\mathcal E}'(\R_+) \cap H^{s, a}_\bo(\R_+)$
    \begin{equation}\label{eq:DI_mellin}
        \mathcal M(Pu)(\sigma) = N(P, \sigma) \mathcal M u(\sigma)
    \end{equation}
    for $-\Im \sigma \in (-\inf E_\rb, a]$.  
\end{lemma}
\begin{Remarks}
    1. For dilation invariant differential operators $P = \sum_{k = 1}^N a_k (xD_x)^k$, $a_k \in \C$, it is easy to see that the associated normal family is simply $N(P, \sigma) = \sum_{k = 1}^N a_k \sigma^k$. 
    
    \noindent
    2. In the literature, the normal operator (or sometimes called the indicial operator) is defined for all b-operators $P \in \Psi^{m, \mathcal E}_\bo$ by restricting the (appropriately weighted) Schwartz kernel of $P$ to ff and taking the Mellin transform. The operators that appear in this paper are truly dilation invariant in a neighborhood of $0$, so we can avoid these discussions. 
\end{Remarks}
\begin{proof}
    It follows from~\eqref{eq:Kff_asymptotics} that $\mathcal M(\pi^* K_\ff)(\sigma)$ is holomorphic in the strip $-\inf E_\rb < \sigma < \inf E_\lb$. The estimate~\eqref{eq:Mk_est} follows from~\eqref{eq:Kff_0} and is simply a consequence of symbol bounds. Similarly, the lower bounds~\eqref{eq:Mk_est_elliptic} follows from the lower bounds on the symbol in the definition of ellipticity.

    From~\eqref{eq:Kff}, observe that 
    \begin{equation*}
        Pu(x) = \int_{\R_+} (\pi^* K_\ff)(x/y) u(y) \, \frac{dy}{y}
    \end{equation*}
    Taking the Mellin transform and using property~\eqref{eq:Mprop3} of the Mellin transform, we find that 
    \[\mathcal M(Pu)(\sigma) = \int_0^\infty \mathcal M(\pi^* K_\ff)(\sigma) y^{-i\sigma} u(y)\, dy = \mathcal M(\pi^* K_\ff)(\sigma) \mathcal Mu(\sigma)\]
    as desired. 
\end{proof}
An immediate consequence of Lemma~\ref{lem:normal_fam} is that if $P \in \widetilde \Psi^{m, \mathcal E}_{\mathrm{dil}}(\R_+)$, then 
\begin{equation}\label{eq:dil-map-prop}
    P: H^{s, a}_{\mathrm{dil}}(\R_+) \to H^{s - m, a}_{\mathrm{dil}}(\R_+) \quad \text{if} \quad a \in (-\inf E_\rb, \inf E_\lb).
\end{equation}

We now present a special case of the full elliptic estimate for b-operators.
\begin{proposition}\label{prop:full_b-est}
    Let $\mathcal E = (E_\lb, E_\rb)$ be a pair of index families such that $-\inf E_\rb < \inf E_\lb$ and $P \in \widetilde \Psi^{0, \mathcal E}_{\mathrm{dil}}(\R_+)$. Let $\chi_0$ be as in~\eqref{eq:zero_cutoff} and let $\chi_1, \chi_2 \in \overline C^\infty_{\mathrm c}(\R_+; [0, 1])$ be such that $\chi_j \equiv 1$ on $\supp \chi_{j - 1}$, $j= 1, 2$. Finally, let $a \in (-\inf E_\rb, \inf E_\lb)$ be such that no zeros of the normal family $N(P, \sigma)$ lie on the line $\{t - ia: t \in \R\}$. Then given that $v \in H^{-N, a}_\bo$, we have the estimate
    \[\|\chi_0 v\|_{H^{s, a}_\bo} \le C\big[\|\chi_1 P \chi_2 v\|_{H^{s, a}_\bo} + \|v\|_{H^{-N, a - \delta}_\bo}\big]\]
    for some $0 < \delta < a + \inf E_\rb$. 
\end{proposition}
\begin{proof}
    1. We first obtain an a priori estimate using by using a small parametrix. Observe that that $\chi_1 P \chi_2 \in \widetilde \Psi^{0, \mathcal E}_\bo(\R_+)$. Therefore, there exists $P_{\mathrm{sm}} \in \Psi^0_\bo(\R_+)$ and $Q_0 \in \widetilde\Psi^{-\infty, \mathcal E}_\bo$ such that
    \[\chi_1 P \chi_2 = P_{\mathrm{sm}} + R_0\]
    where $P_{\mathrm{sm}}$ is uniformly elliptic on an open set containing $\supp \chi_0$. By Proposition~\ref{prop:small_parametrix}, there exists a small parametrix $G \in \Psi_\bo^0(\R_+)$ so that
    \[G \circ (\chi_1 P \chi_2) = \chi_0 + R \quad \text{where} \quad Q = G \circ R_0 \in \widetilde \Psi^{-\infty, \mathcal E}_\bo(\R_+),\]
    where we used Lemma~\ref{lem:full_composition} to find the operator class for $R$. Using mapping properties from Lemmas~\ref{lem:b-map_prop} and~\ref{lem:full_b-map_prop}, we obtain the estimate
    \begin{equation}\label{eq:apriori_b-est}
        \|\chi_0 v \|_{H^{s, a}_\bo} \le C \big[ \|\chi_1 P \chi_2 v\|_{H^{s, a}_\bo} + \|\chi v\|_{H_\bo^{-N, a}} \big]
    \end{equation}
    where $\chi = 1$ near $\supp \chi_2$. 
    \noindent
    2. Since $a$ is chosen so that no zeros of $N(P, \sigma)$ has imaginary part $-ia$, it follows from Lemma~\ref{lem:normal_fam} that we have a uniform lower bound
    \[0 < c \le |N(P, \sigma)| \quad \text{for all} \Im \sigma = -a.\]
    Therefore, it follows from Lemma~\ref{lem:normal_fam} and the mapping property~\eqref{eq:dil-map-prop} that
    \begin{align}
        \|\chi v\|_{H^{-N, a}_\bo} &\le \|\chi_2 v\|_{H^{-N, a}_\bo} + \|v\|_{H^{-N, a - \delta}_\bo} \\
        &\lesssim \|\langle \lambda \rangle^{-N} \mathcal M(\chi_2 v)\|_{L^2(\{\Im \lambda = - a\})}+ \|v\|_{H^{-N, a - \delta}_\bo} \nonumber \\
        &\lesssim \|N(P, \sigma) \langle \lambda \rangle^{-N} \mathcal M(\chi_2 v)\|_{L^2(\{\Im \lambda = - a\})} + \|v\|_{H^{-N, a - \delta}_\bo} \nonumber \\
        &\lesssim \|P\chi_2 v \|_{H^{-N, a}_{\mathrm{dil}}} + \|v\|_{H^{-N, a - \delta}_\bo} \nonumber \\
        &\lesssim \|\chi_1 P \chi_2 v\|_{H^{-N, a}_\bo} + \|v\|_{H^{-N, a - \delta}_\bo} + \|(1 - \chi_1) P \chi_2 v\|_{H^{-N, a}_{\mathrm{dil}}} \label{eq:chiv_bound}
    \end{align}
    It remains to estimate the last term of~\eqref{eq:chiv_bound}.

    \noindent
    3. We further split up the last term of~\eqref{eq:chiv_bound} by 
    \[(1 - \chi_1) P \chi_2 = (1 - \chi)P \chi_2 + (\chi - \chi_1) P \chi_2,\]
    from which we see that 
    \begin{equation}\label{eq:chiv_decay}
        \|(1 - \chi_1) P \chi_2 v\|_{H^{-N, a}_{\mathrm{dil}}} \le \|v\|_{H^{-N, a - \delta}_\bo} + \|(1 - \chi) P \chi_2 v\|_{H^{-N, a}_{\mathrm{dil}}}.
    \end{equation}
    The upshot is that the Schwartz kernel of $(1 - \chi) P \chi_2$ is smooth in the interior of $\R_+^2$. Let $\iota(x) = x^{-1}$ be the inversion map on $\R_+$. In particular, we see that $\mathcal M (\iota^* u)(\sigma) = \mathcal M u(-s)$ where $\iota: \R_+ \to \R_+$ is the inversion map $\iota(x) = x^{-1}$. Therefore, 
    \[u \in H^{s, a}_{\mathrm{dil}}(\R_+) \iff \iota^* u \in H^{s, -a}_{\mathrm{dil}}(\R_+).\]
    Since $P\in \widetilde \Psi^{0, \mathcal E}_{\mathrm{dil}}$, we see that the Schwartz kernel of $\iota^* \circ (1 - \chi_1) P \chi$ is given by 
    \begin{equation}\label{eq:lower_right_kernel}
        \psi(x, y) y^{-1} F(xy)
    \end{equation}
    where $F \in \mathcal A^{E_\rb}_\phg$ and $\psi\in \overline C^\infty_{\mathrm c}(\R_+^2)$ with $\psi \equiv 1$ near $(0, 0)$. To be precise, 
    \[F = \iota^*(\pi^{*} (K_\ff)), \qquad \psi(x, y) = (1 - \chi_1(x^{-1})) \chi(y).\]
    It is easy to see that 
    \[\iota^* \circ (1 - \chi_1) P \chi: H^{-N, -\inf E_{\rb} + \epsilon}_\bo(\R_+) \to H^{N, \inf E_{\rb} - \epsilon}_\bo(\R_+)\]
    for any $\epsilon > 0$. Therefore, for $a \in (-\inf E_\rb, \inf E_\rb)$, 
    \begin{equation}\label{eq:lower_right}
        \|(1 - \chi_1) P \chi v\|_{H^{-N, a}_{\mathrm{dil}}} = \|\iota^* \circ (1 - \chi_1) P \chi v \|_{H^{-N, -a}_\bo} \le \|\chi_1 v\|_{H^{-N, a - \delta}_\bo}
    \end{equation}
    Piecing together the estimates~\eqref{eq:chiv_bound}, \eqref{eq:chiv_decay}, and~\eqref{eq:lower_right}, we see that  
    \[\|\chi v\|_{H^{-N, a}_\bo} \le C\big[\|\chi P \chi v\|_{H^{-N, a}_\bo}+ \|\chi v\|_{H^{-N, a - \delta}_\bo} \big],\]
    which gives the desired estimate upon substituting into the a priori estimate~\eqref{eq:apriori_b-est}. 
\end{proof}

\begin{Remark}
    While we only considered b-operators that act on functions on $\R_+$, it is easy to see that the theory extends to vector-valued functions mapping $\R_+ \to \C^n$. In such cases, b-operators are formed by quantizing an $n \times n$ matrix of symbols applying the same procedure in~\eqref{eq:b-quantization} entry-wise. 
\end{Remark}

\section{Reduction to boundary}\label{sec:reduction_to_boundary}
In this section, we reduce the analysis of the resolvent $(P - \omega^2)^{-1}$ for $\Im \omega > 0$ to the analysis of a family of operators $\mathcal C_\omega$ on the boundary $\partial \Omega$. The reduction uses the method of boundary layer potentials, but this will require more justification near the corners. Furthermore, we compute the Schwartz kernel of $d \mathcal C_\omega$ in \S\ref{sec:kernel_computation}.

\subsection{Fundamental solutions}\label{sec:FS}
We shift our focus to the differential operator
\begin{equation}\label{eq:P(omega)}
    P(\omega):= (1 - \omega^2) \partial_{x_2}^2 - \omega^2 \partial_{x_1}^2, \qquad \omega \in \C, \, 0 < \Re \omega < 1
\end{equation}
on $\R^2_{x_1, x_2}$. Recall that our goal is to understand the spectrum of the operator $P$ defined in~\eqref{eq:P_def}. Formally, $P(\omega)$ is related to $P$ by 
\begin{equation}\label{eq:resolvent_relation}
    P(\omega) = (P - \omega^2) \Delta_\Omega, \qquad P(\omega)^{-1} = \Delta_\Omega^{-1}(P - \omega^2)^{-1}.
\end{equation}
We see that analyzing the resolvent of $P$ is equivalent to analyzing $P(\omega)^{-1}$, which is far more tractable since $P(\omega)$ is a second order differential operator. We first collect some useful properties of $P(\omega)$. Note that $P(\omega)$ can be factorized as
\begin{equation}\label{eq:L_factor}
    P(\omega) = 4 L^+_\omega L^-_\omega, \quad L^\pm_\omega := \frac{1}{2} (\pm \omega \partial_{x_1} + \sqrt{1 - \omega^2} \partial_{x_2}),
\end{equation}
where the square root is taken on the branch $\C \setminus (-\infty, 0]$. When $\Im \omega \neq 0$, $L^\pm_\omega$ are Cauchy--Riemann type operators. On the other hand, if $\lambda \in (0, 1)$, then $L^\pm_\lambda$ are simply linearly independent vector fields. 

Recall the functions $\ell^\pm(x, \lambda)$ from \eqref{eq:dual_factor}, whose level curves formed the trajectories for the chess billiard flow. For $\omega \in (0, 1) + i\R$, we can similarly define
\begin{equation}\label{eq:ell}
    \ell^\pm(x, \omega):= \pm \frac{x_1}{\omega} + \frac{x_2}{\sqrt{1 - \omega^2}},
\end{equation}
which satisfy
\begin{equation}\label{eq:L_dual}
    L_\omega^\pm \ell^\pm(x, \omega) = 1, \quad L_\omega^\mp \ell^\pm(x, \omega) = 0.
\end{equation}

There exist explicit formulas for a fundamental solution of $P(\omega)$. That is, there exists distributions $E_\omega \in \mathcal D'(\R^2)$ with explicit formulas that satisfy
\[P(\omega) E_\omega = \delta_0.\]
Note that $P(\omega)$ is uniformly elliptic when $\Im \omega \neq 0$. On the other hand, as $\omega$ approaches the real line, $P(\omega)$ degenerates to a hyperbolic operator. If $\omega \in (0, 1)$, $P(\omega)$ is the $1 + 1$-dimensional wave operator. In the elliptic case, the fundamental solution is essentially a rescaled version of the Newton potential. As $\omega$ approaches the real line, the fundamental solutions converge in distribution to one of the Feynmann propagators for the wave operator, depending on if the limit is taken from above or below the real line. The explicit formulas are recorded in the following lemma. 
\begin{lemma}\label{lem:fs}
For $\Im \omega \neq 0$, a fundamental solution of $P(\omega)$ is given by the locally integrable function
\begin{equation}
    E_\omega(x) = c_\omega \log A(x, \omega), \qquad x \in \R^2 \setminus \{0\}
\end{equation}
where
\begin{equation}
    A(x, \omega) := \ell^+(x, \omega) \ell^-(x, \omega) = - \frac{x_1^2}{\omega^2} + \frac{x_2^2}{1 - \omega^2}, \qquad c_\omega := \frac{i \sgn \Im w}{4 \pi \omega \sqrt{1 - \omega^2}}.
\end{equation}
    The distributional limits $E_{\lambda \pm i0} := \lim_{\epsilon \to 0} E_{\lambda \pm i \epsilon}$, given by
\begin{equation}\label{eq:lim_fs}
    E_{\lambda \pm i0}(x) = \pm c_\lambda \log(A(x, \lambda) \pm i0), \qquad c_\lambda := \frac{i}{4 \pi \lambda \sqrt{1 - \lambda^2}}.
\end{equation}
are then fundamental solutions for $P(\lambda)$ for $\lambda \in (0, 1)$. 
\end{lemma}

In order to understand the spectrum near the forcing frequency $\lambda$, we will need some regularity in $\lambda$. It is clear that $E_\omega$ is holomorphic on the interior of the regions $(0, 1) \pm i [0, \infty)$. It turns out that $E_\omega$ is also $C^\infty$ up to the boundary in the following sense:
\begin{lemma}\label{lem:fs_holomorphic}
    For every $\varphi \in \CIc(\R^2)$, the map  given by the distributional pairing 
    \[\omega \mapsto (E_\omega, \varphi)\]
    lies in $C^\infty ((0, 1) \pm i[0, \infty))$ and is holomorphic on the interior $(0, 1) \pm i(0, \infty)$. 
\end{lemma}
We refer the reader to \cite[\S4.3]{DWZ} for the proofs of Lemmas \ref{lem:fs} and \ref{lem:fs_holomorphic}.

\subsection{Elliptic boundary value problem}
Consider the problem
\begin{equation}\label{eq:EBVP}
    P(\omega) u = f, \quad u|_{\partial \Omega} = 0, \quad \Re \omega \in (0, 1), \quad \Im \omega \neq 0, \quad f \in \CIc(\Omega),
\end{equation}
where $P(\omega)$ was defined in~\eqref{eq:P(omega)}, and $\Omega$ is a domain with straight corners. Before describing the reduction to boundary, we first need to establish some basic existence and uniqueness properties for the elliptic problem.  

\begin{lemma}\label{lem:basic_EU}
    Let $\Omega$ be a domain with Lipschitz and piecewise smooth boundary. The operator $P(\omega)$ is an isomorphism between $H^1_0(\Omega)$ and $H^{-1}(\Omega)$.
\end{lemma}
\begin{proof}
    Recall that $P(\omega) = (P - \omega^2) \Delta_\Omega^{-1}$. Since $P$ is self-adjoint (see \cite[Lemma~6]{Ralston_73}) and $\Im \omega^2 \neq 0$, we see that $P - \omega^2$ is invertible on $H^{-1}(\Omega)$, hence the lemma follows. 
\end{proof}
\begin{Remark}
    Of course, the direct proof using the Lax--Milgram lemma also works for establishing Lemma \ref{lem:basic_EU}. See for instance \cite[\S6.2]{evans10} and \cite[Chapter 6]{Lax_functional} for such arguments. It is just convenient here to use the self-adjointness of $P$. 
\end{Remark}

If $\partial \Omega$ is smooth, then it follows from standard elliptic boundary regularity theory that solutions to \eqref{eq:EBVP} are smooth up to the boundary. We also need to characterize the solutions near corners. To do so, it is convenient to blow up the corner analogous to the blow up of the double space $\R_+^2$ in \eqref{eq:R2_blowup}. Consider $\Omega \subset \R^2$ with Lipschitz and piecewise-smooth boundary. Let $\mathcal K \subset \partial \Omega$ be the set of points where $\partial \Omega$ fails to be smooth, i.e. the corners. The blow up of $\Omega$ is given by a blown-up domain $\widetilde \Omega$ and a blow-down map $\beta: \widetilde \Omega \to \overline \Omega$. The blown-up domain as a set is given by 
\begin{equation}\label{eq:Omega_tilde_set}
    \widetilde \Omega := (\overline \Omega \setminus \mathcal K) \sqcup S^+ (\mathcal K)
\end{equation}
where $S^+(\mathcal K)$ denotes the inward-pointing unit tangent vectors at $\mathcal K$. The blow-down map is given by 
\begin{equation}
    \beta(x) := \begin{cases} x & \text{if $x \in \overline \Omega \setminus \mathcal K$} \\ \kappa & \text{if $x \in S^+(\kappa)$, $\kappa \in \mathcal K$.} \end{cases}
\end{equation}
See Figure \ref{fig:omega_blowup}.
\begin{figure}
    \centering
    \includegraphics[scale = 0.6]{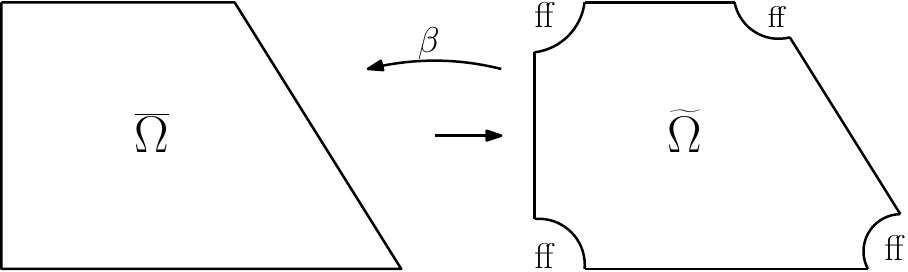}
    \caption{Diagram of the blown-up space $\widetilde \Omega$.}
    \label{fig:omega_blowup}
\end{figure}

It remains to give the natural topology and the smooth structure on $\widetilde \Omega$, and with this structure, $\beta$ is smooth. Let $\mathscr C(\Omega) \subset C^\infty([0, 1], \overline \Omega])$ be the subset of curves such that $\varphi \in \mathscr C(\Omega)$ if and only if $\varphi(0) \in \mathcal K$ and $|\varphi'(0)| = 1$. Note that $\varphi \in \mathscr C(\Omega)$ gives rise to the map 
\[\widetilde \varphi: [0, 1] \to \widetilde \Omega \quad \text{by} \quad \widetilde \varphi(0) = \varphi'(0), \quad \widetilde \varphi(t) = \varphi(t) \ \text{for} \ t > 0.\]
To give the topology, we say that $U \subset \widetilde \Omega$ is open if and only if $U$ intersects $\overline \Omega \setminus \mathcal K$ and $S^+(\mathcal K)$ in open sets, and for every $\varphi \in \mathscr C(\Omega)$ such that $\widetilde \varphi(0) \in U$, there exists a neighborhood $V \subset \mathscr C(\Omega)$ of $\varphi$ and $\epsilon > 0$ such that $\widetilde \psi(t) \in U$ for all $t \in [0, \epsilon)$ and $\psi \in V$. Now it remains to give the smooth structure near $\beta^{-1}(\kappa)$ for every $\kappa \in \mathcal K$. Let $x$ and $x'$ be two local boundary defining functions for the two boundary hypersurfaces near $\kappa$. Observe that the function $\tau = (x-x')/(x + x')$ extends to a continuous function on $\widetilde \Omega$. Smooth functions on $\widetilde \Omega$ are defined to be the functions generated by $\tau$ and $\beta^* C^\infty(\overline \Omega)$. 

The solution $u$ to~\eqref{eq:EBVP} is singular near the corners, but we can characterize this singularity when we lift $u$ to $\partial \Omega$ by the blow-down map $\beta^* u$. The idea is that we are making the geometry more complicated but the singularities are easily understood analytically in the more complicated geometry. Let $\mathcal V_b(\widetilde \Omega)$ denote the space of smooth vector fields in $\widetilde \Omega$ that are tangent to $\partial \widetilde \Omega$. In order to give the precise regularity of solution $u$ to~\eqref{eq:EBVP}, we will show that $V_1 \dots V_N u \in H^1_0(\Omega)$ for any $V_j \in \mathcal V_b(\widetilde \Omega)$ and $N \in \N$. First we need the following technical lemma on difference quotients. For $V \in \mathcal V_\bo(\widetilde \Omega)$, define the difference quotient
\begin{equation}\label{eq:DQ}
    D^h_V u(x) := \frac{u(\varphi_V^h(x)) - u(x)}{h}
\end{equation}
where $\varphi_V^h$ is the time $h$ flow of $V$. Clearly, since the vector field $V$ is tangent to $\partial \widetilde \Omega$, the difference quotient is well-defined and maps $H^1_0(\Omega) \to H^1_0(\Omega)$ and $H^{-1}(\Omega) \to H^{-1}(\Omega)$, although not uniformly as $h \to 0$. 

\begin{lemma}\label{lem:diff_quotient}
    Let $\Omega$ be a domain with straight corners and let $P$ be a constant coefficient second-order differential operator. There exists a finite collection of $\{V_j\}_{j \in J} \subset \mathcal V_\bo(\widetilde \Omega)$ such that $\{V_j\}_{j = 1}^J$ generates $\mathcal V_\bo(\widetilde \Omega)$ over $C^\infty(\widetilde \Omega)$ and
    \begin{equation}\label{eq:DQ_commutator}
        \|[P, D_{V_j}^h]\|_{H^1_0(\Omega) \to H^{-1}(\Omega)} \le C
    \end{equation}
    for all sufficiently small $h > 0$. 
\end{lemma}

\begin{Remarks}
    1. As $h \to 0$, the difference quotient approaches differentiation by $V$. In fact, it is easier to check $V \in \mathcal V_\bo(\widetilde \Omega)$, the commutator with $P$ has the mapping property
    \[[P, V]: H^1_0(\Omega) \to H^{-1}(\Omega).\]
    This can be done directly with an application of  the Poincar\'e inequality as in~\eqref{eq:poincare}. However, this mapping property is not enough, and we must go through the difference quotient. Unlike differentiation by $V$, taking the difference quotient does not drop the Sobolev order. To obtain regularity for the solution $u$ to~\eqref{eq:EBVP}, one might try to feed $Vu$ back into operator and study $P(\omega)(Vu)$. But we do not know a priori that $Vu \in H^1_0(\Omega)$, so we use the difference quotient instead. 

    \noindent
    2. Eventually, we will show that solutions to~\eqref{eq:EBVP} remains regular under repeated differentiation by vector fields that are only tangent to $\ff$ of the blow up, meaning solutions are smooth up to the boundary away from corners. However, differentiating (or taking the difference quotient) with respect to vector fields normal to the boundary kills the Dirichlet boundary condition. Normal regularity will thus be handled separately by using ellipticity of $P(\omega)$.  
\end{Remarks}

\begin{proof}
    1. If $V \in \mathcal V_\bo(\widetilde \Omega)$ is supported away from the corners, then~\eqref{eq:DQ_commutator} is obvious, so it suffices to consider vector fields supported near the corner. Without the loss of generality, we may assume that the corner is at $0$. By a linear change of coordinates, a local coordinate chart for the blow up $\widetilde \Omega$ near $\beta^{-1}(0)$ is then given by 
    \[r = x_2, \quad \tau = \frac{x_1}{x_2}, \quad 0 \le r < \delta, \quad -c_2 \le \tau \le c_1.\]
    In these coordinates, consider the vector fields
    \begin{equation}\label{eq:V1V2}
        V_1 := \chi(r, \tau) r \partial_r \quad \text{and} \quad V_2 = \chi(r, \tau) (c_1 - \tau)(c_2 + \tau) \partial_\tau
    \end{equation}
    where $\chi \in \CIc(\widetilde \Omega)$, $\chi = 1$ near $\beta^*(0)$, and $\supp \chi \subset \{r < \delta\}$. Clearly, $V_1$ and $V_2$ generate the vector fields $\mathcal V_\bo(\widetilde \Omega)$ supported in a sufficiently small neighborhood of $\beta^{-1}(0)$. It remains to show that
    \begin{equation}\label{eq:DQ_simple}
        [\partial_{x_j x_k} , D_{V_\ell}^h] :H^1_0(\Omega) \to H^{-1}(\Omega), \quad j, k, \ell = 1, 2
    \end{equation}
    uniformly for all sufficiently small $h$. Indeed, \eqref{eq:DQ_simple} recovers \eqref{eq:DQ_commutator} up to linear combinations and linear changes of coordinates. 

    \noindent
    2. To go between derivatives in $(x_1, x_2)$ coordinates and $(r, \tau)$ blow-up coordinates, note that 
    \[\partial_{x_1} = r^{-1} \partial_\tau, \quad \partial_{x_2} = \partial_r - \frac{\tau}{r} \partial_\tau, \quad \partial_{\tau} = x_2 \partial_{x_1}, \quad \partial_r = \frac{x_1}{x_2} \partial_{x_1} + \partial_{x_2}.\]
    Clearly, $r^{-1} \partial_\tau$ and $\partial_r$ are bounded operators from $H^1_0(\Omega)$ to $L^2(\Omega)$. By duality, they are also bounded from $L^2 \to H^{-1}(\Omega)$. 
    
    Let $u \in \CIc(\Omega)$ be such that $\supp u \cap \supp(1 - \chi) = \emptyset$. Then for all sufficiently small $h$, we have
    \begin{align*}
        [\partial_{x_1 x_1}, D_{V_1}^h] u &= \frac{r^{-2} u_{\tau \tau}(e^h r, \tau) - r^{-2} u_{\tau \tau}(r, \tau)}{h} \\
        &\qquad - \frac{e^{-h} r^{-2} u_{\tau \tau}(e^h r, \tau) + r^{-2} u_{\tau \tau}(r, \tau)}{h} \\
        &= \frac{1 - e^{-h}}{h} r^{-2} u_{\tau \tau}(e^h r, \tau) \\
        &= \frac{1 - e^{-h}}{h} (r^{-1} \partial_\tau)^2 (u \circ \varphi^h_{V_1}).
    \end{align*}
    Note that $u \mapsto u \circ \varphi^h_{V_1}$ is uniformly bounded from $H^1_0(\partial \Omega) \to H^1_0(\partial \Omega)$ for sufficiently small $h$. Therefore, for all sufficiently small $h$, 
    \begin{equation*}
        \|[\partial_{x_1 x_1}, D_{V_1}^h] u\|_{L^2(\Omega)} \le C\|u\|_{H^1_0(\Omega)}
    \end{equation*}
    The above inequality extends by density to $u \in H^1_0(\Omega)$ supported near the corner, and it is easy to see that it holds for $u$ supported away from $0$ since $V_1$ is only singular near the corner, so it follows that 
    \begin{equation}\label{eq:d11V1}
        \|[\partial_{x_1 x_1}, D_{V_1}^h]\|_{H^1_0(\Omega) \to H^{-1}(\Omega)} < C.
    \end{equation}
    Next, we similarly compute
    \begin{align}
        [\partial_{x_1 x_2}, D_{V_1}^h]u &= \frac{e^h - e^{-h}}{h} r^{-1} \partial_\tau \partial_r (u \circ \varphi^h_{V_1}) \nonumber \\
        &\qquad - \frac{1 - e^{-2h}}{h} \big((r^{-1} \partial_\tau) r^{-1} + \tau (r^{-1} \partial_\tau)^2 \big) (u \circ \varphi^h_{V_1}), \label{eq:d12V1_expansion} \\
        [\partial_{x_2 x_2}, D_{V_1}^h]u &= \frac{e^{2h} - 1}{h} \partial_r^2 (u \circ \varphi^h_{V_1}) - 2\frac{e^h - e^{-h}}{h} \tau r^{-1} \partial_\tau \partial_r (u \circ \varphi^h_{V_1}) \nonumber\\
        & \qquad + 2\frac{1 - e^{-2h}}{h} \big(\tau r^{-2} \partial_\tau + \tau^2 (r^{-1} \partial_\tau)^2 \big) (u \circ \varphi^h_{V_1}). \label{eq:d22V1_expansion}
    \end{align}
    Observe that by Poincar\`e inequality,
    \begin{align}
        \|r^{-1} u\|_{L^2(\Omega)} &\simeq \int_0^\infty \int_a^b |r^{-1} u(r, \tau)|^2 r \, dr d \theta \nonumber \\
        &\le C \int_0^\infty \int_a^b |r^{-1} \partial_\tau u(r, \tau)|^2 r \, dr d\theta \nonumber \\
        &\le \|u\|_{H^1_0}. \label{eq:poincare}
    \end{align}
    Here, `$\simeq$' denotes equivalence in norm. Therefore, it follows from \eqref{eq:d12V1_expansion} that
    \begin{equation}\label{eq:d12V1}
        \|[\partial_{x_j x_2}, D_{V_1}^h]\|_{H^1_0(\Omega) \to H^{-1}(\Omega)} < C, \quad j = 1, 2.
    \end{equation}

    \noindent
    3. Finally, it remains to check the commutator with the vector field $V_2$ defined in \eqref{eq:V1V2}. Let $\widetilde \varphi^h$ denote the time $h$ (one-dimensional) flow of $(c_1 - \tau)(c_2 + \tau) \partial_\tau$. Note that 
    \[\widetilde \varphi^\bullet \in C^\infty([0, \infty)_h \times [-c_1, c_2]_\tau),\]
    and since $\tilde \varphi^0(\tau) = \tau$, we see that both 
    \begin{equation}\label{eq:diff_flow_bdd}
        h^{-1} (1 - \partial_\tau \widetilde \varphi^h(\tau)),\  h^{-1} \partial_{\tau}^2 \widetilde \varphi^h(\tau) \in C^\infty([0, \infty)_h \times [-c_1, c_2]_\tau).
    \end{equation}
    For $u \in \CIc(\Omega)$ such that $\supp u \cap \supp (1 - \chi) = \emptyset$, we compute
    \begin{equation}\label{eq:d11V2_expansion}
        [\partial_{x_1 x_1}, D_{V_2}^h] u = \frac{(\partial_\tau \widetilde \varphi^h(\tau))^2 - 1}{h} r^{-1} u_{\tau \tau}(r, \widetilde \varphi^h(\tau)) + \frac{\partial_\tau^2 \widetilde \varphi^h(\tau)}{h} r^{-2} u_\tau(r, \widetilde \varphi^h(\tau)).
    \end{equation}
    Observe that by \eqref{eq:poincare} and \eqref{eq:diff_flow_bdd}, we have that for all sufficiently small $h$, 
    \begin{equation*}
        \|[\partial_{x_1x_2}, D_{V_2}^h] u\|_{L^2(\Omega)} \le C\|u\|_{H^1_0(\Omega)}.
    \end{equation*}
    Again, $V_2$ is only singular near the corner, so the above holds for $u$ supported away from the corner. Therefore, 
    \begin{equation}\label{eq:d11V2}
        \|[\partial_{x_1x_1}, D_{V_2}^h]\|_{H^1_0(\Omega) \to H^{-1}(\Omega)} < C.
    \end{equation}
    The same inequality holds for the commutator of $D_{V_2}^h$ with $\partial_{x_1 x_2}$ and $\partial_{x_2 x_2}$ following similar computations to \eqref{eq:d11V2_expansion}, only involving more terms as in \eqref{eq:d12V1_expansion} and \eqref{eq:d22V1_expansion}.
\end{proof}

Now we obtain a polyhomogeneous expansion for $u$ near the corner in the blown-up coordinates. This type of result is rather standard, but our setting is special enough to warrant a self-contained proof. See, for instance, Mazzeo \cite[Corollary 4.19]{Mazzeo} for the polyhomogenous expansion near the boundary of solutions to b-elliptic differential equations, and see \cite[\S 5.2]{Hintz_notes} for the polyhomogeneous expansion near corners of Laplace's equation on polygonal domains.  

\begin{proposition}\label{prop:polyhom_EU}
    If $u$ is a solution to~\eqref{eq:EBVP} for $\omega = \lambda + i\epsilon$, then $\beta^* u$ is smooth at the lift of $\partial \Omega \setminus \mathcal K$. Furthermore, if $\kappa \in \mathcal K$ has characteristic ratio $\alpha(\lambda)$ (see Definition~\ref{def:corner_type}), then the pullback by the blow-down map $\beta^* u$ is polyhomogeneous conormal to $\ff$ with expansion
    \begin{equation}\label{eq:u_expansion}
        u(r, \tau) \sim \sum_{k = 1}^\infty r^{\mathfrak{l}_\omega k} w_k(\tau), \quad w_k \in C^\infty([-\alpha_-, \alpha_+]), \quad w_k(-\alpha_-) = w_k(\alpha_+) = 0
    \end{equation}
    near $\beta^{-1}(\kappa)$ where $(r, \tau)$ are the blowup coordinates given in~\eqref{eq:good_corner_coords} and
    \begin{equation}\label{eq:conormal_indices}
        \mathfrak{l}_\omega = \frac{2 \pi i}{i \pi - \log \alpha(\lambda)} + \mathcal O(\epsilon)
    \end{equation}
\end{proposition}
\begin{proof}
1. If $u$ is a solution, then $u \in C^\infty(\Omega)$ by elliptic regularity, so it remains to check smoothness \textit{up to} the boundary and conormality near the corner. 

Let $\mathcal S \subset \mathcal V_\bo(\widetilde \Omega)$ be the collection of vector fields given in Lemma~\ref{lem:diff_quotient}, and let $V \in \mathcal S$. Observe that
\begin{equation}\label{eq:PDu}
    P(\omega) (D^h_{V} u) = D^h_{V} f + [P(\omega), D^h_{V}] u.
\end{equation}
Clearly, 
\[D^h_V f \in \CIc(\Omega)\]
uniformly. Therefore, by Lemma~\ref{lem:diff_quotient}, the right hand side of~\eqref{eq:PDu} is uniformly bounded in $H^{-1}(\Omega)$ for all sufficiently small $h$. Since $P(\omega):H^1_0(\Omega) \to H^{-1}(\Omega)$ is an isomorphism, we conclude that $D^h_Vu$ is uniformly bounded in $H^1_0$. Since $\mathcal S$ generates $\mathcal V_\bo(\widetilde \Omega)$, it follows that 
\[Vu \in H^1_0(\Omega) \quad \text{for all} \quad V \in \mathcal V_\bo(\widetilde \Omega).\]
Now $P(\omega)Vu \in H^{-1}(\Omega)$. Iterating this procedure, we see that 
\[V_1\cdots V_k u \in H^1_0 \quad \text{for all} \quad V_1, \dots, V_k \in \mathcal V_\bo(\widetilde \Omega)\]

\noindent 
2. Next, we show that $V_1 \dots V_k u \in L^2(\Omega)$ for for all $V_1, \dots, V_k \in \mathcal V_\ff(\widetilde \Omega)$, that is, smooth vector fields in $\widetilde \Omega$ tangent to $\ff$. After step 1, it remains to verify regularity in the normal direction. We first work away from corners. Let $(y_1, y_2)$ be local coordinates centered at $x \in \partial \Omega \setminus \mathcal K$ so that $y_2$ is a boundary defining function and $\partial_{y_1}$ is tangent to the boundary. Let $\chi \in \CIc(\overline \Omega)$ be a cutoff with $\chi = 1$ near $x$ with sufficiently small support in the coordinate neighborhood. Since $P(\omega)$ is elliptic and $P(\omega)u(y) = 0$ for $y$ near the boundary, 
\[\chi(y) \partial_{y_2}^2 u(y) = \chi(y) \sum_{\substack{|\alpha| \le 2 \\ \alpha \neq (0, 2)}} c_{\alpha}(y) \partial_y^\alpha u(y)\]
for some $c_\alpha \in C^\infty(\overline \Omega)$. By Step 1, the right hand side lies in $L^2(\Omega)$, so $\chi(y) \partial_{y_2}^2 u \in L^2(\Omega)$. Similarly, we also find that $\partial_{y_2}^2 \partial_{y_1}^k u \in L^2(\Omega)$ for all $k \in \N$. Higher derivatives of $\partial_{y_2}$ are handled inductively. Assume that $\chi(y) \partial_{y_1}^k \partial_{y_2}^n u \in L^2(\Omega)$ for some $n \in \N$ and all $k \in \N$. Then using 
\[\chi(y) \partial_{y_1}^k \partial_{y_2}^{n + 1} u(y) = \chi(y)\sum_{\substack{|\alpha| \le 2 \\ \alpha \neq (0, 2)}} \partial_y^\alpha \partial_{y_1}^k \partial_{y_2}^{n - 1} u(y),\]
we see that $\chi(y) \partial_{y_1}^k \partial_{y_2}^{n + 1} u \in L^2(\Omega)$, completing the induction. 

Now we work near a corner, which we may place at the origin. Since the corner is straight, we can choose linear coordinates $(y_1, y_2)$ so that 
\[\{(y_1, 0): y_1 \in [0, \delta]\} \cup \{(0, y_2): y_2 \in [0, \delta]\} \subset \partial \Omega\] 
for some sufficiently small $\delta > 0$. Then in a neighborhood of the corner, the vector field
\[\chi(y_2/ y_1) y_1 \partial_{y_2}, \quad \chi \in \CIc([0, \infty)), \quad \chi \equiv 1 \quad \text{near} \quad 0\]
lifts to a vector field up tangent to $\ff$, nonvanishing near $\ff \cap \beta^{-1}(\{(y_1, y_2): y_2 = 0, y_1 > 0\})$, and is transversal to the piece of boundary at $\beta^{-1}(\{(y_1, y_2): y_2 = 0, y_1 > 0\})$. Again, by ellipticity, 
\[\chi(y_2/ y_1) y_1^2 \partial_{y_2}^2 u = \chi (y_2/y_1)  \sum_{\substack{|\alpha| \le 2 \\ \alpha \neq (0, 2)}} c_\alpha y_1^2 \partial^\alpha u(y) \quad \text{for} \quad y_1 \in [0, \delta]\]
and some constants $c_\alpha \in \C$. Observe that $\chi(y_2/y_1) y_1 \partial_{y_1} \in \mathcal V_\bo(\widetilde \Omega)$, so it follows from Lemma~\ref{lem:diff_quotient} that the right hand side lies in $L^2(\Omega)$, so normal derivatives near the corner still lies in $L^2(\Omega)$. Higher normal derivatives are again handled inductively. Therefore, taking $\overline C^\infty(\widetilde \Omega)$-linear combination, we indeed see that
\begin{equation}\label{eq:u_conormal}
    V_1 \dots V_k u \in L^2(\Omega) \quad \text{for all} \quad V_1, \dots, V_k \in \mathcal V_\ff(\widetilde \Omega).
\end{equation}
\noindent
3. It remains to show that we have a polyhomogeneous expansion. Without the loss of generality we assume that the corner is at $0$, and we first consider the case that $0$ is a type-$(+, +)$ corner. Blowing up the corner at $0$, the coordinates
\begin{equation}\label{eq:bcoords}
    r = \ell^-(\bullet, \lambda), \quad \tau = \frac{\ell^+(\bullet, \lambda)}{\ell^-(\bullet, \lambda)}, \quad (r, \tau) \in U := [0, \delta)_r \times [-\alpha_-, \alpha_+]_\tau
\end{equation}
for a sufficiently small $\delta$ give a local chart for $\widetilde \Omega$ near $\ff_0 := \beta^{-1}(0)$. Let $\chi \in C_c^\infty(\widetilde \Omega)$ be such that $\chi \equiv 1$ near $\ff_0$ and $\chi$ is supported in the coordinate patch of~\eqref{eq:bcoords}. Then 
\begin{equation}\label{eq:Pchiu}
    r^2 P(\omega)(\chi u) = r^2 [P(\omega), \chi] u + r^2 \chi P(\omega)u
\end{equation}
Note that both sides are supported in $U$, so we may consider~\eqref{eq:Pchiu} on $[0, \infty)_r \times [-\alpha_-, \alpha_+]_\tau$. Taking the Mellin transform of the right hand side in $r$ 
\begin{equation}
    \mathcal M_r([P(\omega), \chi] u + \chi Pu) \in \mathscr O(\C; C^\infty([-\alpha_-, \alpha_+]_\tau)),
\end{equation}
since the right hand side of~\eqref{eq:Pchiu} is clearly supported away from $\ff_0$. 

Now we examine the left hand side of~\eqref{eq:Pchiu}. For simplicity, write $\ell^\pm := \ell^\pm( \bullet, \lambda)$. Recall from~\eqref{eq:L_factor} that we have $P(\omega) = L^+_\omega L^-_\omega$, and
\begin{align}
    L^\mu_\omega \ell^\nu &= \frac{1}{2} \left( \mu \nu \frac{\lambda + i\epsilon}{\lambda} + \sqrt{\frac{1 - (\lambda + i\epsilon)^2}{1 - \lambda^2}} \right) \nonumber\\
    &= \frac{1}{2} \left( \mu \nu\cdot 1 + 1+ i \epsilon\left( \frac{\mu \nu}{\lambda} - \frac{\lambda}{1 - \lambda^2} \right) \right) + \mathcal O(\epsilon^2) \label{eq:Ll_formula}
\end{align}
for $\mu, \nu \in \{+, -\}$. Note that $L^\mu_\omega \ell^\nu$ does not depend on $x$. Then writing $L^\pm_\omega$ in $(r, \tau)$-coordinates, we have
\begin{equation}\label{eq:blow_up_L}
    L_\omega^\pm = (L_\omega^\pm\ell^-) \partial_r + [(L_\omega^\pm \ell^+) - (L^\pm_\omega \ell^-) \tau] r^{-1} \partial_\tau 
\end{equation}
for $(r, \tau) \in U$. Note that $[r, L^\pm_\omega] = -L^\pm_\omega \ell^-$ and $L^\pm$ commute with each other, so 
\begin{equation}\label{eq:ind_factor}
\begin{gathered}
    r^2P(\omega) = (r L^+_\omega - L_\omega^+ \ell^-) (r L^-_\omega), \\
    r^2P(\omega) = (r L^-_\omega - L_\omega^- \ell^-) (r L^+_\omega), 
\end{gathered}
\end{equation}
are two equivalent formulas for $r^2P(\omega)$ near $\ff_0$. Observe that by~\eqref{eq:u_conormal},
\[\mathcal M_r(\chi u)(\sigma, \tau) \in \mathscr O(\{\Im \sigma > 1\}; C^\infty([-\alpha_-, \alpha_+]_\tau)).\]
Therefore, the Mellin transform of the left hand side of~\eqref{eq:Pchiu} is given by
\begin{equation}
    \mathcal M_r(r^2 P(\omega)(\chi u))(\sigma, \tau) = I(\sigma) \mathcal M_r(\chi u)(\sigma, \tau), \quad \Im \sigma > 1
\end{equation}
where $I(\sigma)$ is easily computed using~\eqref{eq:blow_up_L} and~\eqref{eq:ind_factor}. We use the convenient change of variables $s = i \sigma$. Then two equivalent formulas for $I(\sigma)$ are given by
\begin{equation}\label{eq:corner_ind_fam}
    \begin{gathered}
        I(-is) = \big((L^+_\omega \ell^-)(s - 1) + (L^+_\omega \ell^+ - (L^+_\omega \ell^-) \tau) \partial_\tau \big)\big((L^-_\omega \ell^-) s + (L^-_\omega \ell^+ - (L^-_\omega \ell^-)\tau)\partial_\tau \big), \\
        I(-is) = \big((L^-_\omega \ell^-)(s - 1) + (L^-_\omega \ell^+ - (L^-_\omega \ell^-) \tau) \partial_\tau \big)\big((L^+_\omega \ell^-) s + (L^+_\omega \ell^+ - (L^+_\omega \ell^-)\tau)\partial_\tau \big).
    \end{gathered}
\end{equation}
In particular, $I(\sigma):H^1_0([-\alpha_-, \alpha_+]) \to H^{-1}([-\alpha_-, \alpha_+])$ is a holomorphic family of Fredholm operators with index $0$, and $I(\sigma)$ is noninvertible if and only if there exists nonzero $w(\tau) \in C^\infty([-\alpha_-, \alpha_+])$ such that 
\begin{equation}\label{eq:corner_ind_roots_eq}
    I(\sigma) w(\tau) = 0, \qquad w(-\alpha_-) = w(\alpha_+) = 0.
\end{equation}
From~\eqref{eq:Ll_formula}, we have
\begin{equation}\label{eq:Ll}
    \begin{gathered}
    L^-_\omega \ell^+ - (L^-_\omega \ell^-)\tau = -\tau - i \epsilon \left[ \left(\frac{1}{\lambda} + \frac{\lambda}{1 - \lambda^2} \right) + \tau \left(\frac{1}{\lambda} - \frac{\lambda}{1 - \lambda^2} \right) \right] + (1 - \tau) \mathcal O(\epsilon^2), \\
    L^+_\omega \ell^+ - (L^+_\omega \ell^-)\tau = 1 + i \epsilon \left[\left(\frac{1}{\lambda} - \frac{\lambda}{1 - \lambda^2} \right) + \tau \left(\frac{1}{\lambda} + \frac{\lambda}{1 - \lambda^2} \right)\right]+ (1 - \tau)\mathcal O(\epsilon^2).
    \end{gathered}
\end{equation}
In particular, for $\tau \in [-\alpha_-, \alpha_+]$ and all sufficiently small $\epsilon$, the function in~\eqref{eq:Ll} never takes value in $i[0, \infty)$. Therefore, using the two different factorizations of $I(\sigma)$ in~\eqref{eq:corner_ind_fam}, solutions to the ODE $I(\sigma) w(\tau) = 0$ must be linear combinations
\begin{equation}\label{eq:ode_solutions}
    w(\tau) = c_1(L^-_\omega \ell^+ - (L^-_\omega \ell^-)\tau)^s + c_2(L^+_\omega \ell^+ - (L^+_\omega \ell^-)\tau)^s, \quad s = i \sigma
\end{equation}
where $z^s = \exp(s \log z)$ is taken with the branch of $\log$ on $\C \setminus i[0, \infty)$ that agrees with the real-valued $\log$ on the positive real line. It is easy to verify that this makes sense for all sufficiently small $\epsilon = \Im \omega$ in view of~\eqref{eq:Ll}. Furthermore, $w$ must satisfy the boundary conditions in~\eqref{eq:corner_ind_roots_eq}, and we see that $I(\sigma)$ is noninvertible if and only if $s = i \sigma$ satisfies
\begin{equation}\label{eq:corner_ind_roots}
\det
\begin{pmatrix}
    (L^-_\omega \ell^+ - (L^-_\omega \ell^-)\alpha_+)^s & (L^+_\omega \ell^+ - (L^+_\omega \ell^-) \alpha_+)^s \\
    (L^-_\omega \ell^+ + (L^-_\omega \ell^-)\alpha_-)^s & (L^+_\omega \ell^+ + (L^+_\omega \ell^-) \alpha_-)^s 
\end{pmatrix}
= 0, \quad s \neq 0.
\end{equation}
Therefore, for all sufficiently small $\epsilon = \Im \omega$, $I(s)$ is noninvertible if and only if
\[s \log \left(\frac{(-L^-_\omega \ell^+ + (L^-_\omega \ell^-)\alpha_+)(L^+_\omega \ell^+ + (L^+_\omega \ell^-) \alpha_-)}{(L^-_\omega \ell^+ + (L^-_\omega \ell^-)\alpha_-)(L^+_\omega \ell^+ - (L^+_\omega \ell^-) \alpha_+)} \right) - is\pi \in 2 \pi i \Z \setminus \{0\}.\]
Define 
\begin{align*}
    \mathfrak{l}_\omega:=& 2 \pi i \left( \log \left(\frac{(L^+_\omega \ell^+ + (L^+_\omega \ell^-) \alpha_-^{-1})(-L^-_\omega \ell^+ + (L^-_\omega \ell^-)\alpha_+^{-1})}{(L^-_\omega \ell^+ + (L^-_\omega \ell^-)\alpha_-^{-1})(L^+_\omega \ell^+ - (L^+_\omega \ell^-) \alpha_+^{-1})} \right) - i\pi \right)^{-1} \\
    =& \frac{2 \pi i}{i\pi - \log \alpha} + \mathcal O(\epsilon).
\end{align*}
The second equality above follows from~\eqref{eq:Ll}. In particular, observe that $I(-is)$ is noninvertible if and only if $s \in \mathfrak{l}_\omega \Z$ and that
\begin{equation}
    \Re \mathfrak{l}_\omega = \frac{2 \pi^2}{(\log \alpha)^2 + \pi^2} + \mathcal O(\epsilon).
\end{equation}

From~\eqref{eq:Pchiu}, we see that 
\begin{equation}\label{eq:mellin_comparison}
    \mathcal M_r(\chi u)(\sigma, \tau) = I(\sigma)^{-1} \mathcal M_r([P(\omega), \chi] u + \chi Pu),
\end{equation}
so for all sufficiently small $\epsilon$, and we see from~\eqref{eq:u_conormal} that $\mathcal M_r(\chi u)(\sigma, \tau)$ is holomorphic in $\sigma$ for $\Im \sigma > 1$. Therefore, $\mathcal M_r(\chi u)(\sigma, \tau)$ is meromorphic with possible poles in $\{\sigma = \Z \mathfrak{l}_\omega : \Im \sigma \le 1\}$. Since $\chi u \in H^1_0(\Omega)$, the expansion~\eqref{eq:u_expansion} follows from Proposition~\ref{lem:mellin_polyhom}.

\noindent
4. Finally, for the other types of corners, note that if $u$ is a solution to~\eqref{eq:EBVP} in $\Omega$, then $\mathrm{Ref}_j^* u$ is a solution to~\eqref{eq:EBVP} in $\mathrm{Ref}_j(\Omega)$. The result for the other corner types then follows by the remark following Definition~\ref{def:corner_type}. 
\end{proof}

\begin{Remark}
    The above proposition also gives heuristic justification for why we expect singularities to form along a ray coming out of the corner. If one examines the solutions~\eqref{eq:ode_solutions}, we see that as $\epsilon \to 0+$, the functions $w_k$ in the polyhomogeneous expansions~\eqref{eq:u_expansion} forms singularities at $\tau = 0$. This corresponds precisely to the special trajectory $\sr$ in Theorem~\ref{thm:spectral}.
\end{Remark}

\subsection{Distributions on \texorpdfstring{$\partial \Omega$}{partial Omega}} Here, we take a break to discuss the spaces of distributions that will appear when we do the reduction to boundary via single layer potentials. 
\subsubsection{b-Sobolev spaces on \texorpdfstring{$\partial \Omega$}{partial Omega} and \texorpdfstring{$\widetilde \Omega$}{Omega}}
Observe that $\partial \Omega$ can be understood as the union of one-dimensional manifolds with boundary, with the boundary smooth structure inherited from the canonical embedding $\partial \Omega \hookrightarrow \R^2$. By an abuse of notation, we also denote by $\partial \Omega$ the disjoint union of these one-dimensional manifolds. Then the spaces of supported and extendible distributions are well-defined, which we denote by $\dot{\mathcal D}'(\partial \Omega)$ and $\overline{\mathcal D}'(\partial \Omega)$ respectively. Similarly, we have extendible and supported Sobolev spaces and smooth functions. Most importantly, we need b-Sobolev spaces on $\partial \Omega$. Since $\partial \Omega$ consists of compact one-dimensional manifolds with a metric inherited from $\R^2$, b-Sobolev spaces can simply be defined via charts, that is, 
\[v \in H^{s, a}_\bo(\partial \Omega) \iff \varphi_j^* v \in H^{s, a}_\bo (\R_+), \quad \|v\|_{H^{s, a}_\bo(\partial \Omega)} = \sum_{j = 1}^J \|\varphi_j^* v\|_{H^{s, a}_\bo(\R_+)}\]
where $\varphi_j: \R^+ \supset U \to V \subset \partial \Omega$ is a finite collection of charts and $H^{s, a}_\bo(\R_+)$ is as defined in \eqref{eq:b-sob}. 

For particular b-Sobolev spaces on $\partial \Omega$ as a union of one-dimensional manifolds, they can be related to functions on $\partial \Omega$ as a homeomorphic copy of $\mathbb S^1$. Clearly, the spaces $\dot C^\infty(\partial \Omega)$ and $\overline C^\infty (\partial \Omega)$ can be trivially identified with a subset of $L^\infty$ functions on $\partial \Omega \simeq \mathbb S^1$. Let $\rho_\bo \in \overline C^\infty(\partial \Omega; \R)$ be such that 
\begin{enumerate}
    \item $\rho_\bo(\theta) > 0$ for $\theta \in \partial \Omega \setminus \mathcal K$,

    \item $\rho_\bo(\kappa_j) = 0$ for all $j$, and $\rho_\bo' \ge \delta > 0$ in neighborhoods of $\kappa_j$.  
\end{enumerate}
Here, $\rho_\bo$ is analogous to a boundary defining function. Observe that for $a \in (-1, 0)$, we have
\begin{equation}\label{eq:Hbinfty_bd}
\begin{aligned}
    H^{\infty, a}_\bo(\partial \Omega) :=& \bigcap_{s \in \R} H^{s, a}_\bo(\partial \Omega) \\
    =& \left\{v \in H^{a + \ha}(\partial \Omega): \rho_\bo^{- a -\ha}(\rho_\bo \partial_\theta)^k v \in L^2(\partial \Omega) \ \forall \ k \in \N \right\}
\end{aligned}
\end{equation}
Indeed, this follows from Lemma \ref{lem:b_sob_to_sob}, and also note that there is no need to distinguish between $\dot H^{a + \ha}(\partial \Omega)$ and $\bar H^{a + \ha}(\partial \Omega)$ for $a \in (-1, 0)$. Finally, we remark that it is helpful to observe that the numerology is set so that 
\[\rho_\bo^{a + \epsilon} \in H^{\infty, a}_\bo(\Omega) \quad \text{for all} \quad \epsilon > 0.\]
Furthermore, for $a > -1$, $H^{\infty, a}_\bo(\partial \Omega) \subset L^1(\partial \Omega)$ since
\begin{equation}
    \|v\|_{L^1} \le \|\rho_\bo^{-a - \ha} v\|_{L^2} \|\rho_\bo^{a + \ha}\|_{L^2} < \infty
\end{equation}

\subsubsection{Conormal distributions on \texorpdfstring{$\widetilde \Omega$}{widetilde Omega}}
Next, we consider subspaces of extendible distributions of $\overline \Omega$. Note that the space of extendible distributions over $\Omega$ and the space of extendible distributions over the blown up space $\widetilde \Omega$ are isomorphic. There are more smooth vector fields over $\widetilde \Omega$, so we work over $\widetilde \Omega$ to better capture singularities at the corners. Let $\rho_\ff \in \overline C^\infty(\widetilde \Omega)$ be a boundary defining function for the front face $\ff \subset \widetilde \Omega$ of the blow up. We equip $\widetilde \Omega$ with a measure $\mu_\ff$ which satisfies 
\begin{equation}
    \beta_*(\rho_{\ff} \mu_\ff) = \mu_{\mathrm{Leb}}
\end{equation}
where $\beta$ is the blow-down map and $\mu_{\mathrm{Leb}}$ is the measure on $\Omega \subset \R^2$ inherited from the Lebesgue measure on $\R^2$. In the blown-up coordinates $(r, \tau)$ given in the remark following Definition~\ref{def:corner_type}, we see that $\mu_{\ff} \simeq dr d \tau$ near the corners.


Now we can define the conormal distributions on $\widetilde \Omega$. Let $M_1, \dots, M_n$ be a collection of relatively closed one-dimensional submanifolds of $\mathcal N$, and assume that they may intersect pairwise transversally. Denote by 
\[\mathcal V_{M_1, \dots, M_n}(\widetilde \Omega) \subset \mathcal V(\widetilde \Omega)\]
the subset of smooth vector fields on $\widetilde \Omega$ tangent to $M_1, \dots, M_n$. Then define
\begin{multline}\label{eq:conormal_def}
    \mathcal A^{[s]}(\widetilde \Omega; M_1, \dots, M_n) := \big \{u \in \overline{\mathcal D}'(\widetilde \Omega) : \\
    V_1\dots V_k u \in \bar H^s(\widetilde \Omega, \mu_\ff), \ V_j \in \mathcal V_{M_1, \dots, M_n}(\widetilde \Omega), \ k \in \N_0 \big \}.
\end{multline}
For instance, we will use the space $\mathcal A^{[s]}(\widetilde \Omega; \ff, \sr)$ where $\sr$ the trajectories coming out of the corners defined in \eqref{eq:sr_def}. Of course, we can define conormal spaces over any manifold with measure given a collection of submanifolds. In particular, we will also encounter conormal spaces on $\R_+$, $\partial \Omega$, and $\Omega$ with the naturally induced Lebesgue measure on all of them.

Finally, we define b-Sobolev spaces with respect to $\ff$ so that we can define a sensible ``Neumann data'' operator near the corner later. Let $\mathcal V_{\ff}(\widetilde \Omega) \subset \mathcal V(\widetilde \Omega)$ be the smooth vector fields in $\widetilde \Omega$ tangent to $\ff$. Define
\begin{equation}\label{eq:H^infty_ff}
    H_{\ff}^{\infty, a}(\widetilde \Omega) := \left\{u \in \overline{\mathcal D}'(\widetilde \Omega): V_1\cdots V_k \rho_\ff^{-a}u \in L^2(\widetilde \Omega, \rho_\ff^{-1} \mu_\ff), \ V_j \in \mathcal V_{\ff}(\widetilde \Omega), \ k \in \N_0 \right \}.
\end{equation}
The numerology in \eqref{eq:H^infty_ff} is chosen precisely so that 
\[\rho_\ff^{a + \epsilon} \in H_{\ff}^{\infty, a}(\Omega) \quad \text{for all} \quad \epsilon > 0.\]
Comparing \eqref{eq:H^infty_ff} and \eqref{eq:conormal_def}, we see that
\begin{equation}
    H^{\infty, a}_\ff(\widetilde \Omega) = \mathcal A^{[a + 1/2]}(\widetilde \Omega; \ff) \quad \text{for} \quad a \in (-1, 0)
\end{equation}
This helps us go between $\bo$-Sobolev spaces and conormal distributions. Going back and forth between these spaces is sometimes useful since b-pseudodifferential operstors act on b-Sobolev spaces in a nice way and pseudodifferential operators act on conormal distributions in a nice way.

\subsection{Single layer potential}
Equipped with the polyhomogeneous conormal expansion from Proposition \ref{prop:polyhom_EU}, we show that reduction to boundary via the single layer potential is still valid. See \cite{H3} for a detailed study of elliptic boundary value problems in domains with smooth boundary via the single layer potential reduction to boundary.

First observe the following
\begin{lemma}\label{lem:Lu_restriction}
    Let $a > -1$. Let $V$ be a smooth vector field on $\Omega$ such that $\mathbf v := V|_{\partial \Omega}$ is inward-pointing. Then 
    \begin{equation}
        (L^+_\omega u)|_{\partial \Omega}(\theta) := \lim_{\delta \to 0} L^+ u(\theta + \delta \mathbf v)
    \end{equation}
    is well-defined as an operator $(L^+_\omega \bullet)|_{\partial \Omega} : H^{\infty, a+ 1}_\ff \to H^{\infty, a}_\bo(\partial \Omega)$. 
\end{lemma}
\begin{proof}
    Observe that from the explicit formulas computed in \eqref{eq:blow_up_L}, $L^+_\omega$ is a $\overline C^\infty(\widetilde \Omega)$-linear combination of differential operators given by vector fields in $\rho^{-1}_\ff \mathcal V_\ff(\widetilde \Omega)$. Indeed, this is clear away from the corners, and in a neighborhood of a corner, say of type-$(+, +)$, we see that $L^\pm_\omega$ is a $\overline C^\infty(\widetilde \Omega)$-linear combination of $\partial_r$ and $r^{-1} \partial_\tau$ where
    \[r = \ell^-(\bullet, \lambda), \quad \tau = \frac{\ell^+(\bullet, \lambda)}{\ell^-(\bullet, \lambda)}.\]
    Therefore, 
    \[L^+_\omega u: H^{\infty, a + 1}_\ff(\Omega) \to H^{\infty, a}_\ff(\Omega).\]
    The lemma follows upon restricting to the boundary. 
\end{proof}

In view of Lemma~\ref{lem:Lu_restriction}, we may define the ``Neumann data" operator
\begin{equation}\label{eq:neumann_def}
    \mathcal N_\omega: H^{\infty, a + 1}_{\ff}(\Omega) \to H^{\infty, a}_\bo(\partial \Omega; T^*\partial \Omega) \subset L^1(\partial \Omega; T^*\partial \Omega)
\end{equation}
for $a > -1$ by
\begin{equation}
    \mathcal N_\omega u := -2 \omega \sqrt{1 - \omega^2} \mathbf j^*(L^+_\omega u d \ell^+(\bullet, \omega)).
\end{equation}
where $\mathbf j : \partial \Omega \to \overline \Omega$ is the embedding of the boundary. 

Also define $\mathcal I: L^1(\partial \Omega; T^*\partial \Omega) \to \mathcal E'(\R^2)$ by 
\begin{equation}
    \int_{\R^2} \mathcal I v(x) \varphi(x) \, dx := \int_\R \varphi v \quad \text{for all} \quad \varphi \in \CIc(\R^2).
\end{equation}
The left hand side above is understood as a distributional pairing. Essentially, $\mathcal I$ can be thought of as tensoring with a delta distribution supported on $\partial \Omega$.
\begin{lemma}\label{lem:fundamental}
    Let $\omega \in (0, 1) + i(0, \infty)$. Let $u \in H^1_0(\Omega)$ be the solution to \eqref{eq:EBVP} for some $f \in \CIc(\Omega)$. Put $U := \indic_\Omega u \in \mathcal E'(\R^2)$ and $v := \mathcal N_\omega u$. Then 
    \begin{align}
        P(\omega) U &= f - \mathcal I v, \label{eq:fundamental_1}\\
        U &= R_\omega f - R_\omega \mathcal I v \label{eq:fundamental_2}
    \end{align}
\end{lemma}

\begin{proof}
    Let $\varphi \in \CIc(\R^2)$. Then since $u \in H^1_0(\Omega)$, we can integrate by parts to find
    \begin{align*}
        \int_{\R^2} (P(\omega) U) \varphi \, dx &= 4 \int_{\Omega} u L^+_\omega L^-_\omega \varphi \, dx \\
        &= -4 \int_{\Omega} (L^+_\omega u) (L^-_\omega \varphi)\, dx \\
        &= \int_\Omega f \varphi\, dx - 4 \int_{\Omega} L^-_\omega (\varphi L^+_\omega u)\, dx
    \end{align*}
    Observe that 
    \[2 \omega \sqrt{1 - \omega^2} d(\varphi L_\omega^+ u d \ell^+) = -4L^-_\omega(\varphi L^+_\omega u) \, dx.\]
    By Proposition \ref{prop:polyhom_EU} and Lemma \ref{lem:Lu_restriction}, $\varphi L^+_\omega u \, d\ell^+ \in C^\infty(\Omega; T^* \Omega)$ and restricts to $L^1(\partial \Omega)$. Therefore, we can apply Stoke's theorem to conclude
    \[\int_{\R^2} (P(\omega) U) \varphi \, dx = \int (f \varphi - \mathcal I \mathcal N_\omega u) \, dx,\]
    which gives~\eqref{eq:fundamental_1}. Finally, since $U$ is a compactly supported distribution, \eqref{eq:fundamental_2} follows immediately from \eqref{eq:fundamental_1}. 
\end{proof}
 
We are thus motivated to define the single layer potential
\begin{equation}\label{eq:single_layer_op}
    S_\omega: L^1(\partial \Omega; T^* \partial \Omega) \to \overline{\mathcal D}'(\Omega) \quad \text{by} \quad S_\omega v := (R_\omega \mathcal I v)|_{\Omega}.
\end{equation}
The above mapping property is clear since $R_\omega$ is a convolution operator and $\mathcal I v$ is compactly supported.

\subsubsection{Non-real case mapping properties}
Observe that by restricting \eqref{eq:fundamental_2} to $\partial \Omega$, we effectively reduce \eqref{eq:EBVP} to an equation on $\partial \Omega$ where we solve for the boundary data $v = \mathcal N_\omega u$. Therefore, we wish to make sense of $(S_\omega v)|_{\partial \Omega}$.

\begin{lemma}\label{lem:RSP_def}
    Let $S_\omega$ be the single layer potential defined in \eqref{eq:single_layer_op} for $\omega \in (0, 1) \pm (0, \infty)$. Then for $v \in H^{\infty, a}_\bo$ for $a > -1$, $S_\omega v$ is smooth up to $\partial \Omega \setminus \mathcal K$, and the restriction to boundary defines a map 
    \begin{equation}
        \mathcal C_\omega := (S_\omega \bullet)|_{\partial \Omega \setminus \mathcal K}: H_\bo^{\infty, a}(\partial \Omega; T^* \partial \Omega)\to C^\infty(\partial \Omega \setminus \mathcal K).
    \end{equation}
\end{lemma}
$\mathcal C_\omega$ is known as the restricted single layer operator. We remark that, $C^\infty(\partial \Omega \setminus \mathcal K)$ is the space of smooth functions on $\partial \Omega \setminus \mathcal K$ as an open submanifold of $\R^2$. One can also check directly that $\mathcal C_\omega$ maps $H^{\infty, a}_\bo$ to the space of extendible distributions when $\partial \Omega$ is viewed as a closed immersed smooth submanifold of $\R^2$, but this will be clear when we compute the Schwartz kernel of $\mathcal C_\omega$ later in \S\ref{sec:kernel_computation} anyways. 
\begin{proof}
    Let $\delta > 0$ and let $\chi_\bo \in \overline C^\infty(\partial \Omega)$ be a cutoff function supported in a $\delta$-neighborhood of $\mathcal K$ with $\chi_\bo = 1$ near $\mathcal K$. Then we can write
    \begin{equation}
        S_\omega v= S_\omega (\chi_\bo v) + S_\omega ((1 - \chi_\bo) v).
    \end{equation}
    From \cite[Lemma 4.7]{DWZ}, we see that $S_\omega ((1 - \chi_\bo) v) \in C^\infty(\Omega)$ is smooth up to a neighborhood of $\supp (1 - \chi_\bo) \subset \partial \Omega$, that is, away from a $\delta$-neighborhood of the corners. Furthermore, since the fundamental solution $E_\omega$ is smooth away from $0$, it is clear that $R_\omega \mathcal I (\chi_\bo v) \in \mathcal D'(\R^2)$ is smooth away from $\supp (\chi_\bo) \subset \R^2$. Therefore, $S_\omega v$ is smooth up to the $\partial \Omega$ away from a $\delta$-neighborhood of the corner. The lemma then follows by taking smaller and smaller $\delta$ and exhausting $\partial \Omega \setminus \mathcal K$. 
\end{proof}
Restricting \eqref{eq:fundamental_2} to $\Omega$, we see that 
\[u = R_\omega f - S_\omega v_\omega.\]
In view of Lemma~\ref{lem:RSP_def} and the fact that $R_\omega f \in \overline C^\infty(\Omega)$, we see that 
\begin{equation}\label{eq:Cv_reduction}
    \mathcal C_\omega v_\omega = (R_\omega f)|_{\partial \Omega \setminus \mathcal K}.
\end{equation}
Note that $(R_\omega f)|_{\partial \Omega \setminus \mathcal K}$ extends to an element of $\overline C^\infty(\partial \Omega) \subset H^{\infty, 0-}_\bo(\partial \Omega)$. This is useful since upon computing the Schwartz kernel for $d\mathcal C_\omega$, we will see that $d\mathcal C_\omega$ extends to an operator with target space $\overline{\mathcal D}'(\partial \Omega)$.

\subsubsection{Real case mapping properties}
In order to obtain the limiting absorption principle of Theorem \ref{thm:spectral} later, we also need to understand the mapping properties of the limiting single layer operators, which we define as 
\begin{equation}
    S_{\lambda \pm i0}: L^1(\partial \Omega; T^* \partial \Omega) \to \overline{\mathcal D}'(\Omega) \quad \text{by} \quad S_{\lambda \pm i0} := (R_{\lambda \pm i0} \mathcal I v)|_{\Omega}
\end{equation}
Again, the a priori mapping property above is clear since $R_{\lambda \pm i0}$ is a convolution operator and $\mathcal I v$ is compactly supported. 

First, we need the following technical lemma. The distributions $S_{\lambda + i0} v$ will take a particular form near the corners, and we show that such distributions lies in the appropriate conormal space in this lemma. 

\begin{lemma}\label{lem:blowup_conorm}
    Let $u \in \mathcal A^{[s]}([0, \infty))$, $s > 1/2$. Let $v(r, \tau) = u(r\tau)$ and $M = [-1, 1]_\tau \times [0, 1)_r$. Then $v \in \mathcal A^{[s]}(M; [-1, 1] \times\{0\}, \{0\} \times [0, 1))$.
\end{lemma}
\begin{proof}
    1. We first show that $v \in H^s(M)$. Since $H^{s}(\R) \hookrightarrow C^0(\R)$ for $s > 1/2$, we see that $v \in L^2(M)$. It remains to prove that 
    \[|D_x|^s v \in L^2(M), \qquad |D_y|^s v \in L^2(M).\]
    Indeed 
    \begin{align*}
        \||D_x|^s v \|_{L^2(M)}^2 &\le \int \big||\xi|^s (\mathcal F_xv)(\xi, y) \big|^2\, d\xi d y \\
        &= \int |\xi|^{2s} \frac{1}{y^2} \hat u(\xi/y) \, d \xi dy \\
        &= \int |\xi|^{2s} y^{2s - 1} |\hat u(\xi)|^2\, d\xi dy \\
        &\lesssim \|u\|_{H^s}^2.
    \end{align*}
    Similarly, $(D_y)^sv \in L^2(M)$, so $v \in H^s(M)$. 

    \noindent
    2. Observe that 
    \[r\partial_r v(r, \tau) = r\tau u'(r\tau) = r \tau u'(r \tau).\]
    Since $u \in \mathcal A^{[s]}([0, \infty))$, $x \partial_x u \in H^s([0, \infty))$, it follows from Step 1 that $(r \partial_r)^m (\tau \partial_\tau)^n v \in H^s(M)$ for all $m, n \in \N$. Therefore $v \in \mathcal A^{[s]}(M)$. 
\end{proof}

The following two lemmas are all the mapping properties of $S_{\lambda \pm i0}$ we will need. In particular, they tell us how $S_{\lambda \pm i0}$ map singularities at the corners, at the reflections of corners, and away from characteristic points. We first consider singularities at the corners. 
\begin{lemma}\label{lem:S_map_prop_1}
Let $\Omega$ be $\lambda$-simple and let $a \in (-1, 0)$. Then for $s > a + \ha$, 
\begin{equation}\label{eq:S_0_SobMapProp}
    S_{\lambda \pm i0}: H_\bo^{s, a}(\partial \Omega; T^* \partial \Omega)\to \bar H^{a + \frac{3}{2}}(\Omega).
\end{equation}
In fact,
\begin{equation}\label{eq:S_0_CoMapProp}
    S_{\lambda \pm i0}: H_\bo^{\infty, a}(\partial \Omega; T^* \partial \Omega) \to \mathcal A^{[a + \frac{3}{2}]}(\widetilde \Omega; \ff, \sr).
\end{equation}
\end{lemma}
\begin{proof}
1. We focus on the $S_{\lambda + i0}$ case since $S_{\lambda - i0}$ follows similarly. We first consider neighborhoods of the corner, and it suffices to work near a type-$(+, +)$ corner centered at $0$. Let $\ell^\pm = \ell^\pm(\bullet, \lambda)$ be the coordinate functions on $\R^2$. Then the boundary in $(\ell^+, \ell^-)$ coordinates in a sufficiently small neighborhood of the corner at $0$ may be parameterized as 
\begin{equation}\label{eq:S_0_param}
    \mathbf x(\theta) = \begin{cases} (\alpha_+ \theta, \theta), & 0 \le \theta < \epsilon, \\ (\alpha_-\theta, -\theta), & 0 > \theta > -\epsilon. \end{cases}
\end{equation}
Let 
\[v d\theta \in H^{\infty, a}_\bo(\partial \Omega; T^* \partial \Omega), \qquad \supp v \subset (-\epsilon, \epsilon).\] 
Observe that the explicit formulas for $E_{\lambda + i0}$ given in \eqref{eq:lim_fs} can be decomposed as
\begin{equation}\label{eq:S_decomp}
    S_{\lambda + i0} = c_\lambda (S_\lambda^+ + S_\lambda^- + S_\lambda^0), \quad S^\pm_\lambda(x) = \log|\ell^\pm(x)|, \quad S^0_\lambda = i \pi H(-A(x)). 
\end{equation}
In particular, each piece acts on $v d\theta$ by
\begin{gather}
\begin{aligned}
    S^+_\lambda(v d\theta)(\ell^+, \ell^-) &= \int_{-\epsilon}^0 \log|\ell^+ - \alpha_-\theta| v(\theta)\, d \theta + \int_0^\epsilon \log|\ell^+ - \alpha_+\theta| v(\theta)\, d \theta \\
    &= \int_{-\infty}^\infty \log|\ell^+ - \theta| \big(\alpha_-^{-1}H(-\theta)v(\alpha_-^{-1} \theta) + \alpha_+^{-1}H(\theta)v(\alpha_+^{-1} \theta) \big)\, d \theta 
    \end{aligned} \label{eq:S^+}\\
    \begin{aligned}
    S^-_\lambda(v d \theta) (\ell^+, \ell^-) &= \omega \int_{-\epsilon}^0 \log|\ell^- + \theta| v (\theta)\, d \theta + \int_0^\epsilon \log|\ell^- - \theta| v (\theta) \, d \theta \\
    &= \int_0^\infty \log|\ell^- - \theta| \big(v(\theta) + v(-\theta)\big)
    \end{aligned} \label{eq:S^-}\\
    \begin{aligned}
    \frac{1}{i \pi}S^0_\lambda(v d \theta)(\ell^+, \ell^-) &=  \int_{-\infty}^0 H(\ell^- - \theta)v(\theta)\, d \theta + \int_{-\infty}^0  \alpha_-^{-1} H(\theta - \ell^+) v( \alpha_-^{-1}\theta)\, d \theta \\
    & \quad + \int_0^\infty H(\ell^- - \theta) v(\theta) \, d \theta - \int_0^\infty \alpha_+^{-1} H(\ell^+ - \theta) v(\alpha_+^{-1}\theta)\, d \theta
    \end{aligned} \label{eq:S^0}
\end{gather}
By Lemma~\ref{lem:b_sob_to_sob}, it follows that $H(\theta) v(\theta), H(-\theta)v (-\theta) \in H^{a + \ha}_{\comp}(\R)$. Convolution with $\log$ and convolution with the Heaviside function both map $H^{a + \ha}_{\comp}(\R) \to H^{a + \frac{3}{2}}_{\loc}(\R)$. Then it immediately follows from the formulas~\eqref{eq:S^+}, \eqref{eq:S^-}, and \eqref{eq:S^0} that if $\supp v \subset [-\epsilon, \epsilon]$, 
\begin{equation}
    S_{\lambda + i0}(v d \theta) \in \bar H^{a + \frac{3}{2}}(\Omega)
\end{equation}
On the other hand, if $v d \theta$ is smooth and supported away from the corners, it follows from \cite[Lemma 4.8]{DWZ} that $S_{\lambda + i0}(v d \theta) \in \overline{C}^\infty(\Omega)$. Therefore, the mapping property~\eqref{eq:S_0_SobMapProp} follows. 

\noindent
2. For the conormal mapping properties, observe that
\[H(\theta)v(\theta), H(-\theta) v(\theta) \in 
\mathcal A^{[a + \ha]}_{\comp}(\R; {0}).\] Convolution with $\log$ and convolution with the Heaviside function both have the mapping property
\[\mathcal A^{[a + \ha]}_{\comp}(\R; {0}) \to \mathcal A^{[a + \frac{3}{2}]}_{\loc}(\R; {0}).\] 
Then from \eqref{eq:S^+}, we see that 
\[S^+(v d \theta) \in \mathcal A^{[a + \frac{3}{2}]}(\widetilde \Omega; \ff) \subset \mathcal A^{[a + \frac{3}{2}]}(\widetilde \Omega; \ff, \sr)\]
We see from \eqref{eq:S^-} that 
\[\tilde u(\ell^-) := \chi(\ell^-) S^-(v d\theta)(\ell^+, \ell^-) \in \mathcal A^{[a + \frac{3}{2}]}([0, \infty); \{0\})\] 
where $\chi \in \CIc([0, \infty))$ and $\chi \equiv 1$ near $0$. In a sufficiently small neighborhood of the corner, we have the blown-up coordinates
\[r = \ell^-, \qquad \tau = \ell^+/\ell-, \qquad (r, \tau) \in [0, \epsilon) \times [-\alpha_-^{-1}, \alpha_+^{-1}].\]
Observe that for sufficiently small $\epsilon$, 
\[S^-(v d \theta)(r, \tau) = \tilde u(r \tau), \qquad (r, \tau) \in [0, \epsilon) \times [-\alpha_-^{-1}, \alpha_+^{-1}],\]
so it follows from Lemma~\ref{lem:blowup_conorm} that $S^-(v d \theta) \in \mathcal A^{[a + \frac{3}{2}]}(\widetilde \Omega; \ff, \sr)$. Via similar analysis, we also find that $S^0(v d \theta) \in \mathcal A^{[a + \frac{3}{2}]}(\widetilde \Omega; \ff, \sr)$. The mapping property \eqref{eq:S_0_CoMapProp} then follows upon applying \cite[Lemma 4.8]{DWZ} as in the end of Step 1. 
\end{proof}

We will see that the distributions $v$ we feed into $S_{\lambda \pm i0}$ are smooth near noncorner characteristic points, so noncorner characteristic points are covered by the above lemma as well. Now we consider what happens near noncharacteristic points. We consider $S_{\lambda + i0}$ only since $S_{\lambda - i0} v = \overline{S_{\lambda + i0} \bar v}$. 
\begin{lemma}\label{lem:S_map_prop_2}
    Let $\lambda$ be Morse--Smale. Let $x_0 \in \partial \Omega$ be noncharacteristic. Then for all $\chi \in C^\infty(\partial \Omega)$ supported in a sufficiently small neighborhood of $x_0$, we have 
    \begin{equation}
        S_{\lambda + i0} \chi : \bar H^s(\partial \Omega; T^* \partial \Omega) \to \bar H^{s + 1}(\Omega)
    \end{equation}
    and
    \[\WF(S_{\lambda + i0} \chi v) \subset \{(x, \xi) \in N^*_{\sgn(\eta)} \Gamma^{\sgn(\eta)}_\lambda(y) : (y, \eta) \in \WF(\chi v)\}.\]
    Moreover, 
    \begin{equation}\label{eq:Schi_map_prop}
        S_{\lambda + i0} \chi : \mathcal A^{[s]}(\partial \Omega; T^* \partial \Omega; N^*_\pm\{x_0\}) \to \mathcal A^{[s + 1]}(\Omega; N^*_\pm \Gamma_{\lambda}^\pm(x_0)).
    \end{equation} 
\end{lemma}
\begin{Remarks}
    1. Here, the notation $\mathcal A^{[s]}(\partial \Omega; N^*_\pm\{x_0\}) \subset \mathcal A^{[s]}(\partial \Omega; \{x_0\})$ denotes the subset of conormal distribution with wavefront set in $N^*_\pm\{x_0\}$. The same goes for $A^{[s + 1]}(\Omega; N^*_\pm \Gamma_{\lambda}^\pm(x_0))$ where $N^*_\pm\Gamma_{\lambda}^\pm(x_0)$ is defined in~\eqref{eq:N^*_pm}. 
    
    \noindent 
    2. In the case that $\gamma^\pm(x_0)$ is a corner, observe that $\Gamma^\pm(x)$ goes directly into the corner. In this case, we remark that it follows immediately from the formula~\eqref{eq:S_decomp} in the proof and Lemma~\ref{lem:blowup_conorm}
    \[S_{\lambda + i0} \chi : \mathcal A^{[s]}(\partial \Omega; T^* \partial \Omega; N^*_\pm\{x_0\}) \to \mathcal A^{[s + 1]}(\widetilde \Omega; N^*_\pm \Gamma_{\lambda}^\pm(x_0); \ff).\]
    In all the other cases, the space $\mathcal A^{[s + 1]}(\Omega; N^*_\pm \Gamma_{\lambda}^\pm(x_0))$ is smooth up to the corners. 
\end{Remarks}
\begin{proof}
    Fix a b-parameterization on $\partial \Omega$ and consider $v(\theta) d \theta \in \overline{\mathcal D}'(\partial \Omega; T^* \partial \Omega)$. Define for $\chi$ supported away from characteristic points, we can define $\pi^\pm_\lambda(\chi v) \in \mathcal E'(\R)$ by 
    \[\int_\R \pi^\pm_\lambda(\chi v)(s) \varphi(s) \, ds = \int_{\partial \Omega} \chi(x) v(x) \varphi(\ell^\pm(x))\, d \theta(x).\]
    Then if $\chi$ is supported in a sufficiently small neighborhood of $x_0$, we have the formula
    \begin{equation}\label{eq:S_expansion}
        S_{\lambda + i0} (\chi vd \theta)(x) = c_{\lambda} \left( g_+(\ell^+(x)) + g_-(\ell^-(x)) + c_0 \int_{\partial \Omega} v(\theta) \, d \theta \right)
    \end{equation}
    where $c_0 = 0$ or $2 \pi i$ depending on $x_0$ and 
    \begin{equation}
        g_+:= (\pi^+_\lambda(\chi v)) * \log_{\nu^+},\qquad g_-:= (\pi^-_\lambda(\chi v))* \log_{-\nu^-}
    \end{equation}
    where $\nu^\pm := \sgn \partial_\theta \ell^\pm(x_0)$ and $\log_\pm x := \log(x \pm i0)$. See \cite[Lemma 4.9]{DWZ} for the detailed derivation of the formula. Convolution with $\log_\pm$ is a pseudodifferential operator with wavefront set supported on the postive (or negative) frequencies. The lemma then follows from~\eqref{eq:S_expansion} upon careful inspection of the signs.
\end{proof}

\subsection{Restricted single layer potential}\label{sec:kernel_computation}
Recall from Lemma \ref{lem:RSP_def}, we have the restricted single layer potential
\begin{equation}
\begin{gathered}
    \mathcal C_\omega: \dot C^\infty(\partial \Omega; T^* \partial \Omega)) \to \mathcal D' (\partial \Omega \setminus \mathcal K), \\
    \text{given by} \quad \mathcal C_\omega v := (S_\omega v)|_{\partial \Omega \setminus \mathcal K}.
\end{gathered}
\end{equation}
From the mapping properties of $S_{\lambda \pm i0}$ in Lemmas~\ref{lem:S_map_prop_1} and~\ref{lem:S_map_prop_2}, we can also define the limiting operators 
\begin{equation}
\begin{gathered}
    \mathcal C_{\lambda \pm i0}: \dot C^\infty(\partial \Omega; T^* \partial \Omega)) \to \mathcal D' (\partial \Omega \setminus \mathcal K), \\
    \text{given by} \quad \mathcal C_{\lambda \pm i0} v := (S_{\lambda \pm i0} v)|_{\partial \Omega \setminus \mathcal K}.
\end{gathered}
\end{equation}
In terms of the fundamental solution $E_\omega$, whose formula is given in Lemma~\ref{lem:fs}, we see that 
\begin{equation}
    \mathcal C_\omega v (x) = \int_{\partial \Omega} E_\omega(x - y) v(y), \qquad x \in \partial \Omega \setminus \mathcal K
\end{equation}

It will be convenient to compose $\mathcal C_\omega$ with a differential so that the domain and the target are both over the tangent bundle:
\begin{equation}
    d\mathcal C_\omega: \dot C^\infty(\partial \Omega; T^*\partial \Omega) \to \mathcal D'(\partial \Omega \setminus \mathcal K; T^* (\partial \Omega\setminus \mathcal K)).
\end{equation}
Moreover, we will see that composing with the differential in fact allows us to decompose $d \mathcal C_\omega$ as the sum of pulled-back pseudodifferential operators away from the corners and as b-differential operators near the corners. This will be done by explicitly computing the Schwartz kernel of $d \mathcal C_\omega$, from which we will also obtain extended mapping properties for the restricted single layer operator.

The explicit formula for the Schwartz kernel of $d \mathcal C_\omega$ can be computed using the explicit formula for the fundamental solution given in Lemma \ref{lem:fs}. In particular, it follows that the Schwartz kernel $K_\omega$ of $d \mathcal C_\omega$ can be decomposed as 
\begin{equation}
    K_\omega = K^+_\omega + K^-_\omega.
\end{equation}
With respect to a b-parameterization $\mathbf x(\theta)$, $K^\pm_\omega$ are given by 
\begin{equation}\label{eq:kernel_formula}
    K^\pm_\omega(\theta, \theta') = c_\omega \lim_{\delta \to 0+} \frac{ \partial_\theta \ell^\pm(\mathbf x(\theta), \omega))}{\ell^\pm(\mathbf x(\theta) - \mathbf x (\theta') + \delta \mathbf v(\theta), \omega)}
\end{equation}
respectively, where 
\[\mathbf v(\theta) = \mathbf V(\mathbf x(\theta))\]
is an inward pointing vector field on $\partial \Omega$ that descends from a smooth vector field $\mathbf V$ in $\R^2$. For the remainder of this section, we compute the Schwartz kernel for $\omega$ in the upper half plane. Throughout, let $\mathcal J \subset (0, 1)$ be an open set such that $\Omega$ is $\lambda$-simple for every $\lambda \in \mathcal J$. 

\subsubsection{Schwartz kernel away from corner}
Fix a b-parameterization so that $x \in \Omega$ can be identified with $\theta \in \mathbb S^1$. We first consider the case near 
\[\theta, \theta' \notin \mathcal K.\]
In this case, the kernel is identical to the setting \cite[\S 4.6]{DWZ}, and we collect the results in the following proposition.
\begin{proposition}\label{prop:easy_kernel}
    Let $\omega = \lambda + i\epsilon \in \mathcal J + i (0, \epsilon_0)$ for some sufficiently small $\epsilon_0$. Let $U \Subset \partial \Omega \setminus \mathcal K$. Then for $\theta, \theta' \in U \times U$, the Schwartz kernel is given by
    \begin{multline}
        K^\pm_\omega(\theta, \theta') \equiv c_\omega \chi_{\Diag}(\theta, \theta')(\theta - \theta' \pm i0)^{-1} \\
        + \tilde c_\omega^\pm(\theta') \chi_{\Diag}(\gamma^\pm_\lambda(\theta), \theta') (\gamma^\pm_\lambda(\theta) - \theta' + i \epsilon z_\omega^\pm(\theta'))^{-1},
    \end{multline}
    where `$\equiv$' denotes equivalence modulo smooth functions in $\theta$, $\theta'$, and $\omega$ up to $\epsilon = \Im \omega = 0$. Furthermore, 
    \begin{enumerate}
        \item $\chi_{\Diag} \in C^\infty(\mathbb S^1 \times \mathbb S^1)$ is identically $1$ near $\Diag := \{(\theta, \theta): \theta \in \mathbb S^1\}$, and is supported in a small neighborhood of $\Diag$,

        \item $z_\omega^\pm(\theta) \in C^\infty(\mathbb S^1)$ is smooth in $\omega  \in \mathcal J + i[0, \epsilon_0)$ up to the real line, and $\Re z^\pm_\omega(\theta) \ge \delta > 0$,

        \item $\tilde c_\omega^\pm(\theta) \in C^\infty(\mathbb S^1)$ is smooth in $\omega  \in \mathcal J + i[0, \epsilon_0)$ up to the real line and has the formula
        \begin{equation}
        \tilde c_\omega^\pm(\theta') = \frac{c_\omega}{\partial_{\theta'} \gamma^\pm_\lambda(\theta')} + \mathcal O(\epsilon). 
        \end{equation}
    \end{enumerate}
\end{proposition}
In other words, away from the corners, $K^\pm_\omega(\theta, \theta')$ still looks like the sum of a pseudodifferential operator plus the pullback by $\gamma^\pm_\lambda$ of a psuedodifferential operator.

The remaining parts of the kernel consists of the cases when at least one of $\theta$ or $\theta'$ is in a small neighborhood of a corner. This can be separated into four cases that behave differently:
\begin{enumerate}
    \item both $\theta$ and $\theta'$ are near the same corner,
    \item one of $\theta$ or $\theta'$ is near a corner $\kappa$, and the other is near a reflected corner $\gamma_\lambda^\pm(\kappa)$, the sign depending on the corner,
    \item $\theta$ and $\theta'$ are near different corners,
    \item one of $\theta$ or $\theta'$ is near a corner $\kappa$ and the other is away from the corners $\mathcal K$ and the reflected corner $\gamma_\lambda^\pm (\kappa)$. 
\end{enumerate}
The first case will be obtained via direct computation in Proposition~\ref{prop:corner_kernel}. The last three cases are obtained by transferring the results from \cite[\S4.6]{DWZ}. The idea is to smooth out the corner by perturbing only one side of the corner, and then comparing the kernels associated to the smooth perturbed domain to the desired kernel. These last three cases are handled in Proposition~\ref{prop:easy_kernel}. See Figure~\ref{fig:POS} or~\ref{fig:refPOS} for a diagram of where the singularities of the Schwartz kernel lie in the case of one corner. 

\subsubsection{Schwartz kernel near the corner}
Now we consider the case that both $\theta$ and $\theta'$ are in a small neighborhood of a corner. Without the loss of generality, we put the corner at $0$. We compute the explicit formulas for type-$(+, +)$ corners only, since we can deduce propagation estimates near other corner types later using reflection symmetries. Again, we write $\omega = \lambda + i \epsilon$ and use $\ell^\pm = \ell(\bullet, \lambda)$ as coordinate functions on $\R^2$. In $(\ell^+, \ell^-)$-coordinates, we fix a particularly convenient choice of $\lambda$-dependent parameterizations of the boundary near the straight corner by
\begin{equation}\label{eq:good_param}
    \mathbf x(\theta) = \begin{cases} (\alpha_+ \theta, \theta), & \theta \ge 0, \\ (\alpha_- \theta, - \theta), & \theta < 0 \end{cases}
\end{equation}
for $\theta, \theta'$ in a sufficiently small neighborhood $U \subset \R$ of $0$.

For $\theta, \theta' \in U$, we have the following explicit formula for the Schwartz kernel. 

\begin{proposition}\label{prop:corner_kernel}
    Write $\omega = \lambda + i \epsilon$ for $\lambda \in \mathcal J \Subset (0, 1)$ and $0 < \epsilon < \epsilon_0$ for some sufficiently small $\epsilon_0$. Assume that there is a corner at $0$ parameterized by \eqref{eq:good_param} in a small neighborhood $U$. For $\theta, \theta' \in U$, we have
    \begin{equation}
        K^+_\omega(\theta, \theta') = \begin{cases} c_\omega(\theta - \theta' + i0)^{-1} & \theta \cdot \theta' > 0 \\ c_\omega (\theta - \alpha(1 + i \epsilon z^{++}_\omega)\theta')^{-1} & \theta < 0, \ \theta' > 0 \\ c_\omega (\theta - \alpha^{-1}(1 - i \epsilon z^{+-}_\omega)\theta')^{-1} & \theta > 0, \ \theta' < 0\end{cases} 
    \end{equation}
    and
    \begin{equation}
        K^-_\omega(\theta, \theta') = \begin{cases} c_\omega(\theta - \theta' - i0)^{-1} & \theta \cdot \theta' > 0 \\ c_\omega\big(\theta + (1 + i \epsilon z^{-+}_\omega) \theta'\big)^{-1} & \theta < 0, \ \theta' > 0 \\ c_\omega\big(\theta + (1 - i \epsilon z^{--}_\omega)\theta'\big)^{-1} & \theta > 0, \ \theta' < 0\end{cases} 
    \end{equation}
    where $\alpha := \alpha_+/\alpha_-$, $\Im z^{\mu\nu}_\omega = \mathcal O(\epsilon)$, and $\Re z^{\mu\nu}_\omega > \delta > 0$ for $\mu, \nu \in \{+, -\}$ locally uniformly for $\lambda \in \mathcal J$. $z_\omega^{\mu \nu}$ also depends smoothly on $\omega$. 
\end{proposition}
Here, $\alpha$ and $\alpha_\pm$ are the characteristic ratios defined in Definition~\ref{def:corner_type}.
\begin{proof}
1. We compute using~\eqref{eq:kernel_formula}. Since we assume that $\ell^-(x, \lambda) > 0$ for all $x \in U$ a choice of inward pointing vector field $\mathbf v$ is given by 
\[\mathbf v = \frac{1}{2}(-\lambda \hat x_1 + \sqrt{1 - \lambda^2} \hat x_2)\]
in a neighborhood of $\theta = 0$, where $\hat x_j$ are unit vectors in the $x_1$ and $x_2$ directions respectively. Observe that 
\begin{equation}\label{eq:ellv}
\begin{gathered}
    \ell^+(\mathbf v, \omega) = \frac{1}{2} \left(-\frac{\lambda}{\omega} + \sqrt{ \frac{1 - \lambda^2}{1 - \omega^2}} \right) = i \epsilon \frac{1}{2\lambda(1 - \lambda^2)} + \mathcal O(\epsilon^2) \\
    \ell^-(\mathbf v, \omega) = \frac{1}{2} \left(\frac{\lambda}{\omega} + \sqrt{ \frac{1 - \lambda^2}{1 - \omega^2}} \right) = 1 + i \epsilon \frac{2 \lambda^2 - 1}{2 \lambda (1 - \lambda^2)} + \mathcal O(\epsilon^2)
\end{gathered}
\end{equation}

Next we express $\ell^\pm(\bullet, \omega)$ in terms of $\ell^\pm(\bullet, \lambda)$. $\ell^\pm$ with arguments omitted will implicitly denote $\ell^\pm(\bullet, \lambda)$. Then we see that
\begin{align}
    \ell^\pm(\bullet, \omega) &= \frac{1}{2} \left[ \left( \sqrt{\frac{1 - \lambda^2}{1 - (\lambda + i \epsilon)^2}} \pm \frac{\lambda}{\lambda + i \epsilon} \right) \ell^+ + \left( \sqrt{\frac{1 - \lambda^2}{1 - (\lambda + i\epsilon)^2}} \mp  \frac{\lambda}{\lambda + i \epsilon} \right) \ell^- \right] \nonumber \\
    &= \left(1 + i \epsilon \frac{2\lambda^2 - 1}{2 \lambda (1 - \lambda^2)} + \mathcal O(\epsilon^2) \right) \ell^\pm + \left(i \epsilon \frac{1}{2\lambda(1 - \lambda^2)} + \mathcal O(\epsilon^2) \right) \ell^\mp \label{eq:ell_in_ell}
\end{align}
We emphasize that $\mathcal O(\epsilon^2)$ depends smoothly on on $\epsilon$ and $\lambda$ only, and is locally uniform in $\lambda$, and thus uniform for $\lambda \in \mathcal J$. Restricting~\eqref{eq:ell_in_ell} to the boundary using the parameterization~\eqref{eq:good_param}, we find
\begin{equation}\label{eq:ell_pos}
    \begin{gathered}
        \ell^+(\mathbf x(\theta), \omega) = \left(\alpha_+ + i \epsilon \frac{2 \lambda^2 \alpha_+ - \alpha_+ + 1}{2 \lambda(1 - \lambda^2)} + \mathcal O(\epsilon^2) \right) \theta \\
    \ell^-(\mathbf x(\theta), \omega) = \left(1 + i \epsilon \frac{2 \lambda^2  - 1 + \alpha_+}{2 \lambda(1 - \lambda^2)} + \mathcal O(\epsilon^2) \right) \theta
    \end{gathered}
\end{equation}
for $0 < \theta < \delta$, and 
\begin{equation}\label{eq:ell_neg}
\begin{gathered}
    \ell^+(\mathbf x(\theta), \omega) = \left(\alpha_- + i \epsilon \frac{2 \lambda^2 \alpha_- - \alpha_- - 1}{2 \lambda(1 - \lambda^2)} + \mathcal O(\epsilon^2) \right) \theta \\
    \ell^-(\mathbf x(\theta), \omega) = -\left(1 + i \epsilon \frac{2 \lambda^2 - 1 - \alpha_-}{2 \lambda(1 - \lambda^2)} + \mathcal O( \epsilon^2) \right) \theta
\end{gathered}
\end{equation}
for $0 > \theta > -\delta$.

\noindent
2. We first compute the kernel for $\theta, \theta' > 0$. Observe that~\eqref{eq:ellv} and~\eqref{eq:ell_pos} gives
\[\Im \frac{\ell^+(\mathbf v, \omega)}{\partial_\theta \ell^+(\mathbf x(\theta), \omega)} = \Im \frac{i \epsilon \frac{1}{2\lambda (1 - \lambda^2)}}{\alpha_+ + i \epsilon \frac{2 \lambda^2 \alpha_+ - \alpha_+ + 1}{2 \lambda(1 - \lambda^2)}} + \mathcal O(\epsilon^2) > 0\]
for all sufficiently small $\epsilon$. Using~\eqref{eq:kernel_formula}, we find
\begin{equation}\label{eq:diag+1}
    K^+_\omega(\theta, \theta') = c_\omega \lim_{\delta \to 0} \left(\theta - \theta' + \delta \frac{\ell^+(\mathbf v, \omega)}{\partial_\theta \ell^+(\mathbf x(\theta), \omega)} \right)^{-1} = c_\omega (\theta - \theta' +i0)^{-1}.
\end{equation}
Similarly, we see that 
\[\Im \frac{\ell^-(\mathbf v, \omega)}{\partial_\theta \ell^-(\mathbf x(\theta), \omega)} = \Im \frac{1 + i \epsilon \frac{2 \lambda^2 - 1}{2 \lambda (1 - \lambda^2)}}{1 + i \epsilon \frac{2 \lambda^2  - 1 + \alpha_+}{2 \lambda(1 - \lambda^2)}} + \mathcal O(\epsilon^2) = -\epsilon \frac{\alpha_+}{2 \lambda(1 - \lambda^2)} + \mathcal O(\epsilon^2) < 0,\]
so 
\begin{equation}\label{eq:diag+2}
    K^+_\omega(\theta, \theta') = c_\omega \lim_{\delta \to 0} \left(\theta - \theta' + \delta \frac{\ell^-(\mathbf v, \omega)}{\partial_\theta \ell^-(\mathbf x(\theta), \omega)} \right)^{-1} = c_\omega(\theta - \theta' - i0)^{-1}
\end{equation}

Now for $\theta, \theta' < 0$ and using~\eqref{eq:ell_neg}, we have
\[\Im \frac{\ell^+(\mathbf v, \omega)}{\partial_\theta \ell^+(\mathbf x(\theta), \omega)} = \Im \frac{i \epsilon \frac{1}{2 \lambda(1 - \lambda^2)}}{\alpha_- + i \epsilon \frac{2 \lambda^2 \alpha_- - \alpha_- - 1}{2 \lambda(1 - \lambda^2)}} + \mathcal O(\epsilon^2) > 0\]
and 
\[\Im \frac{\ell^-(\mathbf v, \omega)}{\partial_\theta \ell^-(\mathbf x(\theta), \omega)} = -\Im \frac{1 + i \epsilon \frac{2 \lambda^2 - 1}{2 \lambda (1 - \lambda^2)}}{1 + i \epsilon \frac{2 \lambda^2 - 1 - \alpha_-}{2 \lambda(1 - \lambda^2)}} + \mathcal O(\epsilon^2) = -\epsilon \frac{\alpha_-}{2 \lambda( 1- \lambda^2)} + \mathcal O(\epsilon^2) < 0\]
for all sufficiently small $\epsilon$. Therefore, similar to~\eqref{eq:diag+1} and~\eqref{eq:diag+2}, we have 
\begin{equation}\label{eq:diag-}
    K^-_\omega(\theta, \theta') = c_\omega(\theta - \theta' \pm i0)^{-1}
\end{equation}
for $\theta, \theta' < 0$. 

\noindent
3. Next, we consider the cases where $\theta \cdot \theta' < 0$. Since $\theta \neq \theta'$ in these cases, evaluating the limit the kernel formula~\eqref{eq:kernel_formula} is trivial, and~\eqref{eq:ell_pos} and~\eqref{eq:ell_neg} are both linear in $\theta$. Therefore, we have
\begin{equation}
    K(\theta, \theta) = c_\omega \left(\theta - \frac{\ell^\pm(\mathbf x(\theta'), \omega)}{\partial_\theta \ell^\pm(\mathbf x(\theta), \omega))} \right)^{-1}, \qquad \theta, \theta' \in U, \ \theta \cdot \theta' < 0. 
\end{equation}
Now we just compute the four possible cases.
\begin{enumerate}
\item $\theta < 0$, $\theta' > 0$ for $K^+_\omega$:
\begin{align*}
    \frac{\ell^+(\mathbf x(\theta'), \omega)}{\partial_\theta \ell^+(\mathbf x(\theta), \omega)} &= \alpha \left(\frac{1 + i \epsilon \frac{2 \lambda^2 - 1}{2 \lambda (1 - \lambda^2)} + i \epsilon \frac{\alpha^{-1}_+}{2 \lambda (1 - \lambda^2)}}{1 + i \epsilon \frac{2 \lambda^2 - 1}{2 \lambda(1 - \lambda^2)} - i \epsilon \frac{\alpha^{-1}_-}{2 \lambda(1 - \lambda^2)}} + \mathcal O(\epsilon^2) \right) \theta' \\
    &= \alpha \left(1 + i \epsilon \frac{\alpha^{-1}_+ + \alpha^{-1}_-}{2 \lambda(1 - \lambda^2)} + \mathcal O(\epsilon^2)\right) \theta',
\end{align*}
so
\begin{equation}\label{eq:++}
    K^+(\theta, \theta') = c_\omega (\theta - \alpha(1 + i \epsilon z_\omega^{++})\theta')^{-1}, \qquad z_\omega^{++} = \frac{\alpha^{-1}_+ + \alpha^{-1}_-}{2 \lambda(1 - \lambda^2)} + \mathcal O(\epsilon).
\end{equation}

\item $\theta < 0$, $\theta' > 0$ for $K^-_\omega$:
\begin{align*}
    \frac{\ell^-(\mathbf x(\theta'), \omega)}{\partial_\theta \ell^-(\mathbf x(\theta), \omega)} &= -\left(\frac{1 + i \epsilon \frac{2 \lambda^2 - 1}{2 \lambda(1 - \lambda^2)} + i \epsilon \frac{\alpha_+}{2 \lambda(1 - \lambda^2)}}{1 + i \epsilon \frac{2 \lambda^2 - 1}{2 \lambda(1 - \lambda^2)} - i \epsilon \frac{\alpha_-}{2 \lambda (1 - \lambda^2)}} + \mathcal O(\epsilon^2) \right) \theta' \\
    &= -\left(1 + i \epsilon \frac{\alpha_+ + \alpha_-}{2 \lambda(1 - \lambda^2)} + \mathcal O(\epsilon^2)\right) \theta',
\end{align*}
so
\begin{equation}\label{eq:-+}
    K^{-}(\theta, \theta') = c_\omega (\theta + (1 + i \epsilon z^{-+}_\omega) \theta')^{-1}, \qquad z_{\omega}^{-+} = \frac{\alpha_+ + \alpha_-}{2 \lambda(1 - \lambda^2)} + \mathcal O(\epsilon)
\end{equation}

\item $\theta > 0$, $\theta' < 0$ for $K^+_\omega$:
\begin{align*}
    \frac{\ell^+(\mathbf x(\theta'), \omega)}{\partial_\theta \ell^+(\mathbf x(\theta), \omega)} &= \alpha^{-1}\left(\frac{1 + i \epsilon \frac{2 \lambda^2 - 1}{2 \lambda (1 - \lambda^2)} - i \epsilon \frac{\alpha^{-1}_-}{2 \lambda (1 - \lambda^2)}}{1 + i \epsilon \frac{2 \lambda^2 - 1}{2 \lambda(1 - \lambda^2)} + i \epsilon \frac{\alpha^{-1}_+}{2 \lambda(1 - \lambda^2)}} + \mathcal O(\epsilon^2) \right) \theta' \\
    &= \alpha^{-1} \left(1 - i \epsilon \frac{\alpha^{-1}_+ + \alpha^{-1}_-}{2 \lambda(1 - \lambda^2)} + \mathcal O(\epsilon^2)\right) \theta'
\end{align*}
so
\begin{equation}\label{eq:+-}
    K^+(\theta, \theta') = c_\omega (\theta - \alpha^{-1} (1 - i \epsilon z_\omega^{+-})\theta')^{-1}, \qquad z_\omega^{+-} = \frac{\alpha^{-1}_+ + \alpha^{-1}_-}{2 \lambda(1 - \lambda^2)} + \mathcal O(\epsilon)
\end{equation}

\item $\theta > 0$, $\theta' < 0$ for $K^-_\omega$:  
\begin{align*}
    \frac{\ell^-(\mathbf x(\theta'), \omega)}{\partial_\theta \ell^-(\mathbf x(\theta), \omega)} &= -\left(\frac{1 + i \epsilon \frac{2 \lambda^2 - 1}{2 \lambda(1 - \lambda^2)} - i \epsilon \frac{\alpha_-}{2 \lambda(1 - \lambda^2)}}{1 + i \epsilon \frac{2 \lambda^2 - 1}{2 \lambda (1 - \lambda^2)} + i \epsilon \frac{\alpha_+}{2 \lambda(1 - \lambda^2)}} + \mathcal O(\epsilon^2) \right) \theta' \\
    &= - \left(1 - i \epsilon \frac{\alpha_- + \alpha_+}{2 \lambda(1 - \lambda^2)} + \mathcal O(\epsilon^2) \right) \theta'
\end{align*}
Therefore, 
\begin{equation}\label{eq:--}
    K^-(\theta, \theta') = c_\omega (\theta + (1 - i \epsilon z_\omega)\theta')^{-1}, \qquad z^{--}_\omega = \frac{\alpha_- + \alpha_+}{2 \lambda(1 - \lambda^2)} + \mathcal O(\epsilon)
\end{equation} 
\end{enumerate}
Observe that $z^{\mu\nu}_\omega$ given in~\eqref{eq:++}, \eqref{eq:-+}, \eqref{eq:+-}, and \eqref{eq:--} satisfies the desired properties, that is, the leading order part in $\epsilon$ is real and positive.
\end{proof}

\subsubsection{Remaining parts of the Schwartz kernel}
We compute the remaining parts of the kernel by transferring results from the smooth case. In particular, we use the following lemma, which allows us to smooth out the corner by only affecting one side of the corner while retaining $\lambda$-simplicity. See Figure \ref{fig:corner_smoothing}.
\begin{lemma}\label{lem:smoothing}
    Let $\Omega$ be a $\lambda$-simple domain with straight characteristic corners. Let $\kappa \subset \partial \Omega$ be a type-$(+, +)$ corner. Fix $r > 0$ small and define the half balls
    \[U_\pm = \{x \in \R^2: \pm \ell^+(x, \lambda) > 0, \ |x - \kappa| < r \}\]
    Then there exists $\lambda$-simple domains with straight corners $\Omega_{\kappa, \pm}$ such that $\partial \Omega_{\kappa, \pm} \cap B_\kappa(r)$ is smooth and
    \begin{equation}\label{eq:smooth_domain_properties}
        \Omega \subset \Omega_{\kappa, \pm}, \quad \partial \Omega \setminus U_\pm = \partial \Omega_{\kappa, \pm} \setminus U_\pm.
    \end{equation}
\end{lemma}
\begin{proof}
    We construct $\Omega_{\kappa, +}$ since the other sign is similar. We may assume that $\kappa = 0$ and that $r$ is sufficiently small so that the parameterization~\eqref{eq:good_param} holds for $|\theta| < r$. Let $\delta = \min \{r/2, \alpha_+ r/2\}$ and put
    \begin{equation}
        \widetilde \Gamma(\theta) = \begin{cases}-\alpha^{-1}_- \theta & \theta < \delta, \\ \alpha^{-1}_+ \theta & \theta \ge \delta. \end{cases}
    \end{equation}
    It is easy to see that there exists an even function $\psi \in \CIc(\R; \R)$ such that $\supp \psi \subset (-\delta/10, \delta/10)$, $\int \psi = 1$, and that $\Gamma := \psi * \widetilde \Gamma$ has a unique nondegenerate critical point. Then in $(\ell^+, \ell^-)$ coordinates, 
    \[\tilde{\mathbf x}(\theta) = (\theta, \Gamma(\theta))\]
    gives the parameterization for the boundary of $\Omega_{\kappa, +}$. Smoothness and $\lambda$-simplicity follows from the construction. 
\end{proof}
\begin{figure}
    \centering
    \includegraphics[scale = 0.9]{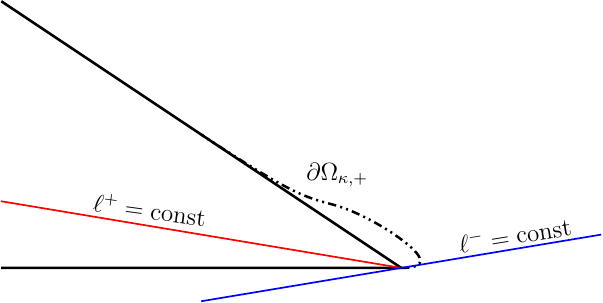}
    \caption{Dotted line represents the new boundary of the smoothed domain $\partial \Omega_{\kappa, +}$ by modifying one side of the corner. By visual inspection, the domain is still $\lambda$-simple.}
    \label{fig:corner_smoothing}
\end{figure}

In the following lemma, we give formulas for the Schwartz kernel of $d \mathcal C_\omega$ when one of the incoming or outgoing variables is in a neighborhood of a type-$(+, +)$ corner. The formula near the other corners can be easily deduced from this case using reflections. 
\begin{proposition}\label{prop:rest_of_kernel}
Let $\Omega$ be Morse--Smale for $\lambda \in \mathcal J \Subset (0, 1)$ and let $\kappa$ be a type-$(+, +)$. Fix a $\lambda$-dependent family of b-parameterizations $\mathbf x: \mathbb S^1 \to \partial \Omega$ such that the implicit dependence on $\lambda$ is smooth, $\mathbf x(0) = \kappa$ and $\gamma^+_\lambda:\mathbb S^1 \to \mathbb S^1$ is smooth. Then for $\epsilon = \Im \omega > 0$ sufficiently small, we have the following formulas. 
\begin{enumerate}
    \item In a sufficiently small neighborhood $U$ of $(0, \gamma_\lambda^+(0))$, there exists $\psi^{\pm}_\omega \in C^\infty(U)$ smooth in $\omega$ up to $\epsilon = \Im \omega = 0$ such that 
    \[K_\omega(\theta, \theta') = \sum_{\pm} H(\pm\theta) \left[\tilde c_\omega^{+}(\theta') (\gamma_\lambda^+(\theta) - \theta' + i \epsilon z_\omega^{+}(\theta'))^{-1} + \psi^{\pm}_\omega(\theta, \theta') \right] \]
    for $(\theta, \theta') \in U$.

    Similarly, in a small neighborhood $V$ of $(\gamma_\lambda^+(0), 0)$, there exists $\varphi^\pm_\omega \in C^\infty(U)$ smooth in $\omega$ up to $\epsilon = \Im \omega = 0$ such that 
    \[K_\omega(\theta, \theta') = \sum_{\pm} H(\pm\theta') \left[\tilde c_\omega^{+}(\theta') (\gamma_\lambda^+(\theta) - \theta' + i \epsilon z_\omega^{+}(\theta'))^{-1} + \varphi^{\pm}_\omega(\theta, \theta') \right] \]
    for $(\theta, \theta') \in V$. 

    Here, $z_\omega^{+} \in \overline C^\infty(\partial \Omega)$ is smooth in $\omega = \lambda + i\epsilon$ up to $\epsilon = 0$. Furthermore $\Re z_\omega^\pm (\theta) \ge \delta > 0$ where $\delta$ is independent of $\epsilon$ and $\theta$.
    
    $\tilde c_\omega^{+} \in C^\infty(\mathbb S^1)$ is also smooth in $\omega$ up to $\epsilon = 0$, and 
    \begin{equation}
        \tilde c_\omega^+(\theta') = \frac{c_\omega}{\partial_{\theta'} \gamma^+_\lambda(\theta')} + \mathcal O(\epsilon).
    \end{equation}
    \item If $\theta_c \in (0, 1)$ is such that $\mathbf x(\theta_c)$ is a corner, then there exists a neighborhood $U$ of $(0, \theta_c)$ and $\psi_\omega^{\mu, \nu} \in C^\infty(U)$ smooth in $\omega$ up to $\epsilon = \Im \omega = 0$ such that 
    \[K_\omega(\theta, \theta') = \sum_{\mu, \nu \in \{+, -\}}H(\mu\theta)H(\nu(\theta' - \theta_c)) \psi^{\mu, \nu}_\omega(\theta, \theta'), \qquad (\theta, \theta') \in U\]

    \item Let $\theta_0 \in \partial \Omega \setminus (\mathcal K \cup \gamma_\lambda^+(\mathcal K, \lambda) \cup \gamma_\lambda^-(\mathcal K, \lambda))$. In a sufficiently small neighborhood $U$ of $(0, \theta_0)$, there exists $\psi^\pm_\omega \in C^\infty(U)$ smooth in $\omega$ up to $\epsilon = \Im \omega = 0$ such that 
    \[K_\omega(\theta, \theta') = \sum_{\pm} H(\pm \theta) \psi_\omega^\pm(\theta, \theta'), \qquad (\theta, \theta') \in U,\]
    and in a sufficiently small neighborhood $V$ of $(\theta_0, 0)$, there exists $\varphi^\pm_\omega \in C^\infty(U)$ locally uniformly in $\omega$ such that 
    \[K_\omega(\theta, \theta') = \sum_{\pm} H(\pm \theta') \varphi_\omega^\pm(\theta, \theta'), \qquad (\theta, \theta') \in U,\]
\end{enumerate}
    
\end{proposition}
\begin{proof}
    In the notation of Lemma~\ref{lem:smoothing}, 
\[K_\omega(\theta, \theta'; \Omega) = K_\omega(\theta, \theta'; \Omega_{0, \pm}) \quad \text{for} \quad \theta, \theta' \notin U_\pm\]
where $K_\omega(\theta, \theta'; \Omega_\bullet)$ denotes the Schwartz kernel of $d \mathcal C_\omega$ for the domain $\Omega_\bullet$. Since $\Omega_{0, \pm}$ is smooth near the corner at $0$, we have the explicit formulas for $K_\omega(\theta, \theta'; \Omega_{0, \pm})$ near $\kappa$.

    Observe that $0$ and and $\gamma_\lambda^\pm(0)$ are both noncharacteristic in the domain $\Omega_{0, \pm}$ as constructed in Lemma~\ref{lem:smoothing}. Therefore (1) follows from \cite[Lemma 4.13]{DWZ}. 

    (2) and (3) follows immediately from Proposition~\ref{lem:smoothing} and the fact that the Schwartz kernel of $d\mathcal C_\omega$ for a smooth $\lambda$-simple domain is smooth away from $\mathrm{Diag} \cup \{(\theta, \theta') : \theta' = \gamma_\lambda^\pm(\theta)\}$. 
\end{proof}

We end this section by collecting some convergence properties of the restricted single layer potential. First, cutting off the operator away from the corners, it is clear that~\cite[Lemma 4.16]{DWZ} still holds in the following sense:
\begin{lemma}\label{lem:ddC_easy}
    Assume that $\Omega$ is $\lambda$-simple for some $\lambda \in (0, 1)$. Then for $s > t$ and $\chi, \chi' \in \CIc(\partial \Omega \setminus \mathcal K)$, 
    \begin{equation}
        \|\chi(\partial_\omega^k \mathcal C_{\omega_j} - \partial_\lambda^k \mathcal C_{\lambda + i0}) \chi' \|_{H^{s + k} \to H^{t + 1}} \to 0
    \end{equation}
    for any sequence $\omega_j \to \lambda$ with $\Im \omega_j > 0$. 
\end{lemma}
Next, if only one of the cutoffs is supported away from the corners, we can still deduce the necessary convergence properties from~\cite[Lemma 4.16]{DWZ}, since the only difference is an application of a Heaviside cutoff. The following lemma then follows immediately from Proposition~\ref{prop:rest_of_kernel}. 
\begin{lemma}\label{lem:ddC_heaviside}
    Assume that $\Omega$ is $\lambda$-simple for some $\lambda \in (0, 1)$. Let $\chi_\bo \in \overline C^\infty(\partial \Omega)$ be such that $\chi_\bo = 1$ near $\mathcal K$, and let $\chi \in \CIc(\partial \Omega \setminus \mathcal K)$. Fix $a \in (-1, 0)$. Then for $s \ge a + \ha$ and $t < a + \ha$, we have
    \begin{equation}
        \|\chi(\partial_\omega^k \mathcal C_{\omega_j} - \partial_\lambda^k \mathcal C_{\lambda + i0}) \chi_\bo \|_{H^{s, a}_\bo \to H^{t + 1 - k}} \to 0.
    \end{equation}
    for any sequence $\omega_j \to \lambda$ with $\Im \omega_j > 0$. 

    On the other hand, for $s > t \ge a - \ha$, we have
    \begin{equation}
        \|\chi_\bo(\partial_\omega^k \mathcal C_{\omega_j} - \partial_\lambda^k \mathcal C_{\lambda + i0}) \chi \|_{H^{s + k} \to H^{t + 1, a}_\bo} \to 0.
    \end{equation}
    for any sequence $\omega_j \to \lambda$ with $\Im \omega_j > 0$. 
\end{lemma}

We remark that in application, we only use the convergence properties in case $k = 0$ in Lemmas~\ref{lem:ddC_easy} and~\ref{lem:ddC_heaviside}. In the case that $k = 1$, we just need the uniform boundedness properties of $\partial_\omega \mathcal C_\omega$ in $\omega$ locally in $\Re \omega$ up to $\Im \omega = 0$. This is clearly implied by the Lemmas.

\section{Propagation estimates}\label{sec:propagation_estimates}
Composing \eqref{eq:Cv_reduction} with the differential, we obtain the reduced equation 
\begin{equation}\label{eq:bdr}
    d\mathcal C_\omega v_\omega = g_\omega, \qquad g_\omega := d(R_\omega f)_{\partial \Omega \setminus \mathcal K} \in \overline C^\infty(\partial \Omega) \subset H^{\infty, 0-}_\bo(\partial \Omega; T^* \Omega).
\end{equation}
It suffices for us to understand $v_\omega$ in the limit as $\Im \omega \to 0$ since~\eqref{eq:fundamental_2} relates $v_\omega$ back to $u_\omega$. In this section, our goal is to obtain high frequency estimates for $v_\omega$. That is, we wish to control the high frequencies of $v_\omega$ using $g_\omega$ and the low frequencies of $v_\omega$ uniformly as $\omega$ approaches the real line. The main result of this section is the global estimate Proposition~\ref{prop:global_semifredholm}. This is crucial for proving the limiting absorption principle of Theorem~\ref{thm:spectral}, and it is proved by stitching together local estimates.

\subsection{Propagation of singularities and radial estimates}
We first prove discrete-time analogs to the microlocal propagation of singularities and the radial point estimates (see, for instance, \cite[\S E.4]{DZ_resonances} for a version of these estimates). These estimates will capture the behavior of $d \mathcal C_\omega$ away from the problematic corners and their reflections.

With a b-parameterization, $\partial \Omega$ is identified with $\mathbb S^1$. Away from the corners, this identification is smooth, so pseudodifferential operators on $\mathbb S^1$ with Schwartz kernel supported in a subset of $\partial \Omega \setminus \mathcal K \times \partial \Omega \setminus \mathcal K$ still makes perfect sense on $\partial \Omega$. The point here is that $d \mathcal C_\omega$ localized away from the corners can still be described as standard pseudodifferential operators. 

\subsubsection{Propagation of singularities away from corners and their reflections}
For convenience, we define the corners with all its reflections as
\begin{equation}
    \widetilde \corner := \corner \cup \gamma^+_\lambda(\corner) \cup \gamma^-_\lambda(\corner). 
\end{equation}
For $v$ that solves~\ref{eq:bdr}, we first estimate $v$ near $\theta_0 \in \partial \Omega \setminus \widetilde{\mathcal K}$. To do so, we need to analyze the Schwartz kernel of $\chi d \mathcal C_\omega$ where $\chi$ is a cutoff supported in a small neighborhood of $\theta_0$. From the kernel computations in \S\ref{sec:kernel_computation}, we see that $\chi d \mathcal C_\omega$ is singular in four regions, corresponding the point itself, the two reflections of the point by $\gamma^\pm$, and the corners. The Schwartz kernel of $\chi d \mathcal C_\omega$ is singular near $(\theta, \theta') = (\theta_0, \theta_0)$, $(\theta_0, \gamma^+(\theta_0))$, $(\theta_0, \gamma^-(\theta_0))$, and $(\theta_0, \kappa_j)$ for $\kappa_j \in \mathcal K$. See Figure~\ref{fig:POS}. We first organize some of the results from the previous section into the following lemma.  
\begin{lemma}\label{lem:psido_no_corner}
    Let $\mathcal J \Subset (0, 1)$ be such that $\Omega$ is $\lambda$-simple for all $\lambda \in \mathcal J$. Write $\omega = \lambda + i \epsilon$ for $\lambda \in \mathcal J$ and sufficiently small $\epsilon > 0$. Let $\theta_0 \in \partial \Omega \setminus \widetilde{\mathcal K}$. Then for $\chi_\cu \in C^\infty(\partial \Omega)$ supported in a sufficiently small neighborhood of $\theta_0$ with $\chi_\cu = 1$ near $\theta_0$, the operator $\chi_\cu d \mathcal C_\omega$ can be understood in pieces as follows.
    \begin{enumerate}
        \item $\chi_\cu d \mathcal C_\omega \widetilde \chi \in \Psi^{0}(\mathbb S^1)$ uniformly in $\omega$, where $\widetilde \chi \in C^\infty(\partial \Omega)$ has sufficiently small support and $\widetilde \chi = 1$ on $\supp \chi_\cu$. Furthermore, $\chi_\cu d \mathcal C_\omega \widetilde \chi$ has principle symbol 
        \[\sigma_0(\chi_c d \mathcal C_\omega \widetilde \chi) = -2 \pi i c_\omega \chi_\cu(\theta) \sgn(\xi),\]
        which means $\chi_\cu d \mathcal C_\omega \widetilde \chi$ elliptic on 
        \[\{(\theta, \xi) \in T^* \mathbb S^1: \chi_\cu(\theta) = 1\}\] 
        uniformly in $\omega$. 
        \item $\chi_\cu d \mathcal C_\omega (\widetilde \chi \circ \gamma^\pm_\lambda) = (\gamma^\pm_\lambda)^* (\chi_\cu \circ \gamma^\pm_\lambda) T^\pm_\omega (\widetilde \chi \circ \gamma^\pm_\lambda)$ where $T^\pm_\omega \in \Psi^0(\mathbb S^1)$ has principle symbol 
        \begin{equation}
            \sigma_0(T_\omega^\pm)=(\mp2 \pi ic_\omega + \mathcal O(\epsilon)) H(\pm\xi) e^{- \epsilon z_\omega^\pm(\theta)|\xi|}
        \end{equation}
        and wavefront set 
        \[\WF\big((\chi_\cu \circ \gamma^\pm_\lambda) T^\pm_\omega (\widetilde \chi \circ \gamma^\pm_\lambda)\big) \subset \{(\theta, \xi) \in T^* \mathbb S^1: \pm\xi > 0\}\]
        uniformly in $\omega$
        \item $\chi_\cu d \mathcal C_\omega (1 - \widetilde \chi - (\widetilde \chi \circ \gamma^-_\lambda) - (\widetilde \chi \circ \gamma^+_\lambda)): H^{-N, a}_\bo(\partial \Omega) \to H^N(\partial \Omega)$ uniformly in $\omega$ for $a > -1$.
    \end{enumerate}
\end{lemma}
We remark that it is implicit in this lemma that the supports of $\widetilde \chi$ and $\widetilde \chi \circ \gamma^\pm_\lambda$ must be disjoint from the corners.
\begin{proof}
    (1) follows immediately from Proposition~\ref{prop:easy_kernel}. Indeed, we see that the kernel of $\chi_c d \mathcal C_\omega \widetilde \chi$ is given by 
    \[c_\omega \chi_c(\theta)\widetilde \chi(\theta') \sum_{\pm} (\theta - \theta' \pm i0)^{-1}\]
    up to smooth functions uniform in $\omega$. The principle symbol and the ellipticity claim in (1) then follows.
    
    For (2), we see from Proposition \ref{prop:easy_kernel} that the kernel of $\chi_\cu d \mathcal C_\omega(\widetilde \chi \circ \gamma^\pm_\lambda)$ is given by 
    \[\chi_\cu(\theta) \widetilde \chi(\gamma^\pm_\lambda(\theta')) \tilde c_\omega(\theta') (\gamma_{\lambda}^\pm(\theta) - \theta' \pm i \epsilon z_\omega^\pm(\theta'))^{-1}\]
    modulo smooth terms uniform in $\omega$. It follows from \cite[Lemma 3.9]{DWZ} that this is a family of pseudodifferential operators pulled back by $\gamma^\pm_\lambda$ with principle symbol given by 
    \[\sigma_0(T^\pm_\omega) = (\mp2 \pi i + \mathcal O(\epsilon))H(\pm\xi) e^{- \epsilon z_\omega^\pm(\theta)|\xi|}.\]
    
    Finally, for (3), recall that the $L^2(\partial \Omega)$-dual of $H^{-N, a}_\bo(\partial \Omega)$ is $H^{N, -a - 1}(\partial \Omega)$. By Proposition~\ref{prop:rest_of_kernel}, we see that for each $\theta \supp \in \chi_\cu$, 
    \[K_\omega(\theta, \bullet) (1 - \widetilde \chi - (\widetilde \chi \circ \gamma^-_\lambda) - (\widetilde \chi \circ \gamma^+_\lambda)) \in H^{\infty, 0-}_{\bo}(\partial \Omega) \subset H^{N, -a - 1}(\partial \Omega)\]
    for $a > -1$ and depending smoothly on $\theta$ and $\omega = \lambda + i\epsilon$ up to $\epsilon = 0$. Therefore the mapping property in (3) follows. 
\end{proof}
\begin{figure}
    \centering
    \includegraphics{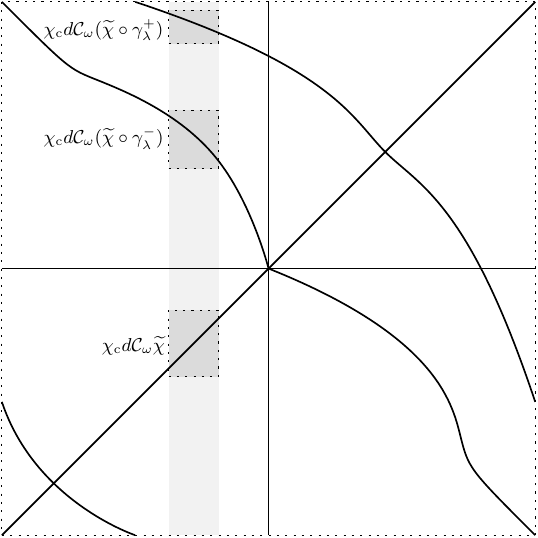}
    \caption{A diagram of the Schwartz kernel of $d\mathcal C_\omega$ with a type-$(+, +)$ corner only. The singularities of the Schwartz kernel are given by the solid black lines. $\chi_\cu d \mathcal C$ as in Lemma~\ref{lem:psido_no_corner} is shaded, and the darker shaded regions correspond to (1) and (2) of the lemma.}
    \label{fig:POS}
\end{figure}

\begin{proposition}\label{prop:POS}
    Let $\mathcal J \Subset (0, 1)$ be such that $\Omega$ is $\lambda$-simple for all $\lambda \in \mathcal J$. Write $\omega = \lambda + i \epsilon$ for $\lambda \in \mathcal J$ and sufficiently small $\epsilon > 0$. Let $\chi, \widetilde \chi \in \CIc(\partial \Omega \setminus \widetilde \corner)$ be such that $\widetilde \chi \equiv 1$ on $\supp \chi$, and assume that $\widetilde \chi$ is supported in a sufficiently small neighborhood of some $\theta_0 \in \partial \Omega \setminus \widetilde{\mathcal K}$. Suppose that $v \in H^{\infty, a}_\bo(\partial \Omega)$ solves~\eqref{eq:bdr} for $a > -1$, then
    \begin{equation}\label{eq:POS}
        \|\Pi^\pm \chi v\|_{H^s} \le C \big(\|\Pi^\mp (\widetilde \chi \circ \gamma_\lambda^\mp) v\|_{H^s} + \|\Pi^\pm \widetilde \chi g\|_{H^s} + \|v\|_{H^{-N, a}_\bo} \big).
    \end{equation}
    The constant $C$ is locally uniform in $\lambda$ and does not depend on $\epsilon$. 
\end{proposition}
\begin{proof}
    We focus on the $\Pi^+$ case since the $\Pi^-$ estimate is completely analogous. Introduce a cutoff $\chi_\cu \in C^\infty(\partial \Omega)$ so that $\chi_\cu = 1$ on $\supp \chi$ and $\widetilde \chi = 1$ on $\supp \chi_\cu$. In other words, $\chi_\cu$ is wedged in between $\chi$ and $\widetilde \chi$. Further assume that $\chi_\cu$ and $\widetilde \chi$ have sufficiently small support so that Lemma~\ref{lem:psido_no_corner} holds. Then by (1) of Lemma~\ref{lem:psido_no_corner} and the elliptic estimates of Lemma~\ref{lem:elliptic},
    \begin{equation}\label{eq:POS_elliptic_step}
        \|\Pi^+ \chi v\|_{H^s} \le C \big(\|\Pi^+ \chi_\cu d \mathcal C_\omega \widetilde \chi v\|_{H^s} + \|\chi_\cu v\|_{H^{-N}} \big).
    \end{equation}
    Now using (2) and (3) of Lemma~\ref{lem:psido_no_corner},
    \begin{equation}\label{eq:POS_decomp_step}
    \begin{aligned}
        \|\Pi^+ \chi_\cu d \mathcal C_\omega \widetilde \chi v\|_{H^s} &\le \|\Pi^+ \chi_\cu d \mathcal C_\omega (\widetilde \chi \circ \gamma^-_\lambda) v\|_{H^s} + \|\Pi^+ \chi_\cu d \mathcal C_\omega (\widetilde \chi \circ \gamma^+_\lambda) v\|_{H^s} \\
        &\qquad + \|\Pi^+ \chi_\cu d \mathcal C_\omega v\|_{H^s} + \|\chi_\cu d \mathcal C_\omega (1 - \widetilde \chi - (\widetilde \chi \circ \gamma^-_\lambda) - (\widetilde \chi \circ \gamma^+_\lambda))v\|_{H^s} \\
        &\le \|(\gamma^-_\lambda)^* \Pi^-(\chi_\cu \circ \gamma^-_\lambda) T^-_\omega (\widetilde \chi \circ \gamma^-_\lambda)v\|_{H^s} \\
        &\qquad + \|(\gamma^+_\lambda)^* \Pi^- (\chi_\cu \circ \gamma^+_\lambda) T^+_\omega (\widetilde \chi \circ \gamma^+_\lambda)v\|_{H^s} \\
        &\qquad \qquad + C \big(\|\Pi^+ \chi_\cu g\|_{H^s} + \|v\|_{H^{-N, a}_\bo} \big). 
    \end{aligned}
    \end{equation}
    The wavefront set condition in (2) of Lemma~\ref{lem:psido_no_corner} implies that $\Pi^- (\chi_\cu \circ \gamma^+_\lambda) T^+_\omega (\widetilde \chi \circ \gamma^+_\lambda) \in \Psi^{-\infty}(\mathbb S^1)$ uniformly in $\omega$, so  
    \begin{equation*}
        \|\Pi^+ \chi_\cu d \mathcal C_\omega \widetilde \chi v\|_{H^s} \le C \big(\| \Pi^-(\widetilde \chi \circ \gamma^-_\lambda) v\|_{H^s} + \|\Pi^+ \chi_\cu g\|_{H^s} + \|v\|_{H^{-N, a}_\bo} \big).
    \end{equation*}
    Combining with~\eqref{eq:POS_elliptic_step} yields 
    \begin{equation*}
        \|\Pi^+ \chi v\|_{H^s} \le C \big(\| \Pi^-(\widetilde \chi \circ \gamma^-_\lambda) v\|_{H^s} + \|\Pi^+ \chi_\cu g\|_{H^s} + \|v\|_{H^{-N, a}_\bo} \big)
    \end{equation*}
    where the constant $C$ may change from line to line but remains independent of $\epsilon$. Finally~\eqref{eq:POS} follows since $\widetilde \chi = 1$ on $\supp \chi_\cu$. 
\end{proof}
For $\theta_0 \in \partial \Omega \setminus \widetilde{\mathcal K}$, this estimate is saying that the positive (or negative) frequencies of $v$ near $\theta_0$ is controlled by the negative (or positive) frequencies of $v$ near $\gamma^-_\lambda(\theta)$ (or $\gamma_\lambda^+(\theta)$), given that $g$ is sufficiently regular. This can be interpreted as a discrete-time version of H\"ormander's theorem on propagation of singularities, which says singularities travel along the bicharacteristics of the Hamiltonian of the symbol of the operator (see \cite[\S23.2]{H3}). In our case, singularities travel along the Hamiltonian dynamics of $P(\lambda)$ restricted to the boundary.

Now we turn our attention to the periodic points $\Sigma_\lambda^\pm$ defined in \eqref{eq:periodic_points}. In a neighborhood of $\Sigma_\lambda$, we have better control over the constant in the estimate in~\eqref{eq:POS}.
\begin{lemma}\label{lem:radial_prep}
    Let $\mathcal J \Subset (0, 1)$ be such that $\Omega$ satisfies the Morse--Smale condition for all $\lambda \in \mathcal J$. Write $\omega = \lambda + i \epsilon$ for $\lambda \in \mathcal J$. Let $\theta_p \in \Sigma_\lambda$. Let $\chi, \widetilde \chi \in C^\infty(\partial \Omega)$ be such that $\widetilde \chi = 1$ on $\supp \chi$, and assume that $\supp \widetilde \chi$ lies in a sufficiently small neighborhood of $\theta_p$. Suppose that $v \in H^{-N, a}_\bo(\partial \Omega)$ solves~\eqref{eq:bdr} for $a > -1$, then
    \begin{equation}\label{eq:radial_prep}
        \|\Pi^\pm \chi v\|_{H^s} \le (1 + \delta)(\partial_\theta\gamma^\mp_\lambda(\theta_p))^{\ha + s}\|\Pi^\mp (\widetilde \chi \circ \gamma_\lambda^\mp) v\|_{H^s} + C \big(\|\Pi^\pm \widetilde \chi g\|_{H^s} + \|v\|_{H^{-N, a}_\bo} \big).
    \end{equation}
    for any $\delta > 0$ and all sufficiently small $\epsilon$ depending on $\delta$. The constant $C$ is locally uniform in $\lambda$ and independent of $\epsilon$.
\end{lemma}
\begin{proof}
    1. We first consider how Sobolev norms of 1-forms scale with dilation. We claim that if $f \in \CIc(\R; T^* \R)$, then 
    \begin{equation}\label{eq:sob_scaling}
        \|M_\sigma^* f\|_{H^s} \le \sigma^{\ha + s} \|f\|_{H^s} + C_{N, s, \sigma} \|f\|_{H^{-N}}
    \end{equation}
    where $M_\sigma:\R \to \R$ is given by $x \mapsto \sigma x$. Indeed, 
    \begin{align*}
        \|M_\sigma^* f\|_{H^s}^2 &= \int_{-\infty}^\infty \langle \xi \rangle^{2s} |\hat f(\xi/\sigma)|^2\, d \xi \\
        &= \int_{-\infty}^\infty \sigma \langle \sigma \xi \rangle^{2s} |\hat f(\xi)|^2 \, d \xi \\
        &\le \int_{-\infty}^\infty \big(\sigma^{1 + 2s} \langle \xi \rangle^{2s} + C_{N, s, \sigma} \langle \xi \rangle^{-2N}\big) |\hat f(\xi)|^2\, d \xi \\
        &= \sigma^{\ha + s} \|f\|_{H^s}^2 + C_{N, s, \sigma} \|f\|_{H^{-N}}^2
    \end{align*}
    as desired. 

    \noindent
    2. Now we consider the problem at hand, and again, we focus on the $\Pi^+$ case since the $\Pi^-$ case is similar. By Lemma~\ref{lem:linear_coords}, we can fix a b-parameterization $\mathbf x: \mathbb S^1 \to \partial \Omega$ so that when $\gamma^\pm_\lambda$ is lifted to $\mathbb S^1$, there exists a neighborhood $U$ of $\theta_p$ such that 
    \[\partial_\theta \gamma^\pm_\lambda (\theta) = \partial_\theta \gamma^\pm_\lambda (\theta_p), \quad \theta \in U.\]
    We assume that $\supp \widetilde \chi \subset U$, and that $\supp \widetilde \chi$ is sufficiently small so that Proposition~\ref{prop:POS} holds. It is easy to see that \eqref{eq:POS_decomp_step} still holds with $\chi_\cu$ constructed in the same as as in the proof. We can gain better control over the first term of the right-hand-side of \eqref{eq:POS_decomp_step}. Observe that 
    \begin{equation}\label{eq:radial_prep_com}
        \|(\gamma^-_\lambda)^* \Pi^-(\chi_\cu \circ \gamma^-_\lambda) T^-_\omega (\widetilde \chi \circ \gamma^-_\lambda)v\|_{H^s} \le \|(\gamma^-_\lambda)^* (\chi_\cu \circ \gamma^-_\lambda) T^-_\omega (\widetilde \chi \circ \gamma^-_\lambda) \Pi^- v \|_{H^s} + C\|v\|_{H^{-N, a}_\bo}. 
    \end{equation}
    By Lemma~\ref{lem:CoV} and (2) of Lemma~\ref{lem:psido_no_corner}, 
    \begin{gather*}
        \sigma_0\big((\gamma^-_\lambda)^* \Pi^-(\chi_\cu \circ \gamma^-_\lambda) T^-_\omega (\widetilde \chi \circ \gamma^-_\lambda) (\gamma^-_\lambda)^*\big) \le 2 \pi|c_\omega| + \mathcal O(\Im \omega), \\
        \WF\big((\gamma^-_\lambda)^* \Pi^-(\chi_\cu \circ \gamma^-_\lambda) T^-_\omega (\widetilde \chi \circ \gamma^-_\lambda) (\gamma^-_\lambda)^*\big) \subset \supp \chi_\cu \times \R_+.
    \end{gather*}
    Thus continuing~\eqref{eq:radial_prep_com} and using Lemma~\ref{lem:symbol_bound}, 
    \begin{equation}\label{eq:radial_prep_cov}
    \begin{aligned}
        \|(\gamma^-_\lambda)^* \Pi^-(\chi_\cu \circ \gamma^-_\lambda) T^-_\omega (\widetilde \chi \circ \gamma^-_\lambda)v\|_{H^s} &\le (2 \pi|c_\omega| + \delta)\|(\gamma^-_\lambda)^* \widetilde \chi \Pi^- v\|_{H^s} + C\|v\|_{H^{-N, a}_\bo} \\
        &\le (2 \pi|c_\omega| + \delta)(\partial_\theta \gamma_\lambda^-(\theta_p))^{\ha + s} \|\Pi^-(\widetilde \chi \circ \gamma^-_\lambda) v\|_{H^s} \\
        &\qquad + C \|v\|_{H^{-N, a}_\bo}
    \end{aligned}
    \end{equation}
    for any $\delta > 0$ and all sufficiently small $\epsilon$ depending on $\delta$. In the second inequality above, we used Step 1 together with the linear choice of coordinates. Again using Lemma~\ref{lem:symbol_bound}, we can be explicit about the constant in~\eqref{eq:POS_elliptic_step} and see that 
    \begin{equation}\label{eq:POS_elliptic_constant}
        \|\Pi^+ \chi v\|_{H^s} \le ((2 \pi|c_\omega|)^{-1} + \delta) \|\Pi^+ \chi_\cu d \mathcal C_\omega \widetilde \chi v\|_{H^s} + C \|\chi_\cu v\|_{H^{-N}}.
    \end{equation}
    for all $\delta > 0$ and some $C$ depending on $\delta$ and $N$. Combining~\eqref{eq:POS_decomp_step}, \eqref{eq:radial_prep_cov}, and~\eqref{eq:POS_elliptic_constant}, we obtain the desired estimate~\eqref{eq:radial_prep}.  
\end{proof}

\subsubsection{Source/sink estimates}\label{sec:source/sink}
By our choice of coordinates from Lemma~\ref{lem:linear_coords}, it is clear that there exists $\chi_\pm, \widetilde \chi_\pm \in \CIc(\Sigma^\pm_\lambda(\delta))$ such that 
\begin{equation}\label{eq:chi_pm}
\begin{gathered}
    \chi_\pm = 1 \ \text{near}\ \Sigma_\lambda^\pm, \qquad \widetilde \chi_\pm = 1 \ \text{near} \ \supp \chi_\pm, \\
    \widetilde \chi_\pm \circ b_\lambda^{\pm 1} \in \CIc(\Sigma^\pm_\lambda(\delta)), \qquad \chi_\pm = 1 \ \text{on} \ \supp(\widetilde \chi_\pm \circ b_\lambda^{\mp 1}).
\end{gathered}
\end{equation}
In the previous section, Proposition~\eqref{prop:POS} roughly tells us that the behavior of $v$ near $\gamma^\pm(\theta_0)$ controls the behavior of $v$ near $\theta_0$. However, this is a tautology near periodic points of the the chess billiard map; the propagation estimate will just tell us that the behavior of $v$ near the periodic points is controlled by the behavior near the periodic points. 

What saves us is the Morse--Smale condition. The fact that the periodic points are hyperbolic allows us to say more about the regularity of $v$ near the periodic points. We start with a low regularity estimate near for positive frequencies near the attracting set $\Sigma_\lambda^+$ (and for negative frequencies near the repulsive set $\Sigma_{\lambda^-}$). We often refer to this as the sink estimate.  
\begin{proposition}\label{prop:sink}
    Let $\mathcal J \Subset (0, 1)$ be such that $\Omega$ satisfies the Morse--Smale condition for all $\lambda \in \mathcal J$. Write $\omega = \lambda + i \epsilon$ for $\lambda \in \mathcal J$ and sufficiently small $\epsilon > 0$. Let $\chi_\pm, \widetilde \chi_\pm \in \CIc(\Sigma_\lambda(\delta))$ be as in~\eqref{eq:chi_pm}. Suppose that $v \in H^{\infty, a}_\bo$, $a > -1$, solves~\eqref{eq:EBVP}. Then for $s < -1/2$, we have
    \begin{multline}\label{eq:high_reg}
        \|\Pi^\mp \chi_\pm v\|_{H^s} \le C \big(\|\Pi^\mp (\widetilde \chi_\pm \circ b_\lambda^{\pm 1} - \chi_\pm) v\|_{H^s} \\
        + \|\Pi^\mp \widetilde \chi_\pm g + \Pi^\pm(\widetilde \chi_\pm \circ \gamma^\pm_\lambda) g\|_{H^s} + 
        \|v\|_{H^{-N, a}_\bo} \big).
    \end{multline}
    The constant $C$ is locally uniform in $\lambda$. 
\end{proposition}
\begin{proof}
    Applying Lemma~\ref{lem:radial_prep} twice, we find that for any $\delta > 0$, 
    \begin{multline}
        \|\Pi^\mp \chi_\pm v\|_{H^s} \le (\sigma_{\mathrm{sink}, \pm}^{\ha + s} + \delta) \|\Pi^\mp (\widetilde \chi_\pm \circ b_\lambda^{\pm 1}) v\|_{H^s} \\
        + C \big(\|\Pi^\mp \widetilde \chi_\pm g + \Pi^\pm(\widetilde \chi_\pm \circ \gamma^\pm_\lambda) g\|_{H^s} + 
        \|v\|_{H^{-N, a}_\bo} \big),
    \end{multline}
    for all sufficiently small $\epsilon$, where
    \[\sigma_{\mathrm{sink}, \pm} = \max\{\partial_\theta b^{\mp 1}_\lambda(\theta): \theta \in \Sigma_\lambda^\pm\}.\]
    In the convenient choice of coordinates from Lemma~\ref{lem:linear_coords}, it is clear that $1 < \sigma_{\mathrm{sink}, \pm}$. Therefore, for $s < -1/2$, \eqref{eq:high_reg} follows upon subtracting $(\sigma_{\mathrm{sink}, \pm}^{\ha + s} + \delta)\|\Pi^\pm \chi_\pm v\|_{H^s}$ from both sides using a sufficiently small $\delta$. 
\end{proof}

For positive frequencies near $\Sigma_\lambda^-$ and negative frequencies near $\Sigma_\lambda^+$, we have a high frequency estimate. We will often refer to this as the source estimate. 
\begin{proposition}\label{prop:source}
    Let $\mathcal J \Subset (0, 1)$ be such that $\Omega$ satisfies the Morse--Smale condition for all $\lambda \in \mathcal J$. Write $\omega = \lambda + i \epsilon$ for $\lambda \in \mathcal J$ and $\epsilon > 0$. Let $\chi_\pm, \widetilde \chi_\pm \in \CIc(\Sigma_\lambda(\delta))$ be as in~\eqref{eq:chi_pm}. Suppose that $v \in H^{\infty, a}_\bo$, $a > -1$, solves~\eqref{eq:EBVP}. Then for $s > -1/2$, 
    \begin{equation}\label{eq:low_reg}
        \|\Pi^\pm \chi_\pm v\|_{H^s} \le C \big(\|\Pi^\pm \widetilde \chi_\pm g + \Pi^\mp(\widetilde \chi_\pm \circ \gamma^\pm_\lambda) g\|_{H^s} + \|v\|_{H^{-N, a}_\bo} \big).
    \end{equation}
    The constant $C$ is locally uniform in $\lambda$ and $\epsilon > 0$ sufficiently small. 
\end{proposition}
\begin{proof}
    Similar to the proof of Proposition~\ref{prop:sink}, we apply Lemma~\ref{lem:radial_prep} twice and find that for any $\delta > 0$, 
    \begin{multline}\label{eq:low_reg_twice}
        \|\Pi^\pm \chi_\pm v\|_{H^s} \le (\sigma_{\mathrm{source}, \pm}^{\ha + s} + \delta)\|\Pi^\pm (\widetilde \chi_\pm \circ b^{\mp 1}_\lambda)v\|_{H^s} \\
        + C \big(\|\Pi^\pm \widetilde \chi_\pm g + \Pi^\mp(\widetilde \chi_\pm \circ \gamma^\pm_\lambda) g\|_{H^s} + 
        \|v\|_{H^{-N, a}_\bo} \big),
    \end{multline}
    for all sufficiently small $\epsilon > 0$, where
    \[\sigma_{\mathrm{source}, \pm} = \max\{\partial_\theta b^{\pm 1}_\lambda(\theta): \theta \in \Sigma_\lambda^\pm\}.\]
    In particular, we see that $\sigma_{\mathrm{source}, \pm} < 1$, so~\eqref{eq:low_reg} follows immediately from~\eqref{eq:low_reg_twice} for $s > -1/2$ since $\chi_\pm \equiv 1$ on the support of $\chi_\pm \circ b_\lambda^{\mp1}$.
\end{proof}

\begin{Remark}
    The sign of $\epsilon = \Im \lambda$ controls the direction of the dynamics, and we have worked exclusively with $\epsilon > 0$. For $\epsilon < 0$, the direction of the dynamics is flipped, so sources become sinks and sinks become sources. The analogous estimates hold up to sign changes for $\epsilon < 0$. 
\end{Remark}

\subsubsection{Propagation near corner reflections}Finally, there is a propagation estimate that is mildly influenced by the corner. All the estimates so far are about controlling $v$ away from the corner as well as the reflection of corners. We see that to control the positive frequencies near $\theta \in \partial \Omega$, we need the negative frequencies near $\gamma^-_\lambda(\theta)$. In neighborhoods of reflections of the corners, the same is essentially still true. To be more precise we work near a type-$(+, +)$ corner since the other estimates can be obtained by symmetry. Furthermore, we center the corner at $0$ for convenience. We start with an analogue of Lemma~\ref{lem:psido_no_corner}. 
\begin{lemma}\label{lem:psido_corner_ref}
    Let $\mathcal J \Subset (0, 1)$ be a sufficiently small subinterval, and assume that $\Omega$ satisfies the Morse--Smale condition for all $\lambda \in \mathcal J$. Write $\omega = \lambda + i \epsilon$ for $\lambda \in \mathcal J$ and sufficiently small $\epsilon > 0$. Then for $\chi_\cu \in C^\infty(\partial \Omega)$ supported in a sufficiently small neighborhood of $\gamma^+_\lambda(0)$ with $\chi_\cu = 1$ near $\gamma^+_\lambda(0)$ for all $\lambda \in \mathcal J$, the operator $\chi_\cu d \mathcal C_\omega$ can be understood in pieces as follows.
    \begin{enumerate}
        \item $\chi_\cu d \mathcal C_\omega \widetilde \chi \in \Psi^{0}(\mathbb S^1)$ uniformly in $\omega$, where $\widetilde \chi \in C^\infty(\partial \Omega)$ has sufficiently small support and $\widetilde \chi = 1$ on $\supp \chi_\cu$. Furthermore, $\chi_\cu d \mathcal C_\omega \widetilde \chi$ has principle symbol 
        \[\sigma_0(\chi_c d \mathcal C_\omega \widetilde \chi) = -2 \pi i c_\omega \chi_\cu(\theta) \sgn(\xi),\]
        which means $\chi_\cu d \mathcal C_\omega \widetilde \chi$ elliptic on 
        \[\{(\theta, \xi) \in T^* \mathbb S^1: \chi_\cu(\theta) = 1\}\] 
        uniformly in $\omega$.  

        \item $\chi_\cu d \mathcal C_\omega (\widetilde \chi \circ \gamma^-_\lambda) = (\gamma^-_\lambda)^* (\chi_\cu \circ \gamma^-_\lambda) T^-_\omega (\widetilde \chi \circ \gamma^-_\lambda)$ where $T^-_\omega \in \Psi^0(\mathbb S^1)$ 
        has principle symbol 
        \begin{equation}
            \sigma_0(T_\omega^-)= (2 \pi ic_\omega + \mathcal O(\epsilon)) H(-\xi) e^{- \epsilon z_\omega^-(\theta)|\xi|}
        \end{equation}
        and wavefront set 
        \[\WF\big((\chi_\cu \circ \gamma^-_\lambda) T^-_\omega (\widetilde \chi \circ \gamma^-_\lambda)\big) \subset \{(\theta, \xi) \in T^* \mathbb S^1: \xi < 0\}\]
        uniformly in $\omega$

        \item $\chi_\cu d \mathcal C_\omega (\widetilde \chi \circ \gamma^+_\lambda) = \sum_{\pm} (\gamma^+_\lambda)^* (\chi_\cu \circ \gamma^+_\lambda) T^+_{\omega, \pm} (\widetilde \chi \circ \gamma^+_\lambda) H(\pm \theta')$ where $T^+_{\omega, \pm} \in \Psi^0(\mathbb S^1)$ has principle symbol 
        \begin{equation}\label{eq:split_T_symbol}
            \sigma_0(T^+_{\omega, \pm}) = (-2 \pi ic_\omega + \mathcal O(\epsilon)) H(\xi) e^{-\epsilon \tilde z^+_{\omega, \pm}(\theta)|\xi|}
        \end{equation}
        uniformly in $\omega$ where $\tilde z^+_{\omega, \pm} \in C^\infty(\mathbb S^1)$ satisfies
        \begin{equation}\label{eq:tilde_z}
            \tilde z^+_{\omega, \pm}(\theta) = z^+_\omega(\theta) \ \text{for} \ \pm \theta > 0\ \text{and}\ \theta \in \supp (\chi \circ \gamma^+_\lambda).
        \end{equation} 

        \item $\chi_\cu d \mathcal C_\omega (1 - \widetilde \chi - (\widetilde \chi \circ \gamma^-_\lambda) - (\widetilde \chi \circ \gamma^+_\lambda)): H^{-N, a}_\bo(\partial \Omega) \to H^N(\partial \Omega)$ uniformly in $\omega$ for $a > -1$.
    \end{enumerate}
\end{lemma}
\begin{Remark}
    Observe that in (3), multiplication by Heaviside function in $\theta$ is unambiguous since $\widetilde \chi \circ \gamma^+_\lambda$ is supported in a small neighborhood of $\gamma^+_\lambda(\theta_0)$. Furthermore, observe that all choices of $\tilde z^+_{\omega, \pm} \in C^\infty(\mathbb S^1)$ gives the same expansion for $\chi_\cu d \mathcal C_\omega(\widetilde \chi \circ \gamma^+_\lambda)$. 
\end{Remark}
\begin{proof}
    (1), (2), and (4) are identical to that of (1), (2) and (3) in Lemma~\ref{lem:psido_no_corner} respectively. 

    For (3), with the same notation from Proposition~\ref{prop:rest_of_kernel}, we see that the Schwartz kernel of $\chi_\cu d \mathcal C_\omega( \widetilde \chi \circ \gamma^+_\lambda)$ is given by 
    \begin{equation}\label{eq:split_kernel}
        \chi_\cu(\theta) \widetilde \chi(\gamma^+_\lambda(\theta')) \sum_{\pm} H(\pm\theta') \big[\tilde c_\omega^+(\theta')(\gamma^+_\lambda(\theta) - \theta' + i \epsilon z^+_\omega(\theta'))^{-1} + \psi^\pm_\omega(\theta, \theta') \big].
    \end{equation}
    Here, $z_\omega^+ \in \overline C^\infty(\partial \Omega)$, which in particular does not pull back to a smooth function on $\mathbb S^1$ by the parameterization. However, $z_\omega^+$ can be extended from the $\theta > 0$ or $\theta < 0$ to $\tilde z^+_{\omega, \pm}$, so~\eqref{eq:split_kernel} can be interpreted as the sum of two pulled-back pseudodifferential operators composed with cutoffs by the Heaviside function. In particular, (3) holds with $T^+_{\omega, \pm} \in \Psi^0_\omega(\mathbb S^1)$ given by the Schwartz kernel 
    \[\partial_\theta \gamma'(\theta) \tilde c_\omega^+(\theta - \theta' + i \epsilon z_\omega^+(\theta'))^{-1}\]
    modulo smooth functions uniform in $\omega$. Then \cite[Lemma 3.9]{DWZ} gives us the principal symbol formula in~\eqref{eq:split_T_symbol} uniformly in $\omega$. 
\end{proof}
\begin{figure}
    \centering
    \includegraphics{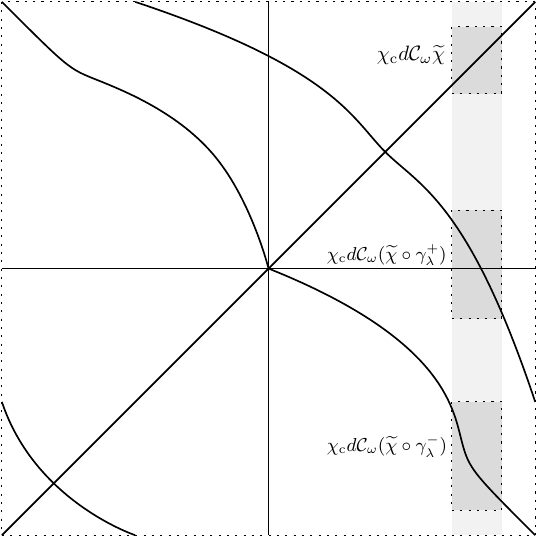}
    \caption{A diagram of the Schwartz kernel of $d\mathcal C_\omega$ with a type-$(+, +)$ corner only. The singularities of the Schwartz kernel are given by the solid black lines. $\chi_\cu d \mathcal C$ as in Lemma~\ref{lem:psido_corner_ref} is shaded, and the darker shaded regions correspond to (1), (2), and (3) of the lemma.}
    \label{fig:refPOS}
\end{figure}

\begin{proposition}\label{prop:ref_corner_POS}
    Let $\mathcal J \Subset (0, 1)$ be such that $\Omega$ is $\lambda$-simple for all $\lambda \in \mathcal J$. Write $\omega = \lambda + i \epsilon$ for $\lambda \in \mathcal J$ and sufficiently small $\epsilon > 0$. Let $\chi, \widetilde \chi \in C^\infty(\partial \Omega)$ be supported in a sufficiently small neighborhood of $\gamma^+_\lambda(0)$ and $\widetilde \chi = 1$ on $\supp \chi$. Assume that $v \in H^{\infty, a}_\bo$ for $a \in (-1, 0)$. Then for $s \in (-\ha, \ha)$, 
    \begin{equation}\label{eq:corner_POS_hard}
    \|\Pi^- \chi v \|_{H^{s}} \le C \big(\|(\widetilde \chi \circ \gamma^+_\lambda) v\|_{H^{s, s - \ha}_\bo} + \|\Pi^- \widetilde \chi g\|_{H^{s}} + \|v\|_{H^{-N, a}_\bo} \big).
    \end{equation}
    For all $s \in \R$, we have
    \begin{equation}\label{eq:corner_POS_easy}
    \|\Pi^+ \chi v \|_{H^s} \le C \big(\|\Pi^-(\widetilde \chi \circ \gamma^-_\lambda) v\|_{H^s} + \|\Pi^+ \widetilde \chi g\|_{H^s} + \|v\|_{H^{-N, a}_\bo} \big).
    \end{equation}
    The constants are locally uniform in $\lambda$ for sufficiently small $\epsilon > 0$. 
\end{proposition}
\begin{proof}
    1. Wedge a cutoff $\chi_\cu \in C^\infty(\partial \Omega)$ in between $\chi$ and $\widetilde \chi$ so that $\chi_\cu = 1$ on $\supp \chi$ and $\widetilde \chi = 1$ on $\supp \chi_\cu$. It follows from (1) of Lemma~\ref{lem:psido_corner_ref} and the elliptic estimates that
    \begin{equation}\label{eq:ref_corner_elliptic}
        \|\Pi^\pm \chi v\|_{H^s} \le C \big(\|\Pi^\pm \chi_\cu d \mathcal C_\omega \widetilde \chi v\|_{H^s} + \|\chi_\cu v\|_{H^{-N}}  \big)
    \end{equation}
    First consider the $\Pi^-$ case. From (2) of Lemma~\ref{lem:psido_corner_ref}, we see that
    \begin{equation}\label{eq:refPOS++}
        \|\Pi^- \chi_\cu d \mathcal C_\omega (\widetilde \chi \circ \gamma^-_\lambda)v \|_{H^s} \le \|\Pi^+(\chi_\cu \circ \gamma_\lambda^-) T^-_\omega (\widetilde \chi \circ \gamma_\lambda^-) v\|_{H^s} \lesssim \|(\widetilde \chi \circ \gamma^-_\lambda)v\|_{H^{-N}}.    \end{equation}
    where the hidden constant is locally uniform in $\lambda$. On the other hand, it follows from (3) of Lemma~\ref{lem:psido_corner_ref} that
    \begin{equation*}
        \|\Pi^- \chi_\cu d \mathcal C_\omega (\widetilde \chi \circ \gamma_\lambda^+)v\|_{H^s} \le \sum_{\pm} \|\Pi^+ (\chi_\cu \circ \gamma_\lambda^+) T^+_{\omega, \pm}(\widetilde \chi \circ \gamma^+_\lambda) H(\pm \theta)v\|_{H^s}
    \end{equation*}
    where $T_{\omega, \pm}^+ \in \Psi^{0}(\mathbb S^1)$ locally uniformly in $\omega$. Therefore, it follows from Lemma~\ref{lem:b_sob_to_sob} that for $s \in (- \ha, \ha)$, 
    \begin{equation*}
        \|\Pi^+ (\chi_\cu \circ \gamma_\lambda^+) T^+_{\omega, \pm}(\widetilde \chi \circ \gamma^+_\lambda) H(\pm \theta)v\|_{H^s} \lesssim \|(\widetilde \chi \circ \gamma_\lambda^+)v\|_{H^{s, s - \ha}_\bo}.
    \end{equation*}
    Therefore,
    \begin{equation}\label{eq:refPOS+-}
        \|\Pi^- \chi_\cu d \mathcal C_\omega (\widetilde \chi \circ \gamma_\lambda^+)v\|_{H^s} \lesssim \|(\widetilde \chi \circ \gamma_\lambda^+)v\|_{H^{s, s - \ha}_\bo}. 
    \end{equation}
    Then applying~\eqref{eq:refPOS++} and~\eqref{eq:refPOS+-} to~\eqref{eq:ref_corner_elliptic} along with (4) of Lemma~\ref{lem:psido_corner_ref}, we find
    \begin{equation*}
        \|\Pi^- \chi v\|_{H^s} \lesssim \|(\widetilde \chi \circ \gamma_\lambda^+)v\|_{H^{s, s - \ha}_\bo} + \|\Pi^- \chi_\cu g \|_{H^s} + \|v\|_{H^{-N, a}_\bo}
    \end{equation*}
    as desired. 

    \noindent
    2. It remains to consider the $\Pi^+$ case. From (2) of Lemma~\ref{lem:psido_corner_ref}, we see that
    \begin{equation}\label{eq:refPOS-+}
    \begin{aligned}
        \|\Pi^+ \chi_\cu d \mathcal C_\omega (\widetilde \chi \circ \gamma^-_\lambda)v \|_{H^s} &\le \|\Pi^-(\chi_\cu \circ \gamma_\lambda^-) T^-_\omega (\widetilde \chi \circ \gamma_\lambda^-) v\|_{H^s} \\
        &\lesssim \|\Pi^-(\widetilde \chi \circ \gamma^-_\lambda)v\|_{s} + \|v\|_{H^{-N, a}_\bo}.
    \end{aligned}
    \end{equation}
    On the other hand, it follows from (3) of Lemma~\ref{lem:psido_corner_ref} that
    \begin{equation*}
        \|\Pi^+ \chi_\cu d \mathcal C_\omega (\widetilde \chi \circ \gamma_\lambda^+)v\|_{H^s} \le \sum_{\pm} \|\Pi^- (\chi_\cu \circ \gamma_\lambda^+) T^+_{\omega, \pm}(\widetilde \chi \circ \gamma^+_\lambda) H(\pm \theta)v\|_{H^s}.
    \end{equation*}
    Observe that $T_{\omega, \pm}^+$ satisfies
    \[\Pi^- T_{\omega, \pm}^+ \in \Psi^{-\infty}(\mathbb S^1)\]
    uniformly in $\omega$. Therefore, for $a > -1$, 
    \[\|\Pi^- (\chi_\cu \circ \gamma_\lambda^+) T^+_{\omega, \pm}(\widetilde \chi \circ \gamma^+_\lambda) H(\pm \theta)v\|_{H^s} \lesssim \|v\|_{H^{-N, a}_{\bo}}.\]
    It then follows that 
    \begin{equation}\label{eq:refPOS--}
        \|\Pi^+ \chi_\cu d \mathcal C_\omega (\widetilde \chi \circ \gamma_\lambda^+)v\|_{H^s} \lesssim \|v\|_{H^{-N, a}_\bo}
    \end{equation}
    Then applying~\eqref{eq:refPOS-+} and~\eqref{eq:refPOS--} to~\eqref{eq:ref_corner_elliptic} along with (4) of Lemma~\ref{lem:psido_corner_ref}, we find
    \[\|\Pi^+ \chi v \|_{H^s} \le C \big(\|\Pi^-(\widetilde \chi \circ \gamma^-_\lambda) v\|_{H^s} + \|\Pi^+ \chi_\cu g\|_{H^s} + \|v\|_{H^{-N, a - \delta}_\bo} \big)\]
    as desired.
\end{proof}

\subsection{Propagation through corner}
It remains to analyze what happens when the propagation hits the corners. The idea is to break $\partial \Omega$ apart at the corners. Then distributions (or rather distributional one-forms) supported near the corner can be considered as $\C^2$-valued distributions on $\R_+$ supported near $0$. Note that this procedure cannot be performed on arbitrary distributions. In particular, if there is a distribution supported at the corner, then there is no canonical way to cut this distribution in half at the corner. On the other hand, if there are two extendible distributions on $\R_+$ equal to $1/x$ near $x = 0$, there is no way canonical way to glue them together and make sense of it distributionally. 

More concretely, we will use Lemmas~\ref{lem:cutting} and~\ref{lem:b_sob_to_sob} to make sense of this cutting and gluing procedure. Again, we focus on a neighborhood of a type-$(+, +)$ corner, which we may center at $0$. Let $\mathbf x$ be a b-parameterization centered at the corner. Fix a cutoff 
\begin{equation}\label{eq:sym_cutoff}
    \chi_\cu \in \overline C^\infty(\partial \Omega), \qquad \chi_\cu = 1 \ \text{near}\ 0, \qquad \chi_\cu = \chi_\cu \circ \gamma^-_\lambda.
\end{equation}
Note that this is possible since $\gamma^-_\lambda$ is an involution that fixes $0$. We call any cutoff satisfying~\eqref{eq:sym_cutoff} a symmetric cutoff, and this choice simply makes the computations near the corner a bit cleaner. In particular, in the coordinates~\eqref{eq:good_param}, symmetric cutoffs are simply even functions. Define the cutting map with respect to $\mathbf x$ by
\begin{equation}\label{eq:cutting}
\begin{gathered}
    \mathfrak X: H^s(\partial \Omega; T^* \partial \Omega) \to H^{s, s - \ha}_\bo(\R_+; \C^2), \qquad s \in (-\ha, \ha),\\
    v(\theta) d\theta \mapsto \begin{pmatrix} H(\theta) \chi_\cu(\theta) v(\theta) \\ H(\theta) \chi_\cu(\theta)v(-\theta) \end{pmatrix} \quad \text{in coordinates $\theta \mapsto \mathbf x(\theta)$}.
\end{gathered}
\end{equation}
Now fix another symmetric cutoff 
\begin{equation}\label{eq:sym_big_cutoff}
    \widetilde \chi \in \overline C^\infty(\partial \Omega), \qquad \widetilde \chi = 1 \ \text{near}\ \supp \chi_\cu, \qquad \widetilde \chi = \widetilde \chi \circ \gamma^-_\lambda.
\end{equation}
We define the gluing map 
\begin{equation}\label{eq:gluing}
\begin{gathered}
    \mathfrak G: H^{s, s - \ha}_\bo(\R_+; \C^2) \to H^s(\partial \Omega; T^* \partial \Omega), \qquad s \in (-\ha, \ha),\\
    \begin{pmatrix} v_1(\theta) \\ v_2(\theta) \end{pmatrix} \mapsto \widetilde \chi(\theta) (v_1(\theta) + v_2(-\theta))\, d \theta \quad \text{in coordinates $\theta \mapsto \mathbf x(\theta)$},
\end{gathered}
\end{equation}
where we use the fact that $v_j \in H^{s, s - \ha}_\bo(\R_+)$ can be identified with $v_j \in H^s(\R)$ by extending by zero. With the correct choice of b-parameterization $\mathbf x$, we will see that 
\begin{equation}\label{eq:Q}
    Q_\omega := \mathfrak X\circ  d\mathcal C_\omega \circ \mathfrak G
\end{equation}
is a family of b-operators elliptic near $0$ uniformly in $\omega$. We will often drop $\omega$ from the notation. 

\begin{lemma}\label{lem:transform_to_b_operator}
    Let $\mathcal J \Subset (0, 1)$ be such that $\Omega$ is $\lambda$-simple for all $\lambda \in \mathcal J$. Let $\omega = \lambda + i\epsilon$ for $\lambda \in \mathcal J$ and $\epsilon > 0$. Assume that the supports of the cutoffs $\chi_\cu$ and $\widetilde \chi$ from~\eqref{eq:sym_cutoff} and~\eqref{eq:sym_big_cutoff} are supported in a neighborhood where the parameterization~\eqref{eq:good_param} holds. Using these cutoffs and the parameterization~\eqref{eq:good_param} to define $\mathfrak X$ and $\mathfrak G$, the resulting operator $Q_\omega$ lies in $\widetilde \Psi^{0, \mathcal E}_\bo(\R_+; \C^2)$ where $\mathcal E = (E_{\lb}, E_{\rb})$ is given by
    \begin{equation}\label{eq:index_set}
        E_\lb = \{(k, 0) : k \in \Z,\, k \ge 0\}, \quad E_\rb = \{(k, 0) : k \in \Z,\, k \ge 1\}.
    \end{equation}
    Moreover, $Q_\omega$ takes the form 
    \begin{equation}\label{eq:DI_form}
        Q_\omega = \chi_\cu Q_{\omega, \mathrm{dil}} \widetilde \chi
    \end{equation}
    for some dilation invariant $Q_{\omega, \mathrm{dil}} \in \Psi^{0, \mathcal E}_{\mathrm{dil}} (\R_+; \C^2)$. Furthermore, $Q_\omega$ is uniformly elliptic on $(\supp(1 - \chi_\cu))^C$ for all sufficiently small $\epsilon$, locally uniformly in $\lambda$.
\end{lemma}
\begin{proof}
Let $\alpha$ be the characteristic ratio for the corner at $0$ as in Definition~\ref{def:corner_type}. Then using the parameterization~\eqref{eq:good_param} to define $\mathfrak X$ and $\mathfrak G$, we see that the Schwartz kernel of $\mathfrak X\circ  d\mathcal C_\omega \circ \mathfrak G$ is given by a $2 \times 2$ matrix with entries in $\overline{\mathcal E}'(\R_+^2)$. In particular, Proposition~\ref{prop:corner_kernel} gives us the exact formula for the Schwartz kernel, the first column of which is given by 
\begin{equation}\label{eq:kernel_C1}
c_\omega \chi_\cu(\theta) \widetilde \chi(\theta') \begin{pmatrix} (\theta - \theta' + i0)^{-1} + (\theta - \theta' - i0)^{-1} \\ -(\theta - (1 + i \epsilon z_\omega^{-+})\theta')^{-1} - (\theta + \alpha(1 + i \epsilon z_\omega^{++}) \theta')^{-1}\end{pmatrix}
\end{equation}  
and the second column given by 
\begin{equation}\label{eq:kernel_C2}
    c_\omega \chi_\cu(\theta) \widetilde \chi(\theta') \begin{pmatrix} (\theta - (1 - i \epsilon z_\omega^{--})\theta')^{-1} + (\theta + \alpha^{-1}(1 - i \epsilon z_\omega^{+-})\theta')^{-1} \\ -(\theta - \theta' + i0)^{-1} - (\theta - \theta' - i0)^{-1}  \end{pmatrix}.
\end{equation}
Observe that 
\begin{equation}
\begin{gathered}
    c_\omega (\theta - \theta' \pm i0)^{-1} = \frac{1}{2 \pi} \int_\R e^{i(\theta - \theta') \xi}  2 \pi i c_\omega H(\pm\xi)\, d \xi,\\
    c_\omega (\theta - (1 \pm i \epsilon z_\omega^{- \pm}) \theta')^{-1} = \frac{1}{2 \pi} \int_\R e^{i (\theta - \theta') \xi} 2 \pi i c_\omega^\pm e^{-\epsilon \tilde z^{- \pm}_\omega \theta |\xi|} H(\mp\xi),
\end{gathered}
\end{equation}
where 
\[\mp i \epsilon \tilde z_\omega^{- \pm} := (1 \pm i\epsilon z_\omega^{-\pm})^{-1} - 1, \quad c_\omega^\pm = \frac{c_\omega}{1 \pm i \epsilon z_\omega^{- \pm}}\]
Note that $\tilde z_\omega^{-\pm}$ and $c_\omega^\pm$ enjoy essentially the same properties as $z_\omega^{-\pm}$ and $c_\omega$ respectively, that is, 
\[\Im \tilde z_\omega^{-\pm} = \mathcal O( \epsilon), \quad \Re \tilde z_\omega^{-\pm} > 0, \quad c^\pm_\omega \to c_\lambda \ \text{as $\epsilon \to 0^+$}.\]
Let
\begin{equation*}
    \tilde q_{\mathrm{sm}}(\theta, \xi) := 2 \pi i \chi_\cu(\theta) \begin{pmatrix} c_\omega \sgn(\xi) & c^-_\omega e^{-\epsilon \tilde z_\omega^{--} |\xi|} H(\xi) \\ -c_\omega^+ e^{-\epsilon  \tilde z_\omega^{-+} |\xi|} H(-\xi) & - c_\omega \sgn(\xi)  \end{pmatrix}.
\end{equation*}
Let $\psi \in \CIc(\R)$ be such that $\supp \psi \subset (-1, 1)$ and $\int \hat \psi = 1$. Then define the symbol 
\begin{equation}
    q_{\mathrm{sm}}(\theta, \xi) = [\hat \psi * p(\theta, \bullet)](\xi) \in S^0_\bo
\end{equation}
uniformly. Here, the slight abuse of notation $p \in S^0_\bo$ means that each entry of $p$ lies in $S^0_\bo$. In particular, we see that
\begin{equation*}
\begin{gathered}
    Q - \Op_\bo(q_{\mathrm{sm}}) \in \widetilde \Psi^{-\infty, \mathcal E}_\bo
\end{gathered}
\end{equation*}
where $\mathcal E$ is given in~\eqref{eq:index_set}. Indeed, we see that $Q$ is the sum of an element of the small b-calculus $\Op_\bo(q_{\mathrm{sm}})$ and an element of $\widetilde \Psi_\bo^{-\infty, \mathcal E}$.

Furthermore, there exists $c > 0$ such that for all sufficiently small $\epsilon > 0$,  
\begin{equation*}
    |\det(q_{\mathrm{sm}}(\theta, \xi))| \ge c \quad \text{for} \quad \theta \in (\supp(1 - \chi_\cu))^C, \quad |\xi| > c^{-1}. 
\end{equation*}
Therefore, the ellipticity claim follows. The dilation invariance property of~\eqref{eq:DI_form} follows by inspection from the Schwartz kernel formulas~\eqref{eq:kernel_C1} and~\eqref{eq:kernel_C2}. 
\end{proof}

Now we are in a position to use the b-elliptic estimates from Proposition~\ref{prop:full_b-est} with the operator $Q$. To apply the proposition, we also need to compute the normal family defined in Lemma~\ref{lem:normal_fam}. In the notation of Lemma~\ref{lem:normal_fam}, it is easy to see from~\eqref{eq:kernel_C1} and~\eqref{eq:kernel_C2} that the restriction of the Schwartz kernel of $Q$ to the front face is given by 
\begin{multline}\label{eq:Kff_matrix}
    c_\omega^{-1}(\pi^*K_\ff)(t) = \left(
    \begin{matrix}
        (t - 1 + i0)^{-1} + (t - 1 - i0)^{-1}\\
        -(t - 1 - i \epsilon z_\omega^{-+})^{-1} - (t + \alpha(1 + i \epsilon z_\omega^{++}))^{-1}
    \end{matrix} \right. \\
    \left.\begin{matrix}
        (t - 1 + i\epsilon z_\omega^{--})^{-1} + (t + \alpha^{-1}(1 - i\epsilon z_\omega^{+-}))^{-1} \\
        -(t - 1 + i0)^{-1} -(t - 1 -i0)^{-1}
    \end{matrix}\right)
\end{multline}
The Mellin transform of $\pi^* K_\ff$ is well defined in the strip $\{\sigma: \Im \sigma \in (0, 1)\}$. To actually compute the Mellin transform, we use the following technical lemma. 

\begin{lemma}\label{lem:convolution}
    Let $h_{\pm\epsilon}(x) = (x \pm i\epsilon)^{-1}$, $\epsilon \ge 0$. Let $x_+^a \in \mathscr S'(\R)$, $\Re a > -1$ be the distribution supported on $[0, \infty)$ and equal to $x^a$ on $[0, \infty)$. Then 
    \[(h_{\pm \epsilon} * x_+^a)(x) = \frac{\pi}{\sin(\pi a)}e^{\mp i \pi a} (x \pm i\epsilon)^a\]
    where $(x \pm i\epsilon)^a$ is defined using the branch of $\log$ on $\C \setminus (-\infty, 0]$ that takes real value on the real line. 
\end{lemma}
\begin{proof}
    Let $\chi_+^a = x_+^a/\Gamma(a + 1)$ where $\Gamma$ denotes the Gamma function. From \cite[Example 7.1.17]{H1}, we have the identity
    \begin{equation*}
    \begin{gathered}
        \mathcal F((x \pm i\epsilon)^a)(\xi) = 2 \pi e^{\pm i\pi a /2} e^{\mp i \epsilon \xi}\chi_\pm^{-a - 1}(\xi), \quad \epsilon > 0, \\
        \mathcal F((x \pm i\epsilon)^a) \to \mathcal F((x \pm i0)^a) = 2 \pi e^{\pm i\pi a /2}\chi_\pm^{-a - 1}(\xi) \quad \text{as} \quad \epsilon \to 0\pm
    \end{gathered}
    \end{equation*}
    Then we can compute the convolution on the Fourier side and find
    \begin{align*}
    (h_{\pm \epsilon} * x_+^a)(x) &= \mathcal F^{-1} \left(\mp 2 \pi i e^{\mp \epsilon \xi} H(\pm \xi) \Gamma(a + 1) e^{-i \pi (a + 1)/2}(\xi - i0)^{-a - 1} \right)\\
    &= \mathcal F^{-1} \left(-2 \pi e^{\mp i \pi a/2} e^{\mp \epsilon \xi} \Gamma(a + 1) \Gamma(-a) \chi_\pm^{-a - 1}(\xi) \right) \\
    &= -2 \pi e^{-i \pi a/2} \Gamma(a + 1) \Gamma(-a) \mathcal F^{-1} \left(e^{\mp \epsilon \xi} \chi_\pm^{-a - 1} (\xi) \right) \\
    &= -\Gamma(a + 1) \Gamma(-a) e^{\mp i \pi a}( x \pm i\epsilon)^a \\
    &= \frac{\pi}{\sin(\pi a)}e^{\mp i \pi a} (x \pm i\epsilon)^a
    \end{align*}
    as desired. 
\end{proof}
Now we can compute the normal family. 

\begin{lemma}
    Let $\Omega$ be $\lambda$-simple and let $\omega = \lambda + i \epsilon $ for sufficiently small $\epsilon$. Let $\alpha = \alpha(\lambda)$ be the characteristic ratio defined in Definition~\ref{def:corner_type} for the corner at $0$. The normal family 
    \[N(Q_\omega, \sigma): \R + i (0, 1)\to \mathrm{Mat}_{2 \times 2}(\C)\] 
    is holomorphic and has finitely many roots in the strip $\R + i (0, 1)$ that converge to the the set
    \begin{equation}\label{eq:zeros}
        \mathcal Z_{\lambda, 0} := \left\{\sigma \in (0, 1) + i \R: i \sigma + 1 = \frac{2 \pi i}{2 \pi - \log \alpha}k \ \mathrm{or} \ i \sigma + 1 = \frac{2 \pi i}{3i \pi + \log \alpha} k, \ k \in \Z \right\}.
    \end{equation}
\end{lemma} 
We will refer to $\mathcal Z_{\lambda, \kappa}$ as the limiting indicial roots for a corner $\kappa$. 

\begin{Remark}
    Recall that the solution $v_\omega$ to \eqref{eq:bdr} is ``Neumann data'' the solution $u$ to \eqref{eq:EBVP}. In particular, $u$ has the polyhomogenous expansion given in Proposition~\ref{prop:polyhom_EU}. Therefore, one might expect that $\mathcal Z_{\lambda, \kappa}$ corresponds exactly to the indices in the polyhomogeneous expansion of $u$ near a corner $\kappa$, in which case $\mathcal Z_{\lambda, \kappa}$ would contain only $i \sigma + 1 \in \mathfrak l_\lambda \Z \cap \{(0, 1) + i \R\}$ where $\mathfrak l_\lambda := \frac{2 \pi i}{2 \pi - \log \alpha(\lambda)}$ is defined in~\eqref{eq:conormal_indices}. However, we see that $\mathcal Z_{\lambda, \kappa}$ contains more than the indices that correspond to $u$. This is because the boundary reduction does not distinguish between solutions supported inside the domain and outside the domain. One can check that the extra indices that appear in $\mathcal Z$ indeed correspond to the analogous ``Neumann data'' of solutions supported outside $\Omega$. 
\end{Remark}
\begin{proof}
We compute the normal family by computing the Mellin transform of~\eqref{eq:Kff_matrix} using Lemma~\ref{lem:convolution}. Observe that for $\sigma \in \R + i(0, 1)$,
\begin{align*}
    \mathcal M[(t - 1 \pm i0)^{-1}](\sigma) &= -\int_{\R_+} (1 - t \mp i0) t^{-i \sigma - 1}\, dt \\
    &= -(h_{\mp 0} * x_+^{-i \sigma - 1})(1) \\
    &= -\tfrac{\pi}{\sin(\pi(- i \sigma - 1))} e^{\pm i \pi(- i \sigma - 1)}.
\end{align*}
Similarly, 
\begin{equation*}
    \mathcal M[(t - 1 \pm i\epsilon z^{\mu \nu}_\omega)^{-1}](\sigma) = -\tfrac{\pi}{\sin(\pi(- i \sigma - 1))} e^{\pm i \pi(- i \sigma - 1)}(1 \mp i \epsilon z_\omega^{\mu \nu})^{-i \sigma - 1}
\end{equation*}
for $\mu, \nu \in \{+, -\}$. Finally, using the branch of log on $\C \setminus (-\infty, 0]$, we find
\begin{align*}
    \mathcal M[(t + \alpha^{-1}(1 - i\epsilon z_\omega^{+-}))^{-1}](\sigma) &= -\tfrac{\pi}{\sin(\pi(- i \sigma - 1))} e^{-i \pi (-i \sigma - 1)}(-\alpha^{-1} + i\epsilon \alpha^{-1} z_\omega^{+-})^{-i \sigma - 1} \\
    &=-\tfrac{\pi}{\sin(\pi(- i \sigma - 1))} (\alpha^{-1} - i\epsilon \alpha^{-1} z_\omega^{+-})^{-i \sigma - 1},
\end{align*}
and 
\begin{align*}
    \mathcal M[(t + \alpha(1 + i\epsilon z_\omega^{++}))^{-1}](\sigma) &= -\tfrac{\pi}{\sin(\pi(- i \sigma - 1))} e^{i \pi (-i \sigma - 1)}(-\alpha - i\epsilon \alpha z_\omega^{++})^{-i \sigma - 1} \\
    &=-\tfrac{\pi}{\sin(\pi(- i \sigma - 1))} (\alpha + i\epsilon \alpha z_\omega^{++})^{-i \sigma - 1}.
\end{align*}
With $s := i \sigma + 1$, we now see that the normal family of $Q$ is given by 
\begin{multline}
    c_\omega^{-1} N(Q, s) = \frac{\pi}{\sin(\pi s)} \left(\begin{matrix} e^{i \pi s} + e^{- i \pi s} \\ -e^{i \pi s}(1 + i \epsilon z_\omega^{-+})^{-s} - (\alpha + i \epsilon \alpha z_\omega^{++})^{-s} \end{matrix} \right. \\
    \left. \begin{matrix} e^{-i \pi s} (1 - i \epsilon z_\omega^{--})^{-s} + (\alpha^{-1} - i\epsilon \alpha^{-1} z_\omega^{+-})^{-s} \\ -e^{i \pi s} - e^{- i \pi s} \end{matrix} \right)
\end{multline}
for $\sigma \in \R + i(-1, 0)$, or equivalently, $s \in (0, 1) + i \R$.

Observe that when $\epsilon = \Im \omega = 0$, $N(Q_{\lambda + i0}, \sigma)$ is noninvertible when 
\begin{equation}\label{eq:b-system_roots}
    \alpha^{-s} e^{-i \pi s} + \alpha^{s} e^{i\pi s} - e^{-2 i \pi s} - e^{2 i \pi s} = 0.
\end{equation}
The left-hand-side factors, so~\eqref{eq:b-system_roots} holds precisely when 
\[\alpha^{-s} - e^{-i \pi s} = 0 \quad \text{or} \quad \alpha^{s} e^{2 i \pi s} - e^{-i \pi s} = 0,\]
which means the roots in  $s = i\sigma + 1$ are given by 
\begin{equation}
    s \in \frac{2 \pi i}{i \pi - \log \alpha} \Z \cap \{(0, 1) + i\R\}, \quad s \in \frac{2 \pi i}{3i \pi - \log(\alpha^{-1})} \Z \cap \{(0, 1) + i\R\}.
\end{equation}
at $\epsilon = 0$. Note that $N(Q, \sigma)$ extends holomorphically in $\epsilon$ for $\epsilon$ in a small neighborhood of $0$, and $N(Q, \sigma)$ is lower bounded as $\Im s \to \infty$ locally uniformly in $\Re s \in (-1, 0)$, so it follows that for all sufficiently small $\epsilon$, $N(Q_\omega, \sigma)$ has finitely many zeros in the strip $\{\sigma \in C: \Im \sigma \in (0, 1)\}$ that converge to the zeros of the $N(Q_{\lambda + i0}, \sigma)$. 
\end{proof}

Now we give the proposition that allows us to propagate through the corner. We consider a type-$(+, +)$ corner, and we will use the parameterization~\eqref{eq:good_param} near the corner throughout the proof. 
\begin{proposition}\label{prop:corner_POS}
Let $\mathcal J \Subset (0, 1)$ be a sufficiently small interval and assume that $\Omega$ is $\lambda$-simple for every $\lambda \in \mathcal J$. Fix symmetric cutoffs $\chi, \widetilde \chi \in \overline C^\infty(\partial \Omega)$ near a corner $\kappa$ in the sense of~\eqref{eq:sym_cutoff} such that $\widetilde \chi = 1$ on $\supp \chi$. Furthermore, assume that the parameterization~\eqref{eq:good_param} holds in the support of $\widetilde \chi$. Let $v$ and $g$ be as in~\eqref{eq:bdr}, and let $a \in (-1, 0)$ be such that 
\[\mathcal Z_{\lambda, \kappa} \cap (\R - ia) = \emptyset\]
for all $\lambda \in \mathcal J$ where $\mathcal Z_{\lambda, \kappa}$ are the limiting indicial roots defined in~\eqref{eq:zeros}. Assume that $v \in H^{\infty, a}_\bo(\partial \Omega)$. Then for $s \ge a + \ha$, 
\begin{equation}
    \|\chi v\|_{H^{s, a}_\bo} \le C \big[\|\Pi^+(\widetilde \chi \circ \gamma^+_\lambda)v\|_{H^s} + \|\widetilde \chi g\|_{H^{s, a}_\bo} + \|v\|_{H^{-N, a - \delta}_\bo}\big]
\end{equation}
where $\delta > 0$ is such that $a - \delta > -1$.
\end{proposition}
\begin{Remark}
    A simple case of this estimate is when we take $s = a + \ha$ and we can state the estimate entirely in terms of regular Sobolev spaces:
    \begin{equation}
        \|\chi v\|_{H^s} \le C \big[\|\Pi^+\widetilde \chi \circ \gamma_\lambda^+)v\| + \|\widetilde g\|_{H^s} + \|v\|_{H^{s - \delta}}  \big]
    \end{equation}
    In other words, up to high frequency effects, $v$ near a type-$(+, +)$ corner at $0$ is controlled by the negative frequencies of $v$ near $\gamma^{-}_\lambda(0)$ and $g$. Also note that $H^s(\mathbb S^1) \hookrightarrow H^{s - \delta}(\mathbb S^1)$ is compact. 
\end{Remark}

\begin{proof}
    1. Wedge a symmetric cutoff $\chi_\cu \in \overline C^\infty(\partial \Omega)$ between $\chi$ and $\widetilde \chi$ so that $\chi_\cu = 1$ on $\supp \chi$ and $\widetilde \chi = 1$ on $\supp \chi_\cu$. Let $\mathfrak X$ be a cutting map defined in~\eqref{eq:cutting} using the cutoff $\chi_\cu$ and let $\mathfrak G$ be the gluing map in~\ref{eq:gluing} defined using $\widetilde \chi$. Then define $Q_\omega$ using $\mathfrak X$ and $\mathfrak G$ along with the parameterization~\eqref{eq:good_param}. By assumption, no zeros of $N(Q_{\lambda + i\epsilon}, \sigma)$ lie in the line $\R - ia$ for all sufficiently small $\epsilon$. Then by the b-elliptic estimate in Proposition~\ref{prop:full_b-est}, 
    \begin{align}
        \|\chi v\|_{H^{s, a}_\bo} &= \|\chi \mathfrak X_0 v\|_{H^{s, a}_\bo} \nonumber \\
        &\le C \big(\|Q_\omega \mathfrak X_0 v\|_{H^{s, a}_\bo} + \|v\|_{H^{-N, a - \delta}_\bo} \big) \nonumber \\
        &\le C \big(\|\chi_\cu d \mathcal C_\omega \widetilde \chi v\|_{H^{s, a}_\bo} + \|v\|_{H^{-N, a - \delta}_\bo} \big) \label{eq:corner_iso}
    \end{align}
    where $\mathfrak X_0$ is a cutting operator defined using a cutoff that is identically 1 on the support of $\widetilde \chi$. The constant $C$ is locally uniform in $\lambda$ and sufficiently small $\epsilon > 0$. 
    
    \noindent
    2. Now we split up the first term on the right hand side of~\eqref{eq:corner_iso} by 
    \begin{equation}\label{eq:corner_pieces}
        \chi_\cu d \mathcal C_\omega \widetilde \chi = \chi_\cu d \mathcal C_\omega - \chi d \mathcal C_\omega(\widetilde \chi \circ \gamma^+_\lambda) - \chi_\cu d \mathcal C_\omega (1 - \widetilde \chi - (\widetilde \chi \circ \gamma^+_\lambda)).
    \end{equation}
    It follows from (2) and (3) of Proposition~\ref{prop:rest_of_kernel} and Lemma~\ref{lem:b-cutting} that
    \begin{equation}\label{eq:corner_smooth_est}
        \|\chi_\cu d \mathcal C_\omega (1 - \widetilde \chi - (\widetilde \chi \circ \gamma^+_\lambda) v\|_{H^{s, a}_\bo} \lesssim \|v\|_{H^{-N, a - \delta}_\bo}
    \end{equation}
    for $s \ge a + \ha$.  
    
    It remains to analyze the term $\chi_\cu d \mathcal C_\omega (\widetilde \chi \circ \gamma^+_\lambda)$. This term is similar to (3) of Lemma~\ref{lem:psido_corner_ref} except that the Heaviside cutoffs are on the outgoing variable rather than the incoming variable. In particular, it follows from from (1) of Proposition~\ref{prop:rest_of_kernel} that there exists $T_{\omega, \pm} \in \Psi^0(\mathbb S^1)$ locally uniformly in $\lambda$ such that 
    \begin{equation}
        \chi_\cu d \mathcal C_\omega(\widetilde \chi \circ \gamma^+_\lambda) = \sum_\pm H(\pm(\theta))(\gamma^+_\lambda)^* (\chi_\cu \circ \gamma_\lambda^+) T^+_{\omega, \pm} (\widetilde \chi \circ \gamma^+_\lambda)
    \end{equation}
    and
    \begin{equation}
        \WF(T_{\omega, \pm}^+) \subset \{(\theta, \xi) \in T^* \mathbb S^1: \xi > 0\}
    \end{equation}
    By Lemma~\ref{lem:b-cutting}, we then see that for $s \ge a + \ha$.
    \begin{equation}\label{eq:corner_ref_est}
        \|\chi_\cu d \mathcal C_\omega (\widetilde \chi \circ \gamma^+_\lambda) v\|_{H^{s, a}_\bo} \lesssim \|\Pi^+ (\widetilde \chi \circ \gamma^+_\lambda) v\|_{H^{s}} + \|(\chi_{0} \circ \gamma^+_\lambda) v\|_{H^{-N}}.
    \end{equation}
    Using~\eqref{eq:corner_smooth_est} and~\eqref{eq:corner_ref_est} to estimate~\eqref{eq:corner_pieces}, we find
    \[\|\chi_\cu d \mathcal C_\omega \widetilde \chi v\|_{H^{s, a}_\bo} \lesssim \|\Pi^+ (\widetilde \chi \circ \gamma^+_\lambda) v\|_{H^{s}} + \|\chi_\cu g\|_{H^{s, a}_\bo} + \|v\|_{H^{-N, a - \delta}_\bo}\]
    Combining with~\eqref{eq:corner_iso} then gives the desired estimate. 
\end{proof}

Finally, we end this section by giving the mapping properties for the differentiated operator $\partial^k_{\omega} d \mathcal C_{\omega}$.  
\begin{lemma}\label{lem:ddC}
    Assume that $\Omega$ is $\lambda_0$-simple for some $\lambda_0 \in (0, 1)$. Let $\chi_\bo \in \overline C^\infty(\partial \Omega)$ be supported in a sufficiently small neighborhood of $\mathcal K$ and let $\chi_\bo = 1$ near $\mathcal K$. Then for all $s \in \R$, $k \in \N$, and $a \in (0, 1)$, there exists $C = C_{s, a, k} > 0$ such that 
    \begin{equation}\label{eq:ddC_converge}
        \|\chi_\bo\partial_\omega^k d \mathcal C_{\omega} \chi_\bo \|_{H^{s + k, a}_\bo \to H^{s, a}_\bo} \le C
    \end{equation}
    for all $\omega = \lambda + i\epsilon$ such that $0 \le \epsilon \le \epsilon_0$ and $|\lambda - \lambda_0| \le \epsilon_0$ for sufficiently small $\epsilon_0 > 0$. Note that this includes $\omega = \lambda + i0$. 
\end{lemma}
\begin{proof}
    Choose $\epsilon_0$ so that $\Omega$ is $\lambda$-simple for all $\lambda \in (\lambda_0 - \epsilon_0, \lambda_0 + \epsilon_0)$. Differentiating the formula for the Schwartz kernel of $d \mathcal C_\omega$ from Proposition~\ref{prop:corner_kernel}, we see that for $\theta$ and $\theta'$ in a small neighborhood of a type-$(+, +)$ corner 
    \begin{equation}
        \partial_\omega^k K^+_\omega(\theta, \theta') = \begin{cases}\partial_\omega^k c_\omega (\theta - \theta' + i0)^{-1} & \theta \cdot \theta' > 0, \\
        \sum_{n = 1}^k c^+_{k, n, \omega} \cdot (\theta')^{n - 1}(\theta - \alpha(1 + i \epsilon z_\omega^{++})\theta')^{-n} & \theta < 0, \ \theta' > 0 \\
        \sum_{n = 1}^k c^-_{k, n, \omega} \cdot (\theta')^{n - 1}(\theta - \alpha^{-1}(1 + i \epsilon z_\omega^{+-})\theta')^{-n} & \theta > 0, \ \theta' < 0
        \end{cases}
    \end{equation}
    where $c^\pm_{k, n, \omega}$ are constants depending smoothly on $\omega$ up to $\Im \omega = 0$. Analogous expressions hold for $\partial_\omega K^-_\omega$. In particular, with $\mathfrak X$ and $\mathfrak G$ as in Lemma~\ref{lem:transform_to_b_operator}, it follows that
    \[\mathfrak X \circ \partial_\omega^k d \mathcal C_\omega \circ \mathfrak G \in \widetilde \Psi_\bo^{0, \mathcal E}\]
    with the same index set $\mathcal E$ from~\eqref{eq:index_set}. Therefore, for every $s \in \R$ and $a \in (0, 1)$, there exists a $C > 0$ such that for $j$ sufficiently large, 
    \begin{equation}
        \|\partial_\theta \chi_\bo \partial_\omega^k \mathcal C_{\omega_j} \chi_\bo\|_{H^{s + k, a}_\bo \to H^{s,a}_\bo} \le C
    \end{equation}
    as desired. 
\end{proof}

\subsection{Global high frequency estimate}
Now we string together all the local estimates. For $v$ solving~\eqref{eq:bdr}, we wish to obtain a global high frequency estimate for $v$ with a compact error.

For every corner $\kappa \in \mathcal K$, we have the limiting index set $\mathcal Z_{\lambda, \kappa}$ defined in~\eqref{eq:zeros}. Denote by
\begin{equation}\label{eq:all_zeros}
    \mathcal Z_{\lambda} = \bigcup_{\kappa \in \mathcal K} \mathcal Z_{\lambda, \kappa}
\end{equation}
the union of the limiting index sets. 
\begin{proposition}\label{prop:global_semifredholm}
    Let $\mathcal J \Subset \Omega$ be such that $\Omega$ is Morse--Smale for every $\lambda \in \mathcal J$, and write $\omega = \lambda + i \epsilon$ for $\lambda \in \mathcal J$ and sufficiently small $\epsilon > 0$. Let $\chi_\pm, \widetilde \chi_\pm \in \overline C^\infty(\partial \Omega)$ be cutoffs near $\Sigma_\lambda^\pm$ that satisfy~\eqref{eq:chi_pm}. Assume that $v_\omega \in H^{\infty, a}_\bo(\partial \Omega)$, $a \in (-1, 0)$, is a solution of~\eqref{eq:bdr}, and assume that $\mathcal Z_{\lambda} \cap \{\R - ia\} = \emptyset$ for all $\lambda \in \mathcal J$. Then for $\beta > 0$, $\delta \in (0, a + 1)$, and $s \ge a + \ha$, we have
    \begin{equation}\label{eq:low_reg_semifredholm}
        \|v_\omega\|_{H^{-\ha - \beta, a}_\bo} \le C \big( \|v_\omega\|_{H^{-N, a - \delta}_\bo} + \sum_{\pm} \|\Pi^\pm \chi_\mp g_\omega\|_{H^{- \ha - \beta}} + \|\Pi^\pm(1 - \chi_\mp) g_\omega\|_{H^{s, a}_\bo} \big)
    \end{equation}
    and
    \begin{equation}\label{eq:high_reg_semifredholm}
        \|\Pi^\pm(1 - \widetilde \chi_\mp) v_\omega\|_{H^{a + \ha, a}_\bo} \le C \big( \|v_\omega\|_{H^{-N, a - \delta}_\bo} + \sum_{\pm} \|\Pi^\pm(1 - \chi_\mp) g_\omega\|_{H^{s, a}_\bo} \big)
    \end{equation}
    The constant $C$ is locally uniform in $\lambda = \Re \omega \in \mathcal J$ and independent of $\epsilon$. 
\end{proposition}
\begin{Remarks}
    1. A particularly useful and simpler version of this estimate is given by 
    a useful special case of this estimate is
    \begin{equation}
        \|v_\omega\|_{H^{-\ha - \beta, a}_\bo} \le C \big( \|v_\omega\|_{H^{-N, \alpha - \delta}_\bo} + \|g_\omega\|_{H^{s, a}_\bo} \big).
    \end{equation}
    
    \noindent
    2. The main point here is that $H^{-\ha - \beta, a}_\bo (\partial \Omega) \hookrightarrow H^{-N, a - \delta}_\bo(\partial \Omega)$ is a compact embedding. This compactness is crucial to establishing a limiting absorption principle. With a limiting absoprtion principle in place, we can then use the local propagation estimates again to better characterize the regularity of the limit. 
\end{Remarks}

\begin{proof}
    Of course, since $\partial \Omega$ is compact, proving~\eqref{eq:low_reg_semifredholm} is equivalent to showing that for every $\theta_0 \in \partial \Omega$, there exists a a cutoff $\chi \in \overline C^\infty(\partial \Omega)$ such that $\chi = 1$ near $\theta_0$ and 
    \[\|\Pi^\pm \chi v\|_{H^{-\ha - \beta, a}_\bo} \le C \big( \|v\|_{H^{-N, \alpha - \delta}_\bo} + \sum_{\pm} \|\Pi^\pm \chi_\mp g\|_{H^{- \ha - \beta}} + \|\Pi^\pm(1 - \chi_\mp) g\|_{H^{s, a}_\bo} \big).\]
    On the other hand, proving~\eqref{eq:high_reg_semifredholm} is equivalent to showing that for every $\theta_0 \in (\supp (1 - \chi_\mp))^\circ$, there exists a a cutoff $\chi \in \overline C^\infty(\partial \Omega)$ such that $\chi = 1$ near $\theta_0$ and 
    \[\|\Pi^\pm(1 - \widetilde \chi_\mp) v_\omega\|_{H^{a + \ha, a}_\bo} \le C \big( \|v_\omega\|_{H^{-N, a - \delta}_\bo} + \sum_{\pm} \|\Pi^\pm(1 - \chi_\mp) g_\omega\|_{H^{s, a}_\bo} \big)\]
    We split this into several cases. 
    
    \noindent
    1. We start with the high regularity source estimate. Take $s \ge a + \ha$. Then by Proposition~\ref{prop:source},
    \begin{equation}\label{eq:start}
        \|(\Pi^+ \chi_+ + \Pi^- \chi_-) v\|_{H^s} \le C \big(\|\Pi^+ \widetilde \chi_+ g + \Pi^- \widetilde \chi_- g\|_{H^s} + \|v\|_{H^{-N, a - \delta}_\bo} \big).
    \end{equation}
    We will propagate all other estimates into this case. 
    
    \noindent
    2. Now consider $\theta_0$ such that $\theta_0 \in (\supp (1 - \chi_\pm))^\circ$ and $b_\lambda^{\pm k}(\theta_0) \notin \mathcal K$ for all $k \in \N_0$. Observe that since corners appear only as characteristic points, this also implies that $\gamma^{\mp}_{\lambda} \circ  b_\lambda^{\pm k}(\theta_0) \notin \mathcal K$ for all $k \in \N_0$. It follows from the Morse--Smale assumption that there exists a neighborhood $U \ni \theta_0$ such that  
    \[b^{\pm k}_\lambda(\theta) \notin \mathcal K \quad \text{and} \quad \gamma^{\mp}_{\lambda} \circ  b_\lambda^{\pm k}(\theta) \notin \mathcal K \quad \text{for all} \quad k \in \N_0, \ \theta \in U.\]
    Choose cutoffs $\chi, \widetilde \chi \in \CIc(U)$ such that $\chi = 1$ near $\theta_0$ and $\widetilde \chi = 1$ near $\supp \chi$. Observe that there exists $N > 0$ such that 
    \[\supp(\widetilde \chi \circ b^{\mp N}_\lambda) \subset b^{\pm N}_{\lambda}(U) \subset \supp \chi_\pm.\]
    Iterating the estimate in Proposition~\ref{prop:POS} (and possibly starting with the estimate~\eqref{eq:corner_POS_easy} if $\theta_0$ is a corner reflection), we see that for all $s \ge a + \ha$, 
    \begin{equation}\label{eq:easy_into_source}
    \begin{aligned}
        \|\Pi^\pm \chi v\|_{H^s} &\le C\big(\|\Pi^\pm(\widetilde \chi \circ b^{\mp N}) v\| + \sum_{\pm} \|\Pi^\pm(1 - \chi_\mp) g\|_{H^{s, a}_\bo} + \|v\|_{H^{-N, a - \delta}_\bo} \big) \\
        &\le C\big(\|\Pi^\pm \chi_\pm v\| + \sum_{\pm} \|\Pi^\pm(1 - \chi_\mp) g\|_{H^{s, a}_\bo} + \|v\|_{H^{-N, a - \delta}_\bo} \big) \\
        &\le C\big(\sum_{\pm} \|\Pi^\pm(1 - \chi_\mp) g\|_{H^{s, a}_\bo} + \|v\|_{H^{-N, a - \delta}_\bo} \big),
    \end{aligned}
    \end{equation}
    where the last inequality follows the source estimate in~\eqref{eq:start}.

    \noindent
    3. Now suppose $\theta_0$ is a corner. Again, we only consider the case that $\theta_0$ is a type-$(+, +)$ corner since the other cases can be deduced via reflection symmetry. Let $\chi, \widetilde \chi$ be symmetric cutoffs near $\theta_0$ in the sense of~\eqref{eq:sym_cutoff}, and assume that $\widetilde \chi = 1$ near $\supp \chi_0$. Furthermore, observe that $\supp \widetilde \chi$ can be made sufficiently small so that 
    \[b_\lambda^k(\theta) \notin \mathcal K \quad \text{for every} \quad \theta \in \gamma^+_{\lambda}(\supp \widetilde \chi).\]
    By Proposition~\ref{prop:corner_POS}, we see that for $s \ge a + \ha$, 
    \begin{equation}\label{eq:corner_into_source}
    \begin{aligned}
        \|\chi v\|_{H^{s, a}_\bo} &\le C \big(\|\Pi^+(\widetilde \chi \circ \gamma^+_\lambda) v\|_{H^s} + \|\widetilde \chi g\|_{H^{s, a}_\bo} + \|v\|_{H^{-N, a - \delta}_{\bo}} \big) \\
        &\le C\big(\|\Pi^+(1 - \chi_-) g\|_{H^{s, a}_\bo} + \|\Pi^-(1 - \chi_+) g\|_{H^{s, a}_\bo} + \|v\|_{H^{-N, a - \delta}_\bo} \big),
    \end{aligned}
    \end{equation}
    where the last inequality follows from~\eqref{eq:easy_into_source}. 

    \noindent
    4. Now assume that $\theta \in (\supp (1 - \chi_\pm))^\circ$ and $b^{\pm k}_\lambda(\theta_0) \in \mathcal K$ for some $k \in \N$. Note that by the Morse--Smale assumption, this occurs for precisely one value of $k \in \N$. We assume that $b^{\pm k}_\lambda(\theta_0) = \kappa$ is a type-$(+, +)$ corner. In the case that $b^{k}_\lambda(\theta_0)$ is a corner for some $k \in \N$, observe that 
    \[(\gamma_\lambda^- \circ b^{k - 1}_{\lambda})(\theta_0) = \gamma^+(\kappa).\]
    Let $\chi_j \in \overline C^\infty(\partial \Omega)$, $j = 1, 2, 3$, be such that $\chi_1 = 1$ near $\theta_0$ and $\widetilde \chi_j = 1$ near $\supp \chi_{j - 1}$ for $j = 2, 3$. If $\chi_3$ has sufficiently small support, it follows from iterating Proposition~\ref{prop:POS} followed by an application of Proposition~\ref{prop:ref_corner_POS} that for $s \ge a + \ha$, 
    \begin{equation}
        \begin{aligned}
            \|\Pi^+ \chi_1 v\|_{H^{a + \ha}} &\le C\big(\|\Pi^-(\chi_2 \circ b^{-(k - 1)}_\lambda \circ \gamma^-_\lambda)v\|_{H^{a + \ha}} \\
            &\qquad \qquad+ \sum_{\pm} \|\Pi^\pm(1 - \chi_\mp) g\|_{H^{s, a}_\bo} + \|v\|_{H^{-N, a - \delta}_\bo} \big) \\
            &\le C \big(\|(\chi_3 \circ b^{-k}_\lambda)v\|_{H^{a + \ha, a}_\bo} + \sum_{\pm} \|\Pi^\pm(1 - \chi_\mp) g\|_{H^{s, a}_\bo} + \|v\|_{H^{-N, a - \delta}_\bo} \big).
        \end{aligned}
    \end{equation}
    Upon possibly further shrinking the support of $\chi_3$ and applying \ref{eq:corner_into_source}, we find that
    \begin{equation}
        \|\Pi^+ \chi_1 v\|_{H^{a + \ha}} \le C\big(\sum_{\pm} \|\Pi^\pm(1 - \chi_\mp) g\|_{H^{s, a}_\bo} + \|v\|_{H^{-N, a - \delta}_\bo} \big).
    \end{equation}
    Observe that $\Pi^- \chi_1 v$ is handled by Step 2 since $b^k(\theta_0) = \kappa$ implies that $b^{-k}(\theta_0) \notin \mathcal K$ for all $k \in \N$. 

    Similarly, if $b^{-k}(\theta_0) = \kappa$ for some $k \in \N$, we instead see that 
    \[b^{k - 1}(\theta_0) = \gamma^+(\kappa).\]
    Therefore, with $\chi_3$ having sufficiently small support near $\theta_0$, we have
    \begin{equation}
        \begin{aligned}
            \|\Pi^- \chi_1 v\|_{H^{a + \ha}} &\le C\big(\|\Pi^+(\chi_2 \circ b^{k - 1}_\lambda)v\|_{H^{a + \ha}} \\
            &\qquad \qquad + \sum_{\pm} \|\Pi^\pm(1 - \chi_\mp) g\|_{H^{s, a}_\bo} + \|v\|_{H^{-N, a - \delta}_\bo} \big) \\
            &\le C \big(\|(\chi_3 \circ b^{-(k - 1)}_\lambda \circ \gamma^-_\lambda)v\|_{H^{a + \ha, a}_\bo} \\
            &\qquad \qquad+ \sum_{\pm} \|\Pi^\pm(1 - \chi_\mp) g\|_{H^{s, a}_\bo} + \|v\|_{H^{-N, a - \delta}_\bo} \big) \\
            &\le C\big(\|g\|_{H^{s, a}_\bo} + \|v\|_{H^{-N, a - \delta}_\bo} \big).
        \end{aligned}
    \end{equation}
    Now, $\Pi^+ \chi_1 v$ is handled by Step 2, so this completes the analysis near points in $\partial \Omega$ that hits a corners under iterated mapping by $b^{\pm 1}_{\lambda}$. 

    \noindent
    5. Only one case remains. We must consider $(\Pi^+ \chi'_- + \Pi^- \chi'_+)v$ where $\chi'_\pm \in \overline C^\infty(\Omega)$ are such that $\chi_\pm = 1$ on $\supp \chi'_\pm$. For this, we use the low regularity sink estimates. Proposition~\ref{prop:sink} implies
    \begin{multline*}
        \|(\Pi^+ \chi_- + \Pi^- \chi_+)v\|_{H^{-\ha - \beta}} \le \big(\sum_\pm \|\Pi^\mp (\widetilde \chi_\pm \circ b_\lambda^{\pm 1} - \chi_\pm) v\|_{H^{-\ha - \beta}} \\
        + \sum_{\pm} \|\Pi^\pm \chi_\mp g\|_{H^{- \ha - \beta}} + 
        \|v\|_{H^{-N, a - \delta}_\bo} \big).
    \end{multline*}
    Note that all the points in $\supp(\widetilde \chi_\pm \circ b_\lambda^{\pm 1} - \chi_\pm)$ falls in one of the cases in Steps 2-4. Therefore, we immediately conclude that 
    \[\|(\Pi^+ \chi_- + \Pi^- \chi_+)v\|_{H^{-\ha - \beta}} \le C\big(\sum_{\pm} \|\Pi^\pm \chi_\mp g\|_{H^{- \ha - \beta}} + \|\Pi^\pm(1 - \chi_\mp) g\|_{H^{s, a}_\bo} + \|v\|_{H^{-N, a - \delta}_\bo} \big),\]
    which completes the proof.
\end{proof}

\subsection{Persistence of conormal regularity}
So far, the propagation of singularities estimates that we have shown in Proposition~\ref{prop:POS} and~\ref{prop:ref_corner_POS} all concern how Sobolev regularity propagates. The shortcoming of these estimate is that the singularities that propagate out of the corner are not fully encapsulated by Sobolev regularity since the regularity at the corner is characterized using b-Sobolev spaces. We now improve the the propagation estimates to show that conormal regularity propagates in some sense. In particular, Propositions~\ref{prop:conorm_POS} and~\ref{prop:conorm_ref} are conormal analogues of Proposition~\ref{prop:POS} and~\ref{prop:ref_corner_POS} respectively. Essentially, we take the proofs of of Proposition~\ref{prop:POS} and~\ref{prop:ref_corner_POS} and iteratively differentiate with respect to vector fields vanishing at a point.

Without the loss of generality, we will again work near a type-$(+, +)$ corner centered at $0$. We fix some b-parameterization $\mathbf x(\theta)$. Mostly to clean up notation later, we have the following lemma which follows by a direct computation. 
\begin{lemma}\label{lem:bvf_conj}
    Let $\rho_{\mathrm{ref}} \in \overline C^\infty(\partial \Omega)$ be supported in a small neighborhood of $\gamma^+_\lambda(0)$ with $\rho_{\mathrm{ref}}(\gamma^+_\lambda(0)) = 0$. Then there exists $\psi_\pm, \varphi_\pm \in \overline C^\infty(\partial \Omega)$ such that 
    \[(\gamma^+)^*(\rho_{\mathrm{ref}} \partial_\theta) (\gamma^+)^*(v(\theta) d \theta) = \begin{cases} (\psi_+(\theta) \theta \partial_\theta + \varphi_+(\theta))v(\theta) d \theta & \theta > 0, \\ (\psi_-(-\theta) \theta \partial_\theta + \varphi_-(-\theta))v(\theta) d \theta & \theta < 0.\end{cases} \]  
\end{lemma}
\begin{proof}
    The pullback on one-forms is explicitly given by 
    \[(\gamma^+)^*(v(\theta) d \theta) = \partial_\theta \gamma^{-1}(\theta) v(\gamma^+(\theta)) d \theta.\] 
    The formula then follows by the product rule and that $\rho_{\mathrm{ref}} \circ \gamma^+ \in \overline C^\infty(\partial \Omega)$ vanishes at $0$. 
\end{proof}
The following proposition is a refinement of the estimate~\eqref{eq:corner_POS_hard} in Proposition~\ref{prop:ref_corner_POS}.
\begin{proposition}\label{prop:conorm_ref}
    Let $\mathcal J \Subset (0, 1)$ and assume that $\Omega$ is $\lambda$-simple for all $\lambda \in \mathcal J$. Write $\omega = \lambda + i \epsilon$ for $\lambda \in \mathcal J$ and sufficiently small $\epsilon > 0$. Let $v = v_\omega \in H^{\infty, a}_\bo(\partial \Omega)$, $a \in (-1, 0)$, be a solution to~\eqref{eq:bdr}. Let $\rho_{\mathrm{ref}}$ be as in Lemma~\ref{lem:bvf_conj}. Fix cutoffs $\chi, \widetilde \chi \in \CIc(\partial \Omega)$ be supported in a small neighborhood of $\gamma^+_\lambda(0)$, and assume that $\chi = 1$ near $\gamma^+_\lambda(0)$ and $\widetilde \chi = 1$ near $\supp \chi$. Then for $s \in (-\ha, \ha)$, 
    \begin{equation}
        \|\Pi^- (\rho_{\mathrm{ref}} \partial_\theta)^k \chi v\|_{H^s} \lesssim \sum_{j = 0}^k \|(\theta \partial_\theta)^j (\widetilde \chi \circ \gamma^+_\lambda) v\|_{H^s} + \| \widetilde \chi g\|_{H^{s + k}} + \|v\|_{H^{-N, a}_\bo}
    \end{equation}
    where the hidden constant is locally uniform in $\lambda$. 
\end{proposition}
Generally speaking, it is more natural to demand conormal regularity of $g$, such control is not needed for the paper since the $g$ in practice will be smooth away from the corners. Here, we choose to simplify the notation rather than having the most precise estimates. 
\begin{proof}
    1. We first consider the case $k = 1$. Let $\chi_\cu \in C^\infty(\partial \Omega)$ be such that $\chi_\cu = 1$ on $\supp \chi$ and $\widetilde \chi = 1$ on $\supp \chi_\cu$. Recall from (1) of Lemma~\ref{lem:psido_corner_ref} that $\chi_\cu d \mathcal C_\omega \widetilde \chi \in \Psi^{0}(\mathbb S^1)$ is uniformly elliptic in $\omega$. Therefore,
    \[[\chi_\cu d \mathcal C_\omega \widetilde \chi, \rho_{\mathrm{ref}} \partial_\theta ] \in \Psi^0(\mathbb S^1), \qquad \WF ([\chi_\cu d \mathcal C_\omega \widetilde \chi, \rho_{\mathrm{ref}} \partial_\theta ]) \subset \supp \chi_\cu \times \R. \]
    Therefore, 
    \begin{equation}\label{eq:conorm1}
    \begin{aligned}
        \|\Pi^- (\rho_{\mathrm{ref}} \partial_\theta) \chi v\|_{H^s} &\lesssim \|\chi_\cu d \mathcal C_\omega \widetilde \chi (\rho_{\mathrm{ref}} \partial_\theta) v\|_{H^s} + \|\widetilde \chi v\|_{H^{-N}} \\
        &\lesssim \|(\rho_{\mathrm{ref}} \partial_\theta) \chi_\cu d \mathcal C_\omega \widetilde \chi v\|_{H^s} + \|\widetilde \chi v\|_{H^s} + \|\widetilde \chi v\|_{H^{-N}}.
    \end{aligned}
    \end{equation}
    It follows from (2) and (4) of Lemma~\ref{lem:psido_corner_ref} that 
    \begin{gather*}
        \|(\rho_{\mathrm{ref}} \partial_\theta) \chi_\cu d \mathcal C_\omega (\widetilde \chi \circ \gamma^-_\lambda)v\|_{H^s} \lesssim \|(\widetilde \chi \circ \gamma^-_\lambda)v\|_{H^{-N}}\\
        \|(\rho_{\mathrm{ref}} \partial_\theta) \chi_\cu d \mathcal C_\omega (1 - \widetilde \chi - (\widetilde \chi \circ \gamma^+_\lambda) - (\widetilde \chi \circ \gamma^-_\lambda))v\|_{H^s} \lesssim \|v\|_{H^{-N, a}_\bo}.
    \end{gather*}
    Then continuing~\eqref{eq:conorm1}, we see that
    \begin{equation}
        \|\Pi^- (\rho_{\mathrm{ref}} \partial_\theta) \chi v\|_{H^s} \lesssim \|(\rho_{\mathrm{ref}} \partial_\theta) \chi_\cu d \mathcal C_\omega (\widetilde \chi \circ \gamma^+_\lambda)v\|_{H^s} +  \|(\rho_{\mathrm{ref}} \partial_\theta) \widetilde \chi g\|_{H^s} + \|\widetilde \chi v\|_{H^s} + \|v\|_{H^{-N, a}_\bo}
    \end{equation}
    By (3) of Lemma~\ref{lem:psido_corner_ref}, we have the decomposition 
    \begin{gather*}
        \chi_\cu d \mathcal C_\omega (\widetilde \chi \circ \gamma^+_\lambda) = \sum_{\pm} (\gamma^+_\lambda)^* (\chi_\cu \circ \gamma^+_\lambda) T^+_{\omega, \pm} (\widetilde \chi \circ \gamma^+_\lambda)H(\pm \theta) \\
        T^+_{\omega, \pm} \in \Psi^0(\mathbb S^1), \qquad \WF(T^+_{\omega, \pm}) \subset \supp \chi_0 \times \R_+.
    \end{gather*}
    Using Lemma~\ref{lem:bvf_conj} with the same notation, we see that 
    \begin{equation}\label{eq:4pieces}
    \begin{aligned}
        (\rho_{\mathrm{ref}} \partial_\theta) \chi_\cu d \mathcal C_\omega \widetilde \chi &= (\gamma_\lambda^+)^*\big[(\gamma_\lambda^+)^*(\rho_{\mathrm{ref}} \partial_\theta) (\gamma_\lambda^+)^*\big] (\gamma^+_\lambda)^* (\chi_\cu \circ \gamma^-_\lambda) T^+_\omega (\widetilde \chi \circ \gamma_\lambda^+)\\
        &= (\gamma_\lambda^+)^*\sum_{\mu, \nu = \pm} H(\mu x)\big(\psi_\mu(\mu x) T^+_{\omega, \nu} (\theta \partial_\theta) + R_{\mu, \nu}\big) H(\nu x)
    \end{aligned}
    \end{equation}
    for some $R_{\mu, \nu} \in \Psi^0(\mathbb S^1)$ uniformly in $\omega$. Applying~\eqref{eq:4pieces} to~\eqref{eq:conorm1}, we find 
    \[\|\Pi^- \rho_{\mathrm{ref}} \partial_\theta \chi v\|_{H^s} \lesssim \|\chi_0 (\theta \partial_\theta) v\|_{H^s} + \|\chi_0 v\|_{H^s} + \|(\rho_{\mathrm{ref}} \partial_\theta) \widetilde \chi g\|_{H^s} + \|\chi v\|_{H^s} + \|v\|_{H^{-N, a}_\bo}.\]

    \noindent
    2. Now we consider the $k > 1$ estimates. First, we claim that for every $k \in \N$, there exists $A_j \in \Psi^0(\mathbb S^1)$ such that 
    \begin{equation}\label{eq:conorm_commutator}
        [\chi_\cu d \mathcal C_\omega \widetilde \chi, (\rho_{\mathrm{ref}} \partial_\theta)^k] = \sum_{j = 0}^{k - 1} A_j (\rho_{\mathrm{ref}} \partial_\theta)^j, \quad \WF(A_j) \subset \supp \chi_\cu. 
    \end{equation}
    Indeed, this holds for $k = 1$, and assuming that~\eqref{eq:conorm_commutator} holds for $k$, we have
    \begin{align*}
        [\chi_\cu d \mathcal C_\omega \widetilde \chi, (\rho_{\mathrm{ref}} \partial_\theta)^{k + 1}] &= (\rho_{\mathrm{ref}} \partial_\theta) [\chi_\cu d \mathcal C_\omega \widetilde \chi, (\rho_{\mathrm{ref}} \partial_\theta)^k] - [\rho_{\mathrm{ref}} \partial_\theta, \chi_\cu d \mathcal C_\omega \widetilde \chi] (\rho_{\mathrm{ref}} \partial_\theta)^k \\
        &= \sum_{j = 0}^{k - 1} [\rho_{\mathrm{ref}} \partial_\theta, A_j] (\rho_{\mathrm{ref}} \partial_\theta)^j + \sum_{j = 0}^{k - 1} A_j (\rho_{\mathrm{ref}} \partial_\theta)^{j + 1} \\
        &\qquad - [\rho_{\mathrm{ref}} \partial_\theta, \chi_\cu d \mathcal C_\omega \widetilde \chi] (\rho_{\mathrm{ref}} \partial_\theta)^k,
    \end{align*}
    which is in the form~\ref{eq:conorm_commutator} for $k + 1$, completing the induction. Applying the elliptic estimate as in Step 1, we have
    \begin{equation}\label{eq:cpk}
    \begin{aligned}
        \|\Pi^-(\rho_{\mathrm{ref}} \partial_\theta)^k \chi v\|_{H^s} &\lesssim \|\chi_\cu d \mathcal C_\omega \widetilde \chi (\rho_{\mathrm{ref}} \partial_\theta)^k v\|_{H^s} + \|\widetilde \chi v\|_{H^{-N}} \\
        &\lesssim \|(\rho_{\mathrm{ref}} \partial_\theta)^k \chi_\cu d \mathcal C_\omega \widetilde \chi v \|_{H^s} + \sum_{j = 0}^{k - 1} \|(\rho_{\mathrm{ref}} \partial_\theta)^k \widetilde \chi v\|_{H^s} + \|\widetilde \chi v\|_{H^{-N}} \\
        &\lesssim \|(\rho_{\mathrm{ref}} \partial_\theta)^k \chi_\cu d \mathcal C_\omega (\widetilde \chi \circ \gamma^+_\lambda) v\|_{H^s} +  \|(\rho_{\mathrm{ref}} \partial_\theta)^k \widetilde \chi g\|_{H^s} \\
        &\qquad + \sum_{j = 0}^{k - 1} \|(\rho_{\mathrm{ref}} \partial_\theta)^k \widetilde \chi v\|_{H^s} + \|v\|_{H^{-N, a}_\bo}.
    \end{aligned}
    \end{equation}
    Applying $\rho_{\mathrm{ref}} \partial_\theta$ repeatedly to~\eqref{eq:4pieces}, we see that there exists $A_{j, \mu, \nu} \in \Psi^0(\mathbb S^1)$ for $j = 1, \dots, k$ and $\mu, \nu = \pm$ with $\WF(A_{j, \mu, \nu}) \subset (\supp \chi_c \circ \gamma^+_\lambda) \times \R$ such that 
    \begin{equation*}
        (\rho_{\mathrm{ref}} \partial_\theta)^k \chi_\cu d \mathcal C_\omega (\widetilde \chi \circ \gamma^+_\lambda) = (\gamma^+_\lambda)^*\sum_{j = 1}^k \sum_{\mu, \nu= \pm} H(\mu x) A_{j, \mu, \nu} (\theta \partial_\theta)^j H(\nu x).
    \end{equation*}
    This gives 
    \begin{equation*}
        \|(\rho_{\mathrm{ref}} \partial_\theta)^k \chi_\cu d \mathcal C_\omega (\widetilde \chi \circ \gamma^+_\lambda)v\|_{H^s} \lesssim \sum_{j = 1}^k \|(\theta \partial_\theta)^j (\widetilde \chi \circ \gamma^+_\lambda) v \|_{H^s}
    \end{equation*}
    Applying this to~\eqref{eq:cpk}, we find
    \begin{multline*}
        \|\Pi^- (\rho_{\mathrm{ref}} \partial_\theta)^k \chi v\|_{H^s} \lesssim \sum_{j = 0}^{k - 1} \|(\rho_{\mathrm{ref}} \partial_\theta)^j \widetilde \chi v\|_{H^s} + \sum_{j = 0}^k \|(\theta \partial_\theta)^j (\widetilde \chi \circ \gamma^+_\lambda) v\|_{H^s} \\
        + \|(\rho_{\mathrm{ref}} \partial_\theta)^k \widetilde \chi g\|_{H^s} + \|v\|_{H^{-N, a}_\bo}.
    \end{multline*} 
    We may inductively apply this estimate to itself by adjusting the cutoffs until we reach
    \begin{multline}
        \|\Pi^- (\rho_{\mathrm{ref}} \partial_\theta)^k \chi v\|_{H^s} \lesssim \|\widetilde \chi v\|_{H^s} + \sum_{j = 0}^k \|(\theta \partial_\theta)^j (\widetilde \chi \circ \gamma^-_0) v\|_{H^s} \\
        + \|\widetilde \chi g\|_{H^{s + k}} + \|v\|_{H^{-N, a}_\bo},
    \end{multline}
    at which point we apply Proposition~\ref{prop:ref_corner_POS} to get the desired estimate. 
\end{proof}

In other words, the conormal regularity at $\gamma^+_\lambda(0)$ is controlled by the conormal regularity at the corner. Of course, this phenomenon is not limited to the corner and its reflection. More generally, for any $x \in \partial \Omega \setminus \widetilde{\mathcal K}$, the positive (or negative) frequency conormal regularity of $v$ at $x$ is controlled by the negative (or positive) frequency conormal regularity of $v$ at $\gamma^\mp(x)$.

Let $x_0 \in \partial \Omega \setminus \widetilde{\mathcal K}$. Fix
\begin{equation}\label{eq:rhovf}
    \rho \in C^\infty(\partial \Omega), \qquad \rho(x_0) = 0, \qquad \rho'(x_0) \neq 0.
\end{equation}
Define
\begin{equation}\label{eq:rhovfpm}
    \rho_\pm := \rho \circ \gamma^\pm_\lambda.
\end{equation}
With pullbacks and derivatives acting on one-forms, a simpler version of Lemma~\ref{lem:bvf_conj} gives
\begin{equation}\label{eq:nvf_cov}
    (\gamma^\pm)^* (\rho \partial_\theta) (\gamma^\pm)^* = \psi_\pm \rho_\pm \partial_\theta + \varphi_\pm
\end{equation}
for some $\psi_\pm, \varphi_\pm \in C^\infty(\partial \Omega)$. 
\begin{proposition}\label{prop:conorm_POS}
    Let $\mathcal J \Subset (0, 1)$ be a sufficiently small subinterval, and assume that $\Omega$ is $\lambda$-simple for all $\lambda \in \mathcal J$. Write $\omega = \lambda + i \epsilon$ for $\lambda \in \mathcal J$ and sufficiently small $\epsilon > 0$. Let $x_0 \in \partial \Omega \setminus \widetilde{\mathcal K}$ and let $\rho$ and $\rho_\pm$ be as in~\eqref{eq:rhovf} and ~\eqref{eq:rhovfpm}. Let $\chi, \widetilde \chi \in C^\infty(\partial \Omega)$ be supported in a small neighborhood of $x_0$ such that $\chi = 1$ near $x_0$ and $\widetilde \chi = 1$ on $\supp \chi$. Then for $a > -1$, we have
    \begin{equation}
        \|\Pi^\pm(\rho \partial_\theta)^k \chi v\|_{H^s} \le C \big(\sum_{j = 0}^{k - 1} \|\Pi^\pm (\rho \partial_\theta)^j \widetilde \chi v\|_{H^s} + \|\widetilde \chi g\|_{H^{s + k}} + \|\chi v\|_{H^{-N, a}_\bo} \big).
    \end{equation}
\end{proposition}
\begin{proof}
    Let $\chi_\cu \in C^\infty(\partial \Omega$ be such that $\chi_\cu = 1$ on $\supp \chi$ and $\widetilde \chi = 1$ on $\supp \chi_\cu$. Similar to Step 2 in the proof of Proposition~\ref{prop:conorm_ref}, we have that for every $k \in \N$, there exists $A_j \in \Psi^0(\mathbb S^1)$ with $\WF (A_j) \subset \supp \chi_\cu \times \R$ such that
    \begin{equation}
        [\chi_\cu d \mathcal C_\omega \widetilde \chi, (\rho \partial_\theta)^k] = \sum_{j = 0}^{k - 1} A_j(\rho \partial_\theta)^j. 
    \end{equation}
    It then follows from the elliptic estimate and Lemma~\ref{lem:psido_no_corner}
    \begin{equation}\label{eq:cPOS_decomp}
    \begin{aligned}
        \|\Pi^\pm(\rho \partial_\theta)^k \chi v\|_{H^s} &\lesssim \|\Pi^\pm (\chi_\cu d \mathcal C_\omega \widetilde \chi) (\rho \partial_\theta)^k v\|_{H^s} + \|\widetilde v\|_{H^{-N}} \\
        &\lesssim \|\Pi^\pm (\rho \partial_\theta)^k \chi_\cu d \mathcal C_\omega \widetilde \chi v\|_{H^s} + \sum_{j = 0}^{k - 1} \|\Pi^\pm (\rho \partial_\theta)^j \widetilde \chi v\|_{H^s} + \|\chi v\|_{H^{-N}} \\
        &\lesssim \|\Pi^\pm (\rho \partial_\theta)^k \chi_\cu d \mathcal C_\omega (\widetilde \chi \circ \gamma^+_\lambda) v\|_{H^s} +  \|\Pi^\pm (\rho \partial_\theta)^k \chi_\cu d \mathcal C_\omega (\widetilde \chi \circ \gamma^-_\lambda) v\|_{H^s}\\
        &\qquad + \|\Pi^\pm (\rho \partial_\theta)^k \chi_\cu g\|_{H^s} +  \sum_{j = 0}^{k - 1} \|\Pi^\pm (\rho \partial_\theta)^j \widetilde \chi v\|_{H^s} + \|\chi v\|_{H^{-N, a}_\bo},
    \end{aligned}
    \end{equation}
    where in the last inequality, we used (3) from Lemma~\ref{lem:psido_no_corner}. Now from (2) in Lemma~\ref{lem:psido_no_corner}, we see that 
    \begin{gather*}
        \Pi^\pm (\rho \partial_\theta)^k \chi_\cu d \mathcal C_\omega (\widetilde \chi \circ \gamma^\pm_\lambda) \in \Psi^{-\infty}(\mathbb S^1),
    \end{gather*}
    and 
    \begin{equation}
        (\gamma^\pm_\lambda)^* \Pi^\mp (\rho \partial_\theta)^k \chi_\cu d \mathcal C_\omega (\widetilde \chi \circ \gamma^\pm_\lambda) = (\gamma^\pm_\lambda)^* (\rho \partial_\theta)^k (\gamma^\pm_\lambda)^* (\chi_\cu \circ \gamma^\pm_\lambda) T^\pm_\omega (\widetilde \chi \circ \gamma^\pm)\Pi^\pm + \mathcal R_\omega^\pm
    \end{equation}
    for some $\mathcal R^\pm_\omega \in \Psi^{-\infty}(\mathbb S^1)$. Observe that by \eqref{eq:nvf_cov}, there exists $\psi_{j, \pm} \in C^\infty(\partial \Omega)$ such that 
    \[(\gamma^\pm_\lambda)^* (\rho \partial_\theta)^k (\gamma^\pm_\lambda)^* = \sum_{j = 0}^k \psi_{j, \pm} (\rho_\pm \partial_\theta)^k.\]
    Therefore, 
    \begin{equation*}
        \|\Pi^\mp (\rho \partial_\theta)^k \chi_\cu d \mathcal C_\omega (\widetilde \chi \circ \gamma^\pm_\lambda) v\|_{H^s} \lesssim \sum_{j = 1}^k \|\Pi^\pm (\rho_\pm \partial_\theta)^k (\widetilde \chi \circ \gamma^\pm)v\|_{H^s} + \|v\|_{H^{-N, a}_\bo}.
    \end{equation*}
    Combining with \ref{eq:cPOS_decomp}, we get
    \begin{multline*}
        \|\Pi^\pm(\rho \partial_\theta)^k \chi v\|_{H^s} \le C \big(\sum_{j = 1}^k \|\Pi^\pm (\rho_\pm \partial_\theta)^k (\widetilde \chi \circ \gamma^\pm)v\|_{H^s} +\sum_{j = 0}^{k - 1} \|\Pi^\pm (\rho \partial_\theta)^j \widetilde \chi v\|_{H^s} \\
        + \|\Pi^\pm (\rho \partial_\theta)^k \widetilde \chi g\|_{H^s} + \|\chi v\|_{H^{-N, a}_\bo} \big).
    \end{multline*}
    Applying this estimate to its own right hand side inductively by adjusting cutoffs and then applying Proposition~\ref{prop:POS} yields the desired estimate. 
\end{proof}

\section{Limiting absorption principle}\label{sec:LAP}
The limiting absorption principle is established using the high frequency estimates developed in Proposition~\ref{prop:global_semifredholm} and the uniqueness of problem~\eqref{eq:EBVP} in the limit $\Im \omega \to 0+$. Again we only consider the limit of $\omega$ approaching the real line from the upper half plane. The arguments for limit from the lower half plane is the same up to sign changes. Throughout the section, let $\mathcal J \Subset (0, 1)$ be such that $\Omega$ is $\lambda$-simple and $b_\lambda$ is Morse--Smale for all $\lambda \in \mathcal J$.

\subsection{Consequences of the high frequency estimate}
We first summarize the consequences of the high frequency estimate for the limit as $\omega$ approaches the real line from the upper half complex plane. Let $\mathcal Z_\lambda$ be the set of limiting indicial roots defined in~\eqref{eq:all_zeros}. 
\begin{lemma}\label{lem:convergence_upgrade}
    Let $\omega_j \to \lambda \in \mathcal J$ with $\Im \omega_j > 0$. Let $-1 < a < a' < 0$ be such that 
    \[\mathcal Z_\lambda \cap \{\R - ia'\} = \emptyset.\] Assume that $v_j \in H^{\infty, a'}_\bo$ satisfies
    \begin{equation}\label{eq:vj_limit_weak}
        \begin{gathered}
            v_j \to v_0 \quad \text{in} \quad H^{-N, a - \epsilon}_\bo \quad \text{for some} \quad N \in \N, \ 0 < \epsilon < a + 1, \\
            d \mathcal C_{\omega_j} v_j \quad \text{is bounded in} \quad H^{s, a'}_{\bo}(\partial \Omega) \quad \text{for all} \quad s \ge a + \ha. 
        \end{gathered}
    \end{equation}
    Then for all sufficiently small $\delta > 0$ and $\chi_\bo, \chi_\pm \in C^\infty(\partial \Omega;[0, 1])$ such that 
    \begin{equation*}
    \begin{gathered}
        \chi_\bo = 1 \quad \text{near} \quad \mathcal K,\qquad \supp \chi \subset \mathcal K + (-\delta, \delta), \\
        \chi_\pm = 1 \quad \text{on} \quad \Sigma^\pm_\lambda(\delta), \qquad \supp \chi_\pm \subset \Sigma^\pm_\lambda(2 \delta),
    \end{gathered}
    \end{equation*}
    we have the following improvements:
    \begin{align}
        v_j \to v_0 \quad &\text{in} \quad H^{-\ha - \beta, a}_\bo \quad \text{for all} \quad \beta > 0, \label{eq:vj_limit_strong}\\
        \chi_\bo v_j \to \chi_\bo v_0  \quad &\text{in} \quad H^{\infty, a}_\bo, \label{eq:v0_bo} \\
        \Pi^\pm(1 - \chi_\mp - \chi_\bo) v_0 &\in \mathcal A^{[a + \ha]}(\partial \Omega; N^*_\pm (\mathscr O^\mp_{\lambda} \setminus \Sigma^\mp_\lambda(\delta))) \label{eq:v0_co} 
    \end{align}
\end{lemma}
\begin{Remark}
    We note here at $\mathscr O^\pm_\lambda \setminus \Sigma^\pm_\lambda(\delta)$ is a finite set of points, so the conormal distributions with respect to this set is well-defined.
\end{Remark}
\begin{proof}
    1. Let $0 < \beta' < \beta$. By Proposition~\ref{prop:global_semifredholm}, it follows from~\eqref{eq:vj_limit_weak} that $\|v_j\|_{H^{-\ha - \beta', a'}_\bo}$ is bounded. The embedding $H^{-\ha - \beta', a'}_\bo(\partial \Omega) \hookrightarrow H^{-\ha - \beta, a}_\bo(\partial \Omega)$ is compact, therefore there exists a subsequence of $v_j$ that converges to $v_0$ in $H^{-\ha - \beta, a}_\bo(\partial \Omega)$. Since $v_j \to v_0$ in $H^{-N, a - \epsilon}(\partial \Omega)$, the whole sequence must then converge to $v_0$ in $H^{-\ha - \beta, a}_\bo(\partial \Omega)$, which gives~\eqref{eq:vj_limit_strong}.

    \noindent
    2. Observe that~\eqref{eq:corner_into_source} gives
    \begin{equation}\label{eq:corner_improvement}
        \|\chi_\bo v_j\|_{H^{s, a'}_\bo} \lesssim \|d \mathcal C_{\omega_j} v_j\|_{H^{s, a'}_{\bo}} + \|v_j\|_{H^{-N, a}_\bo}
    \end{equation}
    for all $s > a' + \ha$. In particular, this shows that $\|\chi_\bo v_j\|_{H^{s, a'}_\bo}$ is bounded. A similar compactness argument from step 1 then shows that $\chi_\bo v_j \to \chi_\bo v_0$ in $H^{s, a}_\bo$ for all $s \gg 1$, which gives~\eqref{eq:v0_bo}.

    \noindent
    3. It remains to verify~\eqref{eq:v0_co}. First observe that if $\chi \in \overline C^\infty(\Omega)$ is such that $\supp \chi$ does not contain any points in $\mathscr O_\lambda^\pm$ or $\mathcal K$, then~\eqref{eq:easy_into_source} applies and we see that 
    \[\|\chi v_j\|_{H^s} \lesssim \|d \mathcal C_{\omega_j} v_j\|_{H^{s, a'}_\bo} + \|v_j\|_{H^{-N, a}_\bo}\]
    for every $s \ge a' + \ha$. Again by a similar compactness argument from Step 1, we conclude that $\chi v_0 \in H^s$ for all $s$. 

    Now consider $\theta_0 \in \mathscr O_\lambda^\pm(\theta)$, and assume that $b^{\mp n}(\theta_0) = \kappa$ for some $n \in \N$ where $\kappa$ is a type $(+, +)$ corner. The other corner types are covered by taking reflections of the domain. Observe that in the case $\theta_0 \in \mathscr O_\lambda^+(\theta)$ and $b^{-n}(\theta_0) = \kappa$, then
    \[b^{-(n - 1)}(\theta_0) = \gamma^+(\kappa).\]
    Therefore, applying Proposition~\ref{prop:conorm_POS} iteratively $2(n - 1)$ times and then applying Proposition~\ref{prop:conorm_ref}, we find that there exists $\chi \in \overline C^\infty(\partial \Omega)$ such that $\chi = 1$ near $\theta_0$ and 
    \begin{equation}\label{eq:co_improvement}
    \begin{aligned}
        \|\Pi^+ (\rho_{\mathrm{ref}} \partial_\theta)^k \chi v_j\|_{H^{a' + \ha}} &\lesssim \sum_{j = 0}^k \|(\theta \partial_\theta)^j \chi_\bo v_j\|_{H^{a' + \ha}} + \|d \mathcal C_{\omega_j} v_j\|_{H^{a' + \ha + k, a'}_\bo} + \|v_j\|_{H^{-N, a}_\bo} \\
        &\lesssim \|\chi_\bo v\|_{H^{a' + \ha + k, a'}_\bo} + \|d \mathcal C_{\omega_j} v_j\|_{H^{a' + \ha + k, a'}_\bo} + \|v_j\|_{H^{-N, a}_\bo} \\
        &\lesssim \|d \mathcal C_{\omega_j} v_j\|_{H^{a' + \ha + k, a'}_{\bo}} + \|v_j\|_{H^{-N, a}_\bo}
    \end{aligned}
    \end{equation}
    The last inequality above follows from~\eqref{eq:corner_improvement}. Since~\eqref{eq:co_improvement} holds for all $k$ and the right hand side is bounded, we use the same compactness argument as in Step 1 to conclude that 
    \begin{equation}\label{eq:local_co_est}
        \Pi^+ \chi v_0 \in \mathcal A^{[a - \ha]}(\partial \Omega; \theta_0).
    \end{equation}

    In the case that $\theta_0 \in \mathscr O_{\lambda}^-(\theta)$ and $b^n(\theta_0) = \kappa$, we instead have
    \[\gamma^- \circ b^{n - 1}(\theta_0) = \gamma^+(\kappa).\]
    Applying Proposition~\ref{prop:conorm_POS} iteratively $2n - 1$ times and then applying Proposition~\ref{prop:conorm_ref} in a similar computation to \ref{eq:co_improvement}, we find that there exists $\chi \in \overline C^\infty(\partial \Omega)$ such that $\chi = 1$ near $\theta_0$ and 
    \[\|\Pi^+ (\rho_{\mathrm{ref}} \partial_\theta)^k \chi v_j\|_{H^{a' + \ha}} \lesssim \|d \mathcal C_{\omega_j} v_j\|_{H^{a' + \ha + k, a'}_{\bo}} + \|v_j\|_{H^{-N, a}_\bo}.\]
    Again following a similar compactness argument to Step 1, we find that $\Pi^- \chi v_0 \in \mathcal A^{[a - \ha]}(\partial \Omega; \theta_0)$. This completes the proof of~\eqref{eq:v0_co}.
\end{proof}

\subsection{Uniqueness of the limiting problem} 
We first need some lemmas about dilation invariant distributions. Restrict the multiplication map $M_\sigma: \R \to \R$ to the interval $(-1, 1)$ and assume that $\sigma \in (0, 1)$. Then we have a map 
\begin{equation}\label{eq:local_dynamics}
    M_\sigma: (-1, 1) \to (-1, 1) \quad \text{by} \quad x \mapsto \sigma x, \quad 0 < \sigma < 1
\end{equation}
The point is that $M_\sigma$ provides a linearized local model of the chess billiard map near the attractors $\Sigma_\lambda^+$. Studying distributions on $(0, 1)$ invariant by $M_\sigma$ will then easily transfer to results on $\partial \Omega$. 

To study such distributions, we first define the \textit{flux} of $u \in \overline{\mathcal D}'([-1, 1])$. Assume that 
\begin{equation}\label{eq:dil_inv_condition}
    M_\sigma^* u = u \quad \text{and} \quad \widetilde \chi u \in \bar{H}^{\ha +}(\R) \quad \text{for every} \quad \widetilde \chi \in \CIc([-1, 1] \setminus \{0\}).
\end{equation}
For such $u$, define
\begin{equation}\label{eq:flux_def}
    \mathbf F(u) := i \int_{-1}^1 (M_\sigma^* \chi - \chi) \bar u \, du
\end{equation}
where $\chi \in \CIc((- \sigma, \sigma))$ and $\chi = 1$ near $0$. Here, the integral is understood as a pairing with a distributional 1-form. In particular, $du \in H^{-\ha +}(\R; T^* \R)$, so the pairing is well-defined. Note that for every $\widetilde \chi \in \CIc((-\sigma, \sigma) \setminus \{0\})$ we have,
\begin{equation}
    \int_{-1}^1 (M_\sigma^* \widetilde \chi - \widetilde \chi) \bar u \, du = \int (M_\sigma^* \widetilde \chi)(\bar u \, du - M_\sigma^* (\bar u du)) = 0,
\end{equation}
where the last equality follows from~\eqref{eq:dil_inv_condition}. Therefore, the definition of the flux is independent of the choice of cutoff. Note that the above holds even if $\widetilde \chi = \indic_{[-\sigma, \sigma]} - \chi$. Then by Lemma~\ref{eq:cutting}, we see that
\begin{equation}
    \mathbf F(u) = i \int_{[-1, -\sigma] \cup [\sigma, 1]} \bar u \, du
\end{equation}
where the integral should be interpreted as $\bar u$ paired with $\indic_{[-1, -\sigma] \cup [\sigma, 1]} du$. We also mention now that the flux is not limited to dilation maps $M_\sigma$. The definition~\ref{eq:flux_def} can be easily adjusted to make sense for any map with $0$ as an attracting fixed point, and the definition would be independent of the choice of cutoffs again.

The following lemma is essentially the same as \cite[Lemma 6.1]{DWZ}. The key difference is that we have to allow for $u$ with lower regularity away from $0$.  
\begin{lemma}\label{lem:local_const}
    Let $M_\sigma$ be as in~\eqref{eq:local_dynamics}. Let $u \in \mathcal D'((-1, 1))$ be such that $M_\sigma^* u = u$ and $\WF(u) \subset \{(x, \xi) \in T^*(-1, 1) : \xi > 0\}$. Furthermore, assume that 
    \begin{equation}\label{eq:reg_assumption}
        \widetilde \chi u \in H^{\ha +}(\R) \quad \text{for every} \quad \widetilde \chi \in \CIc([-1, 1] \setminus \{0\}).
    \end{equation}
    Then $\mathbf F(u) \ge 0$ implies that $u$ is constant. 
\end{lemma}
\begin{proof}
    1. We first extend $u$ to all of $\R$. Fix 
    \begin{equation*}
        \psi \in \CIc((\sigma, \sigma) \setminus [-\sigma^2, \sigma^2]), \quad \sum_{k \in \Z} (M_\sigma^k)^* \psi = 1 \quad \text{for} \quad x \neq 0.
    \end{equation*}
    Define
    \begin{equation*}
        v = \sum_{k \in \Z} (M_\sigma^k)^*(\psi u) \in \mathcal D'(\R \setminus \{0\}).  
    \end{equation*}
    Observe that as elements of $\mathcal D'((-1, 1) \setminus \{0\})$, $u = v$, so we can extend $v$ to an element of $\mathcal D'(\R)$ by setting $v|_{(-1, 1)} = u$. Note that 
    \[v \in \mathscr S'(\R), \quad M_\sigma^* v = v, \quad \WF(v) \subset \{(x, \xi) \in T^* \R : \xi > 0\}.\]
    Furthermore, with the regularity assumption in~\eqref{eq:reg_assumption}, we have
    \[\mathbf F(v) = \mathbf F (u) = i\int_{[-1, -\sigma] \cup [\sigma, 1]} \overline v \, dv\ge 0.\]
    It suffices to show that the extension $v$ is constant. 

    \noindent
    2. Next, we claim that $\hat v(\xi) = 0$ for $\xi < 0$. Decompose $v$ by 
    \[v = v_1 + v_2, \quad v_2 := \sum_{k = 0}^\infty (M_\sigma^k)^*(\psi u), \quad v_1 := v - v_2.\]
    $v_1$ is compactly supported, and $\WF(v_1) \subset \{\xi > 0\}$. Therefore, 
    \[\hat v_1(\xi) = \mathcal O(|\xi|^{-\infty}) \quad \text{as}\quad  \xi \to -\infty.\]
    Similarly, we also see that $\widehat \psi u(\xi) = \mathcal O(|\xi|^{-\infty})$ as $\xi \to -\infty$. Therefore, for every $N \in \N$, 
    \begin{equation*}
        |\hat v_2(\xi)| = \sum_{k = 0}^\infty \sigma^{-k} \left|(\widehat{\psi u})(\sigma^{-k} \xi) \right| \lesssim \sum_{k = 0}^\infty \sigma^{-k} |\sigma^{-k} \xi|^{-N} \lesssim |\xi|^{-N}
    \end{equation*}
    Therefore, $\hat v(\xi) = \mathcal O(|\xi|^{-\infty})$ as $\xi \to -\infty$. Now using invariance by $M_\sigma$, we see that for $\xi < 0$ and $N > 1$, 
    \[|\hat v(\xi)| = |\sigma^{-k} \hat v(\sigma^{-k}\xi)| \lesssim \sigma^{(N - 1) k} |\xi|^{-N} \]
    where the hidden constant is independent of $k$. Taking $k \to \infty$, it follows that 
    \begin{equation}
        \hat v(\xi) = 0 \quad \text{for} \quad \xi < 0.
    \end{equation} 

    \noindent
    3. Since $\hat v \in \mathscr S'(\R)$ is supported on $[0, \infty)$, the (inverse) Fourier-Laplace transform then exists and is holomorphic on the upper half plane:
    \begin{equation*}
        V(z) := \frac{1}{2 \pi} \hat v(e^{i z \bullet}), \quad \Im z > 0.
    \end{equation*}
    Furthermore, $V(\bullet + i\epsilon) \to v$ in $\mathscr S'(\R)$ as $\epsilon \to 0+$ and $V(\bullet + i \epsilon)$ is uniformly bounded in $H^{\ha+}(\gamma_0)$ for all $\epsilon > 0$ where $\gamma_0 = [-1, -\sigma] \cup [\sigma, 1]$. Therefore,
    \begin{equation}\label{eq:v_convergence}
        V(\bullet + i\epsilon) \to v  \quad \text{in} \quad H^{\ha+}(\gamma_0)
    \end{equation}
    Let $\gamma_1$ and $\gamma_\sigma$ be positively oriented semicircle contours in the upper half plane centered at $0$ of radius $1$ and $\sigma$ respectively, see Figure~\ref{fig:contour}. Observe that $V(z) = V(\sigma z)$, so
    \[\int_{\gamma_1} \overline{V(z)} \partial_z V(z) \, dz = \int_{\gamma_1} \sigma \overline{V(\sigma z)} (\partial_z V)(\sigma z)\, dz = \int_{\gamma_\sigma} \overline{V(z)} \partial_z V(z) \, dz.\]
    Let $\Gamma$ be the region bound by $\partial \Gamma = \gamma_1 + \gamma_0 - \gamma_\sigma$. Therefore, by the Cauchy--Pompeiu formula (see \cite[(3.1.9)]{H1}), we find that 
    \begin{align*}
        F(v) &= i \int_{\gamma_0} \overline{V(z)} \partial_z V(z) \, dz  =  i\int_{\partial \Gamma} \overline{V(z)} \partial_z V(z) \, dz \\
        &= -2 \int_{\gamma} \partial_{\overline z}(\overline{V(z)} \partial_z V(z))\, dx dy = -2 \int_{\Gamma} |\partial_z V(z)|^2\, dx dy \le 0,
    \end{align*}
    where we used~\eqref{eq:v_convergence} to justify integrating over $\gamma_0$.  We assumed that $\mathbf F(v) \ge 0$, so $\partial_z V(z) = 0$ in $\Gamma$. Therefore $v$ is constant on $\gamma_0$, and by invariance under $M_\sigma$, $v$ must be constant on $\R \setminus \{0\}$. Therefore $\supp(v - \mathrm{const.})$ is contained in $\{0\}$ for some constant, so it must be a linear combination of $\delta^{(k)}_0$. However, no such distributions are invariant by $M_\sigma$, so $v$ is constant everywhere.  
\end{proof}

\begin{figure}
    \centering
    \includegraphics[scale = 0.7]{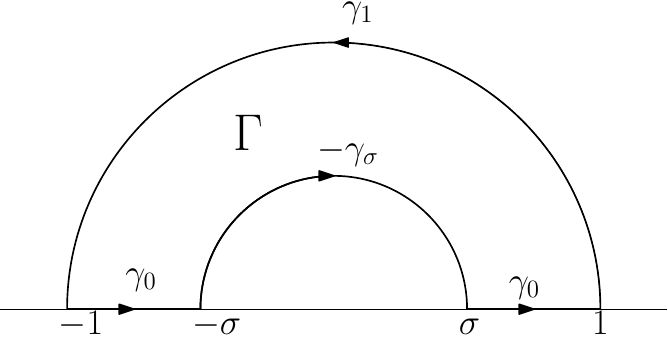}
    \caption{Picture of the region $\Gamma$ and the contour $\partial \Gamma = \gamma_1 + \gamma_0 - \gamma_\sigma$}
    \label{fig:contour}
\end{figure}

The map $M_\sigma$ in Lemma~\ref{lem:local_const} serves as a local model of the dynamics near periodic points in our Morse--Smale system. It implies that under the same regularity assumptions, the only distributions invariant under the chess-billiard map are constant:
\begin{lemma}\label{lem:global_const}
    Assume that $\Omega$ is Morse--Smale with respect to $\lambda \in (0, 1)$. Let $w \in \overline{\mathcal D}'(\partial \Omega)$. Assume that for all $\chi_\pm \in C^\infty(\partial \Omega)$ with $\chi_\pm = 1$ near $\Sigma_\lambda^\pm$ and sufficiently small support, 
    \begin{equation}\label{eq:w_hyp}
        \begin{gathered}
            \mathbf x^* ((1 - \chi_+ - \chi_-) w) \in H^{\ha +} (\mathbb S^1), \\
            \WF(\chi_\pm w) \subset \{(\theta, \xi) \in T^* \partial \Omega : \theta \in \supp \chi_\pm, \ \xi > 0\}
        \end{gathered}
    \end{equation}
    for some $b$-parameterization $\mathbf x$. Then $b_\lambda^* w = w$ implies that $w$ is constant. 
\end{lemma}
\begin{Remark}
    Observe that the regularity condition in~\ref{eq:w_hyp} is independent of the choice of b-parameterization $\mathbf x$  by Lemma~\ref{lem:pullback_continuity}. The point here is that this allows us to characterize the regularity of $w$ through the corners rather than just up to the corners. 
\end{Remark}
\begin{proof}
     Let $n$ be the minimal period of $b_\lambda$ and put $g:= b^n_\lambda$. According to Lemma~\ref{lem:linear_coords}, we may choose $\mathbf x$ so that $b^n(\lambda)$ in $\Sigma^\pm_\lambda(\delta)$ for some sufficiently small $\delta > 0$. Let $\chi_\pm \in C^\infty(\partial \Omega; [0, 1])$ be such that $\chi_\pm = 1$ near $\Sigma^\pm_\lambda$ and
     \begin{equation}\label{eq:chi_pm_supp}
         \supp \chi_\pm \subset \Sigma^\pm_\lambda(\delta/M)
     \end{equation}
     for some
     \[M \ge \max \{g'(\theta_+), g'(\theta_-)^{-1}: \theta_\pm \in \Sigma^\pm_\lambda \}.\]
     Define the fluxes
     \begin{equation}
     \mathbf F_\pm(w) := i \int_{\partial \Omega} ((g^{\pm 1})^* \chi_\pm - \chi_\pm) \bar w\, dw. 
     \end{equation}
     We argue by contradiction and assume that $w$ is not constant. By assumption, $(g^{\pm 1})^* w = w$, and we also have the wavefront set condition~(\ref{eq:w_hyp}) for $\chi_\pm w$. Therefore, by Lemma~\ref{lem:local_const}, we must have 
     \[\mathbf F_+(w) < 0, \qquad \mathbf F_-(w) > 0.\]
     On the other hand, $\mathbf F$ is independent of the choice of cutoffs, and we may choose $\chi_+ := 1 - (g^{-1})^* \chi_-$. This implies $\mathbf F_+ (w) = \mathbf F_-(w)$, an immediate contradiction. 
\end{proof}

Now we deduce the desired uniqueness of the limiting problem using the model case above.

\begin{lemma}\label{lem:uniqueness}
    Let $\lambda \in (0, 1)$ be such that $\Omega$ is Morse--Smale with respect to $\lambda$. Let $v \in \overline{\mathcal D}'(\partial \Omega; T^* \partial \Omega)$. Assume that there exists $\chi_\pm \in \overline C^\infty(\partial \Omega)$ such that $\chi_\pm = 1$ near $\Sigma^\pm_\lambda$ and 
    \begin{equation}\label{eq:v_unique_wf}
        \WF(\chi_\pm v) \subset \{(\theta, \xi) \in T^* \partial \Omega: \theta \in \supp \chi_\pm, \ \mp \xi > 0\}.
    \end{equation}
    Furthermore, assume that for some $a \in (0, 1)$, and $s \ge a + \ha$, 
    \begin{equation}
        \chi v \in H^{s, a}_\bo(\partial \Omega; T^* \partial \Omega) \quad \text{for every} \quad \chi \in \overline C^\infty(\partial \Omega), \ \supp \chi \cap (\Sigma^+_\lambda \cup \Sigma^-_\lambda) = \emptyset,
    \end{equation}
    and that $\mathrm{sing\, supp\,} v$ does not contain any noncorner characteristic points. Then 
    \begin{equation}
        \mathcal C_{\lambda + i0} v = 0, \quad \supp (R_{\lambda + i0} \mathcal I v) \subset \overline{\Omega} \implies v = 0. 
    \end{equation}
\end{lemma}
\begin{proof}
    1. Put $U := R_{\lambda + i0} \mathcal I v \in \mathcal D'(\R^2)$. Note that we can recover $v$ from $U$ by 
    \[P(\lambda) U = \mathcal I v,\]
    so the plan is to show $U = 0$. We first show that $\supp U \subset \partial \Omega$. Let $u = U|_{\Omega^\circ}$. Recall the notation for characteristic points $x_{\min}^\pm$ and $x_{\max}^\pm$ from Definition~\ref{def:lambda_simple}, and put
    \[\ell^\pm_{\min} := \ell^\pm(x_{\min}^\pm) = \min_{x \in \partial \Omega} \ell^\pm(x), \quad \ell^\pm_{\max} := \ell^\pm(x_{\max}^\pm) = \max_{x \in \partial \Omega} \ell^\pm(x), \quad \]
    Note that because $P(\lambda) u = 4L^+_\lambda L^-_\lambda u = 0$, there exists $u_\pm \in \overline{\mathcal D}'((\ell^\pm_{\min}, \ell^\pm_{\max}))$ such that 
    \[u(x) = u_+(\ell^+(x)) - u_-(\ell^-(x)).\]
    The behavior of $u_\pm$ up to the end point s $\ell^\pm_{\min}$ and $\ell^\pm_{\max}$ can be easily characterized. If a characteristic point $x^\pm_{\min}$ or $x^\pm_{\max}$ is not a corner, then $u$ is smooth up to the boundary in a neighborhood of that characteristic point. This means that $u_\pm$ is smooth up to the respective $\ell^\pm_{\min}$ or $\ell^\pm_{\max}$. On the other hand, say $x^+_{\min}$ is a corner. Then it follows from Lemma \ref{lem:S_map_prop_1} that $u$ lies in $\bar H^{a + \frac{3}{2}}(\Omega)$ in a neighborhood of $x^+_{\min}$. This means that for all $\chi^+_{\min} \in \CIc([\ell^-_{\min}, \ell^-_{\max}))$ supported in a sufficiently small neighborhood of $\ell^-_{\min}$, we have
    \begin{equation}\label{eq:u_pm_corner}
        \chi^+_{\min} u_\pm \in \bar H^{\ha + }([\ell^-_{\min}, \ell^-_{\max}]).
    \end{equation} 
    Analogous statements clearly hold for when $x^-_{\min}$ or $x^\pm_{\max}$ are corners. Now we characterize the behavior of $u_\pm$ on the interior of $(\ell^\pm_{\min}, \ell^\pm_{\max})$. If $\chi \in \CIc ((\ell^\pm_{\min}, \ell^\pm_{\max}))$ is supported away from $(\ell^\pm)^{-1}(\Sigma^\pm)$, then it follows from Lemmas~\ref{lem:S_map_prop_1} and \ref{lem:S_map_prop_2} that 
    \begin{equation}\label{eq:u_pm_interior}
        \chi u_\pm \in \bar H^{\ha +}([\ell^\pm_{\min}, \ell^\pm_{\max}]).
    \end{equation}
    If $\chi \in \CIc ((\ell^\pm_{\min}, \ell^\pm_{\max}))$ is supported in a sufficiently small neighborhood of $(\ell^\pm)^{-1}(\Sigma^\pm)$, then it follows from Lemma~\ref{lem:S_map_prop_2} and~\eqref{eq:v_unique_wf} that 
    \begin{equation}\label{eq:u_pm_wf}
        \WF(\chi u_\pm) \subset \{(x, \xi) \in T^*(\ell^\pm_{\min}, \ell^\pm_{\max}) : x \in \supp \chi, \ \xi > 0\}.
    \end{equation}
    Let 
    \[w_\pm := (\ell^\pm)^* u_\pm.\]
    Note that $(\gamma^\pm)^* w_\pm = w_\pm$ and 
    \[w_+ - w_- = (R_{\lambda + i0} \mathcal I v)|_{\partial \Omega} = \mathcal C_{\lambda + i0} v = 0.\]
    Put $w := w_+ = w_-$, and we see that $b^* w = w$. It follows from~\eqref{eq:u_pm_corner} and~\eqref{eq:u_pm_interior} that for every $\chi_\pm \in \overline C^\infty(\partial \Omega)$ with $\chi_\pm = 1$ near $\Sigma^\pm_\lambda$ and sufficiently small support, we have
    \[\mathbf x^*((1 - \chi_+ - \chi_-) w) \in H^{\ha +}(\mathbb S^1)\]
    for any b-parameteriztaion $\mathbf x$. It follows from~\ref{eq:u_pm_wf} that
    \[\WF(\chi_\pm w) \subset \{(0, \xi) \subset T^* \partial \Omega: \theta \in \supp \chi_\pm, \xi > 0\}.\]
    Therefore, $w$ satisfies the hypothesis of Lemma~\ref{lem:global_const}, so $w$ is constant, and so $u = 0$. 

    \noindent
    2. By Step 1, we see that $U$ is supported on $\partial \Omega$. Now we show that $U$ must be supported on the characteristic points. Let $x_0 \in \partial \Omega$ be noncharacteristic. Let $V \subset \R^2$ be a small neighborhood of $x_0$. Pick smooth coordinates $(y_1, y_2)$ centered at $x_0$ so that $\partial \Omega \cap V = \{(y_1, y_2) \in V: y_2 = 0\}$. Since $x_0$ is noncharacteristic, we may assume that $V$ is a sufficiently small neighborhood so that
    \[P(\lambda) = \sum_{|\alpha| \le 2} c_\alpha(y) \partial_y^\alpha, \quad c_\alpha \in C^\infty, \quad c_{(0, 2)}\  \text{is nonvanishing}.\]
    In the $y$ coordinates, $U$ takes the form $U|_V = \sum_{k = 0}^K u_k(y_1) \delta_0^{(k)}(y_2)$ for some $u_k \in \mathcal D'(\R)$. Note that $(P(\lambda)u)|_{V} = \tilde v(y_1) \delta_0(y_2)$ for some $\tilde v \in \mathcal D'(\R)$, and 
    \begin{equation*}
        (P(\lambda)u)|_{V} = c_{(0, 2)} u_K (y_1) \delta_0^{(K + 2)}(y_2) + \sum_{k = 1}^{K - 1} \tilde u_k(y_1) \delta_0^{(k)}.
    \end{equation*}
    for some $\tilde u_k \in \mathcal D'(\R)$. Since $c_{(0, 2)}$ is nonvanishing, we see that $u_K = 0$. Inductively, we then see that $u|_V = 0$. Therefore $U$ is supported the characteristic points, and $\mathcal I v = P(\lambda) U$ therefore must also be supported on the characteristic points. However, $v$ is in $H^{-\ha +}$ near the characteristic points (including the corners), so we must have $v = 0$. 
\end{proof}

\subsection{Boundary data analysis}
We first establish a limiting absorption principle for the boundary data. This allows us to then use the single layer potential to obtain to full limiting absorption principle of Theorem~\ref{thm:spectral}. Let $f \in \CIc(\Omega)$. Fix an open interval $\mathcal J \subset (0, 1)$ such that every $\lambda \in \mathcal J$ satisfies the Morse--Smale condition. Let $u_\omega\in H^1_0(\Omega)$ be the unique solution to the boundary value problem
\begin{equation}\label{eq:Pu=f_reprise}
    P(\omega) u_\omega = f, \qquad u_\omega|_{\partial \Omega} = 0, \qquad \omega \in \mathcal J + i (0, \infty).
\end{equation}
The existence and uniqueness of $u_\omega$ were established in Proposition \ref{prop:polyhom_EU}. We use the same notation from Proposition~\ref{prop:polyhom_EU} and recall that near a corner $\kappa$, we have the polyhomogeneous expansion
\begin{equation*}
    u_\omega(r, \tau) \sim \sum_{k = 1}^\infty r^{\mathfrak{l}_\omega k} w_k(\tau), \quad w_k \in C^\infty([-\alpha_-, \alpha_+]), \quad w_k(-\alpha_-) = w_k(\alpha_+) = 0
\end{equation*}
near $\beta^{-1}(\kappa) \subset \widetilde \Omega$ where $(r, \tau)$ are the blowup coordinates given in~\eqref{eq:good_corner_coords} and
\begin{equation*}
    \mathfrak{l}_{\omega, \kappa} = \mathfrak l_{\lambda, \kappa} + \mathcal O(\epsilon), \quad \mathfrak l_{\lambda, \kappa} =\frac{2 \pi i}{i \pi - \log \alpha_{\kappa}(\lambda)}, \quad \omega = \lambda + i\epsilon \in \mathcal J + i(0, \infty).
\end{equation*}
We consider the boundary data
\begin{equation}\label{eq:neumann_data_reprise}
    v_\omega := \mathcal N_\omega u_\omega
\end{equation}
where $\mathcal N_\omega$ is defined in~\eqref{eq:neumann_def}. In particular, observe that by the mapping property in~\ref{eq:neumann_def},
\begin{equation}\label{eq:v_apriori_reg}
    v_\omega \in H^{\infty, a'}_\bo \quad \text{for any} \quad -1 < a' < \min_{\kappa \in \mathcal K} \Re \mathfrak l_{\lambda, \kappa} - 1,
\end{equation}
but not uniformly so as $\omega$ approaches the real line. We now characterize the limiting behavior. 
\begin{lemma}\label{lem:bd_LAP}
    Let $\omega_j \to \lambda \in \mathcal J$, $\Im \omega_j \in (0, \epsilon_0)$ and $\Re \omega_j \in (\lambda - \epsilon_0, \lambda + \epsilon_0)$ for a sufficiently small $\epsilon_0$. Let $a \in (-1, 0)$ be such that 
    \begin{equation}
        -1 < a < \min_{\kappa \in \mathcal K} \Re(\mathfrak l_{\lambda, \kappa}) - 1.
    \end{equation}
   Then 
    \begin{equation}
        v_{\omega_j} \to v_{\lambda + i0} \quad \text{in} \quad H^{-\ha - \beta, a}_{\bo}(\partial \Omega; T^* \partial \Omega) \quad \text{as} \quad j \to \infty
    \end{equation}
    for all $\beta > 0$.
    
    Furthermore, we have the following regularity improvements. For all sufficiently small $\delta > 0$ and $\chi_\bo, \chi_\pm \in C^\infty(\partial \Omega;[0, 1])$ such that 
    \begin{equation*}
    \begin{gathered}
        \chi_\bo = 1 \quad \text{near} \quad \mathcal K,\qquad \supp \chi \subset \mathcal K + (-\delta, \delta), \\
        \chi_\pm = 1 \quad \text{on} \quad \Sigma^\pm_\lambda(\delta), \qquad \supp \chi_\pm \subset \Sigma^\pm_\lambda(2 \delta),
    \end{gathered}
    \end{equation*} 
    we have
    \begin{equation}\label{eq:traj_conorm}
    \begin{gathered}
        \Pi^\pm (1 - \chi_\mp - \chi_\bo) v_{\lambda + i0} \in \mathcal A^{[a + \ha]} (\partial \Omega; N^*_\pm(\mathscr O^\mp_\lambda \setminus \Sigma_\lambda^\mp(\delta))), \\
        \chi_\bo v_{\lambda + i0} \in H^{\infty, a}_{\bo}.
    \end{gathered}
    \end{equation}
    In fact, $v_{\lambda + i0}$ is the unique distribution such that~\eqref{eq:traj_conorm} holds and 
    \[\mathcal C_{\lambda + i0} v_{\lambda + i0} = (R_{\lambda + i0} f)|_{\partial \Omega}.\]
\end{lemma}

\begin{proof}
    1. First, recall from~\eqref{eq:Cv_reduction} that $v_{\omega_j}$ satisfies
    \begin{equation}\label{eq:dC_bdunique}
        d\mathcal C_{\omega_j} v_{\omega_j} = d(R_{\omega_j} f)|_{\partial \Omega}.
    \end{equation}
    The right hand side of~\eqref{eq:dC_bdunique} is bounded and converges in $H^{\infty, 0 -}_\bo(\partial \Omega; T^* \partial \Omega)$ as $j \to \infty$. Let $\mathcal Z_{\lambda}$ be the limiting indicial roots defined in~\eqref{eq:all_zeros}. For sufficiently small $\epsilon_0$, there exists $a < a' < \min_{\kappa \in \mathcal K} \Re(\mathfrak l_{\lambda, \kappa}) - 1$ such that
    \[\mathcal Z_{\lambda'} \cap \R - ia' = \emptyset \quad \text{for all} \quad \lambda' \in (\lambda - \epsilon_0, \lambda + \epsilon_0).\]
    Note that from~\eqref{eq:v_apriori_reg}, $v_{\omega_j} \in H^{\infty, a'}_\bo$ for all $j$, although not necessarily uniformly. 
    
    We first show $v_{\omega_j}$ is uniformly bounded in $H^{-\ha - \beta, a}_\bo(\partial \Omega;T^* \partial \Omega)$ for any $\beta > 0$, and we proceed by contradiction. We assume by way of contradiction that upon passing to a subsequence, $\|v_{\omega_j}\|_{H^{-\ha - \beta, a}_\bo} \to \infty$. Define the rescaled sequence 
    \begin{equation}\label{eq:vj_rescale}
        \tilde v_j = \frac{v_{\omega_j}}{\|v_{\omega_j}\|_{H^{-\ha - \beta, a}_\bo}}.
    \end{equation}
    Then it follows from \eqref{eq:dC_bdunique} that
    \begin{equation}\label{eq:dCv_rescale_limit}
        d \mathcal C_{\omega_j} \tilde v_j \to 0 \quad \text{in} \quad H^{N, a'}_\bo(\partial \Omega; T^* \partial \Omega),
    \end{equation}
    for all $N \gg 1$. Moreover, $\tilde v_j$ is bounded in $H^{-\ha -\beta, a}(\partial \Omega; T^* \partial \Omega)$. Since the embedding $H^{-\ha -\beta, a}_\bo(\partial \Omega; T^* \partial \Omega) \hookrightarrow H^{-N, a - \epsilon}_\bo(\partial \Omega;T^* \partial \Omega)$ is compact for all sufficiently large $N$, we may assume that $\tilde v_j \to \tilde v_0$ in $H^{-N, a - \epsilon}_\bo(\partial \Omega; T^* \partial \Omega)$ upon passing to a subsequence. By Lemma~\ref{lem:convergence_upgrade},
    \begin{equation}\label{eq:tilde_v_convergence}
    \begin{aligned}
        \tilde v_j \to \tilde v_0 \quad &\text{in} \quad H^{-\ha - \beta, a}_\bo(\partial \Omega; T^* \partial \Omega), \\
        \chi_\bo \tilde v_j \to \chi_\bo \tilde v_0  \quad &\text{in} \quad H^{\infty, a}_\bo (\partial \Omega; T^* \partial \Omega)
    \end{aligned}
    \end{equation}
   By Lemmas~\ref{lem:ddC_easy} and~\ref{lem:ddC_heaviside}, it follows that $d \mathcal C_{\omega_j} \tilde v_j \to d \mathcal C_{\lambda + i0} \tilde v_0$ in $\mathcal D'(\partial \Omega \setminus \mathcal K; T^*(\partial \Omega \setminus \mathcal K))$. Therefore by~\eqref{eq:dCv_rescale_limit} and~\eqref{eq:tilde_v_convergence}, we have 
    \[d\mathcal C_{\omega + i0} \tilde v_0 = 0.\]
    Finally, $\supp(R_{\omega_j} \mathcal I \tilde v_j) = \supp(\indic_\Omega u_\omega - f) \subset \overline \Omega$, so
    \[\supp (R_{\lambda + i0} \mathcal I \tilde v_0) \subset \overline \Omega.\]
    Therefore, all the hypothesis of Lemma~\ref{lem:uniqueness} are satisfied for $\tilde v_0$, and we conclude that $\tilde v_0 = 0$. This contradicts~\eqref{eq:tilde_v_convergence} and the fact that $\|\tilde v_j\|_{H^{-\ha - \beta, a}_\bo} = 1$ due to~\eqref{eq:vj_rescale}. Therefore, we conclude that $v_{\omega_j}$ is uniformly bounded in $H^{-\ha - \beta, a}_\bo(\partial \Omega; T^* \partial \Omega)$.

    \noindent
    2. In particular, Step 1 also implies that $\|v_{\omega_j}\|_{H^{-\ha - \beta', a'}_\bo}$ is bounded where $0 < \beta < \beta'$ and $a < a' < \min_{\kappa \in \mathcal K} \Re(\mathfrak l_{\lambda, \kappa}) - 1$. Using the compactness of the embedding $H^{-\ha - \beta', a'
    }_\bo(\partial \Omega) \hookrightarrow H^{-\ha - \beta, a
    }_\bo(\partial \Omega)$, we see that $v_{\omega_j}$ is precompact in $H^{-\ha - \beta, a}(\partial \Omega)$. Let $v_{\lambda + i0}, w_{\lambda + i0}$ be limit points of $\{v_{\omega_j}\}$ in $H^{-\ha - \beta, a}(\partial \Omega)$. It follows from Lemma~\ref{lem:convergence_upgrade} and~\eqref{eq:dC_bdunique} that $v_{\lambda + i0}$ and $w_{\lambda + i0}$ must satisfy~\eqref{eq:v0_co} and~\eqref{eq:v0_bo}. Therefore, it follows from Lemmas~\ref{lem:ddC_easy} and~\ref{lem:ddC_heaviside} that 
    \[\mathcal C_{\lambda + i0} v_{\lambda + i0} = \mathcal C_{\lambda + i0} w_{\lambda + i0} = (R_{\lambda + i0} f)|_{\partial \Omega}.\]
    Furthermore, $v_{\lambda + i0}$ and $w_{\lambda + i0}$ both satisfy the support condition 
    \[\supp(R_{\lambda + i0} \mathcal Iv_{\lambda + i0}), \ \supp(R_{\lambda + i0} \mathcal Iw_{\lambda + i0}) \subset \overline \Omega.\]
    Therefore, it follows from Lemma~\ref{lem:uniqueness} that $v_{\lambda + i0} = w_{\lambda + i0}$. Therefore, $v_{\omega_j} \to v_{\lambda + i0} \in H^{-\ha - \beta, a}(\partial \Omega)$. Furthermore, \eqref{eq:traj_conorm} follows since $v_{\lambda + i0}$ satisfies~\eqref{eq:v0_co} and~\eqref{eq:v0_bo}. 
\end{proof}

\subsection{Proof of Theorem \ref{thm:spectral}}
1. Let $f \in \CIc(\Omega)$ and $\lambda \in \mathcal J$. We consider the limit $(P - (\lambda + i\epsilon)^2)^{-1} f$ as $\epsilon \to 0+$, since the limit $(P - (\lambda - i \epsilon)^2)^{-1}$ can be obtained by complex conjugation. As such, let $\omega = \lambda + i \epsilon$ for $0 < \epsilon \ll 1$. Let $u_\omega \in H^1_0(\Omega)$ be the solutions to the elliptic boundary value problem~\eqref{eq:EBVP}. Recall that $P - \omega^2 = P(\omega) \Delta_\Omega^{-1}$, so 
\[(P - \omega^2)^{-1} f = \Delta u_\omega \in H^{-1}(\Omega).\]
Let $v_\omega := \mathcal N_\omega u_\omega$ denote the boundary data. In particular, by Lemma~\ref{lem:bd_LAP}, the limit $v_\omega \to v_{\lambda + i0}$ exists in $H^{-\ha - \beta, a}_\bo(\partial \Omega; T^* \partial \Omega)$ locally uniformly in $\lambda \in \mathcal J$. Recall from Lemma~\ref{lem:fundamental} that $u_\omega$ can be recovered from $v_\omega$ by
\[u_\omega = (R_\omega f)|_{\Omega} - S_\omega v_\omega.\]
Since $S_\omega v_\omega = (E_\omega * \mathcal I v_\omega)$, it follows from Lemma~\ref{lem:fs_holomorphic} that 
\begin{equation*}
    u_\omega \to u_{\lambda + i0} := (R_{\lambda + i0} f)|_\Omega - S_{\lambda + i0} v_{\lambda + i0} \quad \text{in} \quad \mathcal D'(\Omega),
\end{equation*}
which immediately implies that 
\begin{equation*}
    (P - \omega^2)^{-1} f \to (P - (\lambda + i0)^2)^{-1} f := \Delta u_{\lambda + i0} \quad \text{in} \quad \mathcal D'(\Omega).
\end{equation*}

\noindent
2. Now it remains to check the more refined regularity properties of $u_{\lambda + i0}$. Essentially, everything follows from the mapping properties of $S_{\lambda + i0}$ from Lemmas~\ref{lem:S_map_prop_1} and~\ref{lem:S_map_prop_2} together with the refined regularity of $v_{\lambda + i0}$ from Lemma~\ref{lem:bd_LAP}. First, observe that 
\[(R_{\lambda + i0} f)|_{\partial \Omega} \in \overline C^\infty(\partial \Omega),\]
so we indeed only need to consider $S_{\lambda + i0} v_{\lambda + i0}$. 

First, the wavefront set property~\eqref{eq:wavefront} follows immediately from~\eqref{eq:traj_conorm} and Proposition~\ref{lem:S_map_prop_2}. Next, fix $\varphi \in \overline C^\infty(\Omega)$ such that $\supp \varphi \cap \Gamma_\lambda = \emptyset$ where $\Gamma_\lambda$ is the periodic trajectory
\begin{equation*}
    \Gamma_\lambda := \bigcup_{x \in \Sigma^+_\lambda} (\Gamma^+_\lambda(x) \cup \Gamma^-_\lambda(x)). 
\end{equation*}
therefore, there exists $\chi_p \in \overline C^\infty(\partial \Omega)$ such that $\chi_p = 1$ in a small neighborhood $\Sigma_\lambda(\delta)$ for sufficiently small $\delta$, and
\[\supp \varphi \cap \bigcup_{x \in \supp \chi} (\Gamma^+_\lambda(x) \cup \Gamma^-_\lambda(x)) = \emptyset.\]
Clearly, 
\begin{equation}\label{eq:Schip_smooth}
    \varphi S_{\lambda + i0} \chi_p v_{\lambda + i0} \in \overline C^\infty(\Omega).
\end{equation}
On the other hand, it follows from Lemmas~\ref{lem:S_map_prop_1} that $S_{\lambda + i0}(1 - \chi_p) v_{\lambda + i0}$ is conormal to $\sr$ and $\ff$, and from Lemma~\ref{lem:S_map_prop_2}, $S_{\lambda + i0}(1 - \chi_p) v_{\lambda + i0}$ is also conormal to 
\begin{equation*}
    \bigcup_{x \in \mathscr O^+_\lambda \setminus \Sigma_\lambda(\delta)} N^*_- \Gamma^-_\lambda(x) \cup \bigcup_{x \in \mathscr O^-_\lambda \setminus \Sigma_\lambda(\delta)} N^*_+ \Gamma^+_\lambda(x).
\end{equation*}
Restricted to the support of $\varphi$, we simply have 
\begin{equation}\label{eq:Schip_bad}
    \varphi S_{\lambda + i0}(1 - \chi_p) v_{\lambda + i0} \in \mathcal A^{[a + \frac{3}{2}]}(\widetilde \Omega; \Lambda^-(\lambda), \sr, \ff)
\end{equation}
for any $a \in (-1, 0)$ such that 
\[a < \min_{\kappa \in \mathcal K} \Re \mathfrak l_{\lambda, k} - 1.\]
Combining~\eqref{eq:Schip_smooth} and~\eqref{eq:Schip_bad}, we find that 
\[\varphi u_{\lambda + i0} \in \mathcal A^{[a + \frac{3}{2}]}(\widetilde \Omega; \Lambda^-(\lambda), \sr, \ff).\]
Then~\eqref{eq:limit_conorm_reg} follows upon taking the Laplacian of $u_{\lambda + i0}$. 

Finally, observe that by Lemma~\ref{lem:S_map_prop_2}, we have
\[S_{\lambda + i0} \chi_p v_{\lambda + i0} \in \bar H^{\frac{1}{2} - }(\Omega),\]
and it follows from Lemma~\ref{lem:S_map_prop_1} that 
\[S_{\lambda + i0} \chi_p v_{\lambda + i0} \in \bar H^{a + \frac{3}{2}}(\Omega).\]
Therefore, $u_{\lambda + i0} \in \bar H^{-\frac{3}{2} -}(\Omega)$, from which~\eqref{eq:lim_regularity} follows.

\section{Evolution problem}\label{sec:evolution}
Finally, we use solve in the internal waves equation~\eqref{eq:internal_waves} using the functional calculus solution~\eqref{eq:functional_solution}. To do so, we need slightly more regularity for the spectral measure near $\lambda^2$. 

\subsection{Regularity of the spectral measure}
The spectral measure is in fact H\"older continuous near $\lambda^2$.

\begin{proposition}\label{prop:holder}
    Let $\mathcal J \Subset (0, 1)$ be such that $\Omega$ is Morse--Smale with respect to every $\lambda \in \mathcal J$. Let $v_{\lambda + i0}$ be as in Lemma~\ref{lem:bd_LAP}, and $a \in (-1, 0)$ be such that 
    \begin{equation}
        -1 < a < \min_{\kappa \in \mathcal K} \Re(\mathfrak l_{\lambda, \kappa}) - 1.
    \end{equation}
    Then
    \begin{equation}
        v_{\lambda + i0} \in C^\eta(\mathcal J; H^{-1/2 - \eta - \beta, a - \eta}_\bo(\partial \Omega))
    \end{equation}
    for any $\beta > 0$ and all sufficiently small $\eta > 0$. Furthermore, for any $\chi^\pm \in \overline C^\infty(\partial \Omega)$ such that $\chi^\pm = 1$ near $\Sigma^\pm_\lambda$, then 
    \begin{equation}
        \Pi^\pm(1 - \chi_\mp) v_{\lambda + i0} \in C^\eta(\mathcal J; H^{a + \ha - \eta, a - \eta}_\bo(\partial \Omega)).
    \end{equation}
\end{proposition}

\begin{proof}
    1. Recall from~\eqref{eq:Cv_reduction} that $v_{\omega_j}$ satisfies
    \begin{equation}
        d\mathcal C_{\omega} v_{\omega} = d(R_{\omega} f)|_{\partial \Omega} = : g_\omega.
    \end{equation}
    Let $\lambda, \lambda' \in \mathcal J$ be distinct and put
    \begin{equation*}
        \omega := \lambda + i\epsilon, \qquad \omega' := \lambda' + i\epsilon
    \end{equation*}
    for some $\epsilon > 0$.
    Observe that 
    \begin{gather}
        d\mathcal C_\omega \left( \frac{v_\omega - v_{\omega'}}{|\lambda - \lambda'|^\eta} \right) = \frac{g_\omega - g_{\omega'}}{|\lambda - \lambda'|^\eta} - \frac{d\mathcal C_\omega v_{\omega'} - d\mathcal C_{\omega'} v_{\omega'}}{|\lambda - \lambda'|^\eta}, \label{eq:Cv=G_diff}\\
        R_\omega \mathcal I \left(\frac{v_\omega - v_{\omega'}}{|\lambda - \lambda'|^\eta} \right) = \left(\frac{R_{\omega} - R_{\omega'}}{|\lambda - \lambda'|^\eta} \right)(f - \mathcal I v_{\omega'}) - \indic_\Omega \left(\frac{u_\omega - u_{\omega'}}{|\lambda - \lambda'|^\eta}\right). \label{eq:diff_support}
    \end{gather}
    
    \noindent
    2. We first characterize the regularity of the right hand side of~\eqref{eq:Cv=G_diff}. It follows from Lemma~\ref{lem:fs_holomorphic} that 
    \[\frac{g_\omega - g_{\omega'}}{|\lambda - \lambda'|^\alpha} \in H^{\infty, 0-}_\bo(\partial \Omega)\]
    uniformly for all $\lambda, \lambda' \in \mathcal J$ and $\epsilon > 0$. It remains to consider the last term in~\eqref{eq:Cv=G_diff}. Applying the mapping properties of $\partial_\omega d \mathcal C_\omega$ from Lemmas~\ref{lem:ddC_easy}, \ref{lem:ddC_heaviside}, and \ref{lem:ddC} to $v_\omega$, we find that
    \begin{enumerate}
        \item for $\chi_\bo \in C^\infty(\partial \Omega)$ such that $\chi_\bo = 1$ near $\mathcal K$ and supported in a sufficiently small neighborhood of $\mathcal K$, 
        \begin{equation}\label{eq:ddC_c}
            \chi_\bo (\partial_\lambda d \mathcal C_\omega)v_\omega \in H^{\infty, a}_{\bo}(\partial \Omega)
        \end{equation}
        uniformly for $\lambda \in \mathcal J$ and $\epsilon \ge 0$ sufficiently small, 
        
        \item for all $\chi_\pm \in C^\infty(\partial \Omega)$ such that $\chi_\pm = 1$ near $\Sigma_\lambda^\pm$ supported in a sufficiently small neighborhood of $\Sigma_\lambda^\pm$, 
        \begin{equation}\label{eq:ddC_away}
            \Pi^\pm(1 - \chi_\mp - \chi_\bo)(\partial_\lambda d \mathcal C_\omega) v_{\omega} \in H^{-\ha + a}(\partial \Omega)
        \end{equation}
        locally uniformly for $\lambda \in \mathcal J$ and $\epsilon \ge 0$ sufficiently small, 

        \item for all $\beta > 0$, 
        \begin{equation}\label{eq:ddC_periodic}
        \begin{gathered}
            \Pi^\pm \chi_{\mp} (\partial_\lambda d \mathcal C_\omega) v_\omega \in H^{-\frac{3}{2} - \beta} (\partial \Omega)\\
            \Pi^\pm \chi_\pm (\partial_\lambda d \mathcal C_\omega) v_\omega \in \overline C^\infty(\partial \Omega)
        \end{gathered}
        \end{equation}
        locally uniformly for $\lambda \in \mathcal J$ and $\epsilon \ge 0$
    \end{enumerate}
    for any for any $a \in (-1, 0)$ be such that $a < \min_{\kappa \in \mathcal K} \Re(\mathfrak l_{\lambda, \kappa}) - 1$. Note that we specify locally uniformly in (2) and (3) since the choice of $\chi_\pm$ can be made independent of $\lambda$ locally in $\lambda$. 
    
    From $(1)$, we see that for any $\eta \in (0, 1)$, 
    \begin{equation}
        \chi_\bo \left(\frac{d\mathcal C_\omega v_{\omega'} - d\mathcal C_{\omega'} v_{\omega'}}{|\lambda - \lambda'|^\eta} \right) \in H^{\infty, a}_\bo(\partial \Omega)
    \end{equation}
     uniformly for all $\lambda, \lambda' \in \mathcal J$ with $0 < |\lambda - \lambda'| \ll 1$ and all sufficiently small $\epsilon > 0$. Now let $\eta \in (0, 1 + a)$. Then by the Sobolev interpolation, we have
    \begin{multline}\label{eq:interpolated_away}
        \left\|\Pi^\pm(1 - \chi_\mp - \chi_\bo) \frac{d\mathcal C_\omega v_{\omega'} - d\mathcal C_{\omega'} v_{\omega'}}{|\lambda - \lambda'|^\eta} \right\|_{H^{\ha + a - \eta}} \\
        \le \|\Pi^\pm(1 - \chi_\mp - \chi_\bo)(d\mathcal C_\omega v_{\omega'} - d\mathcal C_{\omega'} v_{\omega'})\|_{H^{\ha + a}}^{1 - \eta} \\
        \cdot \frac{\|\Pi^\pm(1 - \chi_\mp - \chi_\bo)(d\mathcal C_\omega v_{\omega'} - d\mathcal C_{\omega'} v_{\omega'})\|_{H^{-\ha + a}}^{\eta}}{|\lambda - \lambda'|^\eta}.
    \end{multline}
    It follows from~\eqref{eq:ddC_away} that the right hand side of~\eqref{eq:interpolated_away} is uniformly bounded for all $\lambda, \lambda' \in \mathcal J$ with $0 < |\lambda - \lambda'| \ll 1$ and all sufficiently small $\epsilon > 0$. Similarly,
    \begin{multline}\label{eq:interpolated_periodic}
        \left\|\Pi^\pm\chi_\mp \frac{d\mathcal C_\omega v_{\omega'} - d\mathcal C_{\omega'} v_{\omega'}}{|\lambda - \lambda'|^\eta} \right\|_{H^{-\ha -\beta - \alpha}} \le \|\Pi^\pm \chi_\mp (d\mathcal C_\omega v_{\omega'} - d\mathcal C_{\omega'} v_{\omega'})\|_{H^{-\ha - \beta}}^{1 - \eta} \\
        \cdot \frac{\|\Pi^\pm \chi_\mp (d\mathcal C_\omega v_{\omega'} - d\mathcal C_{\omega'} v_{\omega'})\|_{H^{-\frac{3}{2} - \beta}}^{\eta}}{|\lambda - \lambda'|^\eta},
    \end{multline}
    and from~\eqref{eq:ddC_periodic}, we see that the right hand side is uniformly bounded for all $\lambda, \lambda' \in \mathcal J$ with $|\lambda - \lambda'| \ll 1$ and all sufficiently small $\epsilon > 0$. Also from~\eqref{eq:ddC_periodic}, it is clear that 
    \begin{equation}\label{eq:g_source}
        \Pi^\pm\chi_\pm \frac{d\mathcal C_\omega v_{\omega'} - d\mathcal C_{\omega'} v_{\omega'}}{|\lambda - \lambda'|^\eta} \in C^\infty(\partial \Omega) 
    \end{equation}
    uniformly for all $\lambda, \lambda' \in \mathcal J$ with $|\lambda - \lambda'| \ll 1$ and all sufficiently small $\epsilon > 0$.

    Finally, observe that it follows from from~\eqref{eq:traj_conorm} that if $x_0 \in \partial \Omega$ is a non-corner characteristic point, then for any $N \ge 0$, there exists a cutoff $\chi_0$ with $\chi_0 = 1$ near $x_0$ such that for any $N \ge 0$,
    \begin{equation}\label{eq:g_traj}
    \begin{gathered}
        \left(\sum_{j = 0}^N \chi_0 \circ b_\lambda^{-j} + \chi \circ b_\lambda^{-j} \circ \gamma_\lambda^- \right)  \frac{d\mathcal C_\omega v_{\omega'} - d\mathcal C_{\omega'} v_{\omega'}}{|\lambda - \lambda'|^\eta} \in C^\infty(\partial \Omega), \\
        \left(\sum_{j = 0}^N \chi_0 \circ b_\lambda^{j} + \chi \circ b_\lambda^{j} \circ \gamma_\lambda^+ \right)  \frac{d\mathcal C_\omega v_{\omega'} - d\mathcal C_{\omega'} v_{\omega'}}{|\lambda - \lambda'|^\eta} \in C^\infty(\partial \Omega)
    \end{gathered}
    \end{equation}
    again, locally uniformly for all $\lambda, \lambda' \in \mathcal J$ with $|\lambda - \lambda'| \ll 1$ and all sufficiently small $\epsilon > 0$. In particular, the summed cutoffs simply cuts off near a trajectory starting at $x_0$.

    \noindent
    3. We may assume that $\mathcal J$ is a sufficiently small interval so that all the estimates in Step 2 hold uniformly in the interval. Now we show that for $\lambda, \lambda' \in \mathcal J$ such that $0 < |\lambda - \lambda'| \ll 1$, we have the H\"older bound
    \begin{equation}\label{eq:v_holder_low_reg}
        \left\|\frac{v_{\lambda + i0} - v_{\lambda' + i0}}{|\lambda - \lambda'|^\eta} \right\|_{H^{-\ha - \eta - \beta, a - \eta}_\bo} + \sum_{\pm} \left\|\Pi^\pm(1 - \chi_\mp)\frac{v_{\lambda + i0} - v_{\lambda' + i0}}{|\lambda - \lambda'|^\eta} \right\|_{H^{a + \ha - \eta, a - \eta}_{\bo}} < C.
    \end{equation}
    For the sake of contradiction, assume that there exists sequences $\omega_j = \lambda_j + i \epsilon_j$ and $\omega'_j = \lambda'_j + i \epsilon_j$ such that $0 < |\lambda_j - \lambda_j'| \ll 1$, $\epsilon_j > 0$, $\epsilon_j \to 0$, and $\lambda_j, \lambda'_j \to \lambda \in \mathcal J$, so that
    \begin{equation}
    \begin{gathered}
        \|w_j\|_{H^{-\ha - \eta - \beta, a - \eta}_\bo} + \sum_{\pm} \|\Pi^\pm(1 - \chi_\mp) w_j \|_{H^{a + \ha - \eta, a - \eta}_\bo} \to \infty \quad \text{as} \quad j \to \infty\\
        \text{where} \quad w_j:= \frac{v_{\omega_j} - v_{\omega'_j}}{|\lambda_j - \lambda'_j|^\eta}.
    \end{gathered}
    \end{equation}
    From~\eqref{eq:Cv=G_diff}, $w_j$ satisfies
    \[d\mathcal C_\omega w_j = f_j, \quad \text{where} \quad f_j := \frac{g_{\omega_j} - g_{\omega'_j}}{|\lambda_j - \lambda'_j|^\alpha} - \frac{d \mathcal C_{\omega_j}v_{\omega'_j} - d\mathcal C_{\omega'_j}v_{\omega'_j}}{|\lambda_j - \lambda'_j|^\alpha}.\]
    Rescale by defining 
    \begin{equation}\label{eq:w_rescale}
    \begin{gathered}
        \widetilde w_j := \frac{w_j}{\|w_j\|_{H^{-\ha - \eta - \beta, a - \eta}_\bo} + \sum_{\pm} \|\Pi^\pm(1 - \chi_\mp) w_j \|_{H^{a + \ha - \eta, a - \eta}_\bo}},\\
        \tilde f_j :=  \frac{f_j}{\|w_j\|_{H^{-\ha - \eta - \beta, a - \eta}_\bo} + \sum_{\pm} \|\Pi^\pm(1 - \chi_\mp) w_j \|_{H^{a + \ha - \eta, a - \eta}_\bo}}.
    \end{gathered}
    \end{equation}
    Let $a < a' < \min_{\kappa \in \mathcal K} \Re(\mathfrak l_{\lambda, k}) - 1$. By the three estimates~\eqref{eq:ddC_c}-\eqref{eq:ddC_periodic} in Step 1, we see that
    \[\sum_{\pm} \big(\|\Pi^\pm (1 - \chi_\mp) \tilde f_j\|_{H^{\ha + a' - \eta, a' - \eta}_\bo} + \|\Pi^\pm \chi^\mp \tilde f_j\|_{H^{-\ha - \eta - \beta}} \big) \to 0\]
    as $j \to \infty$. Observe that $\widetilde w_j \in H^{\infty, a'}_\bo$, albeit not uniformly. This a priori regularity allows us to apply~\eqref{eq:low_reg_semifredholm} from Proposition~\ref{prop:global_semifredholm} to find that for any $\beta > 0$, 
    \begin{equation}
        \|\widetilde w_j\|_{H^{-\ha - \eta - \beta, a' - \eta}_\bo} \le C.
    \end{equation}
    By the compactness of the embedding $H^{-\ha - \eta - \beta
    ', a'- \eta}_\bo \hookrightarrow H^{- \ha - \eta - \beta, a' - \eta}_\bo$ for $\beta' < \beta$, we conclude that upon passing to a subsequence, there exists $\widetilde w_0 \in H^{-\ha - \eta - \beta, a - \eta}_\bo$ such that 
    \begin{equation}
        \widetilde w_j \to \widetilde w_0 \quad \text{in} \quad H^{-\ha - \eta - \beta, a}_\bo. 
    \end{equation}
    A similar argument using~\eqref{eq:high_reg_semifredholm} instead of~\eqref{eq:low_reg_semifredholm} from Proposition~\ref{prop:global_semifredholm} further shows that for any $\chi_\pm \in \overline C^\infty(\partial \Omega)$ supported in a small neighborhood of $\Sigma^\pm_\lambda$, 
    \begin{equation}\label{eq:g_high_reg}
        \Pi^\pm(1 - \chi_{\mp}) \widetilde w_j \to \Pi^\pm(1 - \chi_{\mp}) \widetilde w_0 \quad \text{in} \quad H^{a + \ha - \eta, a - \eta}_\bo.
    \end{equation}
    upon further passing to a subsequence. Note that by~\eqref{eq:w_rescale},
    \begin{equation}\label{eq:norm1}
        \|\widetilde w_0\|_{H^{-\ha - \eta - \beta, a - \eta}_\bo} + \sum_{\pm} \|\Pi^\pm(1 - \chi_\mp) \widetilde w_0 \|_{H^{a + \ha - \eta, a - \eta}_\bo} = 1
    \end{equation}
    Furthermore, it follows from Lemmas~\ref{lem:ddC_easy} and~\ref{lem:ddC_heaviside} that
    \begin{equation}\label{eq:Cw=0}
        \mathcal C_{\lambda + i0} \widetilde w_0 = 0.
    \end{equation}

    \noindent
    4. Finally, we claim that $\widetilde w_0$ is smooth near the non-corner characteristic points. Let $x_0 \in \partial \Omega$ be a non-corner characteristic point. Since $\Omega$ is Morse--Smale with respect to $\lambda \in \mathcal J$, it follows that $b_\lambda^k(x_0) \notin \mathcal K$ for all $k \in \Z$. Due to~\eqref{eq:g_source} and~\eqref{eq:g_traj}, Step 2 in the proof of Proposition~\ref{prop:global_semifredholm} still applies, and we find that for every $s \gg 1$, there exists $C_s > 0$ such that
    \begin{equation}
        \|\Pi^\pm \chi_0 \widetilde w_j\|_{H^s} \le C_s.
    \end{equation}
    Together with~\eqref{eq:Cw=0} and~\eqref{eq:g_high_reg}, we conclude from Lemma~\ref{lem:uniqueness} that $\widetilde w_0 = 0$. This contradicts~\eqref{eq:norm1}. 
\end{proof}
\begin{Remark}
    An interesting question that the methods in this paper cannot address is whether or not $v_{\lambda + i0}$ has more than just H\"older regularity in $\lambda$. In particular, if $\partial \Omega$ is smooth, then higher regularity of $v_{\lambda + i0}$ in $\lambda$ is proved in \cite[Proposition 7.4]{DWZ}. However, in our corner setting, the corner necessarily creates a singularity that propagates along a trajectory coming out of the corner. The problem this causes is that $\partial_\omega v_\omega$ does not lie in $H^{\ha +}$ away from the periodic points, so we would be unable to apply the uniqueness in the limit via Lemma~\ref{lem:uniqueness} if we were to differentiate \eqref{eq:Cv_reduction} instead of taking the H\"older quotient as in~\eqref{eq:Cv=G_diff}. 
\end{Remark}

The H\"older regularity of $v_{\lambda \pm i0}$ immediately implies H\"older regularity of the spectral measure for $P$. Denote the spectral measure of $P$ by $dE$. 
\begin{lemma}\label{lem:SM_holder}
    Let $f \in \CIc(\Omega)$ and assume that $\Omega$ is Morse--Smale with respect to $\lambda$. Then there exists $\delta > 0$ and $\nu \in C^\eta((\lambda^2 - \delta, \lambda^2 + \delta); H^{-\frac{3}{2} - \beta - \eta}(\Omega))$ for any $\beta > 0$ such that 
    \begin{equation}
        \left(\int_\R \varphi(z) \, dE_z \right) f = \int_\R \varphi(z) \nu(z)\, dz
    \end{equation}
    for every $\varphi \in \CIc((\lambda^2 - \delta, \lambda^2 + \delta))$. 
\end{lemma}
\begin{proof}
    By Theorem~\ref{thm:spectral}, the spectral measure is absolutely continuous in a neighborhood of $\lambda^2$. By Stone's formula (see \cite[Theorem B.10]{DZ_resonances}), we see that for $z$ near $\lambda^2$, 
    \begin{align}
        2 \pi i \nu(z) &= (P - z - i0)^{-1} f - (P - z + i0)^{-1} f \nonumber \\
        &= \Delta u^+(\sqrt{z}) - \Delta u^-(\sqrt{z}) \label{eq:nu_stones}
    \end{align}
    where
    \[u^\pm(\omega) := \Delta_\Omega^{-1} (P - \omega^{2} \mp i0)^{-1} = R_{\omega + i0} f - S_{\omega \pm i0} v_{\omega \pm i0}, \quad \omega\in (0, 1) \quad \text{near} \quad \lambda.\]
    Here, $v_{\omega \pm i0}$ are as in Lemma~\ref{lem:bd_LAP}. Clearly, $R_{\omega + i0} f \in C^\infty(\Omega)$ and depends smoothly on $\omega$. Using the H\"older regularity of $v_{\omega \pm i0}$ from Proposition~\ref{prop:holder} and the mapping properties of $S_{\omega \pm i0}$ from Lemmas~\ref{lem:S_map_prop_1} and~\ref{lem:S_map_prop_2}, we see that 
    \[S_{\omega \pm i0} v_{\omega \pm i0} \in C^\eta(\mathcal J; \bar H^{\ha - \eta - \beta})\]
    where $\mathcal J \Subset (0, 1)$ is some small interval containing $\lambda$. Therefore, there exists $\delta > 0$ such that 
    \[\Delta u^+(\sqrt{z}) \in C^\eta((\lambda^2 - \delta, \lambda^2 + \delta)_z; \bar H^{- \frac{3}{2} - \eta - \beta}).\]
    The lemma then follows from~\eqref{eq:nu_stones}. 
\end{proof}

Let $\varphi \in \CIc(\mathcal J; [0, 1])$ be such that $\supp \varphi \subset [\lambda^2 - \delta, \lambda^2 + \delta]$ and $\varphi = 1$ on $[\lambda^2 - \delta/2, \lambda^2 + \delta/2]$ for some sufficiently small $\delta > 0$ to be chosen later. Further assume that $\varphi(z + \lambda^2)$ is an even function of $z$. Recall from~\eqref{eq:functional_solution} that the solution to the internal waves equation~\eqref{eq:internal_waves} is given by 
\begin{equation}\label{eq:sol_decomp}
    u(t) = \Delta_\Omega^{-1} \Re\big( e^{i \lambda t}(w_1(t) + r_1(t)) \big)
\end{equation}
where
\begin{equation}\label{eq:wr_def}
    w_1(t) := \varphi(P) \mathbf W_{t, \lambda}(P) f, \quad r_1(t) := (I - \varphi(P)) \mathbf W_{t, \lambda}(P) f. 
\end{equation}
We first show that the contribution from $r_1$ to $u$ is uniformly bounded in $H^1_0(\Omega)$ for all time. 
\begin{lemma}\label{lem:spectral_bound}
    For all $t \ge 0$, we have 
    \begin{equation}
        \|\Re (e^{i \lambda t} r_1(t))\|_{H^{-1}} \le \frac{4}{\delta} \|f\|_{H^{-1}}.
    \end{equation}
\end{lemma}
\begin{proof}
    This follows directly from the functional calculus for $P$. In particular, we see that 
    \[\left|\Re\big( e^{i \lambda t} (1 - \varphi(z)) W_{t, \lambda}(z) \big) \right| = \left|\frac{(\cos(t \sqrt{z}) - \cos(t\lambda))(1 - \varphi(z))}{\lambda^2 - z} \right|\le \frac{4}{\delta}.\]
    Since $P$ is bounded and self-adjoint on $H^{-1}(\Omega)$, the lemma follows. 
\end{proof}

Now we decompose $w_1$ roughly following the strategy in \cite[Lemma 4.1]{Verdiere_Raymond_20}.
\begin{lemma}\label{lem:w_1_decomp}
Let $w_1$ be as defined in~\eqref{eq:wr_def}. Then we have the decomposition
\[w_1 = w_\infty + e(t)\]
where $\|e(t)\|_{\bar H^{-\frac{3}{2} - \beta}} \to 0$ for any $\beta > 0$ and
\[w_\infty := \left(\int \varphi(\sqrt z) (z - \lambda^2 + i0)^{-1} \, dE_z \right)f \]
\end{lemma}
\begin{proof}
    Fix $\beta > 0$. By Proposition~\ref{lem:SM_holder}, for any $0 < \alpha < \beta$, there exists $\nu \in C^\alpha((\lambda^2 - 2\delta, \lambda^2 + 2\delta); \bar H^{-\frac{3}{2} - \beta}(\Omega))$ such that
    \begin{equation}
        w_1(t) = \int_\R \varphi(\sqrt z) \mathbf W_{t, \lambda}(z) \nu(z)\, dz
    \end{equation}
    Since $\nu$ is H\"older continuous, there exists $\mu \in L^1((\lambda^2 - 2\delta, \lambda^2 + 2\delta); \bar H^{-\frac{3}{2} - \beta}(\Omega))$ such that 
    \[\nu(z) = \nu(\lambda^2) + (z - \lambda^2) \mu(z)\]
    Observe that 
    \[\mathbf W_{t, \lambda}(z) \to (z - \lambda^2 + i0) \quad \text{in $\mathcal D'(\R)$} \quad \text{as}\quad t \to \infty.\]
    Therefore, if $\varphi$ is chosen so that $\varphi(z + \lambda^2)$ is an even function of $z$, we have
    \begin{equation}\label{eq:W_converge_easy}
        \int_\R \varphi(z) \mathbf W_{t, \lambda} (z) \nu(\lambda^2) \, dz \to -i \pi \nu(\lambda^2) \quad \text{in} \quad \bar H^{-\frac{3}{2} - \beta}(\Omega).
    \end{equation}
    Next, we have
    \begin{equation}\label{eq:W_converge_hard}
    \begin{aligned}
        \int_{\R} \varphi(z) \mathbf W_{t, \lambda} (z) (z - \lambda^2) \mu(z) \, dz &= \sum_{\pm} \int_{\R} \varphi(z) \frac{1 - e^{it(\lambda \pm \sqrt{z})}}{2 \sqrt{z} (\sqrt{z} \pm \lambda)} (z - \lambda^2) \mu(z) \, dz \\
        &= \sum_{\pm} \int_{\R} \varphi(\omega^2) (1 - e^{i t(\lambda \pm \omega)})(\omega \mp \lambda) \mu(\omega^2)\, d\omega \\
        &= \int_{\R} \varphi(z) \mu(z)\, dz \\
        &\qquad- \sum_{\pm} \int_{\R} e^{it(\lambda \pm \omega)} \varphi(\omega^2)(\omega \mp \lambda) \mu(\omega^2)\, d \omega.
    \end{aligned}
    \end{equation}
    Since $\mu \in L^1$, it follows from the Riemann--Lebesgue lemma that 
    \begin{equation*}
        \sum_{\pm} \int_{\R} e^{it(\lambda \pm \omega)} \varphi(\omega^2)(\omega \mp \lambda) \mu(\omega^2)\, d \omega \to 0 \quad \text{in} \quad \bar H^{-\frac{3}{2} - \beta}(\Omega).
    \end{equation*}
    Moreover, since $\mu(z) = \frac{\nu(z) - \nu(\lambda^2)}{z - \lambda^2}$, we see that 
    \[\int_{\R} \varphi(z) \mu(z)\, dz = \int \varphi(z) \mathrm{p.v.} \frac{1}{z - \lambda^2} \nu(z)\, dz.\]
    Combining~\eqref{eq:W_converge_easy} and \eqref{eq:W_converge_hard}, we then see that there exists $e(t)$ such that $\|e(t)\|_{\bar H^{-\frac{3}{2} - \beta}} \to 0$ as $t \to \infty$ so that 
    \begin{equation}
        \int_\R \varphi(z) \mathbf W_{t, \lambda}(z) \nu(z)\, dz = e(t) + \int_\R \varphi(z) \left(\mathrm{p.v.} \frac{1}{z - \lambda^2} - i \pi \delta_{0}(z - \lambda^2) \right)\nu(z) \, dz.
    \end{equation}
    The lemma then follows since $\mathrm{p.v.} \frac{1}{z - \lambda^2} - i \pi \delta_{0}(z - \lambda^2) = (z - \lambda^2 + i0)^{-1}$. 
\end{proof}

\subsection{Proof of Theorem \ref{thm:evolution}} The result now follows simply by rearranging Lemma~\ref{lem:w_1_decomp} a bit. With $w_\infty$ as defined in Lemma~\ref{lem:w_1_decomp}, it follows from the functional calculus that 
\begin{equation}
    w_\infty - (P - \lambda^2 + i0)^{-1} f \in H^{-1}(\Omega). 
\end{equation}
Combining with Lemma~\ref{lem:spectral_bound}, we see from~\eqref{eq:sol_decomp} that there exists $r(t) \in H^1_0(\Omega)$ uniformly for all $t$ such that the solution to the internal waves equation can be decomposed as
\[u(t) = \Delta_\Omega^{-1} \Re( e^{i \lambda t} (P - \lambda^2 + i0)^{-1} f) + \Delta_\Omega^{-1} e(t) + r(t) \]
where $e(t)$ is the error term in Lemma~\ref{lem:w_1_decomp}. This gives the desired decomposition of Theorem~\ref{thm:evolution}.

\section*{Acknowledgements}
The author would like to thank Semyon Dyatlov, Peter Hintz, Leo Maas, Richard Melrose, and Maciej Zworski for the many insightful discussions throughout the course of this project. The author is partially supported by Semyon Dyatlov's NSF CAREER grant DMS-1749858.

\bibliographystyle{alpha}
\bibliography{b-internal.bib}

\end{document}